\documentclass[11pt, a4paper,reqno]{amsart}
\pdfoutput=1

\usepackage{textcomp,tikz,genyoungtabtikz,enumitem,rotating,latexsym,bm,stmaryrd}
\usetikzlibrary{positioning,intersections}
\usepackage{amsmath,amsthm,amsfonts,amssymb,mathrsfs,pb-diagram}
\usepackage[pagebackref=false,bookmarks=true,colorlinks=true,linktoc=page,citecolor=darkgreen,linkcolor=blue,urlcolor=mediumblue]{hyperref}
\usepackage{caption}
\usepackage{lipsum,wasysym}
\usepackage{mathtools}
\usepackage[a4paper,margin=0.9in]{geometry}
\usepackage{cleveref}
\usepackage{tikz,enumitem,rotating,latexsym,bm,stmaryrd}
\usetikzlibrary{positioning,intersections,shapes.geometric,calc,decorations.pathreplacing}
 \usepackage{upgreek}
\usepackage{color}
\definecolor{mediumblue}{rgb}{0.0, 0.0, 0.8}
\colorlet{darkgreen}{green!50!black} 
\newcommand{\Balgebra}{{\mathbb B}}

\newcommand\con[2]{\mathscr{C}^{#1}_{#2}}
\newcommand\mptn[2]{\mathscr{P}^{#1}_{#2}}
\renewcommand{\geq}{\geqslant}

\newcommand{\Strand}{S}

\renewcommand{\leq}{\leqslant}
\renewcommand{\trianglerighteq}{\trianglerighteqslant}
\renewcommand{\trianglelefteq}{\trianglelefteqslant}

\tikzset{wei/.style=
{red,double=red,double
distance=1pt}}

\usepackage{ltxtable}
\usepackage{tabularx}

\tikzset{wei2/.style={red,double=red,double
distance=1pt}}

\allowdisplaybreaks
\numberwithin{equation}{section}
\usepackage{scalefnt}

\newtheorem{thm}{Theorem}[section]
\newtheorem{cor}[thm]{Corollary}

\newtheorem*{mainA}{Theorem A}
\newtheorem*{mainB}{Theorem B}
\newtheorem*{mainC}{Theorem C}
\newtheorem*{mainD}{Theorem D}

\newtheorem{prop}[thm]{Proposition}
\newtheorem*{prop*}{Proposition}
\newtheorem*{thm*}{Theorem}
\newtheorem*{cor*}{Corollary}

\newtheorem*{conj*}{Conjecture}

\theoremstyle{remark}

\newtheorem*{Acknowledgements*}{Acknowledgements}

\theoremstyle{definition}
\newtheorem{defn}[thm]{Definition}
\newtheorem{eg}[thm]{Example}

\newtheorem{rmk}[thm]{Remark}

\newcommand{\Diag}{A}

\newcommand{\boxla}{\lambda}
\newcommand{\degr}{\mathrm{deg}}

\newcommand{\rad}{\mathrm{rad}}
\newcommand{\res}{{\sf res}}
\newcommand{\Res}{\mathrm{Res}}

\newcommand{\Std}{{\rm Std}}
\newcommand{\SStd}{{\rm SStd}}
\newcommand{\QQ}{{\mathbb Q}}

\newcommand{\TSStd}{\operatorname{\mathcal{T}}}
\newcommand{\Shape}{\operatorname{Shape}}

\newcommand{\la}{\lambda}

\newcommand{\Cell}{A}
\newcommand{\cell}{A}
\newcommand{\acell}{a}

\newcommand{\B}{{\bf B}}
\newcommand{\N}{{\bf M}}

 \newcommand{\SSTS}{\mathsf{S}}

\newcommand{\SSTT}{\mathsf{T}}  
\newcommand{\SSTU}{\mathsf{U}}  
\newcommand{\SSTV}{\mathsf{V}}  
 \newcommand{\SSTQ}{\mathsf{Q}}  
  \newcommand{\SSTR}{\mathsf{R}}  
\newcommand{\str}{\mathsf{r}}  
\newcommand{\sts}{\mathsf{s}}  
\newcommand{\stt}{\mathsf{t}}  
\newcommand{\stu}{\mathsf{u}}  

\newcommand{\weight}{\NN_{>1}\times\ZZ^\ell}  

\newcommand{\stv}{\mathsf{v}}
\newcommand{\stw}{\mathsf{w}}  
\newcommand{\ZZ}{{\mathbb Z}}
\newcommand{\NN}{{\mathbb N}}
\newcommand{\g}{1}
\newcommand{\algebra}{{\mathbb A}}

\newcommand{\bexy}{\stt}
\newcommand{\CC}{\Bbbk}
\newcommand{\C}{{\mathbb Q}}
\newcommand{\RR}{{\mathbb R}}
\newcommand{\R}{{R}}
\newcommand{\Simple}{{\rm D}}

\let\<=\langle
\let\>=\rangle

\tikzset{
    ultra thin/.style= {line width=0.05pt},
    very thin/.style=  {line width=0.2pt},
    thin/.style=       {line width=0.1pt},
    semithick/.style=  {line width=0.6pt},
    thick/.style=      {line width=0.8pt},
    very thick/.style= {line width=1.2pt},
    ultra thick/.style={line width=1.6pt}
}

\crefname{defn}{Definition}{Definitions}
\crefname{thm}{Theorem}{Theorems}
\crefname{prop}{Proposition}{Propositions}
\crefname{lem}{Lemma}{Lemmas}
\crefname{cor}{Corollary}{Corollaries}
\crefname{conj}{Conjecture}{Conjectures}
\crefname{section}{Section}{Sections}
\crefname{subsection}{Subsection}{Subsections}
\crefname{eg}{Example}{Examples}
\crefname{figure}{Figure}{Figures}
\crefname{rem}{Remark}{Remarks}
\crefname{rmk}{Remark}{Remarks}
\crefname{equation}{equation}{equation}

\Crefname{defn}{Definition}{Definitions}
\Crefname{thm}{Theorem}{Theorems}
\Crefname{prop}{Proposition}{Propositions}
\Crefname{lem}{Lemma}{Lemmas}
\Crefname{cor}{Corollary}{Corollaries}
\Crefname{conj}{Conjecture}{Conjectures}
\Crefname{section}{Section}{Sections}
\Crefname{subsection}{Subsection}{Subsections}
\Crefname{eg}{Example}{Examples}
\Crefname{figure}{Figure}{Figures}
\Crefname{rem}{Remark}{Remarks}
\Crefname{rmk}{Remark}{Remarks}

 \newlength{\mylen}
\setbox1=\hbox{$\bullet$}\setbox2=\hbox{\tiny$\bullet$}
\newcommand{\Specht}{{\rm S}}

\usepackage{amsmath}
\newcommand\Item[1][]{%
  \ifx\relax#1\relax  \item \else \item[#1] \fi
  \abovedisplayskip=0pt\abovedisplayshortskip=0pt~\vspace*{-\baselineskip}}
  
\newcommand{\ct}{{\sf ct}}
  \usetikzlibrary{decorations.pathmorphing}
\usetikzlibrary{snakes}
\tikzset{snake it/.style={decorate, decoration=snake}}
\tikzset{zigzag/.style={decorate, decoration=zigzag}}

\def\Item{\item\abovedisplayskip=0pt\abovedisplayshortskip=5pt~\vspace*{-\baselineskip}} 

\begin{document}

\title[The many integral graded  cellular bases  of cyclotomic Hecke algebras] 
{The many integral graded  cellular bases \\   of   Hecke algebras of complex reflection groups }

\newcommand{\Rset}{{^\imath_\mu\!\mathcal{R}_\la^\jmath}}

\author[C. Bowman]{C. Bowman}
\email{chris.bowman-scargill@york.ac.uk}
\address{Department of Mathematics,
University of York, Heslington, York, YO10 5DD, UK }
 
 \renewcommand{\theta}{{{\sigma}}}
 \renewenvironment{abstract}
 {\small
  \begin{center}
  \vspace{-.5em}\vspace{0pt}
  \end{center}
  \list{}{%
    \setlength{\leftmargin}{5mm}
    \setlength{\rightmargin}{\leftmargin}%
  }%
  \item\relax}
 {\endlist}
 
\maketitle 

\!\!
 \begin{abstract}
  We settle several long-standing problems in the theory of  cyclotomic Hecke algebras: 
  for each charge we   construct the integral 
    cellular basis   predicted by   Ariki's categorification  theorem.  We hence
  prove unitriangularity of decomposition matrices and   Martin--Woodcock's conjecture.     \end{abstract}
  
 \!\!
  \section*{Introduction} 
  There are two remarkably successful approaches to the study of    Hecke algebras of symmetric groups: the first is   via geometry and the second is via  categorical Lie theory.  
   The Kazhdan--Lusztig basis  
has   deep  {\em geometric} origins  (arising as the shadow  of  an intersection cohomology sheaf  on a variety);  
   this basis enjoys many positivity properties,  however it is inhomogenous with respect to the  Hecke algebra's graded structure.   
   The graded Murphy basis arises in {\em categorical Lie theory},   it encodes the graded induction and  restriction   along the tower of Hecke algebras, and it is  simpler and more explicit.  
The most important property shared by the Kazhdan--Lusztig and graded Murphy bases is that they are
 both  integral cellular bases     \cite{MR560412,hm10}.

The  complex reflection groups  were classified into the infinite series   $G(\ell,d,n)$ and 34 exceptional cases by Shephard--Todd \cite{st54}; their corresponding  Hecke algebras were later defined by Ariki and Koike \cite{MR1279219,MR1356366}, for the infinite families,  and 
Brou\'e--Malle--Rouquier \cite{MR1637497}, in general.  
 For every {\em real reflection group},   Lusztig
 has constructed       {\em many   different}  Kazhdan--Lusztig bases for the associated Hecke algebras  \cite{MR727851,MR1974442}.      However, this is as far as the geometric picture 
(and the underlying Kazhdan--Lusztig bases!) can be pushed: {\em 
 there do not exist Kazhdan--Lusztig bases for complex reflection groups or their Hecke algebras.}

  Categorical Lie theory picks up where geometry leaves off (one of the most spectacular  examples to-date being \cite{MR3245013}).  
In particular, while complex reflection groups {\em do not} possess Kazhdan--Lusztig bases, Ariki's  categorification theorem  
 suggests that every choice of  charge  
  should give rise to a 
 corresponding  cellular structure on the Hecke algebra of type $G(\ell,1,n)$   \cite{MR1911030}.  
  We prove that every charge does indeed give  rise to an integral  cellular basis on the Hecke algebra of type $G(\ell,1,n)$, as has long been hoped and expected. 
 Namely we generalise the graded Murphy bases    from asymptotic charges \cite{hm10} to {\em all possible charges} on all   Hecke algebras of type $G(\ell,1,n)$.   
(Corresponding  bases for type $G(\ell,d,n)$  can be constructed from ours via Clifford theory \cite{cyclosub}.)

In order to state our main result, we first require some notation. 
For the purposes of the introduction, we let $\Bbbk$  be a  field.  
Given $\sigma=(e;\sigma_0,\sigma_1,\dots,\sigma_{\ell-1})\in \NN_{>1}\times \ZZ^\ell$, 
we define the 
 cyclotomic Hecke algebra 
   to be  the $\Bbbk$-algebra generated by $T_0, T_1, \dots T_{n-1}$ subject to the relations 
\begin{align*} 
(T_i+q)(T_i-1)=0\quad (T_0-q^{\sigma_0}) (T_0-q^{\sigma_1})\dots  (T_0-q^{\sigma_{\ell-1}})=0
\\
  T_iT_j=T_jT_i
\quad
\quad T_i T_{i+1}T_i=  T_{i+1}T_i T_{i+1}
\quad T_0 T_1T_0T_1=   T_1T_0T_1T_0
 \end{align*}
 for $q$ an $e$th root of unity and  $1\leq i , j <n$, $|i-j|>1$. 
The starting point for this paper is the observation that  this presentation  depends only on the reduction of
 $\sigma $ modulo $e$ (which we denote by $   \underline{s}   \in \NN_{>1}\times (\ZZ/e\ZZ)^\ell$). 
We   denote the cyclotomic Hecke algebra by  $ H ^{\Bbbk}_n( \underline{s})$ in order to emphasise the independence   of the actual charge.    
For each distinct integral lift $\sigma \in \NN_{>1}\times \ZZ^\ell$ of $ \underline{s} \in \NN_{>1}\times (\ZZ/e\ZZ)^\ell$, we have a corresponding 
 ${\bf a}_{\sigma}$-order  on $\mptn \ell n$ (the $\ell$-multipartitions of $n$)  due to Lusztig and   an ${\bf a}_{\sigma}$-grading on standard tableaux due to  Uglov.

\begin{mainA}
 The algebra   $ H ^{\Bbbk}_n( \underline{s})$  has many   graded    cellular    structures, one for each integral lift $\sigma \in \NN_{>1}\times \ZZ^\ell$.  The basis 
$$
  \{ \cell^{\sigma}_{\sts  \stt}  \mid
    \lambda \in \mptn \ell n,   \sts,\stt \in \Std_{\sigma}(\lambda)\}$$
 is cellular with respect to    Lusztig's ${\bf a}_{\sigma}$-order on $\mptn \ell n$ and Uglov's ${\bf a}_{\sigma}$-grading on standard tableaux.      
    \end{mainA}

The algebra $ H ^{\Bbbk}_n( \underline{s})$ has many (non-isomorphic) quasi-hereditary covers (one for each integral lift 
$\sigma \in \NN_{>1}\times \ZZ^\ell$) and we shall see  (in \cref{cellularitybreedscontempt2}) that  
each of the bases of  Theorem A arises by idempotent truncation of a  
 cellular basis of the corresponding quasi-hereditary cover.  (In exactly the  same manner that Murphy's basis of $\Bbbk\mathfrak{S}_n$ is obtained from Green's co-determinant basis of the Schur algebra.)  
 
 Theorem A allows  us prove that the decomposition matrices are unitriangular with respect to all  Lusztig ${\bf a}_{\sigma}$-orderings on $\mptn \ell n$ over any field  and explicitly construct  the irreducible modules parameterised by  Ariki's categorification  theorem.  
This completes a long history of work  on this topic 
    \cite{MR1864465,
    MR1443748,
    MR1750939,
    MR2017068,MR2761947,
    MR2510041,MR2847511,MR2886292,MR3465894,MR2905554,MR1317512,DJM,
    MR233664322,MR2336643,MR3465894,MR2905554,MR2799052,
MR2496742,MR1889346,MR2271574,MR2118039,MR2189868,MR2350227,MR2802531}.

\begin{mainB} 
Let $\Bbbk$ be a field. For each integral lift ${\sigma}\in \weight$, 
 the  irreducible  $  H _n^\Bbbk(\underline{s})$-modules are
 explicitly constructible as    canonical quotients of the cell modules   $\Specht^\Bbbk_{\sigma}(\la)$ 
 labelled by  the associated  set of Uglov $\ell$-partitions $\Sigma^\ell_n$  (see \cref{isaididexplain} for definition) 
 and the  decomposition matrix   is uni-triangular with respect to     
Lusztig's ${\bf a}_{\sigma}$-ordering   on $\mptn \ell n$. \end{mainB}

It is worth stressing that there  {\em do not exist}  Kazhdan--Lusztig bases for complex reflection groups (in particular for   type $G(\ell,1,n)$ for $\ell>2$).    
 The Spets   programme seeks  to generalise the
  Kazhdan--Lusztig theory,  existence of finite groups of Lie type, 
and the structural properties of  Hecke algebras from  
 Weyl groups to the wider family of complex reflection groups.  
   Our integral cellular bases generalise one tranche of this theory  (the strong structural properties of Hecke algebras which normally depend on the existence of Kazhdan--Lusztig bases) to type $G(\ell,1,n)$.

 Each of the   integral graded cellular bases we construct provides   us with  a new viewpoint from which to study the Hecke algebra: a  new family of Specht modules, 
   a   new filtration on the projective modules (this was Geck--Rouquier's  motivation for instigating this research programme  in \cite{MR233664322,MR1889346}),   a  new grading  and new unitriangular  ordering  on  the decomposition matrix, 
 and most importantly {\em a  new $\ZZ$-lattice on the Hecke algebra}.
   Therefore our many different integral cellular bases provide us with many new ways to study  the {\em modular} {representations} of Hecke algebras by ``reduction modulo $p$".  Each of our new   $\ZZ$-lattices gives us a new way of  factorising   representation theoretic questions (e.g.   decomposition numbers) via  a two step process: 
   first calculate the decomposition numbers of the Hecke algebra over $\mathbb{Q}$ in terms of  Kazhdan--Lusztig polynomials 
   and then calculate the corresponding  `$p$-modular adjustment matrices'. 
  All known results on Hecke algebras in positive characteristic have been proven within the framework of the {\em asymptotic}  cellular structure of \cite{DJM,hm10} (e.g. the Jantzen sum formula \cite{MR1665333},  homological structure \cite{MR2351381,MR3213527,MR2721060,MR3533560},  branching rules \cite{MR2271584},  and   decomposition numbers    \cite{geordie,ELpaper}).
   We vastly generalise this framework from {\em asymptotic} charges to {\em all} weightings  and hence prove: 
   
   \begin{mainC}[Martin--Woodcock's conjecture] 
There is a  square  submatrix 
  of the  decomposition matrix of  $  H^{\mathbb Q}_n (\underline{s})$ with entries  
  given by the  non-parabolic   Kazhdan--Lusztig polynomials of type $\widehat{A}_{\ell-1}$.  

  \end{mainC}


 Finally, we generalise all the results of Brundan--Kleshchev--Wang \cite{bkw11} to arbitrary weightings; in particular the graded branching rule.  
 Fix  a weighting $\sigma\in \weight$ and the corresponding sets  of 
Specht  and  irreducible  modules 
 $\{ \Specht^\Bbbk_{ \sigma}(\la)  \mid \lambda \in \mptn \ell n\} $ and $ \{L^\Bbbk_{ \sigma}(\la)  \mid \lambda \in \Sigma^\ell_n\subseteq \mptn \ell n\}.$ 
We would like to understand the    structure of 
 the restrictions of these modules to the subalgebra   $   H_{n-1}^\Bbbk(\underline{s})\subset  H_n^\Bbbk(\underline{s})$ 
 (see also \cite{FLOTW99,MR1443748,MR2271584,MR1750939,bkw11,restrict}).  
 

\begin{mainD}
 Let  $\lambda \in \mptn \ell n$ and let $\upalpha_1\rhd_{\sigma}  \upalpha_2 \rhd_{\sigma} \dots \rhd_{\sigma} \upalpha_z$ denote the removable boxes of $\lambda$.  
 Then
the restriction of $  \Specht^\Bbbk_{ \sigma}(\la) $  has an $  H_{n-1}^\Bbbk(\underline{s})$-module filtration 
\! $$
0= \Specht^{z+1,\lambda}_{\sigma} \subset 
\Specht^{z,\lambda}_{\sigma}
\subset \dots 
\subset \Specht^{1,\lambda}_{\sigma}=\Res_{   H_{n-1}^\Bbbk(\underline{s})}( \Specht^\Bbbk_{{\sigma}} (\la))
$$such that for each $1\leq r\leq z$, we have 
$ 
\Specht_{{\sigma}}^\Bbbk (\la-\upalpha_r)\langle 	 \deg (\upalpha_r)	\rangle  
\cong
\Specht^{r,\lambda}_{\sigma} / \Specht^{r+1,\lambda}_{\sigma} .
$ 
 \end{mainD}

 \smallskip
 \noindent{\bf Antecedents:}
   We re-iterate  that there are many different  Kazhdan--Lusztig bases on a given  Hecke algebra   of a real reflection  group, one for each choice of weighting
      \cite{MR1974442,MR727851};
       the ``canonical" basis of \cite{MR560412} is then obtained by restricting ones attention to the trivial weighting. 
         These many weightings   have   applications in Schubert varieties \cite{MR727851}  statistical mechanics \cite{MS94,MW03}  
 and    provide many  different  lenses through which to view and understand a given Hecke algebra.

 \begin{itemize}[leftmargin=*]
\item[$\bullet$] \noindent{\sf Canonical basic sets and cellularity:}   
 The search for a proof of Theorems A and B  
 was the principal  focus of  a book by  Geck--Jacon \cite{MR2799052} and 
 a multitude of conjectures   
  \cite{MR1317512,MR2761947,MR2510041,MR2761946}   
  as well as being one of the motivating factors  for the recent surge of interest in Cherednik algebras \cite{MR2905554,MR2496742,BONNAFEROUQ,BONNAFEROUQ2} and \cite[Section 6]{GGOR03}.    
Over a field of characteristic zero, 
 a huge literature has  
  focussed on constructing  the combinatorial shadows of  our bases in Theorem~A; these shadows are called  ``canonical basic sets" and were  first introduced by Geck--Rouquier \cite{MR1429875,MR1889346}.  
   These combinatorial shadows  have been intensely studied \cite{MR2017068,MR2761947,MR2510041,MR2847511,MR2886292,MR3465894,MR2905554,MR2391649,MR3081654,MR233664322,MR2336643,MR2799052,
MR2496742,MR1889346,MR2271574,MR2118039,MR2189868,MR2350227,MR2802531}
  and have been used to prove unitriangularity of decomposition matrices with respect to the Lusztig ${\bf a}_\sigma$-orderings  over $\mathbb Q$.    
  Our Theorems A and B lift these results   to a higher structural level and extend them to arbitrary fields.

  In the case of   asymptotic charges (for which $\sigma_i \gg \sigma_{i+1}$ for $0\leq i <\ell-1$)  the combinatorics and  basis of Theorem A  coincides with that of \cite[Main Theorem]{hm10} 
   and Theorem A    generalises  the main results  of \cite{hm10}  to all possible charges.
  The existing results on cellular bases of Hecke algebras of type $G(\ell,1,n)$  form along two axes: 
for $\ell\in \{1,2\}$    cellular  Kazhdan--Lusztig theoretic bases   exist for all charges \cite{MR1974442,MR727851,MR2336039};  
for asymptotic charges    cellular Murphy-type bases exist 
for all types $G(\ell,1,n)$ \cite{hm10}.  
   This paper completes the cellularity  picture {\em along both these axes} by constructing   cellular bases for all   charges on all cyclotomic Hecke algebras.

    \smallskip
\item[$\bullet$] \noindent{\sf Parameterising and constructing  irreducible  modules:}   
     Ariki's categorification theorem gives rise to many abstract parameterisations of  irreducible  $ H_n^\Bbbk(\underline{s})$-modules \cite{MR1911030}.  
   The aforementioned  asymptotic cellular structure of \cite{hm10} is the key ingredient in the explicit  construction of   irreducible    modules as canonical quotients of Specht modules labelled by  {\em Kleshchev} $\ell$-partitions  in   \cite{MR1864465,MR1750939}.  
However, the Kleshchev  $\ell$-partitions provide just one of many possible labellings of the nodes in the crystal graph 
  \cite{MR2905554};  
  each such labelling should give rise to an explicit construction   of the  irreducible  modules. 
   In  \cref{isaididexplain}, we provide these many different constructions of the   irreducible  modules (one for each possible charge) and over {\em arbitrary fields}.  
 For each charge, we shall see that the corresponding  irreducibles   are those which survive under the associated {\sf KZ} functor.
  
\end{itemize}

 This paper has gone through many iterations over the past few years, we have provided a discussion  of how to pass between these versions  at the end of the current paper.  
In particular, earlier versions of this paper did not use Theorem A (in conjunction with results of Jacon \cite{MR2350227}) in order to deduce that the Uglov multipartitions label the   irreducible  modules of cyclotomic Hecke algebras (see Theorem B).   
 Using our results, 
  Kerschl has independently used our Theorem A in order to deduce this labelling result \cite{kerschl}.  Our proof is simpler (as it makes use of earlier results \cite{MR2350227}) but Kerschl's proof has the added advantage of providing new lower bounds for the dimensions of these  irreducible  modules.

\section{Weighted  combinatorics of complex reflection groups}\label{sec1}

For the remainder of the paper, unless otherwise specified, let $\Bbbk$ be an arbitrary integral domain.  

 We let $\mathfrak{S}_n$ denote the symmetric group with the usual Coxeter generators $s_{i,i+1}$ for $1\leq i <n$.   
Given  parameters $q$ and $(Q_0,Q_1,\dots, Q_{\ell-1})$    we  define the
 Hecke algebra  of $(\ZZ/\ell\ZZ)\wr \mathfrak{S}_n$ 
  to be the $\Bbbk$-algebra $H_n^\Bbbk(q;Q_0,\dots, Q_{\ell-1})$ generated by $T_0, T_1, \dots T_{n-1}$ subject to the relations 
\begin{align} \label{eqn:eqlabel}
\begin{split}
(T_i+q)(T_i-1)=0\quad (T_0-Q_0) (T_0-Q_1)\dots  (T_0-Q_{\ell-1})=0
\\
  T_iT_j=T_jT_i
\quad
\quad T_i T_{i+1}T_i=  T_{i+1}T_i T_{i+1}
\quad T_0 T_1T_0T_1=   T_1T_0T_1T_0
\end{split}
\end{align}
 for $1\leq i , j <n$ and $|i-j|>1$. We  set  
$ 
X_j = q^{1-j} T_{j-1} \dots T_1 T_0 T_1 \dots T_{j-1}$.  
  Given a {\sf charge} 
$\sigma=(e;\sigma_0,\sigma_1,\dots,\sigma_{\ell-1})\in \weight$ we 
are interested in the specialisation  of the parameters $q=\xi$ a primitive $e$th root of unity and 
$Q_m=q^{\sigma_m}$ for $0\leq m <\ell$.    Given $\sigma=(e;\sigma_0,\sigma_1,\dots,\sigma_{\ell-1})\in \weight$,  we  define the {\sf $e$-charge} 
 to be   
$   \underline{s} =(e;{s}_0,{s}_1,\dots,	{s}_{\ell-1}) \in \NN_{>1}\times (\ZZ/e\ZZ)^\ell$ obtained by reducing the $\ell$-tuple   modulo $e$.   
After specialisation we obtain the algebra  $ H ^{\Bbbk}_n( \underline{s}):= H_n^\Bbbk(\xi;	\xi^{\sigma_0},\dots, \xi^{\sigma_{\ell-1}})$ which we defined in the introduction; 
the notation has been chosen  to emphasise that,  after specialisation,  the definition of the Hecke algebra depends only on the   $e$-charge.

\subsection{Charged $\ell$-partitions}   
Fix a charge     $\theta=(e;\theta_0,\dots, \theta_{\ell-1} ) \in \NN_{>1}\times \ZZ ^\ell$.     
  We define a  {\sf  configuration of boxes}   to be a subset  of  
\begin{equation}\label{assumptions}\tag{$\square$}
\{(r,c,m)\mid r,c,m \in \mathbb{N},  1\leq r,c \leq  n,  
 0\leq m < \ell  \} 
\end{equation}
and we let $\mathscr{C}^\ell_n$ denote the set of all  configurations of $n$ boxes.      
We refer to a box $(r,c,m)$ as being in the $r$th row and $c$th column of the $m$th component of the configuration.  Given a box, $(r,c,m)$,  
we define the {\sf content} of this box  to be  ${\sf ct}(r,c,m)  = \sigma_m+  c - r$ and we define its residue to be ${\rm res}(r,c,m)\equiv {\sf ct}(r,c,m) \pmod e$.  
We refer to a box of residue  $i\in \ZZ/e\ZZ$ as an $i$-box.

  We define a {\sf partition}, $\lambda$,  of $n$ to be a   finite weakly decreasing sequence  of non-negative integers $ (\lambda_1,\lambda_2, \ldots)$ whose sum, $|\lambda| = \lambda_1+\lambda_2 + \dots$, equals $n$.   
 An    {\sf $\ell$-partition}  $\lambda=(\lambda^{(0)},\dots,\lambda^{(\ell-1)})$ of $n$ is an $\ell $-tuple of     partitions  such that $|\lambda^{(0)}|+\dots+ |\lambda^{(\ell-1)}|=n$. 
 We   denote the set of $\ell$-partitions of $n$ by $\mptn {\ell}n$.
 Given  $\lambda=(\lambda^{(0)}, \lambda^{(1)}, \ldots , \lambda^{(\ell-1)}) \in \mptn {\ell}n$, the {\sf Young diagram}   is   the configuration of boxes, 
\[
[\lambda]=\{(r,c,m) \mid  1\leq  c\leq \lambda^{(m)}_r\}.
\]
  We now recall  Lusztig's ${\bf a}_{\sigma}$-ordering on $\ell$-partitions and   Webster's  coarsening of this ordering.  
  
 \begin{defn} \label{domdef}
 Given  
   ${\sigma}\in \weight$   a charge,   
 we write $(r,c,m) <_{\sigma}  (r',c',m')$ if either 
 \begin{itemize}
\item[$(i)$]   $\ct(r,c,m) < \ct(r',c',m')$ or 
\item[$(ii)$]   $\ct(r,c,m) = \ct(r',c',m')$  and $m>m'$
\end{itemize}  
 We write 
 $(r,c,m) \lhd_\theta  (r',c',m')$ if both   $(r,c,m) <_\theta  (r',c',m')$ and $\res(r,c,m)=\res(r',c',m')$.  
\end{defn}

The following formulation  of the Lusztig ${\bf a}_{\sigma}$-ordering is given in \cite[5.6 Proposition]{MR2905554}.  
\begin{defn}[Lusztig's  ${\bf a}_{\sigma}$-ordering]
 For $\la,\mu\in \mptn \ell n$, we write $ \mu \leq_{\sigma}  \la $  if there is a bijective   map ${\sf A}:[\lambda] \to [\mu]$ 
 such that either $  {\sf A}(r,c,m)<_{\sigma}(r,c,m) $ or $  {\sf A}(r,c,m)=(r,c,m) $ for all $(r,c,m)\in \lambda$.  

\end{defn}  
  
  We now rephrase Webster's ordering on $\mptn \ell n$ in such a way that it is easily seen to be a coarsening of Lusztig's ${\bf a}_{\sigma}$-ordering.  
  We reconcile this with Webster's  original diagrammatic definition shortly.  
  
  \begin{defn}[Webster's   ordering]
 For $\la,\mu\in \con \ell n$, we write   $ \mu \trianglelefteq_{\sigma}  \la $  if there is a residue preserving  bijective   map ${\sf A}:[\lambda] \to [\mu]$ 
 such that either $  {\sf A}(r,c,m)\vartriangleleft_{\sigma}(r,c,m) $  or $  {\sf A}(r,c,m)=(r,c,m) $ for all $(r,c,m)\in \lambda$.  

 \end{defn}

We now discuss how   \cref{domdef} and the ensuing orderings on $\mptn \ell n$ can be visualised diagrammatically.  
  Given $\la\in\mptn\ell n$,  the associated  {\sf  (mirrored) $\theta$-Russian array} 
 is defined as follows.
 (We drop the prefix ``mirrored" for the remainder of this paper, we just  highlight now for the reader that our conventions are the opposite of the usual definition of a Russian diagram.)  
For each $0\leq m< \ell$, we place a point on the real line at $\sigma_m-\tfrac{m}{\ell}$ and consider the region bounded by  half-lines at angles $3\pi/4 $ and $\pi/4$.
(Compare the $m/\ell$ removed from the charge with condition $(ii)$ of \cref{domdef}.)
		We tile the resulting quadrant with a lattice of squares, each with diagonal of length $2   $ 
		(this will be important!).  
  We place the   box $(1,1,m)$   at the point  $\sigma_m-\tfrac{m}{\ell}$ on the real line, with rows going northeast from this node, and columns going northwest. 
 Given a fixed charge $\theta \in \weight$ and $\lambda \in \mptn \ell n$, we do not distinguish between the configuration of boxes 
  and its $\theta$-Russian array.

\begin{prop}
We have that $(r,c,m) <_{\sigma}(r',c',m')$ if and only if  the box $(r,c,m)$ appears strictly to the left of the box $(r',c',m')$ in the $\theta$-Russian array.  
\end{prop}

\begin{proof}
 This is clear from the definitions.  Notice that the subtraction $-m/\ell$ ensures that $(ii)$ of \cref{domdef} matches the diagrammatic ordering.  
 \end{proof}

\begin{figure}[ht]\captionsetup{width=0.9\textwidth}
\!\!\!\!\!\!\!\!\!\!\!
$$\scalefont{0.9}
\begin{tikzpicture}[scale=0.7]
\clip(-2.5,-4.5) rectangle (3.5,4);
 \draw(-2,-4)--(4,-4);
 \draw(0,-4.25) node {0};
  \path(0,-4.25)--++(45:0.5)--++(-45:0.5) coordinate (hello) node {1};
  \path(0,-4.25)--++(45:-0.5)--++(-45:-0.5) coordinate (hello2) node {$-1$};
  \path(hello2)--++(45:-0.5)--++(-45:-0.5) coordinate (hello2) node {$-2$};  
    \path(hello)--++(45:0.5)--++(-45:0.5) coordinate (hello) node {2};
        \path(hello)--++(45:0.5)--++(-45:0.5) coordinate (hello) node {3};
            \path(hello)--++(45:0.5)--++(-45:0.5) coordinate (hello) node {4};
 
\begin{scope}{   \path(0,0) coordinate (origin);
 \clip(origin)--(45:8)--(135:8)--(-135:8)--(origin);
  \clip(origin)--++(45:4)--++(135:1)--++(-135:3)--++(135:2)--++(-135:1)--(origin);

  \foreach \i\j in {1,...,19}
  {
    \path (origin)++(45:1*\i cm)  coordinate (a\i);
    \path (origin)++(135:1*\i cm)  coordinate (b\i);
    \path (a\i)++(135:1*\j cm) coordinate (ca\i,\j);
    \path (b\i)++(45:1*\j cm) coordinate (cb\i,\j);

  }
    \foreach \i in {1,...,19}
{  \draw[black,thick] (a\i)--++(135:8);
    \draw[black,thick] (b\i)--++(45:8);}
      \draw[black,thick] (origin)--++(135:8);
            \draw[black,thick] (origin)--++(45:8);

}
\end{scope}
\draw[wei](origin)--++(-90:4);

\path (origin)--++(45:0.5)--++(135:0.5) coordinate (origin);

  \foreach \i in {1,...,4}
  {
     \path (origin)++(45:1*\i cm)  coordinate (xx\i);
    \path (origin)++(135:1*\i cm)  coordinate (xy\i);

  }
\draw(xx1) node {1};
\draw(xx2) node {2};
\draw(xx3) node {3};
\draw(xy1) node {$-1$};  
\draw(xy2) node {$-2$};  
\path(origin)--++(45:0.5)--++(135:0.5) coordinate (xc3);
\draw(origin) node {0};  
\path(xc3)--++(45:0.5)--++(135:0.5) coordinate (xcd);


\path(0,0)--++(45:0.25)--++(-45:0.25)--++(-90:3.2) coordinate (origin);
\draw[wei](origin)--++(-90:0.8);
\begin{scope}
{   
\draw[thick](origin)--++(45:3)--++(135:1)--++(-135:1)--++(135:1)--++(-135:1)--++(135:1)--++(-135:1)--(origin);
 \clip(origin)--++(45:8)--++(135:8)--++(-135:8)--(origin);
 \clip(origin)--++(45:3)--++(135:1)--++(-135:1)--++(135:1)--++(-135:1)--++(135:1)--++(-135:1)--(origin);
   \foreach \i\j in {1,...,19}
  {
    \path (origin)++(45:1*\i cm)  coordinate (a\i);
    \path (origin)++(135:1*\i cm)  coordinate (b\i);
    \path (a\i)++(135:1*\j cm) coordinate (ca\i,\j);
    \path (b\i)++(45:1*\j cm) coordinate (cb\i,\j);

  }
    \foreach \i in {1,...,19}
{  \draw[black,thick] (a\i)--++(135:8);
    \draw[black,thick] (b\i)--++(45:8);}
      \draw[black,thick] (origin)--++(135:8);
            \draw[black,thick] (origin)--++(45:8);
}

  \clip(origin)--++(45:8)--++(135:8)--++(-135:8)--(origin);

\path (origin)--++(45:0.5)--++(135:0.5) coordinate (origin);

  \foreach \i in {1,...,4}
  {
     \path (origin)++(45:1*\i cm)  coordinate (x\i);
    \path (origin)++(135:1*\i cm)  coordinate (y\i);

  }
\draw(x1) node {2};
\draw(x2) node {3};
\draw(y1) node {0};  
\draw(y2) node {$-1$};  
\path(origin)--++(45:0.5)--++(135:0.5) coordinate (c3);
\draw(origin) node {1};  
\path(c3)--++(45:0.5)--++(135:0.5) coordinate (cd);
\draw(cd) node {1};

\end{scope}
\end{tikzpicture}
\qquad 
\begin{tikzpicture}[scale=0.7]
\clip(-2.5,-4.5) rectangle (5,4);
\clip(-2.5,-4.5) rectangle (5,4);
 \draw(-2,-4)--(4,-4);
 \draw(0,-4.25) node {0};
  \path(0,-4.25)--++(45:0.5)--++(-45:0.5) coordinate (hello) node {1};
  \path(0,-4.25)--++(45:-0.5)--++(-45:-0.5) coordinate (hello2) node {$-1$};
  \path(hello2)--++(45:-0.5)--++(-45:-0.5) coordinate (hello2) node {$-2$};  
    \path(hello)--++(45:0.5)--++(-45:0.5) coordinate (hello) node {2};
        \path(hello)--++(45:0.5)--++(-45:0.5) coordinate (hello) node {3};
            \path(hello)--++(45:0.5)--++(-45:0.5) coordinate (hello) node {4};
           \path(hello)--++(45:0.5)--++(-45:0.5) coordinate (hello) node {5};

 
\begin{scope}{   \path(0,0) coordinate (origin);
 \clip(origin)--(45:8)--(135:8)--(-135:8)--(origin);
  \clip(origin)--++(45:4)--++(135:1)--++(-135:3)--++(135:2)--++(-135:1)--(origin);

  \foreach \i\j in {1,...,19}
  {
    \path (origin)++(45:1*\i cm)  coordinate (a\i);
    \path (origin)++(135:1*\i cm)  coordinate (b\i);
    \path (a\i)++(135:1*\j cm) coordinate (ca\i,\j);
    \path (b\i)++(45:1*\j cm) coordinate (cb\i,\j);

  }
    \foreach \i in {1,...,19}
{  \draw[black,thick] (a\i)--++(135:8);
    \draw[black,thick] (b\i)--++(45:8);}
      \draw[black,thick] (origin)--++(135:8);
            \draw[black,thick] (origin)--++(45:8);

}
\end{scope}
\draw[wei](origin)--++(-90:4);

\path (origin)--++(45:0.5)--++(135:0.5) coordinate (origin);

  \foreach \i in {1,...,4}
  {
     \path (origin)++(45:1*\i cm)  coordinate (xx\i);
    \path (origin)++(135:1*\i cm)  coordinate (xy\i);

  }
\draw(xx1) node {1};
\draw(xx2) node {2};
\draw(xx3) node {3};
\draw(xy1) node {$-1$};  
\draw(xy2) node {$-2$};  
\path(origin)--++(45:0.5)--++(135:0.5) coordinate (xc3);
\draw(origin) node {0};  
\path(xc3)--++(45:0.5)--++(135:0.5) coordinate (xcd);


\path(0,0)--++(45:0.25)--++(-45:0.25)
--++(45:1.5)--++(-45:1.5)--++(-90:3.2) coordinate (origin);
\draw[wei](origin)--++(-90:0.8);
\begin{scope}
{   
\draw[thick](origin)--++(45:3)--++(135:1)--++(-135:1)--++(135:1)--++(-135:1)--++(135:1)--++(-135:1)--(origin);
 \clip(origin)--++(45:8)--++(135:8)--++(-135:8)--(origin);
 \clip(origin)--++(45:3)--++(135:1)--++(-135:1)--++(135:1)--++(-135:1)--++(135:1)--++(-135:1)--(origin);
   \foreach \i\j in {1,...,19}
  {
    \path (origin)++(45:1*\i cm)  coordinate (a\i);
    \path (origin)++(135:1*\i cm)  coordinate (b\i);
    \path (a\i)++(135:1*\j cm) coordinate (ca\i,\j);
    \path (b\i)++(45:1*\j cm) coordinate (cb\i,\j);

  }
    \foreach \i in {1,...,19}
{  \draw[black,thick] (a\i)--++(135:8);
    \draw[black,thick] (b\i)--++(45:8);}
      \draw[black,thick] (origin)--++(135:8);
            \draw[black,thick] (origin)--++(45:8);
}

  \clip(origin)--++(45:8)--++(135:8)--++(-135:8)--(origin);

\path (origin)--++(45:0.5)--++(135:0.5) coordinate (origin);

  \foreach \i in {1,...,4}
  {
     \path (origin)++(45:1*\i cm)  coordinate (x\i);
    \path (origin)++(135:1*\i cm)  coordinate (y\i);

  }
\draw(x1) node {5};
\draw(x2) node {6};
\draw(y1) node {3};  
\draw(y2) node {$2$};  
\path(origin)--++(45:0.5)--++(135:0.5) coordinate (c3);
\draw(origin) node {4};  
\path(c3)--++(45:0.5)--++(135:0.5) coordinate (cd);
\draw(cd) node {4};

\end{scope}
\end{tikzpicture}
$$\caption{We picture the 2-partition $((4,1^2) \mid (3,2,1))$ for $(e;\sigma_0,\sigma_1)=(e;0,1)$ and  $(e;\sigma_0,\sigma_1)=(e;0,4)$ respectively for $e\in \NN_{>1}$.  In each box we have placed the content of the box. Notice that the boxes of content $1$ in the second component appear to the left of those of content $1$ in the first component (by half a unit).  
}\label{ALoading3333345}
\end{figure}

\begin{eg}A charge is said to be {\sf asymptotic} if $\sigma_m-\sigma_{m+1} >  n$ for all $0\leq m < \ell-1$.
 For $\theta\in \weight$ a asymptotic charge and $\la,\mu \in \mptn \ell n$, it is easy to see that if $\la \geq_\theta \mu$ if and only if 
 $$\sum_{i=1}^{k-1}|\la^{(i)}| + \sum _{i=1}^{j}\la^{(k)}_i 
 \geq 
 \sum_{i=1}^{k-1}|\mu^{(i)}| + \sum _{i=1}^{j}\mu^{(k)}_i $$
 for all $1\leq k\leq \ell$ and $1\leq j \leq n$.   
\end{eg}

 \begin{eg}
In the case $\theta=(e;\sigma_0,\sigma_1,\dots, \sigma_{\ell-1})\in \weight$ is such that $
0< \sigma_{i} -\sigma_{j} <e$ for $0\leq i <j <\ell$, the $\theta$-dominance order coincides with the ordering on $\mptn \ell n$ considered  in \cite{FLOTW99}. 
This charge is considered in greater detail in \cref{sec2}.  
\end{eg}

\subsection{Charged standard tableaux}

 Given $\la \in \mptn \ell n$, we let   ${\rm Rem} (\la)$ (respectively ${\rm Add}  (\la)$) 
  denote the set of all    removable  (respectively addable) boxes  of the Young diagram of $\la$ so that the resulting diagram is the Young diagram of a $\ell$-partition.  We extend the residue and dominance notation above in the obvious fashion.  Given $i\in \ZZ/e\ZZ$, we let  ${\rm Rem}_i  (\la)\subseteq {\rm Rem}  (\la)$
  and ${\rm Add}_i  (\la)\subseteq {\rm Add}  (\la)$   denote the subsets of boxes of residue $i\in \ZZ/e\ZZ$.

 \begin{defn}
Fix   $\theta\in \weight $.  
Given $\lambda\in \mptn \ell n$, we define a $\theta$-{\sf tableau} of shape $\lambda$ to be a bijective map from the boxes of 
the $\theta$-Russian array of  $\boxla $ to the set 
$\{1,\dots , n\}$ (depicted as a filling the boxes with the corresponding integers).  We define a {\sf  standard tableau} to be a tableau  in which    the entries increase along the rows and columns of each component.  
 We let $\Std_{\theta}(\lambda)$ denote the set of all standard tableaux of shape $\lambda\in\mptn\ell n$. 
    Given 
$\stt\in \Std_{\theta}(\lambda)$, we set $\Shape(\stt)=\la$.  
  Given $1\leq k \leq n$, we let $\stt{\downarrow}_{\{1,\dots ,k\}}$ be the subtableau of $\stt$ whose  entries belong to the set
  $\{1,\dots,k\}$.  
For $\sts, \stt \in \Std_\theta(\la)$ we write $\sts \trianglelefteq_\theta \stt$ if $\Shape(\sts{\downarrow}_{\{1,\dots ,k\}})\trianglelefteq_\theta \Shape(\stt{\downarrow}_{\{1,\dots ,k\}})$ for $1\leq k \leq n$ (one can define   $\leq_\theta$ on 
$\Std_\theta(\la)$ 
similarly).  
\end{defn}

\begin{figure}[ht]\captionsetup{width=0.9\textwidth}
\!\!\!\!\!\!\!\!\!\!
$$\scalefont{0.9}
\begin{tikzpicture}[scale=0.7]
 
\clip(-2.5,-4.5) rectangle (3.5,4);
 \draw(-2,-4)--(4,-4);
 \draw(0,-4.25) node {0};
  \path(0,-4.25)--++(45:0.5)--++(-45:0.5) coordinate (hello) node {1};
  \path(0,-4.25)--++(45:-0.5)--++(-45:-0.5) coordinate (hello2) node {$-1$};
  \path(hello2)--++(45:-0.5)--++(-45:-0.5) coordinate (hello2) node {$-2$};  
    \path(hello)--++(45:0.5)--++(-45:0.5) coordinate (hello) node {2};
        \path(hello)--++(45:0.5)--++(-45:0.5) coordinate (hello) node {3};
            \path(hello)--++(45:0.5)--++(-45:0.5) coordinate (hello) node {4};
 
\begin{scope}{   \path(0,0) coordinate (origin);
 \clip(origin)--(45:8)--(135:8)--(-135:8)--(origin);
  \clip(origin)--++(45:4)--++(135:1)--++(-135:3)--++(135:2)--++(-135:1)--(origin);

  \foreach \i\j in {1,...,19}
  {
    \path (origin)++(45:1*\i cm)  coordinate (a\i);
    \path (origin)++(135:1*\i cm)  coordinate (b\i);
    \path (a\i)++(135:1*\j cm) coordinate (ca\i,\j);
    \path (b\i)++(45:1*\j cm) coordinate (cb\i,\j);

  }
    \foreach \i in {1,...,19}
{  \draw[black,thick] (a\i)--++(135:8);
    \draw[black,thick] (b\i)--++(45:8);}
      \draw[black,thick] (origin)--++(135:8);
            \draw[black,thick] (origin)--++(45:8);

}
\end{scope}
\draw[wei](origin)--++(-90:4);

\path (origin)--++(45:0.5)--++(135:0.5) coordinate (origin);

  \foreach \i in {1,...,4}
  {
     \path (origin)++(45:1*\i cm)  coordinate (xx\i);
    \path (origin)++(135:1*\i cm)  coordinate (xy\i);

  }
\draw(xx1) node {5};
\draw(xx2) node {8};
\draw(xx3) node {9};
\draw(xy1) node {$11$};  
\draw(xy2) node {$12$};  
\path(origin)--++(45:0.5)--++(135:0.5) coordinate (xc3);
\draw(origin) node {1};  
\path(xc3)--++(45:0.5)--++(135:0.5) coordinate (xcd);


\path(0,0)--++(45:0.25)--++(-45:0.25)--++(-90:3.2) coordinate (origin);
\draw[wei](origin)--++(-90:0.8);
\begin{scope}
{   
\draw[thick](origin)--++(45:3)--++(135:1)--++(-135:1)--++(135:1)--++(-135:1)--++(135:1)--++(-135:1)--(origin);
 \clip(origin)--++(45:8)--++(135:8)--++(-135:8)--(origin);
 \clip(origin)--++(45:3)--++(135:1)--++(-135:1)--++(135:1)--++(-135:1)--++(135:1)--++(-135:1)--(origin);
   \foreach \i\j in {1,...,19}
  {
    \path (origin)++(45:1*\i cm)  coordinate (a\i);
    \path (origin)++(135:1*\i cm)  coordinate (b\i);
    \path (a\i)++(135:1*\j cm) coordinate (ca\i,\j);
    \path (b\i)++(45:1*\j cm) coordinate (cb\i,\j);

  }
    \foreach \i in {1,...,19}
{  \draw[black,thick] (a\i)--++(135:8);
    \draw[black,thick] (b\i)--++(45:8);}
      \draw[black,thick] (origin)--++(135:8);
            \draw[black,thick] (origin)--++(45:8);
}

  \clip(origin)--++(45:8)--++(135:8)--++(-135:8)--(origin);

\path (origin)--++(45:0.5)--++(135:0.5) coordinate (origin);

  \foreach \i in {1,...,4}
  {
     \path (origin)++(45:1*\i cm)  coordinate (x\i);
    \path (origin)++(135:1*\i cm)  coordinate (y\i);

  }
\draw(x1) node {6};
\draw(x2) node {10};
\draw(y1) node {3};  
\draw(y2) node {$4$};  
\path(origin)--++(45:0.5)--++(135:0.5) coordinate (c3);
\draw(origin) node {2};  
\path(c3)--++(45:0.5)--++(135:0.5) coordinate (cd);
\draw(cd) node {7};

\end{scope}
\end{tikzpicture}
\qquad 
\begin{tikzpicture}[scale=0.7]

\clip(-2.5,-4.5) rectangle (5,4);
 \draw(-2,-4)--(4,-4);
 \draw(0,-4.25) node {0};
  \path(0,-4.25)--++(45:0.5)--++(-45:0.5) coordinate (hello) node {1};
  \path(0,-4.25)--++(45:-0.5)--++(-45:-0.5) coordinate (hello2) node {$-1$};
  \path(hello2)--++(45:-0.5)--++(-45:-0.5) coordinate (hello2) node {$-2$};  
    \path(hello)--++(45:0.5)--++(-45:0.5) coordinate (hello) node {2};
        \path(hello)--++(45:0.5)--++(-45:0.5) coordinate (hello) node {3};
            \path(hello)--++(45:0.5)--++(-45:0.5) coordinate (hello) node {4};
           \path(hello)--++(45:0.5)--++(-45:0.5) coordinate (hello) node {5};

 
\begin{scope}{   \path(0,0) coordinate (origin);
 \clip(origin)--(45:8)--(135:8)--(-135:8)--(origin);
  \clip(origin)--++(45:4)--++(135:1)--++(-135:3)--++(135:2)--++(-135:1)--(origin);

  \foreach \i\j in {1,...,19}
  {
    \path (origin)++(45:1*\i cm)  coordinate (a\i);
    \path (origin)++(135:1*\i cm)  coordinate (b\i);
    \path (a\i)++(135:1*\j cm) coordinate (ca\i,\j);
    \path (b\i)++(45:1*\j cm) coordinate (cb\i,\j);

  }
    \foreach \i in {1,...,19}
{  \draw[black,thick] (a\i)--++(135:8);
    \draw[black,thick] (b\i)--++(45:8);}
      \draw[black,thick] (origin)--++(135:8);
            \draw[black,thick] (origin)--++(45:8);

}
\end{scope}
\draw[wei](origin)--++(-90:4);

\path (origin)--++(45:0.5)--++(135:0.5) coordinate (origin);

  \foreach \i in {1,...,4}
  {
     \path (origin)++(45:1*\i cm)  coordinate (xx\i);
    \path (origin)++(135:1*\i cm)  coordinate (xy\i);

  }
\draw(xx1) node {5};
\draw(xx2) node {8};
\draw(xx3) node {9};
\draw(xy1) node {$11$};  
\draw(xy2) node {$12$};  
\path(origin)--++(45:0.5)--++(135:0.5) coordinate (xc3);
\draw(origin) node {1};  
\path(xc3)--++(45:0.5)--++(135:0.5) coordinate (xcd);


\path(0,0)--++(45:0.25)--++(-45:0.25)
--++(45:1.5)--++(-45:1.5)--++(-90:3.2) coordinate (origin);
\draw[wei](origin)--++(-90:0.8);
\begin{scope}
{   
\draw[thick](origin)--++(45:3)--++(135:1)--++(-135:1)--++(135:1)--++(-135:1)--++(135:1)--++(-135:1)--(origin);
 \clip(origin)--++(45:8)--++(135:8)--++(-135:8)--(origin);
 \clip(origin)--++(45:3)--++(135:1)--++(-135:1)--++(135:1)--++(-135:1)--++(135:1)--++(-135:1)--(origin);
   \foreach \i\j in {1,...,19}
  {
    \path (origin)++(45:1*\i cm)  coordinate (a\i);
    \path (origin)++(135:1*\i cm)  coordinate (b\i);
    \path (a\i)++(135:1*\j cm) coordinate (ca\i,\j);
    \path (b\i)++(45:1*\j cm) coordinate (cb\i,\j);

  }
    \foreach \i in {1,...,19}
{  \draw[black,thick] (a\i)--++(135:8);
    \draw[black,thick] (b\i)--++(45:8);}
      \draw[black,thick] (origin)--++(135:8);
            \draw[black,thick] (origin)--++(45:8);
}

  \clip(origin)--++(45:8)--++(135:8)--++(-135:8)--(origin);

\path (origin)--++(45:0.5)--++(135:0.5) coordinate (origin);

  \foreach \i in {1,...,4}
  {
     \path (origin)++(45:1*\i cm)  coordinate (x\i);
    \path (origin)++(135:1*\i cm)  coordinate (y\i);

  }
\draw(x1) node {6};
\draw(x2) node {10};
\draw(y1) node {3};  
\draw(y2) node {$4$};  
\path(origin)--++(45:0.5)--++(135:0.5) coordinate (c3);
\draw(origin) node {2};  
\path(c3)--++(45:0.5)--++(135:0.5) coordinate (cd);
\draw(cd) node {7};

\end{scope}
\end{tikzpicture}
$$\caption{   
Tableaux $\sts \in \Std_{(3;0,1)}(\la)$ and  $\stt \in \Std_{(3;0,4)}(\la)$ for  $\la=((4,1^2),(3,2,1))$.  
} 
\label{ALoading33333}
\end{figure} 

 \begin{defn}
We define a  {\sf residue sequence} to be an element ${\imath}=(i_1,\dots,i_n)\in (\ZZ/e\ZZ)^n$.  
 Given 
  $\stt\in \Std_{\theta}(\la)$  we  define the  {\sf  residue sequence}, ${\imath}_\stt $,  as follows,
$$
{\rm res}(\stt) = ({\rm res} (\stt^{-1}(1)), {\rm res} (\stt^{-1}(2)),  \dots, {\rm res}(\stt^{-1}(n)) )    \in (\ZZ/e\ZZ)^n.  
$$

 \end{defn}

 \begin{eg}  \label{oneexample}
 The algebras $H^\Bbbk_{12}(3;0,1)$ and $H^\Bbbk_{12}(3;0,4)$ are isomorphic.  
For  $\la=((4,1^2),(3,2,1))$ we have pictured a tableau $\sts \in \Std_{(3;0,1)}(\la)$ and  $\stt \in \Std_{(3;0,4)}(\la)$ in   \cref{ALoading33333}.  
The residue sequences  $\sts$  and 
 $\bexy$  from \cref{ALoading33333} are all the same and are equal to  $(0,1,0,2,1,2,1,2,0,0,2,1)$.   
\end{eg}

%
%
%

\begin{defn} 
Let $\la\in\mptn\ell n$ and $\stt \in \Std_{\theta}(\la)$. 
  We let $\stt^{-1}(k)$ denote the box in $\stt$ containing the integer $1\leq k\leq n$.   
  Given $1\leq k\leq n$, we let ${\mathcal A}_\stt(k)$, 
(respectively ${\mathcal R}_\stt(k)$)  denote the set of   all addable $\res (\stt^{-1}(k))$-boxes (respectively all
removable   $\res (\stt^{-1}(k))$-boxes)  of the $\ell$-partition $\Shape(\stt{\downarrow}_{\{1,\dots ,k\}})$ which
  are less than  $\stt^{-1}(k)$ in the $\theta$-dominance order  (i.e  those which appear to the {\em  left\color{black}} of $\stt^{-1}(k)$).    

\end{defn}

\begin{defn} Let $\la\in\mptn\ell n$ and $\stt \in \Std_{\theta}(\la)$.  We define the degree of $\stt$ as follows,
$$
\deg(\stt) = \sum_{k=1}^n \left(	|{\mathcal A}_\stt(k)|-|{\mathcal R}_\stt(k)|	\right).
$$
\end{defn}

  \begin{rmk}
 For   $\theta\in \weight$ an asymptotic charge,   our tableaux and grading coincide with those  of  \cite[Section 3]{hm10} and \cite[Section 1]{bkw11}.  
 \end{rmk}

\begin{eg}\label{oneexample}
 We continue with the example above specialising  $e=3$,  and $\theta=(3;0, 1)$ versus $\theta=(3;0, 4)$. 
 The tableau $\sts$ of \cref{ALoading33333} has degree 5: the boxes with entries 5, 7, 8, 9, 10 and 11 have degrees $1, 1, 1, 2, 1, -1 $ respectively and all other boxes have degree 0.  
 The tableau $\stt$ of \cref{ALoading33333} has degree 0: the boxes with entries 2, 3, 4, 9 have degrees $1, -1, 1, -1$ respectively and all other boxes have degree 0.  We note that the boxes of $\stt$ with  entries 6, 7, 8, 10 all have degree zero because they have both an addable and a removable node to their left which cancel out.   
\end{eg}

 \subsection{Charged semistandard tableaux}

We first tilt  the $\theta$-Russian array of $\lambda\in \con\ell n$   ever-so-slightly in the 
 anti\color{black}
clockwise direction so that the top vertex of the box $(r,c,m)$   has $x$-coordinate 
$${\bf I}^{\sigma}_{(r,c,m)} =\ct(r,c,m)-m/\ell - (r+c)\varepsilon$$ for $\varepsilon\ll \tfrac{1}{2n\ell}$ 
 (up to small angle approximation).   
Our assumption that 
$\varepsilon\ll \tfrac{1}{2n\ell}$ implies  that no two boxes in the $\theta$-charged Young diagram of  $\la \in \mathscr{C}^\ell_n$ can have the same $x$-coordinate and thus we have refined the ordering $\leq _\sigma$ of \cref{domdef} to a total ordering on boxes. 
  Given $\lambda\in\mathscr{C}^ \ell_ n$, we let ${\bf I}^{\theta}_\lambda$ denote    
 the  ordered set  of the $\mathbf{I}^\theta_{(r,c,m)}$ for $(r,c,m)\in \lambda$.     
       Given $\lambda\in\mathscr{C}^ \ell_ n$, the associated  {\sf residue sequence}, ${\rm res}(\lambda)$,  of $\lambda$ is given by reading the residues of the boxes of $ \lambda  $ according to the natural ordering on $x$-coordinates. 

\begin{defn} \label{semistandard:defn}
Given $\la , \mu \in \mathscr{C}^\ell_n$ we define   a  tableau, $\SSTT$, of shape $\lambda $ and weight $\mu  $  to be a   bijective map $\SSTT : [\boxla]  \to \mathbf{I}^\theta_\mu$. 
We say that a tableau is {\sf semistandard} if it also satisfies the following  properties   \begin{itemize}
  \item[$(i)$]     $\SSTT(1,1,m)<\sigma_m$,
  \item[$(ii)$]    $\SSTT(r,c,m)< \SSTT(r-1,c,m)  - \g$,
  \item[$(iii)$] 
   $\SSTT(r,c,m) < \SSTT(r,c-1,m) + \g$,   
  \end{itemize}
  for $(r,c,m)\in \lambda$.  
   We  denote the set of all  semistandard tableaux of shape $\lambda$
  and weight $\mu$ by $\SStd_{\theta}(\lambda,\mu)$.   Given $\SSTT \in 
 \SStd_{\theta}(\lambda,\mu)$, we   write $\Shape(\SSTT)=\lambda$.  
  \end{defn}

\begin{figure}[ht!]
\scalefont{0.8} \begin{tikzpicture}[baseline, thick,yscale=0.75,xscale=-0.75]

\path(0,0) coordinate (112);
\path(112)--++(-10:1) coordinate (a1);
\path(112)--++(80:1) coordinate (a2);
\path(112)--++(170:1) coordinate (a3);
\path(112)--++(-100:1) coordinate (a4);
\draw(a1)--(a2)--(a3)--(a4)--(a1);
\fill[red] (a4) circle (2pt);
\draw(112) node {$-2\varepsilon$};

\path(a1)--++(80:1) coordinate (212);
\path(212)--++(-10:1) coordinate (b1);
\path(212)--++(80:1) coordinate (b2);
\path(212)--++(170:1) coordinate (b3);
\path(212)--++(-100:1) coordinate (b4);
\draw(b1)--(b2)--(b3)--(b4)--(b1);

 \draw(212) node {$-1-3\varepsilon$};

\path(a3)--++(80:1) coordinate (312);
\path(312)--++(-10:1) coordinate (c1);
\path(312)--++(80:1) coordinate (c2);
\path(312)--++(170:1) coordinate (c3);
\path(312)--++(-100:1) coordinate (c4);
\draw(c1)--(c2)--(c3)--(c4)--(c1);

\draw(312) node {$-2-4\varepsilon$};

\path(-4.5,0) coordinate (112);
\path(112)--++(-10:1) coordinate (d1);
\path(112)--++(80:1) coordinate (d2);
\path(112)--++(170:1) coordinate (d3);
\path(112)--++(-100:1) coordinate (d4);
\draw(d1)--(d2)--(d3)--(d4)--(d1);

\draw(112) node {$1-3\varepsilon$};

\path(d1)--++(80:1) coordinate (212);
\path(212)--++(-10:1) coordinate (e1);
\path(212)--++(80:1) coordinate (e2);
\path(212)--++(170:1) coordinate (e3);
\path(212)--++(-100:1) coordinate (e4);
\draw(e1)--(e2)--(e3)--(e4)--(e1);

\draw(212) node {$-3-5\varepsilon$};
  
\fill[red] (d4) circle (2pt);
 
\end{tikzpicture}

\caption{A semistandard tableau $\SSTS\in \SStd_{(3;4,0)}	(	( (1^2),(2,1)), ( \varnothing,(2 ,1^3))	 	)	$.
   We have tilted the components of this 2-partition  $\varepsilon$ units anti-clockwise so that the $x$-coordinate of a box $(r,c,m)$ is equal to $\ct(r,c,m)-m/\ell-(r+c)\varepsilon$.  Two possible corresponding basis elements for this tableau are depicted in \cref{othertwo}.}
\label{otherone}
\end{figure}

   \section{Graded  cellular algebras and canonical basic sets}\label{gradedcell}

  Let $\R$ be an integral domain with field of fractions $\mathbb F$.    
 Let $A^\R$ be an associative  $\R$-algebra which is finitely generated and free over  $\R$ and   $\vartheta: \R \to {\mathbb K}$   a ring homomorphism into a field ${\mathbb K}$ such that ${\mathbb K}$ is the field of fractions of $\vartheta(\R)$. We obtain an $\mathbb F$-algebra $A^{\mathbb F} = A^\R \otimes_\R {\mathbb F} $ 
  and a ${\mathbb K}$-algebra $A^{{\mathbb K}} = A^\R \otimes_\R  {\mathbb K} $.   
   We let 
     ${\rm Irr}(A^{{\mathbb K}})$ (respectively ${\rm Irr}(A^{\mathbb F})$) denote the set  
     of all irreducible representations of $A^{{\mathbb K}}$ (respectively $A^{\mathbb F}$) up to isomorphism.   
 The following  generalises the definition of    {\cite[Section 3.1.7]{MR2799052}} to more general  modular systems.  
 
    \begin{defn}   \label{uglovv2} 
Let $A^\R$ be an algebra with representations $\{V_\la^\R\mid   \la\in \Pi\}$.   
  Suppose that      $ A^{\mathbb F} $ is   semisimple  
  and  $\{V_\la^{\mathbb F}:= V_\la^\R\otimes _\R{\mathbb F} \mid   \la\in \Pi\} = {\rm Irr}(A^{\mathbb F})$.         Let $\rhd $ be a partial order on $\Pi$ such that:    \begin{itemize}[leftmargin=*]  
 \item[$(i)$]      Given $L^{\mathbb K}_\la\in {\rm Irr}(A^{{\mathbb K}}) $, let $\mathscr{S}_{\trianglerighteq}(L^{\mathbb K}_\la)=\{  V_\mu^{\mathbb F}  \mid \mu \in \Pi , 
   V_\mu^\R\otimes_\R {\mathbb K}  \text{ has $L_\la^{\mathbb K}$ as a composition factor}\}$.  Then the set $\mathscr{S}_{\trianglerighteq}(L^{\mathbb K}_\la)$ contains a unique minimal element, $V_\la^{\mathbb F}$,  with respect to $\rhd$.  
\item[$(ii)$]    
There exists an injective map   
 ${\rm Irr}(A^{{\mathbb K}}) \to    {\rm Irr}(A^{\mathbb F})$.   
 \item[$(iii)$]     For all $L^{\mathbb K}_\la\in {\rm Irr}(A^{{\mathbb K}})$, we have that  $L^{\mathbb K}_\la$ appears exactly once as a composition factor of  $ V_\la^\R\otimes_\R {\mathbb K} $.      \end{itemize}
 If this holds, we say that  $B^{\mathbb K}_{\rhd}= \{  \la 
 \mid L_\la  \in {\rm Irr}(A^{{\mathbb K}})\}$ is a {\sf canonical basic set} for $A^{\mathbb K}$ and that  
 $${\bf M}_A^{\mathbb K} =  (m_{\la,\mu}) _{\lambda\in \Pi,  \mu\in 
 B^{\mathbb K}_{\rhd}}
\qquad m_{\lambda\mu} 
  =  [ V^{\mathbb K}_\la      :L^{\mathbb K}_\mu]  $$ is the {\sf modular decomposition matrix};  this matrix is uni-triangular with respect to $\rhd$   by $(i)$ and $(iii)$.  
  \end{defn}


\begin{defn}[{\cite[Definition 2.1]{hm10}}]\label{defn1}  
Suppose  $A^\R$ is a $\ZZ$-graded $\R$-algebra  of finite rank over $\R$.
We say that $A$ is a 
{\sf graded cellular algebra} if the following conditions hold.
 The algebra is equipped with a  {datum} $(\Pi,\TSStd,C,\degr)$, where $(\Pi,\trianglerighteq )$ is the {\sf weight poset}.
For each $\lambda \in\Pi$  we have a finite set, denoted $\TSStd(\lambda )$.  
There exist maps
\[
 \textstyle C:{\coprod_{\lambda\in\Pi}\TSStd(\lambda)\times \TSStd(\lambda)}\to A^\R;
      \quad\text{and}\quad
     \degr: {\coprod_{\lambda\in\Pi}\TSStd(\lambda)} \to \R
\]
such that $C$ is injective. We denote $C(\SSTS,\SSTT) = c^\lambda_{\SSTS\SSTT}$  for $\SSTS,\SSTT\in\TSStd(\lambda)$, and
  \begin{enumerate}[leftmargin=*]
    \item[$(1)$] Each    $c^\lambda_{\SSTS\SSTT}$ is homogeneous
	of degree 
$\degr
        (c^\lambda_{\SSTS\SSTT})=\degr(\SSTS)+\degr(\SSTT),$ for
        $\lambda\in\Pi$ and 
      $\SSTS,\SSTT\in \TSStd(\lambda)$.
    \item[$(2)$] The set $\{c^\lambda_{\SSTS\SSTT}\mid\SSTS,\SSTT\in \TSStd(\lambda), \,
      \lambda\in\Pi \}$ is a  
      $\R$-basis of $A^\R$.
    \item[$(3)$]  If $\SSTS,\SSTT\in \TSStd(\lambda)$, for some
      $\lambda\in\Pi$, and $a\in A^\R$ then 
    there exist scalars $r_{\SSTS\SSTU}(a)$, which do not depend on
    $\SSTT$, such that 
      \[ac^\lambda_{\SSTS\SSTT}  =\sum_{\SSTU\in
      \TSStd(\lambda)}r_{\SSTS\SSTU}(a)c^\lambda_{\SSTU\SSTT}\pmod 
      {A     ^{\vartriangleright  \lambda}},\]
      where $A^{\vartriangleright  \lambda} $ is the $ \R$-submodule of $A^\R$ spanned by
$\{c^\mu_{\SSTQ\SSTR}\mid\mu \vartriangleright  \lambda\text{ and }\SSTQ,\SSTR\in \TSStd(\mu )\}.$ 
    \item[$(4)$]  The $\CC$-linear map $*:A^\R\to A^\R$ determined by
      $(c^\lambda_{\SSTS\SSTT})^*=c^\lambda_{\SSTT\SSTS}$, for all $\lambda\in\Pi$ and
      all $\SSTS,\SSTT\in\TSStd(\lambda)$, is an anti-isomorphism of $A^\R$.
   \end{enumerate}
\end{defn}

 Given   $\lambda\in\Pi$, the   {\sf graded cell module} $\Delta_A^\R(\lambda)$ is the graded left $A^\R$-module
  with basis
    $\{c^\lambda_{\SSTS } \mid \SSTS\in \TSStd(\lambda ) \}$.  
    The action of $A^\R$ on $\Delta_A^\R(\lambda)$ is given by
    \[a c^\lambda_{ \SSTS  }  =\textstyle \sum_{ \SSTU \in \TSStd(\lambda)}r_{\SSTS\SSTU}(a) c^\lambda_{\SSTU},\]
    where the scalars $r_{\SSTS\SSTU}(a)$ are the scalars appearing in condition (3) of Definition \ref{defn1}.  
 Suppose that $\lambda \in \Pi$. There is a bilinear form $\langle\ ,\ \rangle_\lambda$ on $\Delta_A^\R(\lambda) $ which
is determined by
\[c^\lambda_{\SSTU\SSTS}c^\lambda_{\SSTT\SSTV}\equiv
  \langle c^\lambda_\SSTS,c^\lambda_\SSTT\rangle_\lambda c^\lambda_{\SSTU\SSTV}\pmod{A ^{\vartriangleright \lambda}} ,\]
for any $\SSTS,\SSTT, \SSTU,\SSTV\in \TSStd(\lambda  )$.  
For every $\lambda \in \Pi$,   we let   $\langle\ ,\ \rangle_\lambda$ denote the bilinear form on $\Delta^\R(\la)$ 
and let  $\rad^\R_A \langle\ ,\ \rangle_\lambda$ denote the  radical of this bilinear form. 
We define $\Delta_A^{\mathbb K}(\lambda)= \Delta_A^\R(\lambda)\otimes _\R {\mathbb K}$ and extend all  the notation above in the obvious manner.   
We set $\Lambda^{\mathbb K}_\rhd=\{
\la 
 \mid  \rad^{\mathbb K} \langle\ ,\ \rangle_\lambda\neq \Delta_A^{\mathbb K}(\la)\}$ and  we  
set   $L^{\mathbb K}_A(\lambda)=\Delta^{\mathbb K} _A(\lambda) / \rad^{\mathbb K} \langle\ ,\ \rangle_\lambda$.   
By \cite[Lemma 2.7]{hm10}, each module  $L^{\mathbb K}_A(\lambda)$ is graded   and simple, and in fact every   irreducible   module is of this form, up to grading shift.   
    The passage between the (graded) cell and   irreducible   modules is recorded in the 
   (graded) {\sf cellular  decomposition matrix},   
  \[ {\bf D}^{{\mathbb K}}_A(t)=(d _{\lambda\mu}(t))_{\lambda\in\Pi,\mu\in \Lambda^{\mathbb K}_\rhd} \qquad d_{\lambda\mu} (t)=\textstyle \sum_{k\in\R} [\Delta_A^{\mathbb K}(\lambda):L_A^{\mathbb K}(\mu)\langle k\rangle]\,t^k \in \mathbb N [t,t^{-1}] . \]
This matrix is uni-triangular with respect to  $\trianglerighteq$;  thus  if $A^{\mathbb F}$ is semisimple,  we have that 
 $ 
\Lambda^{\mathbb K}_\rhd  $ is a canonical basic set for $A^{\mathbb K}$.  
By \cref{uglovv2}$(i)$  canonical basic sets are unique   and so,   matching-up the labelling sets  
 of   semisimple modules and   cell-modules, we     immediately deduce the following:
 
 \begin{prop}\label{uniqueneess}
Suppose that $A^\R$ is   graded cellular  with respect to $\rhd$.  Suppose further   that  $A^{\mathbb F}$ is a semisimple $\mathbb F$-algebra and that    
$A^\R\otimes {\mathbb K}$ has  canonical basic set     $B^{\mathbb K}_{\rhd}$.  
If $V^{\mathbb F}_\la  \cong \Delta_A^{\mathbb F}(\la)$ for all 
$ \la \in   \Pi$, 
    then 
     $\Lambda^{\mathbb K}_\rhd = B_{\rhd}^{\mathbb K}$ and ${\bf D}^{\mathbb K}_A(t) |_{t=1} = {\bf M}_A^{\mathbb K} $.  
 \end{prop}

 \section{The quiver Hecke algebras}     
Let $\Bbbk$ be an arbitrary integral domain.     We emphasise  that the following presentation of the (quiver) Hecke algebra {only depends on  the reduction of the charge modulo $e$}.  

\begin{defn}[\cite{MR2551762,MR2525917,ROUQ}]\label{defintino1}
Fix $e\in \mathbb N_{>1}$ and $\underline{s}\in (\ZZ/e\ZZ)^\ell$.  
The {\sf quiver Hecke algebra}, $\mathcal{H}^\Bbbk_n(\underline{s})$,   is defined to be the unital, associative, finite-dimensional $\CC$-algebra with generators
\begin{align}\label{gnrs}
\{e(\imath) \ | \ \imath=(i_1,\dots,i_n)\in   (\ZZ/e\ZZ)^n\}\cup\{y_1,\dots,y_n\}\cup\{\psi_1,\dots,\psi_{n-1}\},
\end{align}
subject to the    relations  
\begin{align}
\label{rel1.1} e(\imath)e(\jmath)&=\delta_{\imath,\jmath} e(\imath);
\\
\label{rel1.2}
\sum_{\imath \in   (\ZZ/e\ZZ)^n } e(\imath)&=1;\\ 
\label{rel1.3}
y_r		e(\imath)&=e(\imath)y_r;
\\
\label{rel1.4}
\psi_r e(\imath) &= e(s_{r,r+1}\imath) \psi_r;
\\ 
\label{rel1.5}
y_ry_s&=y_sy_r;\\ 
\label{rel1.6}
\psi_ry	_s&=\mathrlap{y_s\psi_r}\hphantom{\smash{\begin{cases}(\psi_{r+1}\psi_r\psi_{r+1}+y_r-2y_{r+1}+y_{r+2})e(\imath)\\\\\\\end{cases}}}\kern-\nulldelimiterspace\text{if } s\neq r,r+1;\\
\label{rel1.7}
\psi_r\psi_s&=\mathrlap{\psi_s\psi_r}\hphantom{\smash{\begin{cases}(\psi_{r+1}\psi_r\psi_{r+1}+y_r-2y_{r+1}+y_{r+2})e(\imath)\\\\\\\end{cases}}}
\kern-\nulldelimiterspace\text{if } |r-s|>1;
\\
\label{rel1.8}
y_r \psi_r e(\imath) &=(\psi_r y_{r+1} \color{black}-\color{black} 
\delta_{i_r,i_{r+1}})e(\imath);\\
\label{rel1.9}
y_{r+1} \psi_r e(\imath) &=(\psi_r y_r \color{black}+\color{black} 
\delta_{i_r,i_{r+1}})e(\imath);\\
\label{rel1.10}
\psi_r^2 e(\imath)&=\begin{cases}
\mathrlap0\phantom{(\psi_{r+1}\psi_r\psi_{r+1}+y_r-2y_{r+1}+y_{r+2})e(\imath)}& \text{if }i_r=i_{r+1},\\
e(\imath) & \text{if }i_{r+1}\neq i_r, i_r\pm1,\\
(y_{r+1} - y_r) e(\imath) & \text{if }i_{r+1}=i_r-1 \ \& \ e\neq2,\\
(y_r - y_{r+1}) e(\imath) & \text{if }i_{r+1}=i_r+1 \ \& \ e\neq2,\\
(y_{r+1} - y_r)(y_r - y_{r+1}) e(\imath) & \text{if } i_{r+1}\neq i_r \ \& \ e=2;
\end{cases}\\
\label{rel1.11}
\psi_r \psi_{r+1} \psi_r &=\begin{cases}
(\psi_{r+1}\psi_r\psi_{r+1} {\color{black}-} 1)e(\imath)& \text{if }i_r=i_{r+2}=i_{r+1}+1 \ \& \ e\neq2,\\
(\psi_{r+1}\psi_r\psi_{r+1} {\color{black}+} 1)e(\imath)& \text{if }i_r=i_{r+2}=i_{r+1}-1 \ \& \ e\neq2,\\
(\psi_{r+1}\psi_r\psi_{r+1}+y_r-2y_{r+1}+y_{r+2})e(\imath)& \text{if }i_r=i_{r+2}\neq i_{r+1} \ \& \ e=2,\\
(\psi_{r+1}\psi_r\psi_{r+1})e(\imath)& \text{otherwise;}
\end{cases} 
\end{align}
{for all admissible $r,s,i,j$. Finally,   the cyclotomic relation:  for   $\imath\in (\ZZ/e\ZZ)^n$, we have that }
\begin{align}
\label{rel1.12} y_1^{\sharp\{s_m \mid s_m= i_1 	\}} e(\imath) &=0.
\end{align}
 \end{defn}

For ease of notation, we have excluded the  $e=\infty$ case from the usual definition of the quiver Hecke algebra (note that the $e=\infty$ and $e>n$ algebras are isomorphic and so this is not important).   
 
\begin{thm}[\cite{MR2551762,MR2525917,ROUQ}]
We have a grading on $\mathcal{H}^\Bbbk_n(\underline{s})$  given by
$$
{\rm deg}(e(\imath))=0 \quad 
{\rm deg}( y_r)=2\quad 
{\rm deg}(\psi_r e(\imath))=
\begin{cases}
-2		&\text{if }i_r=i_{r+1} \\
1		&\text{if }i_r=i_{r+1}\pm 1 \\
2		&\text{if }$e=2$ \text{ and } i_{r+1}\neq i_r \\ 
0 &\text{otherwise} 
\end{cases} $$
\end{thm} 
 \begin{thm}[{\cite[Main Theorem]{MR2551762}}]\label{citemmmmmee}
 Let $\Bbbk$ be a field.  The algebras $\mathcal{H}^\Bbbk_n(\underline{s})$ and  
 $ {H}^\Bbbk_n(\underline{s})$
 are isomorphic as $\Bbbk$-algebras.  
 \end{thm}

 For the remainder of the paper, we will work in the graded setting of  $\mathcal{H}^\Bbbk_n(\underline{s})$.  
 This is because we wish to prove Theorem A in the generality of an arbitrary integral domain $\Bbbk$
   (for which the isomorphism of \cref{citemmmmmee} fails).  

  \section{Quiver Cherednik algebras}\label{Section2}
\label{relationspageofstuff}   

   We now recall Webster's definition of  the {\sf quiver} (or {\sf diagrammatic}) {\sf Cherednik algebra} $\algebra^\Bbbk_n(\sigma)$ for 
  $n\in \mathbb{N}$ and $\sigma \in \weight$.  
We shall see that the  representation theoretic structure of $\algebra^\Bbbk_n(\sigma)$ 
is {\em heavily dependent} on the charge $\sigma \in \weight$.   This is in stark contrast with  
 the structure of  the (quiver) Hecke algebra   which we have seen  is  dependent    only on the modulo $e$ reduction: $( s_0, \dots , s_{\ell-1})\in(\ZZ/e\ZZ)^\ell$.   
 In other words,  the quiver Cherednik algebras {have the extra 
 combinatorial information of \cref{sec1} baked into their definition}.  
 In \cref{mainresults}, we shall apply the {many ``charged"  Schur functors to these quiver Cherednik algebras}  
 in order to obtain {many new presentations of the quiver Hecke algebra which encode the richer structures which cannot be detected using either the classical or KLR presentations}.  
  We have written this section in the style of  
  a self-contained beginner's  guide to the diagrammatic theory and have included many examples.

  \begin{figure}[ht!]
 \[
\begin{tikzpicture}[baseline, thick,yscale=-1,xscale=-3.5]
\scalefont{0.8}

  \draw[wei]  (-1.1+0.5, -2)  to[out=90,in=-90] (-1.1+0.5, 2)  [below, at end] node{$0$};

 
 \draw(-1,-2) rectangle (1.6,2);

\draw[gray, densely dotted, gray!80]  (-.5-0.4,-2)    to[out=60,in=-90](1-0.4,.2) to[out=90,in=-90] 
   (0-0.4,2);

  \draw[gray, densely dotted, gray!80] (.5-0.4,-2) to[out=93,in=-93]   (.5-0.4,0) to[out=87,in=-90]  
  (1-0.4,2);

  \draw[gray, densely dotted, gray!80]  (-0.4+1,-2) to[out=90,in=-90]  (-0.5-0.4,1) to[out=90,in=-90] (-0.4+.5,2);

  \draw[gray, densely dotted, gray!80] (-0.4+0,-2)to     (-0.4,0) to[out=90,in=-90]  node[midway,circle,fill=gray!85,densely dotted, gray,inner sep=2pt]{}
  (-0.4+-.5,2);

  \draw (-.5,-2)    to[out=60,in=-90](1,.2) to[out=90,in=-90] 
  node[below,at end]{$1$}  (0,2);
  \draw (.5,-2) to[out=93,in=-93]   (.5,0) to[out=87,in=-90]  
  node[below,at end]{$2$} (1,2);
  \draw  (1,-2) to[out=90,in=-90]    (-0.5,1) to[out=90,in=-90]
  node[below,at end]{$0$ 
  }   (.5,2);
  \draw (0,-2) to[out=90,in=-90]  
  (-0.05,0)  to[out=90,in=-90] node[below,at end]{$0$ 
  }
   node[midway,circle,fill=black,inner sep=2pt]{}  (-.5,2);
   
  \draw  (1.5,-2) to[out=90,in=-90]    (1.15,0.2) to[out=90,in=-90]
  node[below,at end]{$0$}   (1.5,2);

  \draw[densely dotted, gray!80]  (1.5-0.4,-2) to[out=90,in=-90]    (1.15-0.4,0.2) to[out=90,in=-90]
    (1.5-0.4,2);

\end{tikzpicture} \]
 \caption{
A $\theta$-diagram, $\Diag\in \algebra^\Bbbk_5(3;0)$,   with northern and southern loading
  ${\bf I}^{\theta}_\omega$ for $\omega= (1^5)$.    }
\label{initialdiagram}
\end{figure}  
     
\begin{defn}
We define a  $\theta$-{\sf diagram} of rank $n\in \NN$ and type $(\mu,\la)$ to  be a  {\sf frame},
 $\mathbb{R}\times [0,1]$, with  $n$ distinguished  solid points on the northern and southern boundaries
 given by  $\mathbf{I}^\theta_\mu$ and $\mathbf{I}^\theta_\lambda$ for  
  $\lambda,\mu  \in \con \ell n$ and a 
 collection of {\sf solid strands} each of which starts at a northern point 
 and ends at a southern
  point.  Each solid strand carries a  residue, $i \in \ZZ/e\ZZ$,  say (and we refer to this as a {\sf solid $i$-strand}).  
     We further require that each solid strand has a 
 mapping diffeomorphically to $[0,1]$ via the projection to the $y$-axis.  
  Each solid strand can carry a finite number of dots. 
   We draw
   \begin{itemize}
\item[$(i)$] a ``ghost $i$-strand" $1$ unit  to the right \color{black} of each solid $i$-strand and a 
a ``ghost dot" $1$ unit  to the   right \color{black} of each solid dot; 
\item[$(ii)$]   vertical red lines with  $x$-coordinate $\theta_m  -m/\ell \in 
\QQ$ each of which carries a 
residue $ {s}_m\in \ZZ/e\ZZ$ for $1\leq m\leq \ell$ which we call a {\sf red $s_m$-strand}.
\end{itemize}
Finally, we require that there are no triple points or tangencies involving any combination of strands, ghosts or red lines and no dots lie on crossings. 
 We consider these diagrams {\sf equivalent} if they are related by an isotopy that avoids these tangencies, double points and dots on crossings. 
 \end{defn}

\begin{defn} \label{grsubsec}
We define   the degree of a  $\theta$-diagram  to be the integer obtained  
  by summing over the degrees of all the local neighbourhoods of the diagram, with each neighbourhood contributing to the degree as follows: 
 
\!\!\!\[
\deg\tikz[baseline,very thick,xscale=-1.5,yscale=1.5]{\draw
  (0,.3) -- (0,-.1) node[at end,below,scale=.8]{$i$}
  node[midway,circle,fill=black,inner
  sep=2pt]{};}=
  2 \qquad \deg\tikz[baseline,very thick,scale=1.5]
  {\draw (.2,.3) --
    (-.2,-.1) node[at end,below, scale=.8]{$i$}; \draw
    (.2,-.1) -- (-.2,.3) node[at start,below,scale=.8]{$j$};} =-2\delta_{i,j} \qquad  
  \deg\tikz[baseline,very thick,scale=1.5]{\draw[densely dotted, black] 
  (-.2,-.1)-- (.2,.3) node[at start,below, scale=.8]{$i$}; \draw
  (.2,-.1) -- (-.2,.3) node[at start,below,scale=.8]{$j$};} =\delta_{j,i+1} \qquad %
     \deg\tikz[baseline,very thick,scale=1.5]{ \draw[wei]
  (-.2,-.1)-- (.2,.3) node[at start,below, scale=.8]{$i$}; \draw
  (.2,-.1) -- (-.2,.3) node[at start,below,scale=.8]{$j$};} =\delta_{i,j} 
 \]
and their mirror images.  
\end{defn}

 \begin{defn}
 Let $D$ be a   $\theta$-{\sf diagram}.  We define the {\sf northern} (respectively {\sf southern}) {\sf ordered residue sequence} of $D$ to be the element of $(\ZZ/e\ZZ)^n$ given by reading the residues of the solid strands in $D$ from left to  right along the northern (respectively southern) edge of the frame.  
  \end{defn}
   
  \begin{defn}
  Let $D$ be a  $\theta$-diagram. Suppose $D$ 
 has  distinguished  solid points on the northern and southern boundaries
 given by  $\mathbf{I}^\theta_\mu$ and $\mathbf{I}^\theta_\la$ and northern and southern residue sequence given by $\imath$ and $\jmath \in (\ZZ/e\ZZ)^n$ respectively.  
  We say that a diagram $D  $ is {\sf reduced} if $(i)$ when read from south-to-north $D$   traces out a bijection $:[\la]\to[\mu]$ using the {\em minimal} number of crossings between strands 
  and $(ii)$ $D$ has no dots on any strands.   
 We let ${^\imath_\mu\!\mathcal{R}^\jmath_\la}$ denote the set of all such reduced diagrams.  
 \end{defn}

\begin{eg}
Let $\theta=(3;0)\in \NN_{>1}\times\ZZ$ and $\la=\mu=(1^5)$.    
In \cref{initialdiagram} we picture a $\theta$-diagram.  \end{eg}

  \label{defintino2}
\begin{defn} The {\sf quiver Cherednik algebra}, $\algebra^\Bbbk_n(\theta)$,  is the associative   $\CC$-algebra generated (as a $\Bbbk$-module) by   all inequivalent   $\theta$-diagrams  modulo the   local relations  \ref{rel1} to \ref{rel15} below (here a local relation means one that 
 can be specified by its effect on an arbitrarily small region of the diagram).   
   The product $\acell_1 \acell_2$ of two diagrams $\acell_1,\acell_2 \in \algebra^\Bbbk_n(\theta)$ is then given by putting $\acell_1$ on top of $\acell_2$.
This product is defined to be $0$ unless the southern border of $\acell_1$ is given by the same loading as the northern border of $\acell_2$ with residues of strands matching in the obvious manner, in which case we obtain a new diagram with loading and labels inherited from those of $\acell_1$ and $\acell_2$.  
\end{defn}
    
  \subsection*{Isotopy and   dots through crossings} These are the easiest relations in the quiver Cherednik algebra.  They also serve as a reminder that when we apply a  relation  in a region containing a 
   solid/ghost strand, we must also also has an effect on its corresponding ghost/solid strand 1 unit to the right/left.

\begin{enumerate}[label=(A\arabic*),leftmargin=*] 
\item\label{rel1}  Any diagram may be deformed isotopically; that is,
 by a continuous deformation
 of the diagram which  
 avoids  tangencies, double points and dots on crossings. 
\item\label{rel2} 
Any solid dot     can pass through a crossing   of solid   $i$- and $j$-strands for $i\neq j$ or an arbitrary crossing involving a  ghost strand.  Namely: 
\[   \scalefont{0.8}\begin{tikzpicture}[yscale=0.5,xscale=-0.5,baseline]
  \draw[very thick](-4,0) +(-1,-1) -- +(1,1) node[below,at start]
  {$i$}; \draw[very thick](-4,0) +(1,-1) -- +(-1,1) node[below,at
  start] {$j$}; \fill (-4.5,.5) circle (5pt);
    \node at (-2,0){$=$}; \draw[very thick](0,0) +(-1,-1) -- +(1,1)
  node[below,at start] {$i$}; \draw[very thick](0,0) +(1,-1) --
  +(-1,1) node[below,at start] {$j$}; \fill (.5,-.5) circle (5pt);
  \node at (4,0){ };
\end{tikzpicture}\quad
\begin{tikzpicture}[yscale=0.5,xscale=-0.5,baseline]
  \draw[densely dotted,very thick](-4,0) +(-1,-1) -- +(1,1) node[below,at start]
  {$i$}; \draw[very thick] (-4,0) +(1,-1) -- +(-1,1) node[below,at
  start] {$j$}; \fill (-4.5,.5) circle  (5pt);
    \node at (-2,0){$=$}; \draw[densely dotted,very thick](0,0) +(-1,-1) -- +(1,1)
  node[below,at start] {$i$}; \draw[very thick] (0,0) +(1,-1) --
  +(-1,1) node[below,at start] {$j$}; \fill   (.5,-.5) circle (5pt);
  \node at (4,0){ };
\end{tikzpicture} \quad
\begin{tikzpicture}[yscale=0.5,xscale=-0.5,baseline]
  \draw[densely dotted,very thick](-4,0) +(-1,-1) -- +(1,1) node[below,at start]
  {$i$}; \draw[ very thick] (-4,0) +(1,-1) -- +(-1,1) node[below,at
  start] {$i$}; \fill (-4.5,.5) circle  (5pt);
    \node at (-2,0){$=$}; \draw[densely dotted,very thick](0,0) +(-1,-1) -- +(1,1)
  node[below,at start] {$i$}; \draw[very thick] (0,0) +(1,-1) --
  +(-1,1) node[below,at start] {$i$}; \fill   (.5,-.5) circle (5pt);
  \node at (4,0){ };
\end{tikzpicture}\]
and their mirror images through reflection in the vertical axis hold.

 \item\label{rel3}  We can pass a solid  dot through a
  crossing of  two like-labelled   solid or ghost strands at the expense of  an error term:
\[
\scalefont{0.8}\begin{tikzpicture}[yscale=0.5,xscale=0.5,baseline]
  \draw[very thick](-4,0) +(-1,-1) -- +(1,1) node[below,at start]
  {$i$}; \draw[very thick](-4,0) +(1,-1) -- +(-1,1) node[below,at
  start] {$i$}; \fill (-3.5,.5) circle (5pt);
   \node at (-2,0){$=$}; \draw[very thick](0,0) +(-1,-1) -- +(1,1)
  node[below,at start] {$i$}; \draw[very thick](0,0) +(1,-1) --
  +(-1,1) node[below,at start] {$i$}; \fill (-.5,-.5) circle (5pt);
  \node at (2,0){$-$};
   \draw[very thick](4,0) +(-1,-1) -- +(-1,1)
  node[below,at start] {$i$};
   \draw[very thick](4,0) +(0,-1) --
  +(0,1) node[below,at start] {$i$};
\end{tikzpicture}  \quad \quad \quad\quad
\scalefont{0.8}\begin{tikzpicture}[yscale=0.5,xscale=0.5,baseline]
  \draw[very thick](-4,0) +(-1,-1) -- +(1,1) node[below,at start]
  {$i$}; \draw[very thick](-4,0) +(1,-1) -- +(-1,1) node[below,at
  start] {$i$}; \fill (-3.5,-.5) circle (5pt);
       \node at (-2,0){$=$}; \draw[very thick](0,0) +(-1,-1) -- +(1,1)
  node[below,at start] {$i$}; \draw[very thick](0,0) +(1,-1) --
  +(-1,1) node[below,at start] {$i$}; \fill (-.5,.5) circle (5pt);
  \node at (2,0){$-$}; \draw[very thick](4,0) +(-1,-1) -- +(-1,1)
  node[below,at start] {$i$}; \draw[very thick](4,0) +(0,-1) --
  +(0,1) node[below,at start] {$i$};
\end{tikzpicture}\]
\end{enumerate} 
Ghost dots can pass through any crossing of strands (regardless of their residue) freely.

\begin{eg}For example in \cref{dotcross} we  apply relation \ref{rel3} locally to a region of the   diagram containing the dot in \cref{initialdiagram}; 
 however, moving the dot means we must also move the ghost dot (which must always be 1 unit to the right)
 and undoing the crossing of solid $0$-strands means we must undo the corresponding crossing of ghost strands as in \cref{initialdiagram}.  

\begin{figure}[ht!]

\begin{tikzpicture}[baseline, thick,yscale=-0.8,xscale=-2.8]
\scalefont{0.8}

  \draw[wei]  (-1.1+0.5, -2)  to[out=90,in=-90] (-1.1+0.5, 2)  [below, at end] node{$0$};

 
 \draw(-1,-2) rectangle (1.6,2);

\draw[gray, densely dotted,gray, ]  (-.5-0.4,-2)    to[out=60,in=-90](1-0.4,.2) to[out=90,in=-90] 
   (0-0.4,2);

  \draw[gray, densely dotted,gray, ] (.5-0.4,-2) to[out=93,in=-93]   (.5-0.4,0) to[out=87,in=-90]  
  (1-0.4,2);

  \draw[gray, densely dotted,gray, ]  (-0.4+1,-2) to[out=90,in=-90]  (-0.5-0.4,1) to[out=90,in=-90] (-0.4+.5,2);

  \draw[gray, densely dotted,gray, ] (-0.4+0,-2)to     (-0.4,0) to[out=90,in=-90]  node[very near end,circle,fill=gray!85,densely dotted,gray, gray,inner sep=2pt]{}
  (-0.4+-.5,2);

  \draw (-.5,-2)    to[out=60,in=-90](1,.2) to[out=90,in=-90] 
  node[below,at end]{$1$}  (0,2);
  \draw (.5,-2) to[out=93,in=-93]   (.5,0) to[out=87,in=-90]  
  node[below,at end]{$2$} (1,2);
  \draw  (1,-2) to[out=90,in=-90]    (-0.5,1) to[out=90,in=-90]
  node[below,at end]{$0$ 
  }   (.5,2);
  \draw (0,-2) to[out=90,in=-90]  
  (-0.05,0)  to[out=90,in=-90] node[below,at end]{$0$ 
  }
   node[very near end,circle,fill=black,inner sep=2pt]{}  (-.5,2);
   
  \draw  (1.5,-2) to[out=90,in=-90]    (1.15,0.2) to[out=90,in=-90]
  node[below,at end]{$0$}   (1.5,2);

  \draw[densely dotted,gray, ]  (1.5-0.4,-2) to[out=90,in=-90]    (1.15-0.4,0.2) to[out=90,in=-90]
    (1.5-0.4,2);

\end{tikzpicture} \;  + \; 
\begin{tikzpicture}[baseline, thick,yscale=-0.8,xscale=-2.8]
\scalefont{0.8}

  \draw[wei]  (-1.1+0.5, -2)  to[out=90,in=-90] (-1.1+0.5, 2)  [below, at end] node{$0$};

 
 \draw(-1,-2) rectangle (1.6,2);

\draw[gray, densely dotted,gray, ]  (-.5-0.4,-2)    to[out=60,in=-90](1-0.4,.2) to[out=90,in=-90] 
   (0-0.4,2);

  \draw[gray, densely dotted,gray, ] (.5-0.4,-2) to[out=93,in=-93]   (.5-0.4,0) to[out=87,in=-90]  
  (1-0.4,2);

  \draw[densely dotted,gray]  (1-0.4,-2) to[out=90,in=-90]    (-0.5-0.4,0.8) to[out=90,in=-90]   (-0.4-0.4,1.3)to[out=90,in=-90]
     (-.5-0.4,2);

   \draw[densely dotted,gray] (0-0.4,-2) to[out=90,in=-90]   
  (0.05-0.5,0) to [out=90,in=-90] (-0.3-0.4,1.1)  to[out=90,in=-90]   (.5-0.4,2);

  \draw (.5,-2) to[out=96,in=-94]   (.5,0) to[out=86,in=-90]  
  node[below,at end]{$2$} (1,2);
      
  \draw  (1.5,-2) to[out=90,in=-90]    (1.15,0.2) to[out=90,in=-90]
  node[below,at end]{$0$}   (1.5,2);

  \draw[densely dotted,gray, ]  (1.5-0.4,-2) to[out=90,in=-90]    (1.15-0.4,0.2) to[out=90,in=-90]
    (1.5-0.4,2);

   \draw (-.5,-2)    to[out=60,in=-90](1,.2) to[out=90,in=-90] 
  node[below,at end]{$1$}  (0,2);

  \draw   (1,-2) to[out=90,in=-90]    (-0.5,0.8) to[out=90,in=-90]   (-0.4,1.3)to[out=90,in=-90]
  node[below,at end]{$0$ 
  }   (-.5,2);

   \draw (0,-2) to[out=90,in=-90]   
  (-.05,0) to [out=90,in=-90] (-0.3,1.1)  to[out=90,in=-90] node[below,at end]{$0$ 
  }  (.5,2);

\end{tikzpicture} 
\caption{We apply relation \ref{rel3} to \cref{initialdiagram} in order to move the dot through the crossing at the expense of an error term.  }
\label{dotcross}
\end{figure}
\end{eg}

\subsection*{Undoing double-crossings}
Now we consider how one can undo a pair of strands which cross and then cross again.  The first of these relations,  relation \ref{rel4},  should be familiar from the classical KLR algebra.  
\begin{enumerate}[resume, label=(A\arabic*),leftmargin=*]  
\item\label{rel4} For double-crossings of solid strands with $i\neq j$, we have the following local relations:
\[
\scalefont{0.8}\begin{tikzpicture}[very thick,xscale=-0.5,yscale=0.5,baseline,yscale=1.3]
\draw (-2.8,-1) .. controls (-1.2,0) ..  +(0,2)
node[below,at start]{$i$};
\draw (-1.2,-1) .. controls (-2.8,0) ..  +(0,2) node[below,at start]{$i$};
\node at (-.5,0) {$=$};
\node at (0.4,0) {$0$};
\end{tikzpicture}
\hspace{2cm}
 \begin{tikzpicture}[very thick,xscale=0.5,yscale=0.5,baseline,yscale=1.3]
\draw (-2.8,-1) .. controls (-1.2,0) ..  +(0,2)
node[below,at start]{$i$};
\draw (-1.2,-1) .. controls (-2.8,0) ..  +(0,2)
node[below,at start]{$j$};
\node at (-.5,0) {$=$}; 

\draw (1.8,-1) -- +(0,2) node[below,at start]{$j$};
\draw (1,-1) -- +(0,2) node[below,at start]{$i$}; 
\end{tikzpicture}
\]
\end{enumerate}
  Performing relation \ref{rel4} implicitly involves  undoing the   corresponding double-crossing of ghost strands  at the same time (which we do not picture) and vice versa.  

\begin{eg}
The leftmost diagram in \cref{dotcross} has a double-crossing of two solid 0-strands and  
therefore this leftmost diagram is zero by relation \ref{rel4}.  
 (The observant reader might worry about the fact that a red 0-strand crosses the ghosts of these 0-strands --- however, we shall see that this is not a problem in  relation \ref{rel11}.)
\end{eg}

 
\begin{enumerate}[resume, label=(A\arabic*),leftmargin=*]

\item\label{rel5} If $j\neq i-1$,  then we can freely pass ghosts through solid strands.  That is, we have the following local relations:
\[\scalefont{0.8}
\hspace{.7cm}
 \begin{tikzpicture}[very thick,xscale=-0.5,yscale=0.5,baseline,yscale=1.3]
\draw (-2.8,-1) .. controls (-1.2,0) ..  +(0,2)
node[below,at start]{$i$};
\draw[densely dotted, ] (-1.2,-1) .. controls (-2.8,0) ..  +(0,2)
node[below,at start]{$j$};
\node at (-.5,0) {$=$}; 

\draw[densely dotted, ] (1.8,-1) -- +(0,2) node[below,at start]{$j$};
\draw (1,-1) -- +(0,2) node[below,at start]{$i$}; 
\end{tikzpicture}
\quad\quad \quad \quad 
\hspace{.7cm}\begin{tikzpicture}[very thick,xscale=-0.5,yscale=0.5,baseline,yscale=1.3]
\draw[densely dotted, ] (-2.8,-1) .. controls (-1.2,0) ..  +(0,2)
node[below,at start]{$j$};
\draw (-1.2,-1) .. controls (-2.8,0) ..  +(0,2)
node[below,at start]{$i$};
\node at (-.5,0) {$=$}; 

\draw (1.8,-1) -- +(0,2) node[below,at start]{$i$};
\draw[densely dotted, ] (1,-1) -- +(0,2) node[below,at start]{$j$}; 
\end{tikzpicture}
\]
\end{enumerate} 
\begin{enumerate}[resume, label=(A\arabic*),leftmargin=*]  
\item\label{rel6}\label{rel7}  On the other hand, in the case where $j= i-1$, we have the following local relations:
\[\scalefont{0.8}
 \begin{tikzpicture}[very thick,xscale=0.5,yscale=0.5,baseline,yscale=1.3]

\draw[densely dotted, ] (-2.8-0.3,-1) .. controls (-1.2-0.3,0) ..  +(0,2)
node[below,at start]{$i\;\!$--$\;\!1$};
\draw  (-1.2-0.3,-1) .. controls (-2.8-0.3,0) ..  +(0,2)
node[below,at start]{$i$};
\node at (-.4,0) {$=$};

\draw[densely dotted, ] (0.9,-1) .. controls (1.4,0) ..  +(0,2)
node[below,at start]{$i\;\!$--$\;\!1$};
\draw  (2.5,-1) .. controls (2.1,0) ..  +(0,2)
node[below,at start]{$i$}  node[midway,fill=black,inner sep=2.5pt,circle]{};

\node at (4,0) {$-$};

\draw[densely dotted, ] (0.9+4.5,-1) .. controls (1.4+4.5,0) ..  +(0,2)
node[below,at start]{$i\;\!$--$\;\!1$} node[midway,fill=gray!80,inner sep=2.5pt,circle]{}; 
\draw  (2.5+4.5,-1) .. controls (2.1+4.5,0) ..  +(0,2)
node[below,at start]{$i$ }  ;
 
\end{tikzpicture}
 \qquad\qquad\qquad \begin{tikzpicture}[very thick,xscale=0.5,yscale=0.5,baseline,yscale=1.3]
\draw  (-2.8-0.3,-1) .. controls (-1.2-0.3,0) ..  +(0,2)
node[below,at start]{ $i$};
\draw[densely dotted, ]  (-1.2-0.3,-1) .. controls (-2.8-0.3,0) ..  +(0,2)
node[below,at start]{$i\;\!$--$\;\!1$};
\node at (-.4,0) {$=$}; 

\draw  (0.9,-1) .. controls (1.4,0) ..  +(0,2)
node[below,at start]{$i$} node[midway,fill=black,inner sep=2.5pt,circle]{}; 

\draw[densely dotted, ] (2.5,-1) .. controls (2.1,0) ..  +(0,2)
node[below,at start]{$i\;\!$--$\;\!1$}  ;

\node at (4,0) {$-$};

\draw  (0.9+4.5,-1) .. controls (1.4+4.5,0) ..  +(0,2)
node[below,at start]{ $i$} ;

\draw[densely dotted, ]  (2.5+4.5,-1) .. controls (2.1+4.5,0) ..  +(0,2)
node[below,at start]{$i\;\!$--$\;\!1$}  node[midway,fill=gray!80,inner sep=2.5pt,circle]{};

\end{tikzpicture}
\]
\end{enumerate}

\begin{rmk}\label{mevsyou}It is worth noting that the local diagrammatic regions pictured in the left and right hand sides of 
 relation \ref{rel6} do not have the same degree.  This is because black dots carry degree 2 and ghost dots carry degree 0.  However, we emphasise that by creating a ghost dot (in the local region pictured) we also create de facto solid dot (not pictured!) elsewhere in the diagram.  Thus the overall degree of the diagrams is preserved (as one should expect!).  An example  of how this works in a wider diagram is pictured in \cref{adoublecross}.  
 \end{rmk}

\begin{figure}[ht!]
 \[
 \begin{tikzpicture}[baseline, thick,yscale=-0.8,xscale=-2.8]
\scalefont{0.8}

  \draw[wei]  (-1.1+0.5, -2)  to[out=90,in=-90] (-1.1+0.5, 2)  [below, at end] node{$0$};

 
 \draw(-1,-2) rectangle (1.6,2);

  \draw[densely dotted,gray]  (1-0.4,-2) to[out=90,in=-90]    (-0.5-0.4,0.8) to[out=90,in=-90]   (-0.4-0.4,1.3)to[out=90,in=-90]
   (-.5-0.4,2);

   \draw[densely dotted,gray] (0-0.4,-2) to[out=90,in=-90]   
  (0.05-0.5,0) to [out=90,in=-90] (-0.3-0.4,1.1)  to[out=90,in=-90]   (.5-0.4,2);

  \draw   (1,-2) to[out=90,in=-90]    (-0.5,0.8) to[out=90,in=-90]   (-0.4,1.3)to[out=90,in=-90]
  node[below,at end]{$0$ 
  }   (-.5,2);

   \draw (0,-2) to[out=90,in=-90]   
  (-.05,0) to [out=90,in=-90] (-0.3,1.1)  to[out=90,in=-90] node[below,at end]{$0$ 
  }  (.5,2);

   \draw  (1.5,-2) to[out=90,in=-90]    (1.15,0.2) to[out=90,in=-90]
  node[below,at end]{$0$}   (1.5,2);

  \draw[densely dotted,gray, gray!80]  (1.5-0.4,-2) to[out=90,in=-90]    (1.15-0.4,0.2) to[out=90,in=-90]
    (1.5-0.4,2);

  \draw  (-.5,-2)    to[out=55,in=-90](1,.2) to[out=90,in=-90] 
  node[below,at end]{$1$}  (0,2);
  \draw[gray, densely dotted,gray ]  (-.5-0.4,-2)    to[out=55,in=-90](1-0.4-0.1,.2) to[out=90,in=-90] 
   (0-0.4,2);

    \draw[gray, densely dotted,gray] (.5-0.4,-2) to[out=95,in=-93]   (.5-0.4,-0.2) 
    to[out=87,in=-90]   node[near start,circle,fill=gray!85,inner sep=2pt]{} 
  (1-0.4,2);
 \draw  (.5,-2) to[out=95,in=-93]   (.5,-0.2) to[out=87,in=-90]   node[near start,circle,fill=black,inner sep=2pt]{} 
  node[below,at end]{$2$} (1,2);

\end{tikzpicture} \; - \; 
\begin{tikzpicture}[baseline, thick,yscale=-0.8,xscale=-2.8]
\scalefont{0.8}

  \draw[wei]  (-1.1+0.5, -2)  to[out=90,in=-90] (-1.1+0.5, 2)  [below, at end] node{$0$};

 
 \draw(-1,-2) rectangle (1.6,2);

  \draw[densely dotted,gray]  (1-0.4,-2) to[out=90,in=-90]    (-0.5-0.4,0.8) to[out=90,in=-90]   (-0.4-0.4,1.3)to[out=90,in=-90]
   (-.5-0.4,2);

   \draw[densely dotted,gray] (0-0.4,-2) to[out=90,in=-90]   
  (0.05-0.5,0) to [out=90,in=-90] (-0.3-0.4,1.1)  to[out=90,in=-90]   (.5-0.4,2);

  \draw   (1,-2) to[out=90,in=-90]    (-0.5,0.8) to[out=90,in=-90]   (-0.4,1.3)to[out=90,in=-90]
  node[below,at end]{$0$ 
  }   (-.5,2);

   \draw (0,-2) to[out=90,in=-90]   
  (-.05,0) to [out=90,in=-90] (-0.3,1.1)  to[out=90,in=-90] node[below,at end]{$0$ 
  }  (.5,2);

  \draw  (1.5,-2) to[out=90,in=-90]    (1.15,0.2) to[out=90,in=-90]
  node[below,at end]{$0$}   (1.5,2);

  \draw[densely dotted,gray, gray!80]  (1.5-0.4,-2) to[out=90,in=-90]    (1.15-0.4,0.2) to[out=90,in=-90]
    (1.5-0.4,2);

  \draw  (-.5,-2)    to[out=55,in=-90](1,.2) to[out=90,in=-90] 
  node[very near start,circle,fill=black,inner sep=2pt]{}   node[below,at end]{$1$}  (0,2);
  \draw[gray, densely dotted,gray ]  (-.5-0.4,-2)    to[out=55,in=-90]   (1-0.4-0.1,.2) to[out=90,in=-90] 
  node[very near start,circle,fill=gray!85,inner sep=2pt]{}  (0-0.4,2);

    \draw[gray, densely dotted,gray] (.5-0.4,-2) to[out=95,in=-93]   (.5-0.4,-0.2) to[out=87,in=-90]  
  (1-0.4,2);
 \draw  (.5,-2) to[out=95,in=-93]   (.5,-0.2) to[out=87,in=-90]  
  node[below,at end]{$2$} (1,2);

\end{tikzpicture} 
\]

\!\!\!\!
\caption{Undoing the double-crossing of the ghost $1$-strand and a solid $2$-strand in the rightmost diagram of  \cref{dotcross} using relation \ref{rel7}}
\label{adoublecross}
\end{figure}

\begin{eg}
The rightmost diagram in \cref{dotcross} has a double-crossing of a ghost $1$-strand and a solid 2-strand.  We can continuously deform these strands until they are infinitesimally close together (without creating any tangencies of double points elsewhere in the diagram) and hence undo  this double-crossing using relation \ref{rel7}.   
This is depicted in \cref{adoublecross}.    
(The other diagram  in \cref{dotcross}   has this same double-crossing, but we 
have already seen that this diagram is zero using relation \ref{rel3}.)  
We emphasise that there is a 0-ghost strand which passes through the region between the solid 1-strand and its ghost. It is clear   that this does not hamper our ability to  apply relation  \ref{rel7} to the     region  containing the double-crossing of a ghost $1$-strand and a solid 2-strand.

\end{eg}

\!\!\!\!\!
\subsection*{Pulling a   strand through a    crossing}
We now consider the effect of pulling a strand through a pair of crossing strands.  In other words, our graded versions of the classical braid relation.  
   \begin{enumerate}[resume, label=(A\arabic*),leftmargin=*]  
\item\label{rel9}\label{rel10}    We can pull a solid $i$-strand through a $(i-1)$-ghost-crossing (or a ghost $(i-1)$-strand through a $i$-solid-crossing) at the expense of an error term.     
  \[
\scalefont{0.8}\begin{tikzpicture}[very thick,xscale=1.6,yscale=0.7,baseline]
\draw[densely dotted, ]   (-2.6,-1) -- +(-.8,2)node[below,at start]{$i\;\!$--$\;\!1$};
\draw[densely dotted, ]   (-3.4,-1) -- +(.8,2)node[below,at start]{$i\;\!$--$\;\!1$}; 
 \draw (-3,-1) .. controls (-2.5,0) ..  +(0,2)
node[below,at start]{$i $};

\node at (-2.25,0) {$=$};

\draw[densely dotted, ]  (-1.5+.4,-1) -- +(+-.8,2)node[below,at start]{$i\;\!$--$\;\!1$};
\draw[densely dotted, ]   (-1.5+-.4,-1) -- +(+.8,2)node[below,at start]{$i\;\!$--$\;\!1$};
 \draw  (-1.5+0,-1) .. controls (-1.5-.5,0) ..  +(0,2)
node[below,at start]{$i $};

\node at (-.7,0) {$+$};

 \draw[densely dotted, ]  (-2.9+3.4,-1) -- +(0,2)node[below,at start]{$i\;\!$--$\;\!1$};
\draw[densely dotted, ]   (-2.9+2.6,-1) -- +(0,2)node[below,at start]{$i\;\!$--$\;\!1$};
\draw (-2.9+3,-1) -- +(0,2) node[below,at start]{$i $};
\end{tikzpicture}
\qquad\qquad 
\scalefont{0.8}\begin{tikzpicture}[very thick,xscale=1.6,yscale=0.7,baseline]
\draw  (-2.6,-1) -- +(-.8,2)node[below,at start]{$i$};
\draw  (-3.4,-1) -- +(.8,2)node[below,at start]{$i$}; 
 \draw[densely dotted, ] (-3,-1) .. controls (-2.5,0) ..  +(0,2)
node[below,at start]{$i\;\!$--$\;\!1$};

\node at (-2.25,0) {$=$};

\draw (-1.5+.4,-1) -- +(+-.8,2)node[below,at start]{$i$};
\draw  (-1.5+-.4,-1) -- +(+.8,2)node[below,at start]{$i$};
 \draw[densely dotted, ] (-1.5+0,-1) .. controls (-1.5-.5,0) ..  +(0,2)
node[below,at start]{$i\;\!$--$\;\!1$};

\node at (-.7,0) {$-$};

 \draw (-2.9+3.4,-1) -- +(0,2)node[below,at start]{$i$};
\draw  (-2.9+2.6,-1) -- +(0,2)node[below,at start]{$i$};
\draw[densely dotted ] (-2.9+3,-1) -- +(0,2) node[below,at start]{$i\;\!$--$\;\!1$};
\end{tikzpicture}
\]\end{enumerate}
    \begin{enumerate}[resume, label=(A\arabic*),leftmargin=*]  
\item \label{rel8} All other triples of  solid and ghost strands satisfy the naive braid relation. Diagrammatically, we have that 
 \[\Yvcentermath1
\scalefont{0.8}\begin{tikzpicture}[very thick,scale=0.6 ,baseline]
\draw (-2,-1) -- +(-2,2) node[below,at start]{$k$};
\draw (-4,-1) -- +(2,2) node[below,at start]{$i$};
\draw (-3,-1) .. controls (-4,0) ..  +(0,2)
node[below,at start]{$j$};
\node at (-1,0) {$=$};
\draw (2,-1) -- +(-2,2)
node[below,at start]{$k$};
\draw (0,-1) -- +(2,2)
node[below,at start]{$i$};
\draw (1,-1) .. controls (2,0) ..  +(0,2)
node[below,at start]{$j$};
\end{tikzpicture}\qquad
\quad
\begin{tikzpicture}[very thick,scale=0.6 ,baseline]
\draw[densely dotted] (-2,-1) -- +(-2,2) node[below,at start]{$k$};
\draw (-4,-1) -- +(2,2) node[below,at start]{$i$};
\draw (-3,-1) .. controls (-4,0) ..  +(0,2)
node[below,at start]{$j$};
\node at (-1,0) {$=$};
\draw[densely dotted] (2,-1) -- +(-2,2)
node[below,at start]{$k$};
\draw (0,-1) -- +(2,2)
node[below,at start]{$i$};
\draw (1,-1) .. controls (2,0) ..  +(0,2)
node[below,at start]{$j$};
\end{tikzpicture}
\qquad
\quad
\begin{tikzpicture}[very thick,scale=0.6 ,baseline]
\draw (-2,-1) -- +(-2,2) node[below,at start]{$k$};
\draw[densely dotted] (-4,-1) -- +(2,2) node[below,at start]{$i$};
\draw[densely dotted] (-3,-1) .. controls (-4,0) ..  +(0,2)
node[below,at start]{$j$};
\node at (-1,0) {$=$};
\draw (2,-1) -- +(-2,2)
node[below,at start]{$k$};
\draw[densely dotted] (0,-1) -- +(2,2)
node[below,at start]{$i$};
\draw[densely dotted] (1,-1) .. controls (2,0) ..  +(0,2)
node[below,at start]{$j$};
\end{tikzpicture}
\]
 for any $i,j,k\in \ZZ/e\ZZ$ and their mirror images through reflection in the vertical axis hold.
  Performing the leftmost relation \ref{rel8} implicitly involves manipulating a braid  of three ghost strands at the same time (which we do not picture) and vice versa.
Furthermore, 
\[
\qquad\quad
\begin{tikzpicture}[very thick,scale=0.6 ,baseline]
\draw[densely dotted] (-2,-1) -- +(-2,2) node[below,at start]{$c\vphantom{b}$};
\draw[densely dotted] (-4,-1) -- +(2,2) node[below,at start]{$a\vphantom{b}$};
\draw (-3,-1) .. controls (-2,0) ..  +(0,2)
node[below,at start]{$b$};
\node at (-1,0) {$=$};
\draw[densely dotted] (2,-1) -- +(-2,2)
node[below,at start]{$c\vphantom{b}$};
\draw[densely dotted] (0,-1) -- +(2,2)
node[below,at start]{$a\vphantom{b}$};
\draw (1,-1) .. controls (0,0) ..  +(0,2)
node[below,at start]{$b$};
\end{tikzpicture}
\qquad\quad
\begin{tikzpicture}[very thick,scale=0.6 ,baseline]
\draw (-2,-1) -- +(-2,2) node[below,at start]{$z\vphantom{b}$};
\draw  (-4,-1) -- +(2,2) node[below,at start]{$x\vphantom{b}$};
\draw[densely dotted] (-3,-1) .. controls (-2,0) ..  +(0,2)
node[below,at start]{$y\vphantom{b}$};
\node[densely dotted] at (-1,0) {$=$};
\draw  (2,-1) -- +(-2,2)
node[below,at start]{$z\vphantom{b}$};
\draw (0,-1) -- +(2,2)
node[below,at start]{$x\vphantom{b}$};
\draw[densely dotted] (1,-1) .. controls (0,0) ..  +(0,2)
node[below,at start]{$y\vphantom{b}$};
\end{tikzpicture}
\]  
    and $a,b,c,x,y,z\in \ZZ/e\ZZ$ such that $\delta_{a,b-1,c}=0$,  
  $\delta_{x,y+1,z}=0$.    
     \end{enumerate}

\begin{eg} 
We now illustrate the effect of relations \ref{rel10} and \ref{rel8} by moving the 1-strand in  the leftmost diagram in \cref{adoublecross} rightwards.  
We can do this (without incurring any error terms) until the solid 1-strand meets the 
crossing pair of ghost 0-strands.  This is illustrated in \cref{ataleof2part}.    
We then apply relation \ref{rel10} to the diagram in \cref{ataleof2part} to obtain the sum of diagrams   in \cref{secondpartoftale}.

\begin{figure}[ht!]
 \begin{tikzpicture}[baseline, thick,yscale=-0.8,xscale=-2.8]
\scalefont{0.8}

  \draw[wei]  (-1.1+0.5, -2)  to[out=90,in=-90] (-1.1+0.5, 2)  [below, at end] node{$0$};
  
 \draw(-1,-2) rectangle (1.6,2);

  \draw  (1.5,-2) to[out=90,in=-90]    (1.15,0.2) to[out=90,in=-90]
  node[below,at end]{$0$}   (1.5,2);

  \draw[densely dotted, gray]  (1.5-0.4,-2) to[out=90,in=-90]    (1.15-0.4,0.2) to[out=90,in=-90]
    (1.5-0.4,2);

  \draw[gray, densely dotted]  (-0.4+-.5,-2)    to[out=60,in=-90]   
   (-0.4+-.3,-1) 
     to [out=90,in=-90]    
   (-0.4+-.3,1)  to [out=90,in=-120]   
  node[below,at end]{$1$}  (-0.4+0,2);

  \draw[black]  (-.5,-2)    to[out=60,in=-90]   
   (-.3,-1) 
     to [out=90,in=-90]    
   (-.3,1)  to [out=90,in=-120]   
  node[below,at end]{$1$}  (0,2);
     \draw[gray, densely dotted] (.5-0.4,-2) to[out=95,in=-93]   (.5-0.4,-0.2) 
    to[out=87,in=-90]   node[midway,circle,fill=gray!85,inner sep=2pt]{} 
  (1-0.4,2);

 \draw  (.5,-2) to[out=95,in=-93]   (.5,-0.2) to[out=87,in=-90]   node[midway,circle,fill=black,inner sep=2pt]{} 
  node[below,at end]{$2$} (1,2);


   \draw[densely dotted,gray] (0-0.4,-2) to[out=90,in=-90]   
  (0.05-0.5,0)   to[out=90,in=-90]   (.5-0.4,2);

 \draw[densely dotted,gray]  (1-0.4,-2) to[out=90,in=-90]  
   (-.5-0.4,2);
  \draw   (1,-2) to[out=90,in=-90]  
  node[below,at end]{$0$  
  }   (-.5,2); 

   \draw (0,-2) to[out=90,in=-90]   
  (-.05,0) to   [out=90,in=-90] node[below,at end]{$0$ 
  }  (.5,2);

\end{tikzpicture}

\!\!\!\!
\caption{This diagram is obtained from the leftmost diagram in  \cref{adoublecross} using non-interacting relations (isotopy, and moving the solid (respectively ghost) 1-strand rightwards without crossing any ghost 0-strand (respectively solid 2-strand) neither of which produces an error term. }
\label{ataleof2part}
 \end{figure}

\begin{figure}[ht!]
 $$\begin{tikzpicture}[baseline, thick,yscale=-0.8,xscale=-2.8]
\scalefont{0.8}

  \draw[wei]  (-1.1+0.5, -2)  to[out=90,in=-90] (-1.1+0.5, 2)  [below, at end] node{$0$};
  
 \draw(-1,-2) rectangle (1.6,2);

  \draw  (1.5,-2) to[out=90,in=-90]    (1.15,0.2) to[out=90,in=-90]
  node[below,at end]{$0$}   (1.5,2);

  \draw[densely dotted, gray]  (1.5-0.4,-2) to[out=90,in=-90]    (1.15-0.4,0.2) to[out=90,in=-90]
    (1.5-0.4,2);


  \draw[black]  (-.5,-2)    to[out=60,in=-90]   
   (-.3,-1.2)    to[out=90,in=-90]   
   (-.5,0) 
     to [out=90,in=-120]    
   (-.3,1.2)  to [out=60,in=-110]   
  node[below,at end]{$1$}  (0,2);

  \draw[gray, densely dotted]  (-0.4+-.5,-2)    to[out=60,in=-90]   
   (-0.4+-.3,-1.2)    to[out=90,in=-90]   
   (-0.4+-.5,0) 
     to [out=90,in=-120]    
   (-0.4+-.3,1.2)  to [out=60,in=-110]   
  (-0.4+0,2);

     \draw[gray, densely dotted] (.5-0.4,-2) to[out=95,in=-93]   (.5-0.4,-0.2) 
    to[out=87,in=-90]   node[midway,circle,fill=gray!85,inner sep=2pt]{} 
  (1-0.4,2);

 \draw  (.5,-2) to[out=95,in=-93]   (.5,-0.2) to[out=87,in=-90]   node[midway,circle,fill=black,inner sep=2pt]{} 
  node[below,at end]{$2$} (1,2);


   \draw[densely dotted,gray] (0-0.4,-2) to[out=90,in=-90]   
  (0.05-0.5,0)   to[out=90,in=-90]   (.5-0.4,2);

 \draw[densely dotted,gray]  (1-0.4,-2) to[out=90,in=-90]  
   (-.5-0.4,2);
  \draw   (1,-2) to[out=90,in=-90]  
  node[below,at end]{$0$  
  }   (-.5,2);


   \draw (0,-2) to[out=90,in=-90]   
  (-.05,0) to   [out=90,in=-90] node[below,at end]{$0$ 
  }  (.5,2);

\end{tikzpicture}
\; - \; 
  \begin{tikzpicture}[baseline, thick,yscale=-0.8,xscale=-2.8]
\scalefont{0.8}

  \draw[wei]  (-1.1+0.5, -2)  to[out=90,in=-90] (-1.1+0.5, 2)  [below, at end] node{$0$};
  
 \draw(-1,-2) rectangle (1.6,2);

  \draw  (1.5,-2) to[out=90,in=-90]    (1.15,0.2) to[out=90,in=-90]
  node[below,at end]{$0$}   (1.5,2);

  \draw[densely dotted, gray]  (1.5-0.4,-2) to[out=90,in=-90]    (1.15-0.4,0.2) to[out=90,in=-90]
    (1.5-0.4,2);

  \draw[gray, densely dotted]  (-0.4+-.5,-2)    to[out=60,in=-90]   
   (-0.4+-.3,-1) 
     to [out=90,in=-90]    
   (-0.4+-.3,1)  to [out=90,in=-120]   
  (-0.4+0,2);

  \draw[black]  (-.5,-2)    to[out=60,in=-90]   
   (-.3,-1) 
     to [out=90,in=-90]    
   (-.3,1)  to [out=90,in=-120]   
  node[below,at end]{$1$}  (0,2);
     \draw[gray, densely dotted] (.5-0.4,-2) to[out=95,in=-93]   (.5-0.4,-0.2) 
    to[out=87,in=-90]   node[midway,circle,fill=gray!85,inner sep=2pt]{} 
  (1-0.4,2);

 \draw  (.5,-2) to[out=95,in=-93]   (.5,-0.2) to[out=87,in=-90]   node[midway,circle,fill=black,inner sep=2pt]{} 
  node[below,at end]{$2$} (1,2);

  \draw[densely dotted,gray]  (1-0.4,-2) to[out=90,in=-90]    (0.05-0.3,0) to[out=90,in=-90]  
  (.5-0.4,2);

  \draw  (0.4+1-0.4,-2) to[out=90,in=-90]    (0.4+0.05-0.3,0) to[out=90,in=-90]  node[below,at end]{$0$}
  (0.4+.5-0.4,2);

   \draw[densely dotted,gray] (0-0.4,-2) to[out=90,in=-90]   
  (-0.5+.13,-0.36) 
   to[out=90,in=-90]     (-.5-0.4,2);

   \draw  (0.4+0-0.4,-2) to[out=90,in=-90]   
  (0.4+.13-0.5,-0.36) 
   to[out=90,in=-90]    node[below,at end]{$0$} (0.4+-.5-0.4,2)
   ;

\end{tikzpicture}
 $$

\caption{We apply relation \ref{rel9} to  the diagram in 
\cref{ataleof2part}  thus passing the solid $1$-strand through the  crossing ghost 0-strands at the expensive of an error term.  
}
\label{secondpartoftale}
 \end{figure}

\end{eg}

 \subsection*{The red strands}We now consider the interactions between the red strands and the ghost and solid strands.  One should think of these red strands and playing the analogue of the role of the cyclotomic relation  for the classical KLR algebra.  
 \begin{enumerate}[resume, label=(A\arabic*),leftmargin=*] \item\label{rel11} Ghost strands and ghost dots may pass through red strands freely.  
For $i\neq j$, the solid $i$-strands may pass through red $j$-strands freely.   
If the red and solid strands have the same label, a dot is added to the solid strand when straightening.  Diagrammatically, we have that 
\[
\scalefont{0.8}
\begin{tikzpicture}[very thick,xscale=0.4,yscale=0.5,baseline,yscale=1.3]
\draw (-2.8-0.3,-1) .. controls (-1.2-0.3,0) ..  +(0,2)
node[below,at start]{$i$};
\draw[ wei] (-1.2-0.3,-1) .. controls (-2.8-0.3,0) ..  +(0,2)
node[below,at start]{$i $};
\node at (-.4,0) {$=$}; 

\draw (0.9,-1) .. controls (1.4,0) ..  +(0,2)
node[below,at start]{$i$} node[midway,fill=black,inner sep=2.5pt,circle]{};

\draw[wei] (2.5,-1) .. controls (2.1,0) ..  +(0,2)
node[below,at start]{$i $}  ;
 
\end{tikzpicture}
\qquad\quad
\begin{tikzpicture}[very thick,xscale=0.4,yscale=0.5,baseline,yscale=1.3]
\draw (-2.8-0.3,-1) .. controls (-1.2-0.3,0) ..  +(0,2)
node[below,at start]{$i$};
\draw[ wei] (-1.2-0.3,-1) .. controls (-2.8-0.3,0) ..  +(0,2)
node[below,at start]{$j $};
\node at (-.4,0) {$=$}; 

\draw (0.9,-1) .. controls (1.4,0) ..  +(0,2)
node[below,at start]{$i$} ;

\draw[wei] (2.5,-1) .. controls (2.1,0) ..  +(0,2)
node[below,at start]{$j$}  ;
 
\end{tikzpicture}\qquad\quad
\begin{tikzpicture}[very thick,xscale=0.4,yscale=0.5,baseline,yscale=1.3]
\draw[densely dotted, black] (-2.8-0.3,-1) .. controls (-1.2-0.3,0) ..  +(0,2)
node[below,at start]{$i$};
\draw[ wei] (-1.2-0.3,-1) .. controls (-2.8-0.3,0) ..  +(0,2)
node[below,at start]{$i $};
\node at (-.4,0) {$=$}; 

\draw[densely dotted, black] (0.9,-1) .. controls (1.4,0) ..  +(0,2)
node[below,at start]{$i$};

\draw[wei] (2.5,-1) .. controls (2.1,0) ..  +(0,2)
node[below,at start]{$i $}  ;
 
\end{tikzpicture}
\qquad\quad
\begin{tikzpicture}[very thick,xscale=0.4,yscale=0.5,baseline,yscale=1.3]
\draw[densely dotted, black] (-2.8-0.3,-1) .. controls (-1.2-0.3,0) ..  +(0,2)
node[below,at start]{$i$};
\draw[ wei] (-1.2-0.3,-1) .. controls (-2.8-0.3,0) ..  +(0,2)
node[below,at start]{$j $};
\node at (-.4,0) {$=$}; 

\draw[densely dotted, black] (0.9,-1) .. controls (1.4,0) ..  +(0,2)
node[below,at start]{$i$} ;

\draw[wei] (2.5,-1) .. controls (2.1,0) ..  +(0,2)
node[below,at start]{$j$}  ;
 
\end{tikzpicture}
\]for $i\neq j$ and their mirror images through reflection through the vertical axis hold.  

%
%
%
%
%
%
%
%
%
%
%
%
 \end{enumerate}
 \begin{enumerate}[resume, label=(A\arabic*),leftmargin=*]
\Item\label{rel12}Solid crossings and dots can pass through  red strands, with a
correction term
\[
\scalefont{0.8}\begin{tikzpicture}[very thick,baseline,scale=0.6] \path(0,2.5)--(0,2.5);
\draw[wei] (-3,-1) .. controls (-2,0) ..  +(0,2)
node[at start,below]{$k$};
\draw (-2,-1)  -- +(-2,2)
node[at start,below]{$i$};
\draw (-4,-1) -- +(2,2)
node [at start,below]{$j$};
 
\node at (-1,0) {$=$};
\draw[wei] (1,-1) .. controls (0,0) .. +(0,2)
node[at start,below]{$k$};

\draw (2,-1) -- +(-2,2)
node[at start,below]{$i$};
\draw (0,-1) -- +(2,2)
node [at start,below]{$j$};
 
\node at (2.8,0) {$+ $};
\draw[wei] (6.5,-1) -- +(0,2)
node[at start,below]{$k$};

\draw (7.5,-1) -- +(0,2)
node[at start,below]{$i$};
\draw (5.5,-1) -- +(0,2)
node [at start,below]{$j$};
 \node at (3.8,-.2){$\delta_{i,j,k}$};
\end{tikzpicture}
\]
 \item\label{rel13}
 Any   braid involving a red strand and   not of the form in \ref{rel12} can be undone without cost.  Diagrammatically, we have that 

\[
\scalefont{0.8}\begin{tikzpicture}[scale=0.6,very thick,baseline=2cm]
\draw[wei] (-2,2) -- +(-2,2) node [at start,below]{$k$};
\draw (-3,2) .. controls (-4,3) ..  +(0,2) node [at start,below]{$j$};
\draw (-4,2) -- +(2,2) node [at start,below]{$i$};

\node at (-1,3) {$=$};

\draw[wei] (2,2) -- +(-2,2) node [at start,below]{$k$};
\draw (1,2) .. controls (2,3) ..  +(0,2)node [at start,below]{$j$};
\draw (0,2) -- +(2,2)node [at start,below]{$i$};

\end{tikzpicture}
\quad\quad\quad  
\begin{tikzpicture}[scale=0.6,very thick,baseline=2cm]
\draw[wei] (-2,2) -- +(-2,2) node [at start,below]{$k$};
\draw[densely dotted] (-3,2) .. controls (-4,3) ..  +(0,2)
node [at start,below]{$j$};
\draw (-4,2) -- +(2,2)node [at start,below]{$i$};

\node at (-1,3) {$=$};

\draw[wei] (2,2) -- +(-2,2)
node [at start,below]{$k$};
\draw[densely dotted] (1,2) .. controls (2,3) ..  +(0,2)
node [at start,below]{$j$};
\draw (0,2) -- +(2,2)node [at start,below]{$i$};

\end{tikzpicture}
\quad\quad\quad
\begin{tikzpicture}[scale=0.6,very thick,baseline=2cm]
\draw[wei] (-2,2) -- +(-2,2)node [at start,below]{$k$};
\draw[densely dotted] (-3,2) .. controls (-4,3) ..  +(0,2)
node [at start,below]{$j$};
\draw[densely dotted] (-4,2) -- +(2,2)
node [at start,below]{$i$};

\node at (-1,3) {$=$};

\draw[wei] (2,2) -- +(-2,2)node [at start,below]{$k$};
\draw[densely dotted] (1,2) .. controls (2,3) ..  +(0,2)
node [at start,below]{$j$};
\draw[densely dotted] (0,2) -- +(2,2)
node [at start,below]{$i$};
 
\end{tikzpicture}
\]
\[
\begin{tikzpicture}[scale=0.6,very thick,baseline=2cm]
\draw  (-2,2) -- +(-2,2)node [at start,below]{$j$};
\draw[wei]  (-3,2) .. controls (-4,3) ..  +(0,2)
node [at start,below]{$k$};
\draw (-4,2) -- +(2,2)node [at start,below]{$i$};

\node at (-1,3) {$=$};

\draw   (2,2) -- +(-2,2)node [at start,below]{$j$};
\draw[wei] (1,2) .. controls (2,3) ..  +(0,2)node [at start,below]{$k$};
\draw (0,2) -- +(2,2)
node [at start,below]{$i$};
 
\end{tikzpicture}
\quad\quad\quad
\begin{tikzpicture}[scale=0.6,very thick,baseline=2cm]
\draw[densely dotted]  (-2,2) -- +(-2,2)node [at start,below]{$j$};
\draw[wei]  (-3,2) .. controls (-4,3) ..  +(0,2)
node [at start,below]{$k$};
\draw[densely dotted] (-4,2) -- +(2,2)
node [at start,below]{$i$};

\node at (-1,3) {$=$};

\draw[densely dotted]   (2,2) -- +(-2,2)node [at start,below]{$j$};
\draw[wei] (1,2) .. controls (2,3) ..  +(0,2)
node [at start,below]{$k$};
\draw[densely dotted] (0,2) -- +(2,2)
node [at start,below]{$i$};
 
\end{tikzpicture}
\quad\quad
\quad
\begin{tikzpicture}[scale=0.6,very thick,baseline=2cm]
\draw   (-2,2) -- +(-2,2)node [at start,below]{$j$};
\draw[wei]  (-3,2) .. controls (-4,3) ..  +(0,2)
node [at start,below]{$k$};
\draw[densely dotted] (-4,2) -- +(2,2)
node [at start,below]{$i$};

\node at (-1,3) {$=$};

\draw  (2,2) -- +(-2,2)
node [at start,below]{$j$};
\draw[wei] (1,2) .. controls (2,3) ..  +(0,2)
node [at start,below]{$k$};
\draw[densely dotted] (0,2) -- +(2,2)
node [at start,below]{$i$};
 
\end{tikzpicture}
\]
for any $i,j,k$ and their mirror images through reflection in the vertical axis hold.  
  \item
\label{rel14} Finally, any solid or ghost dot can be pulled through a red strand without cost.  Diagrammatically, we have that 
\[
\scalefont{0.8}\begin{tikzpicture}[very thick,baseline,scale=0.6]
\draw[wei](-3,0) +(1,-1) -- +(-1,1)
node [at start,below]{$k$};
\draw(-3,0) +(-1,-1) -- +(1,1)node [at start,below]{$j$};
\draw[wei](1,0) +(1,-1) -- +(-1,1)node [at start,below]{$k$};

\fill  (-3.5,-.5) circle (5pt);
\node at (-1,0) {$=$};
\draw(1,0) +(-1,-1) -- +(1,1)
node [at start,below]{$j$};
 \fill (1.5,.5) circle (5pt);

\draw[wei](1,0) +(3,-1) -- +(5,1)node [at start,below]{$k$};
\draw[densely dotted](1,0) +(5,-1) -- +(3,1)
node [at start,below]{$j$};
\fill[gray!80] (5.5,-.5) circle (5pt);
\node at (7,0) {$=$};
\draw[wei](5,0) +(3,-1) -- +(5,1)node [at start,below]{$k$};
\draw[densely dotted] (5,0) +(5,-1) -- +(3,1)node [at start,below]{$j$};
\fill [gray!80] (8.5,.5) circle (5pt);
\end{tikzpicture}
\]
for any $j,k$ and their mirror images through reflection in the vertical axis hold.   (We have not added the residues as they play no role here.)
\end{enumerate}

\subsection*{The unsteady relation}
Finally, we have the following non-local idempotent relation.  Before doing so, we note that the unique (up to equivalency) element of 
${^\imath_\mu\!\mathcal{R}^\imath_\mu}$ with {\em no crossing strands} is an idempotent by construction.  We refer to any such  diagram as the {\sf weight idempotent}, ${\sf 1}^\imath_\mu$.  
When the northern (equivalently, southern) residue sequence of the weight idempotent  is that of the box configuration,
 we simply  write 
 ${\sf 1}_\mu:= {\sf 1}^{{\rm res}(\mu)}_\mu$.

\begin{enumerate}[resume, label=(A\arabic*),leftmargin=*]
\item\label{rel15}
Any weight idempotent   in which a solid strand is at least  $  n$ units to the  right  of the  rightmost red-strand is referred to as unsteady and set to be equal to zero.
\end{enumerate}

\begin{eg}
Consider the leftmost diagram in \cref{secondpartoftale}.  This diagram has a solid 1-strand which can be passed through the red 0-strand  (without any error term) using relation \ref{rel11}.  This strand can then be pulled arbitrarily  far to the right  and hence is unsteady.  Thus the leftmost diagram in \cref{secondpartoftale} is zero by relations \ref{rel11} and \ref{rel15}.
\end{eg}

  \begin{rmk}\label{noninteracting}
We refer to   relations \ref{rel1}, \ref{rel2}, \ref{rel5}, \ref{rel8}, \ref{rel13} and \ref{rel14}, the latter relation in  \ref{rel4}, and the three rightmost relations in  \ref{rel11} as {\sf non-interacting relations}.
These   relations   pull strands through one another in the na\"ive fashion (without acquiring     error terms or dots).
\end{rmk}

\section{The combinatorics of diagrams and box configurations}

In  
this section we introduce the combinatorial language  and  corresponding diagrammatic  relations  needed for proving   the  main results 
of this paper.

\subsection{The Bruhat ordering }     
  We  define a subword of $\underline{w}=s_{r_1,r_1+1}s_{r_2,r_2+1}\dots s_{r_\ell,r_\ell+1}$  to be a sequence
${\bf t}=(t_1,t_2,\dots ,t_\ell)\in\{0,1\}^\ell$ and we set 
 $\underline{w}^{\bf t} :=
 s_{r_1,r_1+1}^{t_1}s_{r_2,r_2+1}^{t_2}\dots s_{r_\ell,r_\ell+1}^{t_\ell}
 $.   
 We let $\leq $ denote the strong  Bruhat order: namely $y\leq w$ if for some (or equivalently, every) 
reduced expression $\underline{w}$ there exists   a subword $\stt$ and a reduced expression  $\underline{y}$ such that $\underline{w}^{\bf t}=\underline{y}$.  
 We are now ready to define a new ordering on $\sigma$-diagrams; this ordering will be the key to our inductive proofs.  

 \begin{defn}
 Let $\theta \in \weight$ and $\la, \mu \in \con\ell n$.  
 Given $\Diag\in  {\sf 1}^\imath _\la  \algebra^\Bbbk_n(\theta) {\sf 1}^\jmath_\mu$, we let $[\Diag] \in \mathfrak{S}_{2n+\ell}$ denote the 
underlying word (in the Coxeter generators of $\mathfrak{S}_{2n+\ell}$) given by forgetting the types (solid, dashed, red) of strand and the distances between end points --- in other words, simply viewing the diagram as a permutation in $\mathfrak{S}_{2n+\ell}$.  
  Given two 
   diagrams  $\Diag, \Diag'\in {\sf 1}^\imath _\la  \algebra^\Bbbk_n(\theta) {\sf 1}^\jmath_\mu$
  we write $\Diag'\trianglerighteq _\theta \Diag$ if 
$ [\Diag] \leq [\Diag']$ in the Bruhat ordering on the permutations $ [\Diag] , [\Diag'] \in \mathfrak{S}_{2n+\ell}$.  
We let $\ell [\Diag]$ denote the length of a reduced expression of $[\Diag]$.  
  \end{defn}

  \begin{rmk}
If $\Diag, \Diag'\in {\sf 1}^\imath _\la  \algebra^\Bbbk_n(\theta) {\sf 1}^\jmath_\mu$ are  equivalent  diagrams, then the words 
$[\Diag']$ and  $[ \Diag] $ differ only by application of the commuting Coxeter relations ($s_{i,i+1} s_{j,j+1}= s_{j,j+1} s_{i,i+1} $ for $ |i-j|>1$). 
Therefore if two words differ by only the commuting Coxeter relations,
 we do not distinguish between these words.  
\end{rmk}

For $n\in \mathbb{N}$ and $\theta \in \weight$   we emphasise that a $\theta$-diagram  from $\algebra^\Bbbk_n(\theta)$ has $n$ solid strands, $n$ ghost strands and $\ell$ red strands;  and  our ordering 
considers all possible crossings of these $2n+\ell$ strands.

\begin{eg}
The diagram from $\algebra^\Bbbk_5(0)$ in \cref{ataleof2part} 
has  14  crossings between its 5 solid, 5 ghost, and 1 red strands.  
The underlying word is 
$s_{8,9} s_{4,5}     s_{3,4}  s_{5,6}s_{6,7}s_{4,5}   s_{5,6} s_{7,8}s_{8,9} s_{7,8}
s_{6,7} s_{7,8}	
s_{9,10} s_{10,11} s_{9,10} s_{8,9} \in \mathfrak{S}_{11}$.  
The righthand diagram of \cref{secondpartoftale} is obtained from the previous diagram by undoing a crossing of solid strands (and the  corresponding crossing of their ghosts); this diagram has 12 crossings and the underlying permutation is  
$s_{8,9} s_{4,5}   s_{3,4} s_{5,6}s_{6,7}
s_{4,5}     s_{6,7}   	  s_{7,8}		s_{9,10} s_{10,11}s_{9,10} s_{8,9} \in \mathfrak{S}_{11}$.  
Thus the latter diagram dominates the former in the Bruhat ordering.  
\end{eg}

 \subsection{Brick combinatorics}\label{definebrick} 
We require a language for discussing the effect of moving a single  $i$-strand (and its ghost) through an idempotent ${\sf 1}_\la$ for $\la \in \mptn \ell n$.  We can restrict our attention to understanding the boxes $(r,c,m)\in\la$ of residue 
$j\in \ZZ/e\ZZ$  such that $|j-i|\leq 1$.  This leads us to   introduce a combinatorial language of  $i$-diagonals and bricks. This will be essential for the proofs of  \cref{cellularitybreedscontempt0.5,corollary,hfsaklhsalhskafhjksdlhjsadahlfdshjksadflhafskhsfajk}, but can be skipped by the light-touch reader.

\begin{defn}
Let $\la \in \con \ell n$. Given $  \kappa  \in \NN$ and $0\leq m < \ell$, we   refer to the set of boxes  
\[
\mathbf{D}_{m,\kappa}= \{(r,c,m)\in \la
  \mid  \ct(r,c,m) \in \{\kappa-1,\kappa,\kappa+1\}			\}
\]
as the associated {\sf diagonal}.
If $\kappa$ is greater than, less than, or equal to $\sigma_m$, we say that the  diagonal is a right, left, or centred   diagonal respectively.
If $\la \in \mptn \ell n$, we say that a  diagonal is {\sf addable}, {\em removable}, or {\sf invisible} if $\bf D$ contains an addable   box of $\boxla $ of content $\kappa$, a  removable box of $\boxla $ of content $\kappa$, or no such box respectively.   
Given $i\in \ZZ/e\ZZ$ we refer to any diagonal $\mathbf{D}_{m,\kappa}$ such that $i\equiv \kappa$ modulo $e$ as an
 $i$-diagonal.  
\end{defn}

  \begin{figure}[ht]
\begin{center}\scalefont{0.6}
\begin{tikzpicture}[yscale=1,xscale=-1]
   \path (0,0) coordinate (origin);     \draw[wei2] (0,0)   circle (2pt);
  \draw[thick] (origin) 
   --++(135:10*0.4)
   --++(45:1*0.4)
  --++(-45:1*0.4)	
  --++(45:2*0.4)
  --++(-45:3*0.4)
  --++(45:1*0.4)
  --++(-45:1*0.4)
  --++(-45:1*0.4)--++(45:1*0.4)
    --++(45:1*0.4)
  --++(-45:1*0.4)
  --++(45:1*0.4)
  --++(-45:1*0.4)
  --++(45:1*0.4)
  --++(-45:1*0.4)
  --++(45:2*0.4)
  --++(-45:1*0.4)
    --++(-135:10*0.4);
     \clip (origin) 
     --++(135:10*0.4)
   --++(45:1*0.4)
  --++(-45:1*0.4)	
  --++(45:2*0.4)
  --++(-45:3*0.4)
  --++(45:1*0.4)
  --++(-45:1*0.4)
  --++(-45:1*0.4)--++(45:1*0.4)
    --++(45:1*0.4)
  --++(-45:1*0.4)
  --++(45:1*0.4)
  --++(-45:1*0.4)
  --++(45:1*0.4)
  --++(-45:1*0.4)
  --++(45:2*0.4)
  --++(-45:1*0.4)
    --++(-135:10*0.4);
   \path (45:1cm) coordinate (A1);
  \path (45:2cm) coordinate (A2);
  \path (45:3cm) coordinate (A3);
  \path (45:4cm) coordinate (A4);
  \path (135:1cm) coordinate (B1);
  \path (135:2cm) coordinate (B2);
  \path (135:3cm) coordinate (B3);
  \path (135:4cm) coordinate (B4);
  \path (A1) ++(135:3cm) coordinate (C1);
  \path (A2) ++(135:2cm) coordinate (C2);
  \path (A3) ++(135:1cm) coordinate (C3);
  \foreach \i in {1,...,19}
  {
    \path (origin)++(45:0.4*\i cm)  coordinate (a\i);
    \path (origin)++(135:0.4*\i cm)  coordinate (b\i);
    \path (a\i)++(135:4cm) coordinate (ca\i);
    \path (b\i)++(45:4cm) coordinate (cb\i);
    \draw[thin,gray] (a\i) -- (ca\i)  (b\i) -- (cb\i); 
   \draw  (origin)++(45:0.2)++(135:0.2) node {${0}$} ; 
   \draw  (origin)
   ++(45:0.4)   ++(135:0.4)
   ++(45:0.2)++(135:0.2) node {${0}$} ; 
    \draw  (origin)
   ++(45:0.4)   ++(135:0.4)   ++(45:0.4)   ++(135:0.4)
   ++(45:0.2)++(135:0.2) node {${0}$} ; 
    \draw  (origin)
   ++(45:0.4)   ++(135:0.4)    ++(45:0.4)   ++(135:0.4)   ++(45:0.4)   ++(135:0.4)
   ++(45:0.2)++(135:0.2) node {${0}$} ; 
     \draw  (origin)
   ++(45:0.4)   ++(135:0.4)      ++(45:0.4)   ++(135:0.4)    ++(45:0.4)   ++(135:0.4)   ++(45:0.4)   ++(135:0.4)
   ++(45:0.2)++(135:0.2) node {${0}$} ; 
      \draw  (origin)
   ++(45:0.4) 
   ++(45:0.2)++(135:0.2) node {${4}$} ; 
      \draw  (origin)
   ++(45:0.4)++(45:0.4)   ++(135:0.4)
   ++(45:0.2)++(135:0.2) node {${4}$} ; 
    \draw  (origin)
  ++(45:0.4) ++(45:0.4)   ++(135:0.4)   ++(45:0.4)   ++(135:0.4)
   ++(45:0.2)++(135:0.2) node {${4}$} ; 
    \draw  (origin)
  ++(45:0.4) ++(45:0.4)   ++(135:0.4)    ++(45:0.4)   ++(135:0.4)   ++(45:0.4)   ++(135:0.4)
   ++(45:0.2)++(135:0.2) node {${4}$} ; 
     \draw  (origin)
 ++(45:0.4)  ++(45:0.4)   ++(135:0.4)      ++(45:0.4)   ++(135:0.4)    ++(45:0.4)   ++(135:0.4)   ++(45:0.4)   ++(135:0.4)
   ++(45:0.2)++(135:0.2) node {${4}$} ; 
      \draw  (origin)
   ++(135:0.4) 
   ++(45:0.2)++(135:0.2) node {${1}$} ; 
      \draw  (origin)
   ++(135:0.4) ++(45:0.4)   ++(135:0.4)
   ++(45:0.2)++(135:0.2) node {${1}$} ; 
    \draw  (origin)
   ++(135:0.4)  ++(45:0.4)   ++(135:0.4)   ++(45:0.4)   ++(135:0.4)
   ++(45:0.2)++(135:0.2) node {${1}$} ; 
    \draw  (origin)
   ++(135:0.4)  ++(45:0.4)   ++(135:0.4)    ++(45:0.4)   ++(135:0.4)   ++(45:0.4)   ++(135:0.4)
   ++(45:0.2)++(135:0.2) node {${1}$} ; 
     \draw  (origin)
   ++(135:0.4)   ++(45:0.4)   ++(135:0.4)      ++(45:0.4)   ++(135:0.4)    ++(45:0.4)   ++(135:0.4)   ++(45:0.4)   ++(135:0.4)
   ++(45:0.2)++(135:0.2) node {${1}$} ; 

    \path (origin) ++(135:2cm) coordinate (XXXXXX);

   \draw  (XXXXXX)++(45:0.2)++(135:0.2) node {${0}$} ; 
   \draw  (XXXXXX)
   ++(45:0.4)   ++(135:0.4)
   ++(45:0.2)++(135:0.2) node {${0}$} ; 
    \draw  (XXXXXX)
   ++(45:0.4)   ++(135:0.4)   ++(45:0.4)   ++(135:0.4)
   ++(45:0.2)++(135:0.2) node {${0}$} ; 
    \draw  (XXXXXX)
   ++(45:0.4)   ++(135:0.4)    ++(45:0.4)   ++(135:0.4)   ++(45:0.4)   ++(135:0.4)
   ++(45:0.2)++(135:0.2) node {${0}$} ; 
     \draw  (XXXXXX)
   ++(45:0.4)   ++(135:0.4)      ++(45:0.4)   ++(135:0.4)    ++(45:0.4)   ++(135:0.4)   ++(45:0.4)   ++(135:0.4)
   ++(45:0.2)++(135:0.2) node {${0}$} ; 
      \draw  (XXXXXX)
   ++(45:0.4) 
   ++(45:0.2)++(135:0.2) node {${4}$} ; 
      \draw  (XXXXXX)
   ++(45:0.4)++(45:0.4)   ++(135:0.4)
   ++(45:0.2)++(135:0.2) node {${4}$} ; 
    \draw  (XXXXXX)
  ++(45:0.4) ++(45:0.4)   ++(135:0.4)   ++(45:0.4)   ++(135:0.4)
   ++(45:0.2)++(135:0.2) node {${4}$} ; 
    \draw  (XXXXXX)
  ++(45:0.4) ++(45:0.4)   ++(135:0.4)    ++(45:0.4)   ++(135:0.4)   ++(45:0.4)   ++(135:0.4)
   ++(45:0.2)++(135:0.2) node {${4}$} ; 
     \draw  (XXXXXX)
 ++(45:0.4)  ++(45:0.4)   ++(135:0.4)      ++(45:0.4)   ++(135:0.4)    ++(45:0.4)   ++(135:0.4)   ++(45:0.4)   ++(135:0.4)
   ++(45:0.2)++(135:0.2) node {${4}$} ; 
      \draw  (XXXXXX)
   ++(135:0.4) 
   ++(45:0.2)++(135:0.2) node {${1}$} ; 
      \draw  (XXXXXX)
   ++(135:0.4) ++(45:0.4)   ++(135:0.4)
   ++(45:0.2)++(135:0.2) node {${1}$} ; 
    \draw  (XXXXXX)
   ++(135:0.4)  ++(45:0.4)   ++(135:0.4)   ++(45:0.4)   ++(135:0.4)
   ++(45:0.2)++(135:0.2) node {${1}$} ; 
    \draw  (XXXXXX)
   ++(135:0.4)  ++(45:0.4)   ++(135:0.4)    ++(45:0.4)   ++(135:0.4)   ++(45:0.4)   ++(135:0.4)
   ++(45:0.2)++(135:0.2) node {${1}$} ; 
     \draw  (XXXXXX)
   ++(135:0.4)   ++(45:0.4)   ++(135:0.4)      ++(45:0.4)   ++(135:0.4)    ++(45:0.4)   ++(135:0.4)   ++(45:0.4)   ++(135:0.4)
   ++(45:0.2)++(135:0.2) node {${1}$} ; 
  
     \draw  (XXXXXX)++(45:0.2)++(135:0.2)++(-45:0.4) node {${4}$} ; 
     \draw  (XXXXXX)++(45:0.2)++(135:0.2)++(135:0.4*4) node {${4}$} ; 
  
    \path (origin) ++(45:2cm) coordinate (XXXXXX);

   \draw  (XXXXXX)++(45:0.2)++(135:0.2)++(-135:0.4) node {${1}$} ; 
   \draw  (XXXXXX)++(45:0.2)++(135:0.2) node {${0}$} ; 
   \draw  (XXXXXX)
   ++(45:0.4)   ++(135:0.4)
   ++(45:0.2)++(135:0.2) node {${0}$} ; 
    \draw  (XXXXXX)
   ++(45:0.4)   ++(135:0.4)   ++(45:0.4)   ++(135:0.4)
   ++(45:0.2)++(135:0.2) node {${0}$} ; 
    \draw  (XXXXXX)
   ++(45:0.4)   ++(135:0.4)    ++(45:0.4)   ++(135:0.4)   ++(45:0.4)   ++(135:0.4)
   ++(45:0.2)++(135:0.2) node {${0}$} ; 
     \draw  (XXXXXX)
   ++(45:0.4)   ++(135:0.4)      ++(45:0.4)   ++(135:0.4)    ++(45:0.4)   ++(135:0.4)   ++(45:0.4)   ++(135:0.4)
   ++(45:0.2)++(135:0.2) node {${0}$} ; 
      \draw  (XXXXXX)
   ++(45:0.4) 
   ++(45:0.2)++(135:0.2) node {${4}$} ; 
      \draw  (XXXXXX)
   ++(45:0.4)++(45:0.4)   ++(135:0.4)
   ++(45:0.2)++(135:0.2) node {${4}$} ; 
    \draw  (XXXXXX)
  ++(45:0.4) ++(45:0.4)   ++(135:0.4)   ++(45:0.4)   ++(135:0.4)
   ++(45:0.2)++(135:0.2) node {${4}$} ; 
    \draw  (XXXXXX)
  ++(45:0.4) ++(45:0.4)   ++(135:0.4)    ++(45:0.4)   ++(135:0.4)   ++(45:0.4)   ++(135:0.4)
   ++(45:0.2)++(135:0.2) node {${4}$} ; 
     \draw  (XXXXXX)
 ++(45:0.4)  ++(45:0.4)   ++(135:0.4)      ++(45:0.4)   ++(135:0.4)    ++(45:0.4)   ++(135:0.4)   ++(45:0.4)   ++(135:0.4)
   ++(45:0.2)++(135:0.2) node {${4}$} ; 
      \draw  (XXXXXX)
   ++(135:0.4) 
   ++(45:0.2)++(135:0.2) node {${1}$} ; 
      \draw  (XXXXXX)
   ++(135:0.4) ++(45:0.4)   ++(135:0.4)
   ++(45:0.2)++(135:0.2) node {${1}$} ; 
    \draw  (XXXXXX)
   ++(135:0.4)  ++(45:0.4)   ++(135:0.4)   ++(45:0.4)   ++(135:0.4)
   ++(45:0.2)++(135:0.2) node {${1}$} ; 
    \draw  (XXXXXX)
   ++(135:0.4)  ++(45:0.4)   ++(135:0.4)    ++(45:0.4)   ++(135:0.4)   ++(45:0.4)   ++(135:0.4)
   ++(45:0.2)++(135:0.2) node {${1}$} ; 
     \draw  (XXXXXX)
   ++(135:0.4)   ++(45:0.4)   ++(135:0.4)      ++(45:0.4)   ++(135:0.4)    ++(45:0.4)   ++(135:0.4)   ++(45:0.4)   ++(135:0.4)
   ++(45:0.2)++(135:0.2) node {${1}$} ; 
        \draw  (origin)
    ++(45:9*0.4)  
   ++(45:0.2)++(135:0.2) node {${1}$} ; 
  }
\end{tikzpicture} \qquad \begin{tikzpicture}[yscale=1,xscale=-1]
   \path (0,0) coordinate (origin);     \draw[wei2] (0,0)   circle (2pt);
  \draw[thick] (origin) 
   --++(135:10*0.4)
   --++(45:1*0.4)
  --++(-45:1*0.4)	
  --++(45:2*0.4)
  --++(-45:3*0.4)
  --++(45:1*0.4)
  --++(-45:1*0.4)
  --++(-45:1*0.4)--++(45:1*0.4)
    --++(45:1*0.4)
  --++(-45:1*0.4)
  --++(45:1*0.4)
  --++(-45:1*0.4)
  --++(45:1*0.4)
  --++(-45:1*0.4)
  --++(45:2*0.4)
  --++(-45:1*0.4)
    --++(-135:10*0.4);
      \clip (origin) 
     --++(135:10*0.4)
   --++(45:1*0.4)
  --++(-45:1*0.4)	
  --++(45:2*0.4)
  --++(-45:3*0.4)
  --++(45:1*0.4)
  --++(-45:1*0.4)
  --++(-45:1*0.4)--++(45:1*0.4)
    --++(45:1*0.4)
  --++(-45:1*0.4)
  --++(45:1*0.4)
  --++(-45:1*0.4)
  --++(45:1*0.4)
  --++(-45:1*0.4)
  --++(45:2*0.4)
  --++(-45:1*0.4)
    --++(-135:10*0.4);
     
  {      \draw[thick,black,fill=pink](origin)--++(135:0.8)--++(45:0.4)--++(-45:0.4)--++(45:0.4)--++(-45:0.4)--++(-135:0.8);
\path(origin)--++(45:0.4)--++(135:0.4) coordinate (origin);
        \draw[thick,black,fill=pink](origin)--++(135:0.8)--++(45:0.4)--++(-45:0.4)--++(45:0.4)--++(-45:0.4)--++(-135:0.8);
\path(origin)--++(45:0.4)--++(135:0.4) coordinate (origin);
        \draw[thick,black,fill=pink](origin)--++(135:0.8)--++(45:0.4)--++(-45:0.4)--++(45:0.4)--++(-45:0.4)--++(-135:0.8);
\path(origin)--++(45:0.4)--++(135:0.4) coordinate (origin);
        \draw[thick,black,fill=pink](origin)--++(135:0.8)--++(45:0.4)--++(-45:0.4)--++(45:0.4)--++(-45:0.4)--++(-135:0.8);
 }

  {    
  \path(0,0)--++(45:5*0.4)  coordinate (origin);
    \draw[thick,black,fill=lime](origin)--++(135:0.4)--++(-135:0.4)--++(-45:0.4)--++(45:0.4) ;

    \draw[thick,black,fill=pink](origin)--++(135:0.8)--++(45:0.4)--++(-45:0.4)--++(45:0.4)--++(-45:0.4)--++(-135:0.8);
\path(origin)--++(45:0.4)--++(135:0.4) coordinate (origin);
        \draw[thick,black,fill=pink](origin)--++(135:0.8)--++(45:0.4)--++(-45:0.4)--++(45:0.4)--++(-45:0.4)--++(-135:0.8);
\path(origin)--++(45:0.4)--++(135:0.4) coordinate (origin);
      }

  {    
  \path(0,0)--++(45:10*0.4)  coordinate (origin);
    \draw[thick,black,fill=lime](origin)--++(135:0.4)--++(-135:0.4)--++(-45:0.4)--++(45:0.4) ;
       }

 {    
  \path(0,0)--++(135:5*0.4)  coordinate (origin);
    \draw[thick,black,fill=yellow](origin)--++(45:0.4)--++(-45:0.4)--++(-135:0.4)--++(135:0.4) ;
         \draw[thick,black,fill=pink](origin)--++(135:0.8)--++(45:0.4)--++(-45:0.4)--++(45:0.4)--++(-45:0.4)--++(-135:0.8);
\path(origin)--++(45:0.4)--++(135:0.4) coordinate (origin);
        \draw[thick,black,fill=pink](origin)--++(135:0.8)--++(45:0.4)--++(-45:0.4)--++(45:0.4)--++(-45:0.4)--++(-135:0.8);
\path(origin)--++(45:0.4)--++(135:0.4) coordinate (origin);
        \draw[thick,black,fill=cyan!50](origin)--++(135:0.8)--++(45:0.4)--++(-45:0.8)--++(-135:0.4); 
        
       }

 {    
  \path(0,0)--++(135:10*0.4)  coordinate (origin);
    \draw[thick,black,fill=yellow](origin)--++(45:0.4)--++(-45:0.4)--++(-135:0.4)--++(135:0.4) ;
       }

   \path (0,0) coordinate (origin);     \draw[wei2] (0,0)   circle (2pt);
    \end{tikzpicture} 
\end{center} 
\caption{Examples of  $0$-diagonals for $e=5$ (4 are addable and 1 is invisible).   On the left we   highlight  all boxes in all 0-diagonals in the partition.  On the right we  illustrate  how these diagonals are built up from bricks.  Here 
$\B_1$ is pink, $\B_3$ is cyan, $\B_4$ is yellow, and $\B_5$ is green. 
}
\label{case1sum}
\end{figure}
  We shall now describe all ways of building $i$-diagonals of $\ell$-partitions  from the set of bricks 
  $\B_k$  for $k=1,\dots ,6$  depicted in \cref{bricksleft} and the empty brick, $\B_7$.  
    We shall also require three distinct bricks ${\bf M}_1, {\bf M}_2, {\bf M}_3$ which represent the important 
box-configurations in which {\em some boxes are missing}. 
Namely, for a given $i$-box $(r,c,m)\in \boxla $  the  cases  ${\bf M}_1, {\bf M}_2, {\bf M}_3$  correspond to a missing box in $\la$ to the south-west, south-east, or both  respectively.  These are depicted in \cref{ahlsjfdhalkjsdf}.

\begin{figure}[h]
\begin{center}\scalefont{0.8}
\begin{tikzpicture}[xscale=-0.9,yscale=0.9]
   \path (0,0) coordinate (origin);
      \path  (origin)--++(-90:0.5) node {${\bf B}_1$};
   
  \draw[thick] (origin) 
   --++(135:2*0.8)
   --++(45:1*0.8)
  --++(-45:1*0.8)	
   --++(45:1*0.8)
  --++(-45:1*0.8)	     --++(-135:2*0.8)
;     \clip (origin) 
  (origin) 
    --++(135:2*0.8)
   --++(45:1*0.8)
  --++(-45:1*0.8)	
   --++(45:1*0.8)
  --++(-45:1*0.8)	     --++(-135:2*0.8);     
   \path (45:1cm) coordinate (A1);
  \path (45:2cm) coordinate (A2);
  \path (45:3cm) coordinate (A3);
  \path (45:4cm) coordinate (A4);
  \path (135:1cm) coordinate (B1);
  \path (135:2cm) coordinate (B2);
  \path (135:3cm) coordinate (B3);
  \path (135:4cm) coordinate (B4);
  \path (A1) ++(135:3cm) coordinate (C1);
  \path (A2) ++(135:2cm) coordinate (C2);
  \path (A3) ++(135:1cm) coordinate (C3);
   \foreach \i in {1,...,19}
  {
    \path (origin)++(45:0.8*\i cm)  coordinate (a\i);
    \path (origin)++(135:0.8*\i cm)  coordinate (b\i);
    \path (a\i)++(135:4cm) coordinate (ca\i);
    \path (b\i)++(45:4cm) coordinate (cb\i);
    \draw[thin,gray] (a\i) -- (ca\i)  (b\i) -- (cb\i); 
 \draw  (origin)
    ++(45:0.4)++(135:0.4) node {${i}$} ;
  \draw  (origin)
    ++(45:0.4)++(135:1.2) node {${i+1}$} ;
      \draw  (origin)
      ++(45:1.2)++(135:0.4)  node {${i-1}$} ;
  }  
\end{tikzpicture}
 \quad
 \begin{tikzpicture}[xscale=-0.9,yscale=0.9]
   \path (0,0) coordinate (origin);
     \path  (origin)--++(-90:0.5) node {${\bf B}_2$};
  \draw[thick] (origin) 
   --++(135:1*0.8)
   --++(45:2*0.8)
  --++(-45:1*0.8)	
      --++(-135:2*0.8)
;     \clip (origin) 
   --++(135:1*0.8)
   --++(45:2*0.8)
  --++(-45:1*0.8)	
      --++(-135:2*0.8)
;     
   \path (45:1cm) coordinate (A1);
  \path (45:2cm) coordinate (A2);
  \path (45:3cm) coordinate (A3);
  \path (45:4cm) coordinate (A4);
  \path (135:1cm) coordinate (B1);
  \path (135:2cm) coordinate (B2);
  \path (135:3cm) coordinate (B3);
  \path (135:4cm) coordinate (B4);
  \path (A1) ++(135:3cm) coordinate (C1);
  \path (A2) ++(135:2cm) coordinate (C2);
  \path (A3) ++(135:1cm) coordinate (C3);
   \foreach \i in {1,...,19}
  {
    \path (origin)++(45:0.8*\i cm)  coordinate (a\i);
    \path (origin)++(135:0.8*\i cm)  coordinate (b\i);
    \path (a\i)++(135:4cm) coordinate (ca\i);
    \path (b\i)++(45:4cm) coordinate (cb\i);
    \draw[thin,gray] (a\i) -- (ca\i)  (b\i) -- (cb\i); 
 \draw  (origin)
    ++(45:0.4)++(135:0.4) node {${i}$} ;
  \draw  (origin)
    ++(45:0.4)++(135:1.2) node {${i+1}$} ;
      \draw  (origin)
      ++(45:1.2)++(135:0.4)  node {${i-1}$} ;
  }  
\end{tikzpicture}
\quad
 \begin{tikzpicture}[xscale=-0.9,yscale=0.9]
   \path (0,0) coordinate (origin);
     \path  (origin)--++(-90:0.5) node {${\bf B}_3$};
  \draw[thick] (origin) 
   --++(135:2*0.8)
   --++(45:1*0.8)
  --++(-45:2*0.8)	
      --++(-135:1*0.8)
;     \clip (origin) 
  --++(135:2*0.8)
   --++(45:1*0.8)
  --++(-45:2*0.8)	
      --++(-135:1*0.8)
;     
   \path (45:1cm) coordinate (A1);
  \path (45:2cm) coordinate (A2);
  \path (45:3cm) coordinate (A3);
  \path (45:4cm) coordinate (A4);
  \path (135:1cm) coordinate (B1);
  \path (135:2cm) coordinate (B2);
  \path (135:3cm) coordinate (B3);
  \path (135:4cm) coordinate (B4);
  \path (A1) ++(135:3cm) coordinate (C1);
  \path (A2) ++(135:2cm) coordinate (C2);
  \path (A3) ++(135:1cm) coordinate (C3);
   \foreach \i in {1,...,19}
  {
    \path (origin)++(45:0.8*\i cm)  coordinate (a\i);
    \path (origin)++(135:0.8*\i cm)  coordinate (b\i);
    \path (a\i)++(135:4cm) coordinate (ca\i);
    \path (b\i)++(45:4cm) coordinate (cb\i);
    \draw[thin,gray] (a\i) -- (ca\i)  (b\i) -- (cb\i); 
 \draw  (origin)
    ++(45:0.4)++(135:0.4) node {${i}$} ;
  \draw  (origin)
    ++(45:0.4)++(135:1.2) node {${i+1}$} ;
      \draw  (origin)
      ++(45:1.2)++(135:0.4)  node {${i-1}$} ;
  }  
\end{tikzpicture} 
\quad
 \begin{tikzpicture}[xscale=-0.9,yscale=0.9]
   \path (0,0) coordinate (origin);
     \path  (origin)--++(-90:0.5) node {${\bf B}_4$};
  \draw[thick] (origin) 
   --++(135:1*0.8)
   --++(45:1*0.8)
  --++(-45:1*0.8)	
   --++(-135:1*0.8)
;     \clip (origin) 
  (origin) 
   --++(135:1*0.8)
   --++(45:1*0.8)
  --++(-45:1*0.8)	
   --++(-135:1*0.8)
;     
   \path (45:1cm) coordinate (A1);
  \path (45:2cm) coordinate (A2);
  \path (45:3cm) coordinate (A3);
  \path (45:4cm) coordinate (A4);
  \path (135:1cm) coordinate (B1);
  \path (135:2cm) coordinate (B2);
  \path (135:3cm) coordinate (B3);
  \path (135:4cm) coordinate (B4);
  \path (A1) ++(135:3cm) coordinate (C1);
  \path (A2) ++(135:2cm) coordinate (C2);
  \path (A3) ++(135:1cm) coordinate (C3);
  \foreach \i in {1,...,19}
  {
    \path (origin)++(45:0.8*\i cm)  coordinate (a\i);
    \path (origin)++(135:0.8*\i cm)  coordinate (b\i);
    \path (a\i)++(135:4cm) coordinate (ca\i);
    \path (b\i)++(45:4cm) coordinate (cb\i);
    \draw[thin,gray] (a\i) -- (ca\i)  (b\i) -- (cb\i); 
\draw  (origin)
    ++(45:0.4)++(135:0.4) node {${i-1}$} ;
 \draw  (origin)
    ++(45:0.4)++(135:1.2) node {${i}$} ; 
     \draw  (origin)
    ++(45:1.2)++(135:1.2) node {${i-1}$} ; 
  }  
\end{tikzpicture} 
\quad
\begin{tikzpicture}[xscale=-0.9,yscale=0.9]
   \path (0,0) coordinate (origin);  \path  (origin)--++(-90:0.5) node {${\bf B}_5$};
  \draw[thick] (origin) 
   --++(135:1*0.8)
   --++(45:1*0.8)
  --++(-45:1*0.8)	
   --++(-135:1*0.8)
;     \clip (origin) 
  (origin) 
   --++(135:1*0.8)
   --++(45:1*0.8)
  --++(-45:1*0.8)	
   --++(-135:1*0.8)
;     
   \path (45:1cm) coordinate (A1);
  \path (45:2cm) coordinate (A2);
  \path (45:3cm) coordinate (A3);
  \path (45:4cm) coordinate (A4);
  \path (135:1cm) coordinate (B1);
  \path (135:2cm) coordinate (B2);
  \path (135:3cm) coordinate (B3);
  \path (135:4cm) coordinate (B4);
  \path (A1) ++(135:3cm) coordinate (C1);
  \path (A2) ++(135:2cm) coordinate (C2);
  \path (A3) ++(135:1cm) coordinate (C3);
   \foreach \i in {1,...,19}
  {
    \path (origin)++(45:0.8*\i cm)  coordinate (a\i);
    \path (origin)++(135:0.8*\i cm)  coordinate (b\i);
    \path (a\i)++(135:4cm) coordinate (ca\i);
    \path (b\i)++(45:4cm) coordinate (cb\i);
    \draw[thin,gray] (a\i) -- (ca\i)  (b\i) -- (cb\i); 
\draw  (origin)
    ++(45:0.4)++(135:0.4) node {${i+1}$} ;
   }

\end{tikzpicture}\quad
\begin{tikzpicture}[xscale=-0.9,yscale=0.9]
   \path (0,0) coordinate (origin);
     \path  (origin)--++(-90:0.5) node {${\bf B}_6$};
  \draw[thick] (origin) 
   --++(135:1*0.8)
   --++(45:1*0.8)
  --++(-45:1*0.8)	
   --++(-135:1*0.8)
;     \clip (origin) 
  (origin) 
   --++(135:1*0.8)
   --++(45:1*0.8)
  --++(-45:1*0.8)	
   --++(-135:1*0.8)
;     
   \path (45:1cm) coordinate (A1);
  \path (45:2cm) coordinate (A2);
  \path (45:3cm) coordinate (A3);
  \path (45:4cm) coordinate (A4);
  \path (135:1cm) coordinate (B1);
  \path (135:2cm) coordinate (B2);
  \path (135:3cm) coordinate (B3);
  \path (135:4cm) coordinate (B4);
  \path (A1) ++(135:3cm) coordinate (C1);
  \path (A2) ++(135:2cm) coordinate (C2);
  \path (A3) ++(135:1cm) coordinate (C3);
   \foreach \i in {1,...,19}
  {
    \path (origin)++(45:0.8*\i cm)  coordinate (a\i);
    \path (origin)++(135:0.8*\i cm)  coordinate (b\i);
    \path (a\i)++(135:4cm) coordinate (ca\i);
    \path (b\i)++(45:4cm) coordinate (cb\i);
    \draw[thin,gray] (a\i) -- (ca\i)  (b\i) -- (cb\i); 
\draw  (origin)
    ++(45:0.4)++(135:0.4) node {${i}$} ;
   }
 \end{tikzpicture} \quad 
\begin{tikzpicture}[xscale=-0.9,yscale=0.9]
   \path (0,0) coordinate (origin);
     \path  (origin)--++(-90:0.5) node {${\bf M}_1$};
  \draw[thick] (origin) 
   --++(135:2*0.8)
   --++(45:2*0.8)
   --++(-45:1*0.8)  --++(-135:1*0.8)	
   --++(-45:1*0.8)	     --++(-135:1*0.8);
       \clip (origin) 
 (origin) 
   --++(135:2*0.8)
   --++(45:2*0.8)
   --++(-45:1*0.8)  --++(-135:1*0.8)	
   --++(-45:1*0.8)	     --++(-135:1*0.8); 
     \path (45:1cm) coordinate (A1);
  \path (45:2cm) coordinate (A2);
  \path (45:3cm) coordinate (A3);
  \path (45:4cm) coordinate (A4);
  \path (135:1cm) coordinate (B1);
  \path (135:2cm) coordinate (B2);
  \path (135:3cm) coordinate (B3);
  \path (135:4cm) coordinate (B4);
  \path (A1) ++(135:3cm) coordinate (C1);
  \path (A2) ++(135:2cm) coordinate (C2);
  \path (A3) ++(135:1cm) coordinate (C3);
   \foreach \i in {1,...,19}
  {
    \path (origin)++(45:0.8*\i cm)  coordinate (a\i);
    \path (origin)++(135:0.8*\i cm)  coordinate (b\i);
    \path (a\i)++(135:4cm) coordinate (ca\i);
    \path (b\i)++(45:4cm) coordinate (cb\i);
    \draw[thin,gray] (a\i) -- (ca\i)  (b\i) -- (cb\i); 
 \draw  (origin)
    ++(45:0.4)++(135:0.4) node {${i}$} ;
  \draw  (origin)
    ++(45:0.4)++(135:1.2) node {${i+1}$} ;
    \draw  (origin)
      ++(45:1.2)++(135:1.2)  node {${i}$} ;
  }  
\end{tikzpicture}
\quad
\begin{tikzpicture}[xscale=-0.9,yscale=0.9]
   \path (0,0) coordinate (origin);
     \path  (origin)--++(-90:0.5) node {${\bf M}_2$};
  \draw[thick] (origin) 
   --++(45:2*0.8)
   --++(135:2*0.8)
   --++(-135:1*0.8)  --++(-45:1*0.8)	
   --++(-135:1*0.8)	     --++(-45:1*0.8);
        \clip (origin) 
(origin) 
   --++(45:2*0.8)
   --++(135:2*0.8)
   --++(-135:1*0.8)  --++(-45:1*0.8)	
   --++(-135:1*0.8)	     --++(-45:1*0.8);
     \path (45:1cm) coordinate (A1);
  \path (45:2cm) coordinate (A2);
  \path (45:3cm) coordinate (A3);
  \path (45:4cm) coordinate (A4);
  \path (135:1cm) coordinate (B1);
  \path (135:2cm) coordinate (B2);
  \path (135:3cm) coordinate (B3);
  \path (135:4cm) coordinate (B4);
  \path (A1) ++(135:3cm) coordinate (C1);
  \path (A2) ++(135:2cm) coordinate (C2);
  \path (A3) ++(135:1cm) coordinate (C3);
   \foreach \i in {1,...,19}
  {
    \path (origin)++(45:0.8*\i cm)  coordinate (a\i);
    \path (origin)++(135:0.8*\i cm)  coordinate (b\i);
    \path (a\i)++(135:4cm) coordinate (ca\i);
    \path (b\i)++(45:4cm) coordinate (cb\i);
    \draw[thin,gray] (a\i) -- (ca\i)  (b\i) -- (cb\i); 
 \draw  (origin)
    ++(45:0.4)++(135:0.4) node {${i}$} ;
    \draw  (origin)
      ++(45:1.2)++(135:1.2)  node {${i}$} ;
            \draw  (origin)
      ++(45:1.2)++(135:0.4)  node {${i-1}$} ;
  }  
\end{tikzpicture}
\quad
\begin{tikzpicture}[xscale=-0.9,yscale=0.9]
   \path (0,0) coordinate (origin);
        \path  (origin)--++(-90:0.4) node {${\bf M}_3$};
  \draw[thick] (origin) 
   --++(45:1*0.8)
   --++(135:2*0.8)
   --++(45:1*0.8)  --++(-45:1*0.8)	
   --++(-135:2*0.8)	     --++(-45:1*0.8);
        \clip(origin) 
   --++(45:1*0.8)
   --++(135:2*0.8)
   --++(45:1*0.8)  --++(-45:1*0.8)	
   --++(-135:2*0.8)	     --++(-45:1*0.8);
     \path (45:1cm) coordinate (A1);
  \path (45:2cm) coordinate (A2);
  \path (45:3cm) coordinate (A3);
  \path (45:4cm) coordinate (A4);
  \path (135:1cm) coordinate (B1);
  \path (135:2cm) coordinate (B2);
  \path (135:3cm) coordinate (B3);
  \path (135:4cm) coordinate (B4);
  \path (A1) ++(135:3cm) coordinate (C1);
  \path (A2) ++(135:2cm) coordinate (C2);
  \path (A3) ++(135:1cm) coordinate (C3);
   \foreach \i in {1,...,19}
  {
    \path (origin)++(45:0.8*\i cm)  coordinate (a\i);
    \path (origin)++(135:0.8*\i cm)  coordinate (b\i);
    \path (a\i)++(135:4cm) coordinate (ca\i);
    \path (b\i)++(45:4cm) coordinate (cb\i);
    \draw[thin,gray] (a\i) -- (ca\i)  (b\i) -- (cb\i); 
 \draw  (origin)
    ++(45:0.4)++(135:0.4) node {${i}$} ;
    \draw  (origin)
      ++(45:1.2)++(135:1.2)  node {${i}$} ;
            \draw  (origin)
      ++(45:1.2)++(135:0.4)  node {${i-1}$} ;
  }  
\end{tikzpicture}\end{center}
\caption{The bricks $\mathbf{B}_i$ and  $\mathbf{M}_j$ for $1\leq i \leq 6$ and $1\leq j \leq 3$.
The $\mathbf{B}_7$ brick is a single red $i$-strand (i.e., it corresponds to an empty box configuration).  
}
\label{bricksleft}\label{ahlsjfdhalkjsdf}
\end{figure}

Fix $\la \in \mptn \ell n$ and consider some fixed  component $1\leq m \leq \ell$.
We build an addable $i$-diagonal, $\bf D$, in this component by  placing a  a $\mathbf{B}_4$, $\mathbf{B}_5$, or $\mathbf{B}_7$ at the base (for diagonals to the right, left, or centred on the node $(1,1,m)$); we then place some number (possibly zero) of   $\mathbf{B}_1$ bricks on top.
If  $\bf D$ is  invisible then we place either a $\B_2$ or $\B_3$ brick on top of the addable $i$-diagonal.
If  $\bf D$ is  removable then we place a $\B_6$-brick  on top of the addable $i$-diagonal.
Examples  of how to construct such an $i$-diagonal are depicted in   \cref{case1sum}.

 \subsection{Brick diagrams}\label{resolvebrick} We gather  some easy results concerning the effect of pulling an $i$-strand through the diagram corresponding to one of these $i$-bricks.   
   In order to go back and forth between $\sigma$-diagrams and box configurations, we make the following intuitive definition.

   \begin{defn}\label{idempotentsssss} Associated to $\lambda\in{\mathscr{C}}^\ell_n$,  $\imath \in (\ZZ/e\ZZ)^n$, we have an  idempotent 
 ${\sf 1}_\la^{ \imath}$ given by the diagram with northern/southern points 
 ${\bf I}^\theta_\la$, no crossing strands,  and northern/southern  residue sequence   given by  $\imath \in (\ZZ/e\ZZ)^n$.    
 Each brick,   ${\bf B}_i$, is itself a  box-configuration  with associated  idempotent ${\sf 1}_{{\bf B}_i} $ for $1\leq i \leq 6$ (similarly for ${\bf M}_j$ for $1\leq j \leq 3$).  
   For $(r,c,m) \in \la$, we let $y_{(r,c,m)}{\sf 1}^{\imath}_\la $ be the diagram obtained by adding a dot to ${\sf 1}^{\imath}_\la$ on the 
strand   labelled by the box $(r,c,m)$.  
  \end{defn}

\begin{defn}
For $\la\in\con \ell n $ and $\jmath \in (\ZZ/e\ZZ)^n$, we let 
 ${\mathcal Y}^\jmath_\la$ denote the   subalgebra $ 
\langle {\sf 1}^{\jmath}_\la,  y_{(r,c,m)}{\sf 1}^{\jmath}_\la\mid \text{ for } (r,c,m) \in[\la]\rangle\subset \algebra^\Bbbk_n(\sigma)$.  
\end{defn}

   \!\!\!
   \begin{figure}[ht!]
   $$
\scalefont{0.8} \begin{tikzpicture}[baseline, thick,yscale=0.55,xscale=-2.5]

   \draw(-1.3,-2) rectangle (0.9,2);
   
  \draw[densely dotted, black,snake it] (-0.525+-0.5,-2) to[out=90,in=-90]   
  (-0.525+-0.5,2);

   \draw[densely dotted, black] (-0.425-0.15,-2) to[out=90,in=-90]   
   (-0.425+0-0.15,2);

   \draw[densely dotted, black,zigzag] (-0.425+0.5,-2) to[out=90,in=-90]   
    (-0.425+0.5,2);

  \draw[snake it] (-0.5,-2) to[out=90,in=-90]   
  node[below,at start]{$i+1$ }  (-0.5,2);
   \draw (0,-2) to[out=90,in=-90]   
  node[below,at start]{$i$ }  (0,2);
 \draw [zigzag]    (0.525,-2) to[out=90,in=-90]   
  node[below,at start]{$i\!-\!1$ }   (0.525,2);

\end{tikzpicture} 
\quad\quad 
 \begin{tikzpicture}[baseline, thick,yscale=0.55,xscale=-2.5]

   \draw(-1.5,-2) rectangle (0.45,2);
   
  \draw[densely dotted, black,snake it] (-0.525+-0.425,-2) to[out=90,in=-90]   
  (-0.525+-0.425,2);
   \draw[densely dotted, black] (-0.425+0,-2) to[out=90,in=-90]   
   (-0.425+0,2);

  \draw (-0.15,-2) to[out=90,in=-90]   
  node[below,at start]{$i$ }  (-0.15,2);
 \draw[densely dotted, black] (-0.475-0.15,-2) to[out=90,in=-90]   
 (-0.475-0.15,2);

  \draw[snake it] (-0.525,-2) to[out=90,in=-90]   
  node[below,at start]{$i\!+\!1$ }  (-0.525,2);
   \draw (0,-2) to[out=90,in=-90]   
  node[below,at start]{$i$ }  (0,2);

\end{tikzpicture}\quad\quad 
\begin{tikzpicture}[baseline, thick,yscale=0.55,xscale=-2.5]

   \draw(-0.75,-2) rectangle (1,2);

   \draw[densely dotted, black] (-0.475+0,-2) to[out=90,in=-90]   
   (-0.475+0,2);
   \draw[densely dotted, black] (-0.475+0.15,-2) to[out=90,in=-90]   
 (-0.475+0.15,2);

   \draw[densely dotted, black,zigzag] (-0.425+0.5,-2) to[out=90,in=-90]   
    (-0.425+0.5,2);


   \draw (0,-2) to[out=90,in=-90]   
  node[below,at start]{$i$ }  (0,2);
   \draw (0.15,-2) to[out=90,in=-90]   
  node[below,at start]{$i$ }  (0.15,2);

   \draw [zigzag]    (0.525,-2) to[out=90,in=-90]   
  node[below,at start]{$i\!-\!1$ }   (0.525,2);

\end{tikzpicture}$$ 

\!\!\!\!\!\!\caption{The diagrams ${\sf 1}_{\B_1}$,  ${\sf 1}_{\N_1}$   and  ${\sf 1}_{\N_2}$ respectively.    We have applied isotopy to    the strands to make it clearer which ghost belongs to which solid strand.   }
\end{figure}

\begin{defn}\label{anorder}
Given an $i$-diagonal, $\mathbf{D}$, we enumerate the bricks in $\mathbf{D}$ according to their height within the $i$-diagonal starting with the brick at the base of  $\mathbf{D}$ first   (which is one of $\mathbf{B}_4$, $\mathbf{B}_5$, or $\mathbf{B}_6$) and finishing with the top brick.
\end{defn}

\begin{rmk}For ease of discussion, we  assume $e>2$ (the $e=2$ case is similar, but the $i$-diagonals overlap and one must consider the diagonals  in turn).   
Let $S$ be a solid $i$-strand and suppose we wish to pull $S$ and its ghost $S'$ through an $i$-diagonal ${\bf D}$.  
One can factorise this calculation by considering each brick in the diagonal in turn.  This is immediate
 from the definitions, but is slightly non-intuitive because we visualise the $(i-1)$-boxes as being to the left of the $i$-boxes which are in turn to the left of the $(i+1)$-boxes.  However a momentary glance at \cref{Orderillus} reveals that this intuition is wrong: we will   not encounter two successive $i$-strands at any point in the process, as they are separated by a ghost $(i-1)$-strand (and recall that $S$ and $S'$ both commute with the solid $(i-1)$-strand, it is {\em only} the {\em ghost}  $(i-1)$-strand that is of interest!).

Indeed, the only strands of interest in ${\bf D}$ are the ghost $(i-1)$-strands, the solid $i$-strands, and the solid $(i+1)$-strands.  
For ${\bf D}$ as pictured in \cref{Orderillus},   the 
ghost $(i-1)$-strands  have $x$-coordinates 
 $-3\varepsilon,-5\varepsilon$,
 the solid 
 ghost $i$-strands  have $x$-coordinates 
 $-2\varepsilon,-4\varepsilon$
and the 
solid $(i+1)$-strands  have $x$-coordinates 
 $1-3\varepsilon,1-5\varepsilon$.  
As we pull the strand $S$ {\em and its ghost} through the strands labelled by boxes  with $x$-coordinate $x - k\varepsilon$,   
the order in which these interactions occur is determined by    $k\in \mathbb N$   and is
 independent of $x \in \mathbb N$.   
Thus  one can factorise the calculation by considering each brick in the diagonal in turn, as claimed. 
 \end{rmk}

\begin{figure}[ht!]
$$\!\!\!\!\!\! \begin{minipage}{2.5cm}
\begin{tikzpicture}[baseline, xscale=-1.5,yscale=1.5]
  
  \path(0,0)--++(90:-0.1); 
  \draw[very thick,fill=gray!40]   (0,0) --++(135:2*0.8)
   --++(135:1*0.8)
     --++(45:1*0.8)   --++(45:1*0.8)
 --++(-45:1*0.8)  
 --++(45:1*0.8)  
  --++(-45:1*0.8)  
 --++(-45:1*0.8)     --++(-135:1*0.8)  
 --++(-135:2*0.8)    
; 
    \clip(0,0) 
    --++(135:2*0.8)
   --++(135:1*0.8)
     --++(45:1*0.8)   --++(45:1*0.8)
 --++(-45:1*0.8)  
 --++(45:1*0.8)  
  --++(-45:1*0.8)  
 --++(-45:1*0.8)     --++(-135:1*0.8)  
 --++(-135:2*0.8)    
;

  \draw[thick,fill=magenta!30] (origin) 
   --++(135:2*0.8)
   --++(45:1*0.8)
  --++(-45:1*0.8)	
   --++(45:1*0.8)
  --++(-45:1*0.8)	     --++(-135:2*0.8)
;

   \path (0,0) coordinate (origin);
  \draw[thick] (origin) 
   --++(135:2*0.8)
   --++(45:1*0.8)
  --++(-45:1*0.8)	
   --++(45:1*0.8)
  --++(-45:1*0.8)	     --++(-135:2*0.8)
;       \path (45:1cm) coordinate (A1);
  \path (45:2cm) coordinate (A2);
  \path (45:3cm) coordinate (A3);
  \path (45:4cm) coordinate (A4);
  \path (135:1cm) coordinate (B1);
  \path (135:2cm) coordinate (B2);
  \path (135:3cm) coordinate (B3);
  \path (135:4cm) coordinate (B4);
  \path (A1) ++(135:3cm) coordinate (C1);
  \path (A2) ++(135:2cm) coordinate (C2);
  \path (A3) ++(135:1cm) coordinate (C3);
   \foreach \i in {1,...,19}
  {
    \path (origin)++(45:0.8*\i cm)  coordinate (a\i);
    \path (origin)++(135:0.8*\i cm)  coordinate (b\i);
    \path (a\i)++(135:4cm) coordinate (ca\i);
    \path (b\i)++(45:4cm) coordinate (cb\i);
    \draw[thin,gray] (a\i) -- (ca\i)  (b\i) -- (cb\i); 
 \draw  (origin)
    ++(45:0.4)++(135:0.4) node {${-2\varepsilon }$} ;
  \draw  (origin)
    ++(45:0.4)++(135:1.2) node {${1-3\varepsilon }$} ;
      \draw  (origin)
      ++(45:1.2)++(135:0.4)  node {${-1-3\varepsilon }$} ;
  }

     \path (origin) 
--++(135:1*0.8)
   --++(45:1*0.8)  coordinate (origin);

  \draw[thick,fill=cyan!30] (origin) 
   --++(135:2*0.8)
   --++(45:1*0.8)
  --++(-45:1*0.8)	
   --++(45:1*0.8)
  --++(-45:1*0.8)	     --++(-135:2*0.8)
;   
  
   \path (45:1cm) coordinate (A1);
  \path (45:2cm) coordinate (A2);
  \path (45:3cm) coordinate (A3);
  \path (45:4cm) coordinate (A4);
  \path (135:1cm) coordinate (B1);
  \path (135:2cm) coordinate (B2);
  \path (135:3cm) coordinate (B3);
  \path (135:4cm) coordinate (B4);
  \path (A1) ++(135:3cm) coordinate (C1);
  \path (A2) ++(135:2cm) coordinate (C2);
  \path (A3) ++(135:1cm) coordinate (C3);
   \foreach \i in {1,...,19}
  {
    \path (origin)++(45:0.8*\i cm)  coordinate (a\i);
    \path (origin)++(135:0.8*\i cm)  coordinate (b\i);
    \path (a\i)++(135:4cm) coordinate (ca\i);
    \path (b\i)++(45:4cm) coordinate (cb\i);
    \draw[thin,gray] (a\i) -- (ca\i)  (b\i) -- (cb\i); 
 \draw  (origin)
    ++(45:0.4)++(135:0.4) node {${-4\varepsilon }$} ;
  \draw  (origin)
    ++(45:0.4)++(135:1.2) node {${1-5\varepsilon }$} ;
      \draw  (origin)
      ++(45:1.2)++(135:0.4)  node {${-1-5\varepsilon }$} ;
  }

\end{tikzpicture}
\end{minipage} \qquad \qquad\qquad\qquad
\begin{minipage}{9cm}\scalefont{0.7} \begin{tikzpicture}[baseline, thick,yscale=1.1,xscale=-4.3]

   \draw(-1.3,-2) rectangle (0.9,2);
   \clip(-1.3,-2.3) rectangle (0.9,2.3);
   
  \draw[densely dotted, magenta!20, snake it] (-0.525+-0.5,-2) to[out=90,in=-90]   
  (-0.525+-0.5,2);

   \draw[densely dotted,magenta!30 ] (-0.425-0.15,-2) to[out=90,in=-90]   
   (-0.425+0-0.15,2);

   \draw[densely dotted, very thick, zigzag,magenta] (-0.425+0.5,-2) to[out=90,in=-90]   
    (-0.425+0.5,2);

  \draw[snake it,very thick,magenta] (-0.5,-2) to[out=90,in=-90]   
  node[below,at start]{$i\!+\!1$ }  (-0.5,2);
   \draw[very thick,magenta]  (0,-2) to[out=90,in=-90]   
  node[below,at start]{$i$ }  (0,2);
 \draw [zigzag,magenta!20, very thick]    (0.525,-2) to[out=90,in=-90]   
  node[below,at start]{$i\!-\!1$ }   (0.525,2);

\draw[densely dotted, cyan!20, snake it] (0.15+-0.525+-0.5,-2) to[out=90,in=-90]   
  (0.15+-0.525+-0.5,2);

   \draw[densely dotted,cyan!20,] (0.15+-0.425-0.15,-2) to[out=90,in=-90]   
   (0.15+-0.425+0-0.15,2);

   \draw[densely dotted,  very thick,zigzag,cyan] (0.15+-0.425+0.5,-2) to[out=90,in=-90]   
    (0.15+-0.425+0.5,2);

  \draw[snake it,very thick,cyan] (0.15+-0.5,-2) to[out=90,in=-90]   
  node[below,at start]{$i\!+\!1$ }  (0.15+-0.5,2);
   \draw[very thick,cyan] (0.15+0,-2) to[out=90,in=-90]   
  node[below,at start,very thick]{$i$ }  (0.15+0,2);
 \draw [zigzag,cyan!20, very thick]    (0.15+0.525,-2) to[out=90,in=-90]   
  node[below,at start]{$i\!-\!1$ }   (0.15+0.525,2);


  
\end{tikzpicture} 
\end{minipage} $$
\caption{
On the left we picture two ${\bf B}_1$ bricks (within the partition $(3^2,2)$) with the $x$-coordinates of their top nodes recorded in each box (we have assumed for ease of notation that the bottom of the lowest node has $x$-coordinate $x=0$).  
On the right we picture the corresponding    idempotent corresponding to this $i$-diagonal,  emphasising (with bolder colour) the strands which do not commute with an $i$-strand $S$ or  its ghost $S'$ (for $e>2$).    
Notice that the $i$-strands are not adjacent, but rather they are separated by a ghost $(i-1)$-strand.} 
\label{Orderillus}
\end{figure}

We will need to work by induction along the dominance ordering on box configurations.  In order to do this, we   need to be able to apply relations locally in a diagram and hence rewrite   local regions of diagrams in terms of box configurations.  The key to doing this is the   following  relations:  \begin{equation}\label{N3}
{\scalefont{0.8}
 \begin{tikzpicture}[baseline, thick,yscale=0.5,xscale=-4]
     \draw  (-0.55,-2) to[out=90,in=-90]   
  node[below,at start]{$i $ }  (-0.4,2);
   \draw  (-0.4,-2) to[out=90,in=-90]   
  node[below,at start]{$i $ }  (-0.55,2);
   \draw(-0.655,0) node {$=$};
    \draw(-0.8,0) node {$-$};
 \end{tikzpicture} 
 \begin{tikzpicture}[baseline, thick,yscale=0.5,xscale=-4]
     \draw  (-0.55,-2) to[out=90,in=-90] 
       node[below,at start]{$i $ }  (-0.4,0)   to[out=90,in=-90] 
        (-0.55,2);
   \draw  (-0.4,-2) to[out=90,in=-90]   node[below,at start]{$i $ }   
   node[at end, circle,fill=black,inner sep=2pt]{} 
   (-0.55,0)   to[out=90,in=-90] 
  (-0.4,2);
\end{tikzpicture}
\qquad
\qquad \qquad \qquad 
 \begin{tikzpicture}[baseline, thick,yscale=0.5,xscale=-4]
     \draw  (-0.55,-2) to[out=90,in=-90]   
  node[below,at start]{$i $ }  (-0.55,2);
   \draw  (-0.4,-2) to[out=90,in=-90]   
  node[below,at start]{$i $ }  (-0.4,2);
   \draw(-0.675,0) node {$=$};
 \end{tikzpicture}\;
\begin{tikzpicture}[baseline, thick,yscale=0.5,xscale=-4]
     \draw  (-0.55,-2) to[out=90,in=-90] 
       node[below,at start]{$i $ }  (-0.4,0)   to[out=90,in=-90] 
     node[near end, circle,fill=black,inner sep=2pt]{}   (-0.55,2);
   \draw  (-0.4,-2) to[out=90,in=-90]   node[below,at start]{$i $ }   
   node[at end, circle,fill=black,inner sep=2pt]{} 
   (-0.55,0)   to[out=90,in=-90] 
  (-0.4,2);
\end{tikzpicture}
\;\; - \;
\begin{tikzpicture}[baseline, thick,yscale=0.5,xscale=4]
     \draw  (-0.55,-2) to[out=90,in=-90] 
      node[near start, circle,fill=black,inner sep=2pt]{}  
      node[at end, circle,fill=black,inner sep=2pt]{} 
       node[below,at start]{$i $ }  (-0.4,0)   to[out=90,in=-90] 
      (-0.55,2);
   \draw  (-0.4,-2) to[out=90,in=-90]   node[below,at start]{$i $ }   
   (-0.55,0)   to[out=90,in=-90] 
  (-0.4,2);
\end{tikzpicture}}
\end{equation}
%
Both these relations follow by multiple applications of  relation \ref{rel3}.    
 Let ${\N_1}$ be a brick containing the nodes $\{(r,c-1,m),(r+1,c-1,m),(r+1,c,m)\}$.  
\cref{whatslovegottodo} illustrates how we can rewrite  the diagram ${\sf 1}_{\N_1}$ by first applying the leftmost equality in relation \ref{rel9} followed by the leftmost equality  in \cref{N3} (applied to both diagrams).   
We hence obtain the following:
\begin{equation}\scalefont{0.8}\label{M1relation} \begin{tikzpicture}[baseline, thick,yscale=0.5,xscale=-4]
     \draw  (-0.55,-2) to[out=90,in=-90]   
  node[below,at start]{$i $ }  (-0.55,2);
   \draw  (-0.4,-2) to[out=90,in=-90]   
  node[below,at start]{$i $ }  (-0.4,2);
  
   \draw[densely dotted]   (-0.55-0.5,-2) to[out=90,in=-90]   
(-0.55-0.5,2)  ;
   \draw[densely dotted]  (-0.4-0.5,-2) to[out=90,in=-90]   
(-0.4-0.5,2) ;

   \draw   (-0.975,-2) to[out=90,in=-90]   
  node[below,at start]{$i\!+\!1 $ }  (-0.975,2);

 \end{tikzpicture}\;\;\;\;=\;\;\;\;
\begin{tikzpicture}[baseline, thick,yscale=0.5,xscale=-4]
     \draw  (-0.55,-2) to[out=90,in=-90] 
       node[below,at start]{$i $ }  (-0.4+0.15,0)   to[out=90,in=-90] 
 (-0.55,2);
   \draw  (-0.4,-2) to[out=90,in=-90]   node[below,at start]{$i $ }   
   node[midway, circle,fill=black,inner sep=2pt]{} 
  (-0.4,2);

    \draw[densely dotted]    (-0.5+-0.55,-2) to[out=90,in=-90] 
  (-0.5+-0.4+0.15,0)   to[out=90,in=-90] 
 (-0.5+-0.55,2);
   \draw[densely dotted]    (-0.5+-0.4,-2) to[out=90,in=-90]  
   node[midway, circle,fill=gray!85,inner sep=2pt]{} 
  (-0.5+-0.4,2);

  \draw (-0.425-0.55,-2) to 
node[below, at start] {$i\!+\!1$}     (-0.425-0.55,2) ;

\end{tikzpicture}\;\;\;\;\; -\;\;\;\;
\begin{tikzpicture}[baseline, thick,yscale=0.5,xscale=-4]
     \draw  (-0.55,-2) to[out=90,in=-90] 
       node[below,at start]{$i $ }  (-0.4,0)   to[out=90,in=-90] 
 (-0.55,2);
   \draw  (-0.4,-2) to[out=90,in=-90]   node[below,at start]{$i $ }   
   node[at end, circle,fill=black,inner sep=2pt]{} 
   (-0.55,0)   to[out=90,in=-90] 
  (-0.4,2);

\draw[densely dotted]  (-0.5-0.55,-2) to[out=90,in=-90] 
 (-0.4-0.5,0)   to[out=90,in=-90] 
   (-0.55-0.5,2);
   \draw[densely dotted]  (-0.4-0.5,-2) to[out=90,in=-90] node [at end, circle,fill=gray!85,inner sep=2pt]{}     
   (-0.55-0.5,0)   to[out=90,in=-90] 
  (-0.4-0.5,2);
  
  \draw (-0.425-0.55,-2)
  to [out=84, in=-84]
node[below, at start] {$i\!+\!1$}     (-0.425-0.55,2) ;

\end{tikzpicture}
\end{equation}
Consider the idempotent corresponding to 
 any  $2\times2$ square of boxes $\{(r,c,m),(r-1,c,m),(r,c-1,m),(r-1,c-1,m)\}$  and place a dot on the strand labelled by $(r,c,m)$; the diagram and the $2\times 2$-array is depicted on the lefthand-side of the equation in \cref{LEMMA1!}. 
  We can pull this dotted $i$-strand and its ghost rightwards through the ghost $(i-1)$-strand using relation \ref{rel6} to obtain a dotted and an undotted diagram.   We hence obtain the following:
  \begin{equation} \label{squarebrickrelation}
\scalefont{0.8}
 \begin{tikzpicture}[baseline, thick,yscale=0.5,xscale=-1.7]

   
   \draw[densely dotted, gray!80] (-0.475+0,-2) to[out=90,in=-90]   
   (-0.475+0,2);
   \draw[densely dotted, gray!80 ] (-0.425+0.5,-2) to[out=90,in=-90]   
    (-0.425+0.5,2);

  \draw  (-0.4,-2) to[out=90,in=-90]   
  node[below,at start]{$i\!+\!1$ }  (-0.4,2);
   \draw (0,-2) to[out=90,in=-90]   
  node[below,at start]{$i$ }  (0,2);
 \draw     (0.525,-2) to[out=90,in=-90]   
  node[below,at start]{$i\!-\!1$ }   (0.525,2);

  \draw[densely dotted]  (-0.4-0.425,-2) to[out=90,in=-90]   
  (-0.4-0.425,2);

\draw[densely dotted, gray!80 ] (-0.475+1.5-0.15-0.2,-2) to[out=90,in=-90]   
node[at end, circle,fill=gray!80,inner sep=2pt]{}  (-0.475+2-1.97+0.15,0)to[out=90,in=-90]   
    (-0.475+1.5-0.15-0.2,2);

 \draw    (1.525-0.35,-2) to[out=90,in=-90]    node[below,at start]{$i $ }
 node[at end, circle,fill=black,inner sep=2pt]{}  (1.525 -1.5+0.15,0) to[out=90,in=-90]   
     (1.525-0.35,2);

\end{tikzpicture}\qquad =\qquad 
 \begin{tikzpicture}[baseline, thick,yscale=0.5,xscale=-1.7]
 
     \draw[densely dotted] (-0.475+0,-2) to[out=90,in=-90]   
   (-0.475+0,2);
   \draw[densely dotted ] (-0.425+0.5,-2) to[out=90,in=-90]   
   node[midway, circle,fill=gray!80,inner sep=2pt]{}   (-0.425+0.5,2);

  \draw  (-0.4,-2) to[out=90,in=-90]   
  node[below,at start]{$i\!+\!1$ }  (-0.4,2);
   \draw (0,-2) to[out=90,in=-90]   
  node[below,at start]{$i$ }  (0,2);
 \draw     (0.525,-2) to[out=90,in=-90]   
 node[midway, circle,fill=black,inner sep=2pt]{}  node[below,at start]{$i\!-\!1$ }   (0.525,2);

  \draw[densely dotted]  (-0.4-0.425,-2) to[out=90,in=-90]   
  (-0.4-0.425,2);

%
%

\draw[densely dotted ] (-0.475+1.5-0.15-0.2,-2) to[out=90,in=-90]   
  			   (-0.475+1.5-1-0.5 +0.15,0) 			to[out=90,in=-90]   
    (-0.475+1.5-0.15-0.2,2);

 \draw    (1.525-0.35,-2) to[out=90,in=-90]    node[below,at start]{$i $ }
   (-0.475+1.5-1 +0.15,0) to[out=90,in=-90]   
     (1.525-0.35,2);

\end{tikzpicture} \qquad +\qquad 
\begin{tikzpicture}[baseline, thick,yscale=0.5,xscale=-1.7]

   
   \draw[densely dotted] (-0.475+0,-2) to[out=90,in=-90]   
   (-0.475+0,2);
   \draw[densely dotted ] (-0.425+0.5,-2)  to[out=82,in=-82]   
    (-0.425+0.5,2);

    \draw (0,-2) to[out=90,in=-90]   
  node[below,at start]{$i$ }  (0,2);
 \draw     (0.525,-2) to[out=82,in=-82]   
  node[below,at start]{$i\!-\!1$ }   (0.525,2);

  \draw  (-0.4,-2) to 
  node[below,at start]{$i\!+\!1$ }  (-0.4,2);

  \draw[densely dotted]  (-0.4-0.425,-2) to[out=90,in=-90]   
  (-0.4-0.425,2);

\draw[densely dotted ] (-0.475+1.5-0.15-0.2,-2) to[out=90,in=-90]    
(-0.475+1.75-2.8+0.15+0.15+0.3+0.1+0.2+0.2+0.1,0)to[out=90,in=-90]   
    (-0.475+1.5-0.15-0.2,2);

 \draw    (1.525-0.35,-2) to[out=90,in=-90]   
  (1.535 -2.6+0.15+0.3+0.2+0.2+0.2+0.1,0) to[out=90,in=-90]   
     (1.525-0.35,2);

\end{tikzpicture}
\end{equation} 
In order to associate  \cref{M1relation,squarebrickrelation} to manipulations of brick diagrams, 
we must consider the intersection of the line $y=1/2$ with these diagrams.  
These are idempotents corresponding to certain bricks;   we depict the corresponding bricks in  \cref{whatslovegottodo,LEMMA1!}.

\begin{figure}[ht!]
\!\!\!\!$$
\scalefont{0.8}
\begin{tikzpicture}[xscale=1,yscale=1]
 
  \path (0.4,0)   coordinate (origin);
   \draw[white](-3.4,0) rectangle (14,2.5); 
   
   \begin{scope}
   {

    \path (origin)   coordinate (origin2);
       \draw[densely dotted]  (origin)
    --++(50:0.8*3)--++(140:0.8*3)   --++(-130:0.8*3) --++(-40:0.8*3); 
    \draw[densely dotted]  (origin)
    --++(50:0.8*1)--++(140:0.8*3)    ;
 \draw[densely dotted]  (origin)
    --++(50:0.8*2)--++(140:0.8*3)    ;
      \draw[densely dotted]  (origin)
      --++(140:0.8*1) --++(50:0.8*3); 
           \draw[densely dotted]  (origin)
      --++(140:0.8*2) --++(50:0.8*3); 

  \draw[thick] (origin) 
   --++(140:1*0.8)   --++(50:1*0.8)   --++(140:1*0.8)
   --++(50:1*0.8)
   --++(-40:2*0.8)	     --++(-130:2*0.8)
;

    \clip (origin) 
   --++(140:1*0.8)   --++(50:1*0.8)   --++(140:1*0.8)
   --++(50:1*0.8)
   --++(-40:2*0.8)	     --++(-130:2*0.8)
;       
   \path (50:1cm) coordinate (A1);
  \path (50:2cm) coordinate (A2);
  \path (50:3cm) coordinate (A3);
  \path (50:4cm) coordinate (A4);
  \path (140:1cm) coordinate (B1);
  \path (140:2cm) coordinate (B2);
  \path (140:3cm) coordinate (B3);
  \path (140:4cm) coordinate (B4);
  \path (A1) ++(140:3cm) coordinate (C1);
  \path (A2) ++(140:2cm) coordinate (C2);
  \path (A3) ++(140:1cm) coordinate (C3);
   \foreach \i in {1,...,19}
  {
    \path (origin)++(50:0.8*\i cm)  coordinate (a\i);
    \path (origin)++(140:0.8*\i cm)  coordinate (b\i);
    \path (a\i)++(140:4cm) coordinate (ca\i);
    \path (b\i)++(50:4cm) coordinate (cb\i);
    \draw[thin,gray] (a\i) -- (ca\i)  (b\i) -- (cb\i); 
 \draw  (origin)
    ++(50:0.4)++(140:0.4) node {${i}$} ;
 \draw  (origin)
    ++(50:0.4)++(140:0.4)++(50:0.4)++(140:0.4)++(50:0.4)++(140:0.4) node {${i}$} ;

  \draw  (origin)
    ++(50:0.4)++(140:1.2) node {${i-1}$} ;
      \draw  (origin)
      ++(50:1.2)++(140:0.4)  node {${i+1}$} ;
  }  }

\end{scope}

  \path (origin)   --++(0:2.1)--++(90:1.7) node {$=$}; 
     \path (origin)   --++(0:6.65)--++(90:1.7) node {$-$};

  \path (5,0)--++(140:0.8)--++(50:0.8)  coordinate (origin);

  \path (5,0)coordinate (origin2);
    \draw[densely dotted]  (origin2)
    --++(50:0.8*3)--++(140:0.8*3)   --++(-130:0.8*3) --++(-40:0.8*3); 
    \draw[densely dotted]  (origin2)
    --++(50:0.8*1)--++(140:0.8*3)    ;
 \draw[densely dotted]  (origin2)
    --++(50:0.8*2)--++(140:0.8*3)    ;
      \draw[densely dotted]  (origin2)
      --++(140:0.8*1) --++(50:0.8*3); 
           \draw[densely dotted]  (origin2)
      --++(140:0.8*2) --++(50:0.8*3); 

\begin{scope}
    \path (origin)--++(140:1*0.8)    coordinate (origin2);
   \fill[cyan!20]  (origin2) 
   --++(50:1*0.8)
   --++(-40:1*0.8)	     --++(-130:1*0.8)
;

  {
  \draw[thick] (origin) 
   --++(140:1*0.8)   --++(50:1*0.8)   --++(140:1*0.8) 
   --++(50:1*0.8)   --++(-40:1*0.8)  --++(-130:1*0.8)
   --++(-40:2*0.8)	     --++(-130:1*0.8)   --++(140:1*0.8)
;     \clip (origin) 
     --++(140:1*0.8)   --++(50:1*0.8)   --++(140:1*0.8) 
   --++(50:1*0.8)   --++(-40:1*0.8)  --++(-130:1*0.8)
   --++(-40:2*0.8)	     --++(-130:1*0.8)   --++(140:1*0.8)
;       
   \path (50:1cm) coordinate (A1);
  \path (50:2cm) coordinate (A2);
  \path (50:3cm) coordinate (A3);
  \path (50:4cm) coordinate (A4);
  \path (140:1cm) coordinate (B1);
  \path (140:2cm) coordinate (B2);
  \path (140:3cm) coordinate (B3);
  \path (140:4cm) coordinate (B4);
  \path (A1) ++(140:3cm) coordinate (C1);
  \path (A2) ++(140:2cm) coordinate (C2);
  \path (A3) ++(140:1cm) coordinate (C3);
   \foreach \i in {1,...,19}
  {
    \path (origin)++(50:0.8*\i cm)  coordinate (a\i);
    \path (origin)++(140:0.8*\i cm)  coordinate (b\i);
    \path (a\i)++(140:4cm) coordinate (ca\i);
    \path (b\i)++(50:4cm) coordinate (cb\i);
    \draw[thin,gray] (a\i) -- (ca\i)  (b\i) -- (cb\i); 
 \draw  (origin)
    ++(50:0.4)++(140:0.4) node {${i}$} ;
 \draw  (origin)
    ++(50:0.4)++(140:0.4)++(50:0.4)++(140:0.4)++(50:0.4)++(140:0.4) node {${i}$} ;

  \draw  (origin)
    ++(50:0.4)++(140:1.2) node {${i-1}$} ;
      \draw  (origin)
      ++(50:1.2)++(140:0.4)
      ++(-40:0.8)++(-130:0.8)
        node {${i+1}$} ;
  }  }
  \end{scope}

  \path (9.4,0)    coordinate (origin);

\begin{scope}
    \draw[densely dotted]  (origin)
    --++(50:0.8*3)--++(140:0.8*3)   --++(-130:0.8*3) --++(-40:0.8*3); 
    \draw[densely dotted]  (origin)
    --++(50:0.8*1)--++(140:0.8*3)    ;
 \draw[densely dotted]  (origin)
    --++(50:0.8*2)--++(140:0.8*3)    ;
      \draw[densely dotted]  (origin)
      --++(140:0.8*1) --++(50:0.8*3); 
           \draw[densely dotted]  (origin)
      --++(140:0.8*2) --++(50:0.8*3); 

    \path (origin)--++(140:1*0.8)    coordinate (origin2);
  \fill[cyan!20]  (origin2) 
   --++(50:1*0.8)
   --++(-40:1*0.8)	     --++(-130:1*0.8)
;

  {
  \draw[thick] (origin) 
   --++(140:1*0.8)   --++(50:1*0.8)   --++(140:1*0.8) 
   --++(50:2*0.8)   --++(-40:1*0.8)  --++(-130:2*0.8)
   --++(-40:1*0.8)	     --++(-130:1*0.8) 
;     \clip (origin) 
      --++(140:1*0.8)   --++(50:1*0.8)   --++(140:1*0.8) 
   --++(50:2*0.8)   --++(-40:1*0.8)  --++(-130:2*0.8)
   --++(-40:1*0.8)	     --++(-130:1*0.8) 
;       
   \path (50:1cm) coordinate (A1);
  \path (50:2cm) coordinate (A2);
  \path (50:3cm) coordinate (A3);
  \path (50:4cm) coordinate (A4);
  \path (140:1cm) coordinate (B1);
  \path (140:2cm) coordinate (B2);
  \path (140:3cm) coordinate (B3);
  \path (140:4cm) coordinate (B4);
  \path (A1) ++(140:3cm) coordinate (C1);
  \path (A2) ++(140:2cm) coordinate (C2);
  \path (A3) ++(140:1cm) coordinate (C3);
   \foreach \i in {1,...,19}
  {
    \path (origin)++(50:0.8*\i cm)  coordinate (a\i);
    \path (origin)++(140:0.8*\i cm)  coordinate (b\i);
    \path (a\i)++(140:4cm) coordinate (ca\i);
    \path (b\i)++(50:4cm) coordinate (cb\i);
    \draw[thin,gray] (a\i) -- (ca\i)  (b\i) -- (cb\i); 
 \draw  (origin)
    ++(50:0.4)++(140:0.4) node {${i}$} ;
 \draw  (origin)
    ++(50:0.4)++(140:0.4)++(50:0.4)++(140:0.4)++(50:0.4)++(140:0.4) node {${i}$} ;

  \draw  (origin)
    ++(50:0.4)++(140:1.2) node {${i-1}$} ;
      \draw  (origin)
      ++(50:1.2)++(140:0.4)
      ++(50:0.8)++(140:0.8)
        node {${i+1}$} ;
  }  }

\end{scope}

\end{tikzpicture}$$

\caption{  The brick diagrams depict the intersection of the $\theta$-diagrams in \cref{M1relation} with the line $y=1/2$.  
We  shade  the box  corresponding to  the   decorated strand in each case. }
\label{whatslovegottodo}
\end{figure}

 \!\!\!\!
 \begin{figure}[ht!]

\begin{tikzpicture}[xscale=1,yscale=1]
 
  \path (-1.5,0)   coordinate (origin);
   \draw[white](-4,0) rectangle (10.6,2.5); 
   
   \begin{scope}
   {

    \path (origin)--++(50:1*0.8) --++(140:1*0.8)  coordinate (origin2);
  \draw[densely dotted]  (origin)
    --++(50:0.8*3)--++(140:0.8*3)   --++(-130:0.8*3) --++(-40:0.8*3); 
    \draw[densely dotted]  (origin)
    --++(50:0.8*1)--++(140:0.8*3)    ;
 \draw[densely dotted]  (origin)
    --++(50:0.8*2)--++(140:0.8*3)    ;
      \draw[densely dotted]  (origin)
      --++(140:0.8*1) --++(50:0.8*3); 
           \draw[densely dotted]  (origin)
      --++(140:0.8*2) --++(50:0.8*3); 
  \fill[cyan!20]  (origin2) 
   --++(140:1*0.8)
   --++(50:1*0.8)
   --++(-40:1*0.8)	     --++(-130:1*0.8)
;

  \draw[thick] (origin) 
   --++(140:2*0.8)
   --++(50:2*0.8)
   --++(-40:2*0.8)	     --++(-130:2*0.8)
;

    \clip (origin) 
   --++(140:2*0.8)
   --++(50:2*0.8)
   --++(-40:2*0.8)	     --++(-130:2*0.8)
;       
   \path (50:1cm) coordinate (A1);
  \path (50:2cm) coordinate (A2);
  \path (50:3cm) coordinate (A3);
  \path (50:4cm) coordinate (A4);
  \path (140:1cm) coordinate (B1);
  \path (140:2cm) coordinate (B2);
  \path (140:3cm) coordinate (B3);
  \path (140:4cm) coordinate (B4);
  \path (A1) ++(140:3cm) coordinate (C1);
  \path (A2) ++(140:2cm) coordinate (C2);
  \path (A3) ++(140:1cm) coordinate (C3);
   \foreach \i in {1,...,19}
  {
    \path (origin)++(50:0.8*\i cm)  coordinate (a\i);
    \path (origin)++(140:0.8*\i cm)  coordinate (b\i);
    \path (a\i)++(140:4cm) coordinate (ca\i);
    \path (b\i)++(50:4cm) coordinate (cb\i);
    \draw[thin,gray] (a\i) -- (ca\i)  (b\i) -- (cb\i); 
 \draw  (origin)
    ++(50:0.4)++(140:0.4) node {${i}$} ;
 \draw  (origin)
    ++(50:0.4)++(140:0.4)++(50:0.4)++(140:0.4)++(50:0.4)++(140:0.4) node {${i}$} ;

  \draw  (origin)
    ++(50:0.4)++(140:1.2) node {${i-1}$} ;
      \draw  (origin)
      ++(50:1.2)++(140:0.4)  node {${i+1}$} ;
  }  }

\end{scope}

  \path (origin)   --++(0:2.3)--++(90:1.7) node {$=$}; 
     \path (origin)   --++(0:7.8)--++(90:1.7) node {$+$};

  \path (3.6,0)  coordinate (origin);
  \draw[densely dotted]  (origin)
    --++(50:0.8*3)--++(140:0.8*3)   --++(-130:0.8*3) --++(-40:0.8*3); 
    \draw[densely dotted]  (origin)
    --++(50:0.8*1)--++(140:0.8*3)    ;
 \draw[densely dotted]  (origin)
    --++(50:0.8*2)--++(140:0.8*3)    ;
      \draw[densely dotted]  (origin)
      --++(140:0.8*1) --++(50:0.8*3); 
           \draw[densely dotted]  (origin)
      --++(140:0.8*2) --++(50:0.8*3);

\begin{scope}
    \path (origin)--++(140:1*0.8)    coordinate (origin2);
  \fill[cyan!20]  (origin2) 
   --++(140:1*0.8)
   --++(50:1*0.8)
   --++(-40:1*0.8)	     --++(-130:1*0.8)
;

  {
  \draw[thick] (origin) 
   --++(140:2*0.8)
   --++(50:2*0.8)
   --++(-40:2*0.8)	     --++(-130:2*0.8)
;     \clip (origin) 
   --++(140:2*0.8)
   --++(50:2*0.8)
   --++(-40:2*0.8)	     --++(-130:2*0.8)
;       
   \path (50:1cm) coordinate (A1);
  \path (50:2cm) coordinate (A2);
  \path (50:3cm) coordinate (A3);
  \path (50:4cm) coordinate (A4);
  \path (140:1cm) coordinate (B1);
  \path (140:2cm) coordinate (B2);
  \path (140:3cm) coordinate (B3);
  \path (140:4cm) coordinate (B4);
  \path (A1) ++(140:3cm) coordinate (C1);
  \path (A2) ++(140:2cm) coordinate (C2);
  \path (A3) ++(140:1cm) coordinate (C3);
   \foreach \i in {1,...,19}
  {
    \path (origin)++(50:0.8*\i cm)  coordinate (a\i);
    \path (origin)++(140:0.8*\i cm)  coordinate (b\i);
    \path (a\i)++(140:4cm) coordinate (ca\i);
    \path (b\i)++(50:4cm) coordinate (cb\i);
    \draw[thin,gray] (a\i) -- (ca\i)  (b\i) -- (cb\i); 
 \draw  (origin)
    ++(50:0.4)++(140:0.4) node {${i}$} ;
 \draw  (origin)
    ++(50:0.4)++(140:0.4)++(50:0.4)++(140:0.4)++(50:0.4)++(140:0.4) node {${i}$} ;

  \draw  (origin)
    ++(50:0.4)++(140:1.2) node {${i-1}$} ;
      \draw  (origin)
      ++(50:1.2)++(140:0.4)  node {${i+1}$} ;
  }  }

\end{scope}

   \path (9,0)   coordinate (origin);
  \draw[densely dotted]  (origin)
    --++(50:0.8*3)--++(140:0.8*3)   --++(-130:0.8*3) --++(-40:0.8*3); 
    \draw[densely dotted]  (origin)
    --++(50:0.8*1)--++(140:0.8*3)    ;
 \draw[densely dotted]  (origin)
    --++(50:0.8*2)--++(140:0.8*3)    ;
      \draw[densely dotted]  (origin)
      --++(140:0.8*1) --++(50:0.8*3); 
           \draw[densely dotted]  (origin)
      --++(140:0.8*2) --++(50:0.8*3); 

\begin{scope}

  {
    
  \draw[thick] (origin) 
   --++(50:2*0.8)
   --++(140:3*0.8)
   --++(-130:1*0.8)     --++(-40:2*0.8)
      --++(-130:1*0.8)
    --++(-40:1*0.8)
  	    ;
      \clip(origin) 
--++(50:2*0.8)
   --++(140:3*0.8)
   --++(-130:1*0.8)     --++(-40:2*0.8)
      --++(-130:1*0.8)
    --++(-40:1*0.8)
  	    ; 
   \path (50:1cm) coordinate (A1);
  \path (50:2cm) coordinate (A2);
  \path (50:3cm) coordinate (A3);
  \path (50:4cm) coordinate (A4);
  \path (140:1cm) coordinate (B1);
  \path (140:2cm) coordinate (B2);
  \path (140:3cm) coordinate (B3);
  \path (140:4cm) coordinate (B4);
  \path (A1) ++(140:3cm) coordinate (C1);
  \path (A2) ++(140:2cm) coordinate (C2);
  \path (A3) ++(140:1cm) coordinate (C3);
   \foreach \i in {1,...,19}
  {
    \path (origin)++(50:0.8*\i cm)  coordinate (a\i);
    \path (origin)++(140:0.8*\i cm)  coordinate (b\i);
    \path (a\i)++(140:4cm) coordinate (ca\i);
    \path (b\i)++(50:4cm) coordinate (cb\i);
    \draw[thin,gray] (a\i) -- (ca\i)  (b\i) -- (cb\i); 
 \draw  (origin)
    ++(50:0.4)++(140:0.4) node {${i}$} ;
     

     \draw  (origin)
    ++(-130:0.4)++(-40:0.4) node {${  i}$} ;
 \draw  (origin)
    ++(140:0.4+0.8)++(50:0.4+0.8)  node  {${  i}$} ;
 \draw  (origin)
    ++(140:0.4+0.8+0.8)++(50:0.4+0.8)  node  {${  i-1}$} ;
 
  \draw  (origin)
    ++(50:1.2)++(140:0.4) node {${i+1}$} ;
  }  }

\end{scope}

\end{tikzpicture}
\caption{
The brick diagrams depict the intersection of the $\theta$-diagrams in \cref{squarebrickrelation} with the line $y=1/2$.  
We  shade  the box  corresponding to  the   decorated strand in each case.
 }
\label{LEMMA1!}
\end{figure}

 The following technical definition will allow  us to handle the inductive error terms of \cref{M1relation,squarebrickrelation} by induction on the dominance ordering.  
 In particular, notice that if we set the node $(r,c,m)$ to be the missing node   in   the ${\bf M}_1$ brick     in \cref{N3}  then the maps  $\phi^{X_1}_{r,c,m}$  and $\phi^{X_2}_{r,c,m}$  describe  the intersection of the line $y=1/2$ with  the two terms on the righthand-side of \cref{M1relation}.  
 Similarly $\phi^{\B_1}_{(r,c,m)}$ describes the second term on the righthand-side of  \cref{LEMMA1!}.

\begin{defn}\label{phimap}Let $1\leq r,c\leq   n$ and $0\leq m < \ell$.  
  Given $\la \in \con \ell n$, we set 
$\phi (\la)\in  \con \ell n$ to be the  box configuration $\phi (\la)=\{\phi  (r',c',m') \mid (r',c',m') \in \la\}  $ for $\phi$ any one of the three following maps:
\begin{align*}
\phi^{N}_{(r,c,m)}(r',c',m')&= 
\begin{cases}
(r'+1,c'+1,m')  &\text{if $m= m'$ and  $r'> r $ and  $c'>c$} \\
 (r'+1,c'+1,m')&\text{if }(r'+1,c'+1,m')=(r,c,m)    \\
 (r',c',m')	&\text{otherwise.}
\end{cases}
\\
\phi^{NE}_{(r,c,m)}(r',c',m')	&= 
\begin{cases}
(r'+1,c'+1,m')  &\text{if $m= m'$ and  $r'\geq r $ and  $c'>c$} \\
 (r'+1,c'+1,m')&\text{if }(r'+1,c'+1,m')=(r,c,m)    \\
 (r',c',m')	&\text{otherwise.}
 \end{cases}
\\
\phi^{NW}_{(r,c,m)}(r',c',m')&= 
\begin{cases}
(r'+1,c'+1,m')  &\text{if $m= m'$ and  $r'> r $ and  $c'\geq c$} \\
 (r'+1,c'+1,m')&\text{if }(r'+1,c'+1,m')=(r,c,m)    \\
 (r',c',m')	&\text{otherwise.}
\end{cases}
\end{align*}\end{defn}


\section{The integral cellular basis of the   quiver Cherednik algebra} \label{cellllllllllll}

In this section, we prove that   $\algebra^\Bbbk_n(\theta )$ is cellular for any  $\theta \in \weight$  and over  $\Bbbk$ an arbitrary    integral domain, we define the associated Schur functor  ${\sf E}^\theta_\omega$ relating the quiver Cherednik and Hecke algebras, and we generalise and strengthen the structural  results of \cite{bkw11,MR3732238} and \cite[Section 2.3]{MR2525917}.   
    The framework we developed in \cref{definebrick,resolvebrick} allows us to proceed by induction on  $(\con \ell n,\rhd_\theta)$. 
  In \cref{isomer} we directly  match-up the presentations of the KLR algebra and (a subalgebra of) the quiver Cherednik algebra for the first time, this should be of independent interest.  
  Over   $\mathbb  C$, cellularity of $\algebra^{\mathbb  C}_n(\theta )$ is proven in \cite{MR3732238} by applying  the isomorphism in  \cite[Theorem 4.5]{MR3732238}
to   the ungraded versions of these algebras (this isomorphism   generalises that of \cite{MR2551762} and  only holds for $\mathbb  C$).   
We   take this opportunity to add a little flesh to the bones of the ideas \cite{MR3732238}. 
  We also  prove  a number of new structural results concerning the action of the algebra $\algebra^\Bbbk_n(\theta)$ on the cellular basis (generalising   \cite{bkw11}).

 \begin{figure}[ht!]
$$\scalefont{0.8} \begin{tikzpicture}[baseline, thick,yscale=0.75,xscale=-0.75]

\path(0,0) coordinate (112);
\path(112)--++(-10:1) coordinate (a1);
\path(112)--++(80:1) coordinate (a2);
\path(112)--++(170:1) coordinate (a3);
\path(112)--++(-100:1) coordinate (a4);
\draw(a1)--(a2)--(a3)--(a4)--(a1);

\path(a1)--++(80:1) coordinate (212);
\path(212)--++(-10:1) coordinate (b1);
\path(212)--++(80:1) coordinate (b2);
\path(212)--++(170:1) coordinate (b3);
\path(212)--++(-100:1) coordinate (b4);
\draw(b1)--(b2)--(b3)--(b4)--(b1);

\path(a3)--++(80:1) coordinate (312);
\path(312)--++(-10:1) coordinate (c1);
\path(312)--++(80:1) coordinate (c2);
\path(312)--++(170:1) coordinate (c3);
\path(312)--++(-100:1) coordinate (c4);
\draw(c1)--(c2)--(c3)--(c4)--(c1);

\path(-4.5,0) coordinate (112);
\path(112)--++(-10:1) coordinate (d1);
\path(112)--++(80:1) coordinate (d2);
\path(112)--++(170:1) coordinate (d3);
\path(112)--++(-100:1) coordinate (d4);
\draw(d1)--(d2)--(d3)--(d4)--(d1);

\path(d1)--++(80:1) coordinate (212);
\path(212)--++(-10:1) coordinate (e1);
\path(212)--++(80:1) coordinate (e2);
\path(212)--++(170:1) coordinate (e3);
\path(212)--++(-100:1) coordinate (e4);
\draw(e1)--(e2)--(e3)--(e4)--(e1);

\path(0,10) coordinate (u112);
\path(u112)--++(-10:1) coordinate (ua1);
\path(u112)--++(80:1) coordinate (ua2);
\path(u112)--++(170:1) coordinate (ua3);
\path(u112)--++(-100:1) coordinate (ua4);
\draw(ua1)--(ua2)--(ua3)--(ua4)--(ua1);

\path(ua1)--++(80:1) coordinate (u212);
\path(u212)--++(-10:1) coordinate (ub1);
\path(u212)--++(80:1) coordinate (ub2);
\path(u212)--++(170:1) coordinate (ub3);
\path(u212)--++(-100:1) coordinate (ub4);
\draw(ub1)--(ub2)--(ub3)--(ub4)--(ub1);

\path(ua3)--++(80:1) coordinate (u312);
\path(u312)--++(-10:1) coordinate (uc1);
\path(u312)--++(80:1) coordinate (uc2);
\path(u312)--++(170:1) coordinate (uc3);
\path(u312)--++(-100:1) coordinate (uc4);
\draw(uc1)--(uc2)--(uc3)--(uc4)--(uc1);

\path(ub2)--++(-10:1) coordinate (u322);
\path(u322)--++(-10:1) coordinate (ue1);
\path(u322)--++(80:1) coordinate (ue2);
\path(u322)--++(170:1) coordinate (ue3);
\path(u322)--++(-100:1) coordinate (ue4);
\draw(ue1)--(ue2)--(ue3)--(ue4)--(ue1);

\path(ue1)--++(80:1) coordinate (u222);
\path(u222)--++(-10:1) coordinate (ud1);
\path(u222)--++(80:1) coordinate (ud2);
\path(u222)--++(170:1) coordinate (ud3);
\path(u222)--++(-100:1) coordinate (ud4);
 \draw(ud1)--(ud2)--(ud3)--(ud4)--(ud1);


 \draw(a2)--(ua2);
 \path(a2)--++(-90:1)--++(180:1) coordinate (ga2);
\draw[densely dotted, gray, gray!80] (ga2)--++(90:10.3);
\draw(b2)--++(90:10);
\path(b2)--++(-90:1)--++(180:1) coordinate (gb2);
\draw[densely dotted, gray, gray!80] (gb2)--++(90:10.3);

\path(ue2)--++(-90:1)--++(180:1) coordinate (ue2g);
\path(ue2)--++(-90:6) coordinate (ue2below);
\path(ue2below)--++(180:1) coordinate (ue2gbelow);
\draw(c2)--++(90:1)to [out=90,in=-90] (ue2below) to [out=90,in=-90] (ue2);

 \path(c2)--++(-90:1)--++(180:1) coordinate (gc2);
 
\draw[densely dotted, gray, gray!80](gc2)--++(90:2)to [out=90,in=-90] (ue2gbelow) --++(90:5.3);

\path(uc2)--++(-90:4.2) coordinate(belowuc2);
\path(uc2)--++(-90:0.7)--++(180:1) coordinate(uc2g);
\path(uc2g)--++(-90:3.5) coordinate(belowuc2g);
\draw (d2)--++(90:3) to [out=90,in=-90](belowuc2)--(uc2);
\path(d2)--++(-90:1)--++(180:1) coordinate (gd2);
\draw[densely dotted, gray, gray!80] (gd2)--++(90:4) to [out=90,in=-90](belowuc2g)--(uc2g);

\path(ud2)--++(-90:5.2) coordinate(belowud2);
\path(ud2)--++(-90:0.7)--++(180:1) coordinate(ud2g);
\path(ud2g)--++(-90:4.3) coordinate(belowud2g);

\draw (e2)--++(90:2) to [out=90,in=-90](belowud2)--(ud2);
 
\path(e2)--++(-90:1)--++(180:1) coordinate (ge2);
\draw[densely dotted, gray, gray!80] (ge2)--++(90:3) to [out=90,in=-90](belowud2g)--(ud2g);


 
\draw[fill=white](a1)--(a2)--(a3)--(a4)--(a1);
 
\draw[fill=white](b1)--(b2)--(b3)--(b4)--(b1);
 
\draw[fill=white](c1)--(c2)--(c3)--(c4)--(c1);
 
\draw[fill=white](d1)--(d2)--(d3)--(d4)--(d1);
 

\draw(-6.3,3) rectangle (4,8);

\path(-4.5,10)--++(-100:1) coordinate (up);
\draw[red](d4)--(up);
\draw[red](a4)--(ua4); 
\fill[red](d4) circle(3pt);
\fill[red](a4) circle(3pt);
\fill[red](ua4) circle(3pt);
\fill[red](up) circle(3pt);

\end{tikzpicture}
\qquad
\begin{tikzpicture}[baseline, thick,yscale=0.75,xscale=-0.75]

\path(0,0) coordinate (112);
\path(112)--++(-10:1) coordinate (a1);
\path(112)--++(80:1) coordinate (a2);
\path(112)--++(170:1) coordinate (a3);
\path(112)--++(-100:1) coordinate (a4);
\draw(a1)--(a2)--(a3)--(a4)--(a1);

\path(a1)--++(80:1) coordinate (212);
\path(212)--++(-10:1) coordinate (b1);
\path(212)--++(80:1) coordinate (b2);
\path(212)--++(170:1) coordinate (b3);
\path(212)--++(-100:1) coordinate (b4);
\draw(b1)--(b2)--(b3)--(b4)--(b1);

\path(a3)--++(80:1) coordinate (312);
\path(312)--++(-10:1) coordinate (c1);
\path(312)--++(80:1) coordinate (c2);
\path(312)--++(170:1) coordinate (c3);
\path(312)--++(-100:1) coordinate (c4);
\draw(c1)--(c2)--(c3)--(c4)--(c1);

\path(-4.5,0) coordinate (112);
\path(112)--++(-10:1) coordinate (d1);
\path(112)--++(80:1) coordinate (d2);
\path(112)--++(170:1) coordinate (d3);
\path(112)--++(-100:1) coordinate (d4);
\draw(d1)--(d2)--(d3)--(d4)--(d1);

\path(d1)--++(80:1) coordinate (212);
\path(212)--++(-10:1) coordinate (e1);
\path(212)--++(80:1) coordinate (e2);
\path(212)--++(170:1) coordinate (e3);
\path(212)--++(-100:1) coordinate (e4);
\draw(e1)--(e2)--(e3)--(e4)--(e1);

\path(0,10) coordinate (u112);
\path(u112)--++(-10:1) coordinate (ua1);
\path(u112)--++(80:1) coordinate (ua2);
\path(u112)--++(170:1) coordinate (ua3);
\path(u112)--++(-100:1) coordinate (ua4);
\draw(ua1)--(ua2)--(ua3)--(ua4)--(ua1);

\path(ua1)--++(80:1) coordinate (u212);
\path(u212)--++(-10:1) coordinate (ub1);
\path(u212)--++(80:1) coordinate (ub2);
\path(u212)--++(170:1) coordinate (ub3);
\path(u212)--++(-100:1) coordinate (ub4);
\draw(ub1)--(ub2)--(ub3)--(ub4)--(ub1);

\path(ua3)--++(80:1) coordinate (u312);
\path(u312)--++(-10:1) coordinate (uc1);
\path(u312)--++(80:1) coordinate (uc2);
\path(u312)--++(170:1) coordinate (uc3);
\path(u312)--++(-100:1) coordinate (uc4);
\draw(uc1)--(uc2)--(uc3)--(uc4)--(uc1);

\path(ub2)--++(-10:1) coordinate (u322);
\path(u322)--++(-10:1) coordinate (ue1);
\path(u322)--++(80:1) coordinate (ue2);
\path(u322)--++(170:1) coordinate (ue3);
\path(u322)--++(-100:1) coordinate (ue4);
\draw(ue1)--(ue2)--(ue3)--(ue4)--(ue1);

\path(ue1)--++(80:1) coordinate (u222);
\path(u222)--++(-10:1) coordinate (ud1);
\path(u222)--++(80:1) coordinate (ud2);
\path(u222)--++(170:1) coordinate (ud3);
\path(u222)--++(-100:1) coordinate (ud4);
 \draw(ud1)--(ud2)--(ud3)--(ud4)--(ud1);


 \draw(a2)--(ua2);
 \path(a2)--++(-90:1)--++(180:1) coordinate (ga2);
\draw[densely dotted, gray, gray!80] (ga2)--++(90:10.3);
\draw(b2)--++(90:10);
\path(b2)--++(-90:1)--++(180:1) coordinate (gb2);
\draw[densely dotted, gray, gray!80] (gb2)--++(90:10.3);

\path(ue2)--++(-90:1)--++(180:1) coordinate (ue2g);
\path(ue2)--++(-90:6) coordinate (ue2below);
\path(ue2below)--++(180:1) coordinate (ue2gbelow);
\draw[black](c2)--++(90:2)to [out=90,in=-90] (ue2below) to [out=90,in=-90] (ue2);

 \path(c2)--++(-90:1)--++(180:1) coordinate (gc2);
 
\draw[densely dotted, gray](gc2)--++(90:3)to [out=90,in=-90] (ue2gbelow) --++(90:5.3);

\path(uc2)--++(-90:4.2) coordinate(belowuc2);
\path(uc2)--++(-90:0.7)--++(180:1) coordinate(uc2g);
\path(uc2g)--++(-90:3.5) coordinate(belowuc2g);
\draw (d2)--++(90:3) to [out=90,in=-90](belowuc2)--(uc2);
\path(d2)--++(-90:1)--++(180:1) coordinate (gd2);
\draw[densely dotted, gray, gray!80] (gd2)--++(90:4) to [out=90,in=-90](belowuc2g)--(uc2g);

\path(ud2)--++(-90:7.2) coordinate(belowud2);
\path(ud2)--++(-90:0.7)--++(180:1) coordinate(ud2g);
\path(ud2g)--++(-90:6.3) coordinate(belowud2g);

\draw[black] (e2)--++(90:1) to [out=90,in=-90](belowud2)--(ud2);
 
\path(e2)--++(-90:1)--++(180:1) coordinate (ge2);
\draw[densely dotted, gray] (ge2)--++(90:2) to [out=90,in=-90](belowud2g)--(ud2g);


 
\draw[fill=white](a1)--(a2)--(a3)--(a4)--(a1);
 
\draw[fill=white](b1)--(b2)--(b3)--(b4)--(b1);
 
\draw[fill=white](c1)--(c2)--(c3)--(c4)--(c1);
 
\draw[fill=white](d1)--(d2)--(d3)--(d4)--(d1);
 

\draw(-6.3,3) rectangle (4,8);

\path(-4.5,10)--++(-100:1) coordinate (up);
\draw[red](d4)--(up);
\draw[red](a4)--(ua4); 
\fill[red](d4) circle(3pt);
\fill[red](a4) circle(3pt);
\fill[red](ua4) circle(3pt);
\fill[red](up) circle(3pt);

\end{tikzpicture}
$$  

\!\!\!\!\!
\caption{Two distinct diagrams $\Cell_\SSTS$ and $\Cell_\SSTS'$ associated to the tableau $\SSTS\in \SStd_{(3;4,0)}(	( (1^2),(2,1)), ( \varnothing,(2,1^3))	 	)$  as in \cref{otherone}. }
\label{othertwo}
\end{figure}

\begin{defn}Given $\SSTS $ a  tableau of shape $\lambda$ and weight $\mu$, we let 
$\Cell_\SSTS \in {\sf 1}_\mu^{\imath}\algebra^\Bbbk_n(\theta)
 {\sf 1}_\la^{\res(\la)} $ denote any reduced diagram tracing out the bijection $\SSTS :[\la]\to [\mu]$.    
Given 
$\SSTS$, $\SSTT$ a pair of  tableaux of shape $\lambda$  (and possibly distinct weights) 
we set $\Cell_{\SSTS\SSTT}=\Cell_\SSTS \Cell_\SSTT^*$ where $\Cell^\ast_\SSTT$
is the diagram obtained from $\Cell_\SSTT $ by flipping it through the
horizontal axis.
  
 \end{defn}

  
%
%
%
%
%
%

\subsection{Right justification}
 The following total order refines the dominance order from \cref{sec1}.  
Given $i\in \ZZ/e\ZZ$ and $(r,c,m)$ and $(r',c',m')$ two $i$-boxes,  we write $(r,c,m) \preceq (r',c',m')$ if 
  \begin{itemize}
\item[$(i)$]   $\ct(r,c,m) < \ct(r',c',m')$ or 
\item[$(ii)$]   $\ct(r,c,m) = \ct(r',c',m')$  and either $m>m'$ or $m=m'$ and $r+c\leq r'+c'$.  
\end{itemize}  
For $\la,\mu\in \con \ell n$, we write $ \mu \preceq_{\sigma}  \la $  if there is a bijective   map ${\sf A}:[\lambda] \to [\mu]$ 
 such that   $  {\sf A}(r,c,m)\preceq_{\sigma}(r,c,m) $   for all $(r,c,m)\in \lambda$.  
 Given $\la,\mu \in     \con \ell n$ we note that $\la \lhd _\sigma \mu$ implies $\la \prec_\sigma \mu$.  
 We write $(r,c,m)\prec_{\rm co}  (r',c',m')$ if $(r,c,m)\prec_{\rm co}  (r',c',m')$ and there does not exist any 
$(r'',c'',m'')$ such that $(r,c,m)\prec  (r'',c'',m'')\prec  (r',c',m')$.  
In which case, we say that $(r,c,m)$ and $ (r',c',m')$ are {\sf consecutive} and say that the latter {\sf immediately follows} the former.

  \begin{rmk}\label{moredominant}
 Recall the definition of $\phi$ from \cref{phimap}.  
 We have that $\la\succ_\sigma \phi(\la)$, however $\la$ and $\phi(\la)$ are not relatable in the dominance ordering.  In particular  $\mu \rhd_\sigma \la$ if and only if $\mu \rhd_\sigma  
 \phi (\la)$.   \end{rmk}

   \begin{defn}
 Given  $\xi \in \mptn \ell n$ and 
    $(r,c,m)\in \xi$ we say that $(r,c,m)$ is {\sf right-justified}  if one of the following holds: 
    $(i)$ $(r-1,c,m)\in \xi$ 
      $(ii)$ $(r,c-1,m)\in \xi$
            $(iii)$ $(r-1,c-1,m)\in \xi$ $(iv)$ $r=c=1$.  
 We say that $\xi \in \con \ell n$ is {\sf right-justified}  if and only if every $(r,c,m)\in \xi$ is right justified.  
        \end{defn}

 Let $\xi  \in \con \ell n$ and suppose that $(r,c,m)\in \xi  $ is not right justified. 
We set $(r',c',m')$ equal to the box immediately following $(r,c,m)$ in the order $\prec_\sigma$  
and   $\xi'=(\xi\cup\{(r',c',m')\})\setminus \{(r,c,m)\}$. 
We say that $\xi'$ is obtained from $\xi$ by right-justifying the box $(r,c,m)$.  
 More generally, suppose there exists a chain
 $$
 \xi = \xi ^{(0)}\prec_\sigma  \xi^{(1)}
 \prec_\sigma \dots\prec_\sigma  \xi^{(r)}= \xi'
 $$ 
 and suppose that $\xi^{(i+1)}$ is obtained from $\xi^{(i)}$ by 
 right-justifying some box $(r_i,c_i,m_i)\in \xi^{(i)}$; then we say that $\xi'$ is 
    is obtained from $\xi$ by {\sf right-justification}.

\begin{figure}[ht]\captionsetup{width=0.9\textwidth}
\!\!\!\!\!\!\!\!\!\!
$$\scalefont{0.9}
\begin{tikzpicture}[scale=0.8]
 
\clip(-3.5,-4.5) rectangle (3.5,4.65);
 \draw(-2,-4)--(4,-4);
 \draw(0,-4.25) node {0};
  \path(0,-4.25)--++(45:0.5)--++(-45:0.5) coordinate (hello) node {1};
  \path(0,-4.25)--++(45:-0.5)--++(-45:-0.5) coordinate (hello2) node {$-1$};
  \path(hello2)--++(45:-0.5)--++(-45:-0.5) coordinate (hello2) node {$-2$};  
    \path(hello)--++(45:0.5)--++(-45:0.5) coordinate (hello) node {2};
        \path(hello)--++(45:0.5)--++(-45:0.5) coordinate (hello) node {3};
            \path(hello)--++(45:0.5)--++(-45:0.5) coordinate (hello) node {4};
 
\begin{scope}{ 
  \path(0,0.4) coordinate (origin);

  \path(origin)--++(135:0.5)--++(45:1.5) node {$1$};
    \path(origin)--++(135:0.5)--++(45:2.5) node {$2$};

 \fill[pink](origin)--++(135:4) --++(45:1)--++(-45:1)--++(-135:1); 
\path(origin)--++(135:2)--++(45:2) coordinate (hi);
%
  \draw[thick](origin)--++(135:2) --++(45:1)  --++(135:1) --++(-135:1)
 --++(135:1) --++(45:1) --++(-45:1) 
 --++(45:1) --++(-45:1) 
  --++(-135:1) --++(-45:1)
   --++(45:2) --++(-45:1) --++(-135:3) ; 

%

\clip
 (origin)--++(135:2) --++(45:1)  --++(135:1) --++(-135:1)
 --++(135:1) --++(45:1) --++(-45:1) 
 --++(45:1) --++(-45:1) 
  --++(-135:1) --++(-45:1)
   --++(45:2) --++(-45:1) --++(-135:3) ;

  \foreach \i\j in {1,...,19}
  {
    \path (origin)++(45:1*\i cm)  coordinate (a\i);
    \path (origin)++(135:1*\i cm)  coordinate (b\i);
    \path (a\i)++(135:1*\j cm) coordinate (ca\i,\j);
    \path (b\i)++(45:1*\j cm) coordinate (cb\i,\j);

  }
    \foreach \i in {1,...,19}
{  \draw[black,thick] (a\i)--++(135:8);
    \draw[black,thick] (b\i)--++(45:8);}
      \draw[black,thick] (origin)--++(135:8);
            \draw[black,thick] (origin)--++(45:8);

}
\end{scope}
\draw[wei](origin)--++(-90:4.4);

\path (origin)--++(45:0.5)--++(135:0.5) coordinate (origin);

  \foreach \i in {1,...,4}
  {
     \path (origin)++(45:1*\i cm)  coordinate (xx\i);
    \path (origin)++(135:1*\i cm)  coordinate (xy\i);

  }
 \draw(xy3) node {$0$};
 \draw(xy1) node {$2$};  
\path (xy1) --++(45:1)--++(135:1)  node {$2$}; 

\path(origin)--++(45:0.5)--++(135:0.5) coordinate (xc3);
\draw(origin) node {0};  
\path(xc3)--++(45:0.5)--++(135:0.5) coordinate (xcd);


\path(0,0)--++(45:0.25)--++(-45:0.25)--++(-90:3.2) coordinate (origin);

\path(origin)--++(45:2)--++(135:1)coordinate (hi);
\fill[lime!80](hi)--++(45:1)--++(-45:1)--++(-135:1);

  \fill[cyan!20](origin)--++(135:3) --++(45:1)--++(-45:1)--++(-135:1);

\draw[wei](origin)--++(-90:0.8);
\begin{scope}
{   
\draw[thick]
(origin)--++(135:1)--++(45:1)--++(135:1)--++(-135:1)--++(135:1)--++(45:1)--++(-45:1)--++(45:1)--++(-45:1)--++(45:1)--++(-45:1)--++(-135:1)
 --++(135:1)--++(-135:1)--++(-45:1)--++(-135:1);

 \clip(origin)--++(135:1)--++(45:1)--++(135:1)--++(-135:1)--++(135:1)--++(45:1)--++(-45:1)--++(45:1)--++(-45:1)--++(45:1)--++(-45:1)--++(-135:1)
 --++(135:1)--++(-135:1)--++(-45:1)--++(-135:1);

 \clip(origin)--++(45:3)--++(135:1)--++(-135:1)--++(135:1)--++(-135:1)--++(135:1)--++(-135:1)--(origin);
   \foreach \i\j in {1,...,19}
  {
    \path (origin)++(45:1*\i cm)  coordinate (a\i);
    \path (origin)++(135:1*\i cm)  coordinate (b\i);
    \path (a\i)++(135:1*\j cm) coordinate (ca\i,\j);
    \path (b\i)++(45:1*\j cm) coordinate (cb\i,\j);

  }
    \foreach \i in {1,...,19}
{  \draw[black,thick] (a\i)--++(135:8);
    \draw[black,thick] (b\i)--++(45:8);}
      \draw[black,thick] (origin)--++(135:8);
            \draw[black,thick] (origin)--++(45:8);
}

  \clip(origin)--++(45:8)--++(135:8)--++(-135:8)--(origin);

\path (origin)--++(45:0.5)--++(135:0.5) coordinate (origin);

  \foreach \i in {1,...,4}
  {
     \path (origin)++(45:1*\i cm)  coordinate (x\i);
    \path (origin)++(135:1*\i cm)  coordinate (y\i);

  }
\draw(x2) node {0};
\draw(y2) node {$2$};  
 \path(origin)--++(45:0.5)--++(135:0.5) coordinate (c3);
\draw(origin) node {1};  
\path(c3)--++(45:0.5)--++(135:0.5) coordinate (cd);
\draw(cd) node {1};

\end{scope}
\end{tikzpicture}
\qquad
\qquad
\begin{tikzpicture}[scale=0.8]
 
\clip(-3.5,-4.5) rectangle (3.5,5.5);
 \draw(-2,-4)--(4,-4);
 \draw(0,-4.25) node {0};
  \path(0,-4.25)--++(45:0.5)--++(-45:0.5) coordinate (hello) node {1};
  \path(0,-4.25)--++(45:-0.5)--++(-45:-0.5) coordinate (hello2) node {$-1$};
  \path(hello2)--++(45:-0.5)--++(-45:-0.5) coordinate (hello2) node {$-2$};  
    \path(hello)--++(45:0.5)--++(-45:0.5) coordinate (hello) node {2};
        \path(hello)--++(45:0.5)--++(-45:0.5) coordinate (hello) node {3};
            \path(hello)--++(45:0.5)--++(-45:0.5) coordinate (hello) node {4};
 
\begin{scope}{ 
  \path(0,0.4) coordinate (origin);

  \fill[cyan!20](origin)--++(135:4) --++(45:3)--++(-45:1)--++(-135:1)--++(135:1)--++(-135:2);

   \fill[lime!80](origin)--++(45:4) --++(135:1)--++(-135:1)--++(-45:1);  
  \path(origin)--++(135:0.5)--++(45:1.5) node {$1$};
    \path(origin)--++(135:0.5)--++(45:2.5) node {$2$};
        \path(origin)--++(135:3.5)--++(45:2.5) node {$2$};
        \path(origin)--++(135:0.5)--++(45:3.5) node {$0$};
\path(origin)--++(135:2)--++(45:2) coordinate (hi);

 \draw[thick](origin)--++(135:2) --++(45:1)  --++(135:1) 
 --++(45:1) --++(135:1)
 --++(45:1) --++(-45:1) --++(-135:1)  --++(-45:1) 
  --++(-135:1) --++(-45:1)--++(45:3)--++(-45:1)--++(-135:4);

\clip
 (origin)--++(135:2) --++(45:1)  --++(135:1) 
 --++(45:1) --++(135:1)
 --++(45:1) --++(-45:1) --++(-135:1)  --++(-45:1) 
  --++(-135:1) --++(-45:1)--++(45:3)--++(-45:1)--++(-135:3); 

  \foreach \i\j in {1,...,19}
  {
    \path (origin)++(45:1*\i cm)  coordinate (a\i);
    \path (origin)++(135:1*\i cm)  coordinate (b\i);
    \path (a\i)++(135:1*\j cm) coordinate (ca\i,\j);
    \path (b\i)++(45:1*\j cm) coordinate (cb\i,\j);

  }
    \foreach \i in {1,...,19}
{  \draw[black,thick] (a\i)--++(135:8);
    \draw[black,thick] (b\i)--++(45:8);}
      \draw[black,thick] (origin)--++(135:8);
            \draw[black,thick] (origin)--++(45:8);

}
\end{scope}
\draw[wei](origin)--++(-90:4.4);

\path (origin)--++(45:0.5)--++(135:0.5) coordinate (origin);

  \foreach \i in {1,...,4}
  {
     \path (origin)++(45:1*\i cm)  coordinate (xx\i);
    \path (origin)++(135:1*\i cm)  coordinate (xy\i);

  }
 \draw(xy1) node {$2$};  
\path (xy1) --++(45:1)--++(135:1)  node {$2$}; 

\path(origin)--++(45:0.5)--++(135:0.5) coordinate (xc3);
\draw(origin) node {0};  
\path(xc3)--++(45:0.5)--++(135:0.5) coordinate (xcd);


\path(0,0)--++(45:0.25)--++(-45:0.25)--++(-90:3.2) coordinate (origin);

\path(origin)--++(45:2)--++(135:1)coordinate (hi);


\draw[wei](origin)--++(-90:0.8);
\fill[pink](origin)--++(135:2)--++(45:2)--++(135:1)--++(-135:1)--++(-45:1)--++(-135:1) ; 
\draw[wei](0,0.4)--++(-90:4.4);
 \path(origin)--++(135:2.5)--++(45:1.5) node {0};

\begin{scope}
{   
\draw[thick]
(origin)--++(135:1)--++(45:1)
--++(135:1)
--++(135:1)--++(45:1)--++(-45:1)--++(-135:1)
--++(45:1)--++(-45:1) --++(-135:1)--++(-45:1)--++(-135:1);

 \clip(origin)--++(135:1)--++(45:1)
--++(135:1)
--++(135:1)--++(45:1)--++(-45:1)--++(-135:1)
--++(45:1)--++(-45:1) --++(-135:1)--++(-45:1)--++(-135:1);

 \clip(origin)--++(45:3)--++(135:1)--++(-135:1)--++(135:1)--++(-135:1)--++(135:1)--++(-135:1)--(origin);
   \foreach \i\j in {1,...,19}
  {
    \path (origin)++(45:1*\i cm)  coordinate (a\i);
    \path (origin)++(135:1*\i cm)  coordinate (b\i);
    \path (a\i)++(135:1*\j cm) coordinate (ca\i,\j);
    \path (b\i)++(45:1*\j cm) coordinate (cb\i,\j);

  }
    \foreach \i in {1,...,19}
{  \draw[black,thick] (a\i)--++(135:8);
    \draw[black,thick] (b\i)--++(45:8);}
      \draw[black,thick] (origin)--++(135:8);
            \draw[black,thick] (origin)--++(45:8);
}

  \clip(origin)--++(45:8)--++(135:8)--++(-135:8)--(origin);

\path (origin)--++(45:0.5)--++(135:0.5) coordinate (origin);

  \foreach \i in {1,...,4}
  {
     \path (origin)++(45:1*\i cm)  coordinate (x\i);
    \path (origin)++(135:1*\i cm)  coordinate (y\i);

  }
\draw(x2) node {0};
\draw(y2) node {$2$};  
 \path(origin)--++(45:0.5)--++(135:0.5) coordinate (c3);
\draw(origin) node {1};  
\path(c3)--++(45:0.5)--++(135:0.5) coordinate (cd);
\draw(cd) node {1};  
 \path(cd)--++(135:1)node {0};

\end{scope}
\end{tikzpicture}
 $$\caption{For $e=3$ and $\sigma =(0,1)\in\ZZ^2$ we depict  a box configuration and  its right justification, respectively.  
  } 
\label{ALoading333334}
\end{figure}

\subsection{A spanning set of the algebra}\label{span}
In this subsection we provide a spanning set for $\algebra^\Bbbk_n(\theta)$ and provide analogues of a number of results from \cite{bkw11}.  This section is inspired by the     ideas of  \cite{MR3732238,bkw11}.     
In this section      we shall assume, without loss of generality, that $s_1\geq s_2\geq \dots \geq s_m$ (this is simply for ease of notation when describing the maximal and minimal elements of the dominance order  and one can simply reorder the charge if necessary).

\begin{prop} \label{step1}
Any $\sigma$-diagram $\Diag \in {\sf 1}_\la ^\imath \algebra^\Bbbk_n(\sigma) {\sf 1}_\mu  ^\jmath $ can be written in the form
$\Diag = a {\sf 1}_\xi a'  $ for $\xi$ a right-justified box configuration such that
 $\xi \succeq_\sigma \la, \mu$.   Moreover any idempotent ${\sf 1}_\xi$ for $\xi \in\con \ell n$ 
 belongs to $\algebra^\Bbbk_n(\sigma){\sf 1}_{\xi'} \algebra^\Bbbk_n(\sigma)$ for $\xi'\succeq \xi$ obtained from  $\xi $ by right justification.  
  \end{prop}

\begin{proof}{\renewcommand{\prec}{\lessdot}
\renewcommand{\succ}{\gtrdot}
 Let $y\in [2\varepsilon,1-2\varepsilon]$.  
Let $\Strand$ denote a solid or red strand in the diagram and let
 $x(\Strand)$ denote the $x$-coordinate of this strand  at the point  where it intersects  the line $\{y\}\times \RR$.
We write $\Strand_1 \prec \Strand_2$ if 
\begin{itemize}
\item[$(i)$] $\Strand_1$ and $\Strand_2$  are solid-strands of the same residue and $0< x(\Strand_1)-x(\Strand_2)\leq 2\varepsilon$;
\item[$(ii)$] $\Strand_1$ is a solid $(i+1)$-strand  $\Strand_2$  is a solid  $i$-strand   and $\g < x(\Strand_1)-x(\Strand_2)\leq \g+\varepsilon$; 
\item[$(iii)$] $\Strand_1$ is a solid $(i-1)$-strand  $\Strand_2$  is a solid  $i$-strand   and  $  - \g  \leq x(\Strand_1)-  x(\Strand_2)<   \varepsilon - \g  $;    
\item[$(iv)$]  $\Strand_1$ is a red $i$-strand  $\Strand_2$  is a solid  $i$-strand   and $0< x(\Strand_1)-x(\Strand_2)\leq   2 \varepsilon$.  
\end{itemize}
We extend this to a partial ordering on strands by taking the transitive closure.
We  first explain how, applying {\em only the  non-interacting relations} to our diagram $\Diag $, 
we can group the strands into $\prec$-equivalence classes.  
 We   then show that these equivalence classes correspond to the components of a box-configuration.  
The reader might already see how the definition of right justification of boxes 
intuitively captures this process (such a reader is invited to skip the rest of this proof, as it is merely an in depth description of this process).  
 We remark that cases $(i)$ to $(iii)$ correspond to $i$-bricks   ${\bf M}_3$,  $\B_3$, $\B_2$.    

 Consider the region $\Diag \cap (\RR \times [y-2\varepsilon,y+2\varepsilon])$ of our diagram $\Diag $.  We may assume   that there are no   crossing   strands in this region and moreover that all strands in this region are vertical lines (by applying local isotopy \ref{rel1} if necessary, we can move any crossings above or below the region and straighten all the strands within the region).
     
Let   $\Strand_2$ denote a strand  in the diagram $\Diag   \cap (\RR \times [y-2\varepsilon,y+2\varepsilon])$. 
We pull  the strand $\Strand_2$ rightwards under the process outlined above.  
 Let $\Strand_1$ denote any strand in  $\Diag   \cap (\RR \times [y-2\varepsilon,y+2\varepsilon])$ which $\Strand_2$ interacts with during this process of being pulled rightwards.  
 Then either $(i)$ the $\Strand_2$-strand   passes through the  $\Strand_1$-strand using the non-interacting relations or
$(ii)$ the $\Strand_2$-strand comes to a halt at a point such that the strands $\Strand_1$ and $\Strand_2$ are in one of cases $(i)$ to $(iv)$ above.  
Having obtained $\Diag '$ (which we assume is not zero under relation \ref{rel15}) by pulling all solid strands as far right as possible 
 in this manner (while keeping the red strands fixed) we find that the solid and red strands in the diagram $\Diag '$ have naturally gathered into $\ell$ distinct $\prec$-connected components (each containing precisely 1 red strand).
This is simply because $(i)$ the difference between the $x$-coordinates of two red strands is 
equal to an element of $\tfrac{1}{\ell}\ZZ$  
$(ii)$ $\varepsilon \ll \tfrac{1}{2n\ell}$
$(iii)$ if $\Strand_1\prec \Strand_2 \prec \dots \prec \Strand_k$ for $1\leq k \leq n$, then $|x(\Strand_k)- x(\Strand_1)| \leq 2n\varepsilon \pmod \ZZ$ using  $(i)$ to $(iv)$.

We now consider the $m$th $\prec$-connected component, denoted  $\Theta_m$, of strands containing the vertical red-strand with $x$-coordinate $\sigma_m-m/\ell $ and residue $s_m\in (\ZZ/e\ZZ)$.  
If $\Theta_m$ contains no vertical strands, then this $\prec$-connected component corresponds to  an empty component of the box configuration and we are done.  
Assume $\Theta_m$ contains at least one solid strand (i.e.~$|\Theta_m|>1$).  
Then by $(i)$ to $(iv)$ there exists at least one solid $ {s}_m$-strand, $\Strand_1\in \Theta_m$, 
in the region $[\sigma_m-m/\ell-2\varepsilon,\sigma_m-m/\ell)$.
If $|\Theta_m|>2$, then by $(i)$ to $(iv)$  there exists $\Strand_2\in \Theta_m$ such that  one of the following holds $(i)$ $\Strand_2$ is a solid ${s}_m$-strand with $x(\Strand_2)\in [\sigma_m-m/\ell -4\varepsilon ,  \sigma_m-m/\ell)$  
or $(ii)$ $\Strand_2$ is a solid $({s}_m-1)$-strand with $x(\Strand_2)\in [\sigma_m-m/\ell-\g- 3\varepsilon, \sigma_m-m/\ell-\g )$ 
or
$(iii)$ 
$\Strand_2$ is a solid $(s_m+1)$-strand with $x(\Strand_2)\in [\sigma_m-m/\ell-\g- 3\varepsilon, \sigma_m-m/\ell-\g)$.  
Continuing in this fashion, 
we find that  
 for  $x\in \ZZ/e\ZZ$ 
 the solid strands in the  region  $\cup_{k\in\NN}[ke+x-m/\ell- 2n\varepsilon, ke+x-m/\ell  )$ 
 are precisely   the  solid $x$-strands in $\Theta_m$.  
By isotopy, we can  
assume that each strand in $\Theta_m$ has maximal $x$-coordinate  such that the strands are still related under $\succ$. 
In so doing, we find that   each  $i$-strand  intersects the line  $y $ at some point equal to $\mathbf{I}^\theta_{(r,c,m)} $ for some     $i$-box $(r,c,m)$.
We now restrict to the region 
$X=\Diag '  \cap (\RR \times [y- \varepsilon,y+ \varepsilon])$ and by isotopy we can assume that all strands in this region are vertical. We hence have that $X={\sf1}_\xi$ for some $\xi \in \con \ell n$ and that any pair  of vertical strands in $X={\sf1}_\xi$ is  of the form $(i)$, $(ii)$, $(iii)$, or $(iv)$ above.   
   In particular if $(r,c,m)\in \xi$ then this implies  that at least one of $(r-1,c,m),(r,c-1,m),(r-1,c-1,m)\in \xi$ or $r=c=1$ as required.

For the second claim, let $X={\sf 1}_\xi$ for $\xi\in \con \ell n$.  
Moving a strand corresponding to   $(r,c,m)\in \xi$ rightwards using non-interacting relations (in the process above!) corresponds to the process of right-justifying the box $(r,c,m)\in \xi$.  The result follows.  
 }  \end{proof}

\begin{cor}
The algebra $\algebra_n^\Bbbk(\sigma)$ is {finitely} generated 
by the set of all reduced diagrams and the elements $y_{(r,c,m)}{\sf 1} _\mu$ for $\mu \in  \con \ell n$.  
\end{cor}   

\begin{proof}
As in the proof of  \cref{step1}, we can make successive horizontal cuts to any $\Diag \in \algebra_n^\Bbbk(\sigma)$ until 
we have rewritten $\Diag = \Diag _1 \Diag _2 \dots \Diag _k$ such that each $\Diag _i$ is either reduced (with northern and southern loadings   given by 
 box configurations) or is obtained by adding a single dot to a strand in a weight idempotent.  
\end{proof}

\begin{rmk}
The upshot of \cref{step1} is that we can organise our proofs (which will involve manipulating $\sigma$-diagrams) by induction on the dominance order on 
box-configurations.  In future proofs, we can gloss over any steps involving the non-interacting relations and instead focus on 
 manipulating $\sigma$-diagrams corresponding to  right-justified  box-configurations.   
   \end{rmk}
   
   \begin{prop}\label{aprop}
   For $\mu \in \con \ell n$ with $(r,c,m) \in \mu$  and $\la \in \con \ell n\setminus \mptn \ell n $, we have that
$$ y_{(r,c,m)}{\sf 1} _\mu \in \algebra^{\rhd \mu}_n(\theta)
\qquad
 {\sf 1} _\lambda \in \algebra^{\rhd \la}_n(\theta)   
 $$
where $ \algebra^{\rhd  \mu }_n(\theta) =
\algebra^\Bbbk_n(\theta) \langle {\sf1} _\nu \mid \nu\in \mptn \ell n, \nu \rhd \mu\rangle\algebra^\Bbbk_n(\theta)$ and the elements 
are defined in \cref{idempotentsssss}.  
   
   \end{prop}
   
   \begin{proof}
%
We shall prove both statements  simultaneously  by (intertwined) reverse induction on  the dominance ordering on box-configurations.    
   Let  $\xi:=((n), \varnothing, \dots ,\varnothing)\in \con \ell n$.  
If $\xi' \in \con \ell n$ is such that $\xi'  \vartriangleright \xi$, then it easy to see that   ${\sf 1}_{\xi'}=0$ under relation \ref{rel15}. 
Applying  relation  \ref{rel6} or \ref{rel11} to $y_{(r,c,m)}{\sf 1}_\xi$ as necessary, we have that 
  $y_{(r,c,m)}{\sf 1}_\xi=0$ by relation \ref{rel15}.  
Therefore the base case for induction holds.
Now, assume that $\xi\in \con \ell n$ is right-justified (by \cref{step1}) and suppose that the result has been proven for all box-configurations strictly more dominant than $\xi$.  
We refine our induction by the natural ordering on  $r+c\geq 2$. We assume that for all   $r'+c' < r+c$, we have
\begin{itemize}[leftmargin=*]
\item[$(\alpha)$]   if $(r',c',m')\not\in \xi$ this implies   $(r'+1,c',m)\not\in \xi$ and  $(r',c'+1,m)\not\in \xi$.  
 \item[$(\beta)$]  $y_{(r',c',m')}{\sf 1}_\xi \in \algebra^{\rhd \xi}_n(\sigma)$.  
\end{itemize}
We first check the base case  for our   inductive assumptions for which  $r+c=2$. 
\begin{itemize}[leftmargin=*]
 \item[(a)] We have that $r+c=2$ and $(1,1,m)\not\in \xi $ for some $0\leq m <\ell$.  We can pull the strand labelled by $(1,2,m)$ or $(2,1,m)$ to the right using the non-interacting relations  until the strand labelled by $(1,2,m)$ or $(2,1,m)$
encounters the $(i+1)$- respectively $(i-1)$-box immediately following  $ {(1,2,m)}$ or ${(2,1,m)}$.  
The resulting diagram factors through an idempotent ${\sf 1}_{\xi' }$ for    
 $\xi' \in \con \ell n$ such that   $\xi' \rhd \xi$ and   
 the result follows by induction on the dominance ordering on box-configurations.
\item[(b)]  We have that $r+c=2$,   $(1,1,m) \in \xi $, and 
we have that a dot on the   strand labelled by $(1,1,m) \in \xi $ for some $0\leq m <\ell$.
We apply relation \ref{rel11}  to the solid strand labelled by $(1,1,m)$ and the red strand with $x$-coordinate $\sigma_m-m/\ell$.  We can now pull the solid strand rightwards   to obtain a more dominant box configuration; the result follows by induction on the $\rhd _\sigma$ ordering on $\con \ell n$.  
\end{itemize}
Now for the inductive step.  
\begin{itemize}[leftmargin=*]
\item[(A)]
We have that $(r,c,m) \not\in \xi$ for some $r+c>2$, but that  $(a,b,m) \in \xi$ for all $2\leq a+b < r+c$.  
   If $(1,c+1,m)\in \xi$  and $(1,c,m)\not \in \xi$ 
   (similarly if $(r+1,1,m)\in \xi$  and $(r,1,m)\not \in \xi$) then  then  one can argue as in the $r=c=1$ case above. 
We now assume this is not the case.  
 Since $\xi$ is right justified,    there are three cases to consider (corresponding to the ${\bf M}_1$,  $\N_2$, and ${\bf M}_3$ bricks).   
 \begin{itemize}
 \item[$(\N1)$]  Suppose $(r-1,c+1,m), (r,c+1,m) \in \xi$ (and note that  $(r-1,c,m)\in \xi$ by induction).   
    We apply the relation depicted in \cref{M1relation,whatslovegottodo} 
     to the   the triple of strands in $\sf 1_\xi$ labelled by $(r,c+1,m)$, $(r-1,c+1,m)$ and $(r-1,c,m)$.  We hence obtain  a sum of two diagrams $X_2-X_1$.  
We have that  $X_1  \cap (\RR\times \{1/2\}) = y_{(r-1,c,m)}{\sf 1}_{\xi'}$ for $\xi'=\phi^{NE}_{(r-1,c+1,m)}(\xi)$. 
We have that $X_2  \cap (\RR\times \{1/2\}) ={\sf1}_{\xi''}$ where $\xi''=\phi^{N}_{r-1,c,m}(\xi)$ and we note 
 that  ${(r-1,c,m)}\not\in \xi''$.  
 Therefore $X_1$ and $X_2$ factor  through  idempotents    strictly more dominant than $\xi'$ and $\xi''$ respectively (by induction on $r+c$)  and therefore both factor through an idempotent strictly more dominant than $\xi$ by \cref{moredominant}.  
 \item[$(\N2)$] Suppose   $(r+1,c,m)\in \xi$ and and $(r+1,c-1,m)\in \xi$  (and $(r,c-1,m)\in \xi$ by induction).
  This case is similar to $(\N1)$ except that we use the {\em rightmost} (not the leftmost)  relation \ref{rel6} in
  the analogue of  \cref{M1relation}. 
 \item[$(\N3)$] 
Now suppose that  $ (r-1,c+1,m)\not \in \xi$ (the case $(r+1,c-1,m)\not\in \xi$ is identical).  
We apply the rightmost equation of  \cref{N3} to the strands labelled by $(r-1,c,m)$ and $(r,c+1,m)$ to obtain a sum of two diagrams
$Y'_1-Y_2'$ which both factor through the diagram  $y_{(r-1,c,m)}{\sf 1}_{\xi}$.   
Therefore both factor through an idempotent strictly more dominant than $\xi$  by induction on $(r+c)$.   
 \end{itemize}
\item[(B)] If $r=1$ or  $c=1$  then  one can argue as in the $r=c=1$ case above   with the  exception that  we replace the reference  to relation \ref{rel11}  with \ref{rel7}.  
For $r,c>1$,   our inductive assumption implies 
 $(r-1,c,m), (r,c-1,m), (r-1,c-1,m)\in  \xi$.  
    We apply the relation depicted in \cref{M1relation}       to the   the quadruple of strands in $\sf 1_\xi$ labelled by  $(r-1,c,m), (r,c-1,m)$, $(r-1,c-1,m)$, and $(r,c,m)$.  We hence obtain  a sum of two diagrams $Z_1+Z_2$.  
We have that  $Z_1  \cap (\RR\times \{1/2\})= y_{(r,c-1,m)}{\sf 1}_{\xi}$.   
We have that $Z_2  \cap (\RR\times \{1/2\}) = {\sf1}_{\xi'}$ where $\xi'=\phi^{NW}_{r,c-1,m}(\xi)$ and we note  that  ${(r,c-1,m)}\not\in \xi'$.  
 Therefore $Z_1$ and $Z_2$ factor  through  idempotents    strictly more dominant than $\xi$ and $\xi'$ respectively (by induction on $r+c$)  and so both factor through an idempotent strictly more dominant than $\xi$ (by \cref{moredominant}).  \qedhere
 \end{itemize}
 \end{proof}

\begin{thm}\label{cellularitybreedscontempt0.5}  
We let 
 $  \lambda^{(0)} \geq \lambda^{(1)}\geq\lambda^{(2)}\geq\ldots\geq  \lambda^{(m)}  $  denote any  total refinement of the order, $\rhd_\theta$, on $\mptn \ell n$.
 The $\Bbbk$-algebra $\algebra^\Bbbk_n(\theta )$ has a filtration 
\[
0 \subset   \algebra^{\geq  \lambda^{(0)} }_n(\theta) \subseteq  \algebra^{\geq  \lambda^{(1)} }_n(\theta) \subseteq \dots
\subseteq   \algebra^{\geq  \lambda^{(m)}}_n(\theta)  = \algebra^\Bbbk_n(\theta ).
\]
where $ \algebra^{\geq  \lambda }_n(\theta) =
\algebra^\Bbbk_n(\theta)\langle {\sf1} _\nu \mid \nu\in \mptn \ell n, \nu \geq \la\rangle\algebra^\Bbbk_n(\theta)$.  
\end{thm}
\begin{proof}
This follows immediately from \cref{step1,aprop}
\end{proof}
 
 \begin{prop}\label{move a dot down} Let $\la, \mu \in \con  \ell n$  with $(r,c,m) \in [\mu]$, $\imath,\jmath \in (\ZZ/e\ZZ)^n$, 
and    $\Diag \in \Rset$.  
 We have that 
 $$  y_{(r,c,m)}{\sf 1}^{\imath}_\mu  
  \Diag = \Diag    y_{(r',c',m')}{\sf 1}^{\imath}_\la   +  \sum_{
 \begin{subarray}c
 \Diag '  \in \Rset  \\
  \Diag  '\rhd_\theta \Diag  
 \end{subarray}
 }     \Diag ' . $$  
 Here
 $(r',c',m') \in[\la]$  on the southern edge  is connected to $(r,c,m)$ on the northern edge.      \end{prop}
 \begin{proof}
Now consider a general diagram   $\Diag    \in {\sf 1}^{\imath}_\mu\algebra^\Bbbk_n(\theta){\sf 1}_\la^{\jmath}$.   
If there is a dot placed at the top of the strand labelled by  $(r,c,m)\in[\mu]$ we move this dot along the strand towards the bottom of the diagram using (homogenous)  relations \ref{rel2}, \ref{rel3} and \ref{rel14}.
We hence rewrite $\Diag $ as  a linear combination of diagrams  $\Diag' \in 
{\sf 1}_\mu^\imath \algebra^\Bbbk_n(\theta ){\sf 1}_\la^\jmath$ where  each $\Diag '$   differs from $\Diag $ only in that   one or zero crossings of like-labelled strands have been undone and  there is either   zero or one   dots along the southernmost edge. This amounts to removing zero or one of the Coxeter generators in the  reduced expression $[\Diag ]$ to obtain a reduced expression of $[\Diag ']$. Hence this sum is over diagrams $\Diag '$ such that $\Diag '\rhd_\theta \Diag $, as required.   
  \end{proof}
 \begin{eg}
 A step in the   procedure outlined in the proof of \cref{move a dot down} is  carried out in obtaining  \cref{initialdiagram} 
 from  \cref{dotcross}.   \end{eg}

\begin{prop}\label{mastumoto}
 
 Let $\la, \mu \in \con  \ell n$ and $\imath,\jmath \in (\ZZ/e\ZZ)^n$.  
If $w=s_{t_1,t_1+1}\dots s_{t_m,t_m+1}$ and $w=s_{r_1,r_1+1}\dots s_{r_m,r_m+1}$ are two reduced expressions and 
   $\Diag , \Diag ' \in
\Rset $
  are two reduced diagrams  with $[\Diag ]= s_{t_1,t_1+1}\dots s_{t_m,t_m+1}$ and 
  $[\Diag ']=s_{r_1,r_1+1}\dots s_{r_m,r_m+1}$, then 
 $$\Diag    = \Diag ' + \sum_{
 \begin{subarray}c
 \Diag '' \in \Rset  \\
  \Diag ''\rhd_\theta \Diag ' 
 \end{subarray}
 }   \Diag '' f_{\Diag ''}(y)\in
  {\sf 1}^{\imath}_\mu\algebra^\Bbbk_n(\theta){\sf 1}_\la^{\jmath}
  \quad \text{for some $f_{\Diag ''}(y)\in {\mathcal Y}^\jmath_\la$}.$$
 \end{prop}
\begin{proof}Recall Matsumoto's theorem  states that any two reduced expressions for $w\in \mathfrak{S}_n$ differ only by applying a sequence of the relations $s_{i,i+1}s_{i+1,i+2}s_{i,i+1}=s_{i+1,i+2}s_{i,i+1} s_{i+1,i+2}$. 

By assumption  neither $\Diag $ nor $\Diag '$ contains a double crossing or a dot on any strand.  Therefore we can obtain  $\Diag $ from $\Diag '$  by   applying {\em only} the strand-through-crossing relations \ref{rel8}, \ref{rel9}, and \ref{rel12} successively. 
These are precisely the relations   which rewrite a subproduct $s_{i,i+1}s_{i+1,i+2}s_{i,i+1}$ in $[\Diag ']$ in the form 
$s_{i+1,i+2}s_{i,i+1} s_{i+1,i+2}\pm {1}_{\mathfrak{S}_{n+2\ell}}$.  Thus   one can rewrite the diagram in the required form at the expense of some error terms $\Diag ''$ such that $[\Diag '']=u<w$, however we note that $\Diag ''$ may no longer be reduced.  If $\Diag ''$ is reduced, then $g_{\Diag ''}(y)=\pm 1$.    If $\Diag ''$ is not reduced, then rewriting $\Diag ''$ as a linear combination of reduced diagrams involves  reapplying some combination of relations \ref{rel8}, \ref{rel9},   \ref{rel12}, or   \ref{rel5}  (which simply creates more error terms $\Diag '''$ with 
$[\Diag ''']=v\leq u<w$ and $g_{\Diag '''}(y)=\pm1$) 
  and \ref{rel6} (which creates  error terms $\Diag '''$ with 
$[\Diag ''']=v<u<w$ and $g_{\Diag '''}(y) $ a polynomial of degree 1).   
Once this process terminates we are left with a combination of more dominant  diagrams, but with dots in the middle of the diagram (which we now need to move southwards).  
 We can isotope  the neighbourhood of any diagram (in particular any region containing a dot)
  so that it is of the form ${\sf 1}_{\nu}$ for some $\nu \in \con \ell n$; we then apply 
 \cref{move a dot down} to deduce the result.  
\end{proof}

\begin{prop}\label{mastumoto2}
 Let $\la, \mu \in \con  \ell n$ and $\imath,\jmath \in (\ZZ/e\ZZ)^n$.  
 Let    $\Diag  \in
 {\sf 1}^{\imath}_\mu\algebra^\Bbbk_n(\theta){\sf 1}_\la^{\jmath} 
 $
 be a diagram which is {\em not reduced}, then 
 $$\Diag    =    \sum_{
 \begin{subarray}c
 \Diag '  \in \Rset  \\
  \Diag  '\rhd_\theta \Diag  
 \end{subarray}
 }     \Diag '  f_{\Diag '}(y)\in
  {\sf 1}^{\imath}_\mu\algebra^\Bbbk_n(\theta){\sf 1}_\la^{\jmath} \quad \text{for some $f_{\Diag ''}(y)\in {\mathcal Y}^\jmath_\la$}.$$
 \end{prop}
\begin{proof}
 Any non-reduced diagram $\Diag \in  {\sf 1}^{\imath}_\mu\algebra^\Bbbk_n(\theta){\sf 1}_\la^{\jmath}$ contains either a double-crossing of strands or a dot on a strand. By \cref{move a dot down} we can restrict our attention to the former   case.  We choose  $y\in [\varepsilon,1-\varepsilon]$ minimal such that   
$Y=\Diag \cap(\RR\times(y,1])$ contains this double crossing.  
We have that $X=\Diag \cap(\RR\times[0,y ))$ 
does not contain this double-crossing and so  $X$  is a reduced diagram.   
  By \cref{mastumoto} we can assume that   the northern most crossing of strands in $X$ is equal to the crossing of strands in $Y$  modulo a combination of diagrams $X'\rhd_\theta X$.  
 We have that $\Diag =YX = \sum_{X'\rhd X}  YX'=  YX + \sum_{\Diag '\rhd \Diag } \Diag '$ where the second term in the sum is of the required form.  We are now free to undo the double-crossing in $YX\cap [\RR \times (y-\varepsilon,y+\varepsilon)]$ using relation \ref{rel5}, \ref{rel6}, \ref{rel11}.  In any  case, the result is  either 1 or 2 diagrams  with this crossing undone (possibly at the expense of acquiring some dots) and so both diagrams are strictly more dominant than  $\Diag $ in the Bruhat ordering.  By \cref{move a dot down} we can remove any dots to obtain a linear combination of {diagrams} $\Diag ''$  which strictly dominate $\Diag $ and such that $\Diag ' \cap (\RR\times[0,y])$ is reduced.   
If each $\Diag '$ is reduced the result follows; if not, then there exists some $1>y'>y>0$ for which $\Diag ' \cap (\RR\times[0,y ))$ is not reduced and we can repeat the above argument.  Repeating as necessary, the result follows.    
\end{proof}

We hence immediately generalise the spanning set of \cite[Section 2.3]{MR2525917} to our algebras. Namely, we can write any element of $\algebra^\Bbbk_n(\theta)$ as a linear combination of   reduced diagrams with dots along the southernmost edge.  
\begin{cor}\label{arikibasis}
The algebra $\algebra^\Bbbk_n(\theta)$ is free  as a $\Bbbk$-module and spanned by
 $$\{
  \Diag     f_{\Diag }(y)
  \mid  
\la,\mu \in \con \ell n, \imath,\jmath \in (\ZZ/e\ZZ)^n,
\Diag \in {^\imath_\mu\!\mathcal{R}^\jmath_\la},
  f_\Diag (y) \in {\mathcal Y}^\jmath_\la
  \}$$ 
\end{cor}
 
We view the following theorem as the ``2-sided" version of \cref{arikibasis}.

 \begin{thm} \label{cellularitybreedscontempt}
 Let   $\Bbbk$  be a   integral domain.  
  The $\Bbbk$-algebra  $\algebra^\Bbbk_n(\theta )$ is spanned by    \[
  \{{ \Cell}_{\SSTS  \SSTT} \mid \SSTS \in \SStd_{\theta}(\lambda,\mu), \SSTT\in \SStd_{\theta}(\lambda,\nu), 
 \lambda  \in \mptn {\ell}n,  \mu, \nu  \in \con {\ell}n\}.
\]
 \end{thm}

\begin{proof}
In \cref{cellularitybreedscontempt0.5} we saw that any diagram $\Diag \in \algebra^\Bbbk_n(\theta )$ can be written as a linear combination of elements of the form $x_\lambda {\sf 1}_\lambda  y_\lambda^\ast $  for $x_\la, y_\la \in \algebra^\Bbbk_n(\theta){\sf 1}_\la$ and  $\la\in\mptn \ell n$.
It remains to show that the elements $x_\lambda$ and $ y_\lambda$ can be chosen so that 
$(i)$ $x_\la$ and $ y_\la$ are reduced diagrams  (neither   has a dot on any strand, or contains any ``double-crossing''),
$(ii)$ the span of these  elements is independent of the choices of reduced expression for $[x_\la]$ and $[ y_\la]$.    
We shall then conclude that the algebra is spanned by $\Cell_{\SSTS\SSTT}$ for tableaux $\SSTS,\SSTT$ of shape $\lambda\in \mptn\ell n$.  
Finally it will remain to show that the set of  $\Cell_{\SSTS\SSTT}$ for a pair of {\em semistandard} tableaux $\SSTS,\SSTT$ span the algebra.

The result for $\lambda=((n),\varnothing,\dots,\varnothing)$ is trivial.  
We assume that $(i)$  is proven for all partitions strictly more dominant than $\la$.  
 For the remainder of the proof, we refine our induction by  proceeding along the Bruhat order on $x_\la$ and $ y_\la$ and consider the span of elements  modulo 
\begin{equation}\label{workmod}
{\rm Span}_\Bbbk
 \{\Cell_{\SSTS\SSTT} \mid \Cell_\SSTS \rhd _\theta x_\la \text{  or  }\Cell_\SSTT\rhd _\theta  y_\la
  \} + \algebra_n^{\rhd \la}.  
\end{equation}

 If $x_\la$ or $ y_\la$ has a dot or a double crossing, then $x_\la {\sf1}_\la  y_\la^\ast$ is zero modulo 
  \cref{workmod} 
 by  \cref{move a dot down} and \cref{mastumoto2}.  Similarly, given   any two reduced diagrams 
 $x_\la$ and $x'_\lambda$ tracing out the same bijection $:[\la]\to [\mu]$, we have that
  $x_\la {\sf 1}_\la  y_\la - x_\la'{\sf 1}_\la  y_\la$ belongs to \cref{workmod} by \cref{mastumoto}.  
  Thus $(i)$ and $(ii)$ hold by induction.

  We note that any   bijection $:[\la]\to [\mu]$ is encoded as a  tableau of shape $\la$ and weight $\mu$ and so 
  $\{\Cell_{\SSTS\SSTT}\mid \SSTS, \SSTT \text{ are tableaux of shape }\la\}$ is   a spanning set by definition and our proof of $(i)$ and $(ii)$.  
 It only remains to show that the subset of  semistandard tableaux (within the wider class of tableaux) index a spanning set.  
We consider $\Cell_\SSTS$ (or $\Cell_\SSTT^\ast)$ such that $\SSTS$ (or $\SSTT$) violates the semistandardness condition.
In other words, one of the following holds: 
 $(i)$  $\SSTS(1,1,m)>s_m$ 
 $(ii)$  $\SSTS(r,c,m)> \SSTS(r-1,c,m)  +\g$ or 
 $(iii)$  $\SSTS(r,c,m)> \SSTS(r,c-1,m) -\g$.
In each case, we obtain a ``bad crossing".  We can choose to draw our diagram $\Cell_\SSTS$ so that this crossing appears at the bottom of the diagram using \cref{mastumoto2} and induction on the Bruhat ordering.  
These crossings are as follows,
\begin{itemize}
   \item[$(i)$]      the solid strand corresponding to $(1,1,m)$ passes to the  right   of the red $\theta_m$-strand,    \item[$(ii)$]     the ghost strand corresponding to $(r,c,m)$ passes to the  right   of the solid strand corresponding to $(r-1,c,m)$,
  \item[$(iii)$]  the solid strand corresponding to $(r,c,m)$ passes to the  right   of the ghost strand corresponding to $(r-1,c,m)$. 
 \end{itemize}
In each case the strand labelled by the  box $(r,c,m)$ is now free to move  right  wards using the process outlined in \cref{cellularitybreedscontempt0.5} and hence belongs to $\algebra^{\rhd \la}_n(\theta)$.  \end{proof}

\subsection{The Schur functor}\label{schurdefn}
We   define   the {\sf Schur} or {\sf KZ} functor  
 relating    the  quiver Hecke and Cherednik algebras.   
 We let $\omega
\in \mptn\ell n$ denote the   unique element which is  minimal in the $\theta$-dominance order. 
For  weakly decreasing   $s_1\geq s_2\geq \dots \geq s_\ell$ we have that  $\omega=(\varnothing,\varnothing,\dots ,\varnothing ,(1^n))$.   
 We let ${\sf E}^\theta_\omega$ denote the associated  {\sf Schur idempotent} 
 $ 
 {\sf E}^\theta_\omega= \sum_{\imath \in (\ZZ/e\ZZ)^n}{\sf 1}^{\imath}_{\omega} 
 $. 
 
\begin{defn}
Given $\SSTT \in \SStd_{\theta}(\la,\omega)$ 
(respectively $\stt \in \Std_{\theta}(\la)$) we define the reading word $R(\SSTT)$ (respectively $r(\stt)$) 
 to be the ordered sequence of boxes  $(r_k,c_k,m_k)$ for $1\leq k \leq n$ 
 under the ordering  $(r_k,c_k,m_k) < (r_{k'},c_{k'},m_{k'})$ if and only if 
 $\SSTT(r_k,c_k,m_k)< \SSTT(r_{k'},c_{k'},m_{k'})$ (respectively 
 $\stt(r_k,c_k,m_k)< \stt(r_{k'},c_{k'},m_{k'})$).  
 \end{defn}

 \begin{prop}\label{comb}
Let   $\theta \in (\ZZ/e\ZZ)^\ell$   and $\theta\in \weight$ be a charge.  For  $\lambda \in \mptn \ell n$, we have a  
 bijection
$$
\varphi : \Std_{\theta}(\lambda) \rightarrow  \SStd_{\theta}(\lambda,\omega) .
$$
 given by $\varphi(\stt)=\SSTT$ if and only if   $r(\stt)=R(\SSTT)$.  
 \end{prop}
 
 \begin{proof}
We order the boxes  $(r,1,\ell)$ for $1\leq r \leq n$ of the $\ell$-partition $\omega\in \mptn \ell n$ by the natural numbering on $\{1,\ldots, n\}$.  Clearly $\mathbf{I}^\theta_{(r,1,\ell)}<\mathbf{I}^\theta_{(r',1,\ell)}$ if and only if $1\leq r' < r\leq n$.  
Therefore the set of maps $\{\SSTT \mid  \SSTT: [\boxla]  \to \bf{I}^\theta_\omega\}$  is  in bijection with the set of
 tableaux of shape $\la$. 
  This map is simply given by identifying 
 the entry    $ \mathbf{I}^\theta_{(r,1,\ell)}\in {\mathbb R}$ in a box of $\SSTT$  with the entry  $r\in \NN$ in  a box of $\stt$.
  It remains to show that $\SSTT$ is semistandard if and only if $\stt$ is standard.  
Condition $(i)$ of \cref{semistandard:defn} is empty as  ${\bf I}_{(r,1,\ell)}<s_m $ for all $1\leq m \leq \ell$.  
Conditions  $(ii)$  and $(iii)$ of  \cref{semistandard:defn} simply  correspond to the 
   conditions that $\stt(r,c,m)>\stt(r-1,c,m)$ and  $\stt(r,c,m)>\stt(r,c-1,m)$ respectively 
  and the fact that ${\bf I}^{\theta}_{(r+1,1,\ell)}<{\bf I}^{\theta}_{(r,1,\ell)}-1$ for $1\leq r \leq n$.   
  \end{proof}
  
     Over a field, the theorem below follows  from \cite[Theorem 4.5]{MR3732238} and 
 \cite[Theorem 5.3]{Webster(b)}.  
Our proof proceeds by matching-up the presentations in \cref{defintino1} and \cref{defintino2}     and is valid over a   integral domain.  By matching up these presentations explicitly, we see   how Webster's diagrammatics generalises that of Khovanov--Lauda \cite{MR2525917}.  We  also generalise Webster's results to an arbitrary    integral domain.  
     
 \begin{thm}\label{isomer}   
Let $\Bbbk$ be a    integral domain.  
Let $\underline{s}\in \NN_{>1}\times(\ZZ/e\ZZ)^\ell$ and let    $\theta\in \weight$ be any integral lift.  We have  an isomorphism of graded $\Bbbk$-algebras  
 $$ \sigma   : \mathcal H_n^\Bbbk(\underline{s}) \to  {\sf E}^\theta_\omega \algebra^\Bbbk_n(\theta)  {\sf E}^\theta_\omega $$
 which is determined  as follows  
\begin{align*}
\sigma  ( e(\imath )) &= 1^{\imath}_\omega\\
\sigma  ( y_r e(\imath) ) &= 
\begin{minipage}{145mm} \scalefont{0.7}    \begin{tikzpicture}[xscale=-0.95,yscale=0.95] 
  \draw[thick] (-2,0) rectangle (12.6,2);
   \foreach \x in {-1.5,-0.5,2.5}  
     {\draw[wei2] (\x,0)--(\x,2); }
     \node [wei2,below] at (-1.5,-0.07)  {\tiny ${s}_0$};
          \node [wei2,below] at (2.5,-0.07)  {\tiny ${s}_{\ell\!-\!1}$};
        \draw[fill,white]   (0.2,1.9) rectangle (1.75,2.1);  
     \draw[fill,white] (0.2,0.1) rectangle (1.75,-0.1);  
     \draw[thick,dotted] (0.1,0.) --  (1.75,-0);       \draw[thick,dotted] (0.1,2) --  (1.75,2);  
          \node [wei2,below] at (-0.5,-0.07)  {\tiny ${s}_1$};
        \node [below] at (2.9,0)  {\tiny $i_1$};
                      \node [below] at (3.8,0)  {\tiny $i_2$};           \draw  (2.9,0)  to  (2.9,2) ;          \draw[densely dotted, black]  (2.3,0)  to  (2.3,2) ; 
                    \draw  (3.8,0)  to  (3.8,2) ;          \draw[densely dotted, black]  (3,0)  to  (3,2) ; 
                      \draw[densely dotted, black]  (3.9,0)  to  (3.9,2) ; 
                       \draw[densely dotted, black]  (3.9,0)  to  (3.9,2) ; 
                       \draw[fill,white]   (4.2,2.1) rectangle (-0.5+5.8,-0.1);  
     \draw[thick,dotted] (4.18,0.) --  (-0.5+5.8,-0);         \draw[thick,dotted] (-0.5+4.2,2) --  (-0.5+5.8,2);    
       \node [below] at (-0.5-0.9+7.1,0)  {\tiny $i_{r-2}$};
         \draw  (-0.5-0.9+7,0)  to  (-0.5-0.9+7,2) ; 
                  \draw[densely dotted, black]  (-0.4-0.9+7,0)  to  (-0.4-0.9+7,2) ; 
       \node [below] at (-0.5+7.2,0)  {\tiny $i_{r-1}$};
           \draw  (-0.5+7.3,0)  to  (-0.5+7.3,2) ;         
   \node [below] at (-0.5+8.1,0)  {\tiny $i_{r}$};  
    \draw[fill]   (-0.5+8.2,1)           circle (1.75pt); 
      \node [below] at (-0.5+9,0)  {\tiny $i_{r+1}$};
        \draw  (-0.5+8.2,0)  to  (-0.5+8.2,2) ;             \draw  (-0.5+9.1,2)  to  (-0.5+9.1,0) ;     
        \draw[densely dotted, black]    (-0.5-0.8+8.2,0)  to  (-0.5-0.8+8.2,2) ;             \draw[densely dotted, black]    (-0.5-0.8+9.13,2)  to  (-0.5-0.8+9.13,0) ;     
          \node [below] at (-0.5+10,0)  {\tiny $i_{r+2}$};
        \draw  (-0.5+10,0)  to  (-0.5+10,2) ;         \draw[densely dotted, black]  (-0.5+9.2,0)  to  (-0.5+9.2,2) ; 
         \draw[densely dotted, black]  (-0.5+10.1,0)  to  (-0.5+10.1,2) ; 
               \node [below] at (0.4+10,0)  {\tiny $i_{r+3}$};
        \draw  (0.4+10,0)  to  (0.4+10,2) ;         \draw[densely dotted, black]  (0.4+9.2,0)  to  (0.4+9.2,2) ; 
        \draw[densely dotted, black]  (0.4+10.1,0)  to  (0.4+10.1,2) ; 
                           \draw[fill,white]   (10.8,2.1) rectangle (11.6,-0.1);  
      \draw[thick,dotted](10.8,2) --  (11.6,2);         \draw[thick,dotted] (10.8,0) --  (11.6,0);    
                  \node [below] at (12,0)  {\tiny $i_{n}$};
        \draw  (12,0)  to  (12,2) ;    
   \end{tikzpicture}\end{minipage}\\
   \sigma  ( \psi_r e(\imath)) & = 
\begin{minipage}{145mm} \scalefont{0.7}    \begin{tikzpicture}[xscale=-0.95,yscale=0.95] 
  \draw[thick] (-2,0) rectangle (12.6,2);
   \foreach \x in {-1.5,-0.5,2.5}  
     {\draw[wei2] (\x,0)--(\x,2); }
     \node [wei2,below] at (-1.5,-0.07)  {\tiny ${s}_0$};
          \node [wei2,below] at (2.5,-0.07)  {\tiny ${s}_{\ell\!-\!1}$};
        \draw[fill,white]   (0.2,1.9) rectangle (1.75,2.1);  
     \draw[fill,white] (0.2,0.1) rectangle (1.75,-0.1);  
     \draw[thick,dotted] (0.1,0.) --  (1.75,-0);       \draw[thick,dotted] (0.1,2) --  (1.75,2);  
          \node [wei2,below] at (-0.5,-0.07)  {\tiny ${s}_1$};
           \node [below] at (2.9,0)  {\tiny $i_1$};
                      \node [below] at (3.8,0)  {\tiny $i_2$};
           \draw  (2.9,0)  to  (2.9,2) ;          \draw[densely dotted, black]  (2.3,0)  to  (2.3,2) ; 
                    \draw  (3.8,0)  to  (3.8,2) ;          \draw[densely dotted, black]  (3,0)  to  (3,2) ; 
                      \draw[densely dotted, black]  (3.9,0)  to  (3.9,2) ; 
                       \draw[densely dotted, black]  (3.9,0)  to  (3.9,2) ; 
                       \draw[fill,white]   (4.2,2.1) rectangle (-0.5+5.8,-0.1);  
     \draw[thick,dotted] (4.18,0.) --  (-0.5+5.8,-0);         \draw[thick,dotted] (-0.5+4.2,2) --  (-0.5+5.8,2);    
       \node [below] at (-0.5-0.9+7.1,0)  {\tiny $i_{r-2}$};
         \draw  (-0.5-0.9+7,0)  to  (-0.5-0.9+7,2) ; 
                  \draw[densely dotted, black]  (-0.4-0.9+7,0)  to  (-0.4-0.9+7,2) ; 
       \node [below] at (-0.5+7.2,0)  {\tiny $i_{r-1}$};
           \draw  (-0.5+7.3,0)  to  (-0.5+7.3,2) ;         
   \node [below] at (-0.5+8.1,0)  {\tiny $i_{r}$};
      \node [below] at (-0.5+9,0)  {\tiny $i_{r+1}$};
        \draw  (-0.5+8.2,0)  to [out=90,in=-90] (-0.5+9.1,2) ;             \draw  (-0.5+8.2,2)  to [out=-90,in=90] (-0.5+9.1,0) ;     
        \draw[densely dotted, black]  (-0.5+8.2-0.8,0)  to [out=90,in=-90] (-0.5+9.1-0.8,2) ;             \draw[densely dotted, black]  (-0.5+8.2-0.8,2)  to [out=-90,in=90] (-0.5+9.1-0.8,0) ;     
          \node [below] at (-0.5+10,0)  {\tiny $i_{r+2}$};
        \draw  (-0.5+10,0)  to  (-0.5+10,2) ;         \draw[densely dotted, black]  (-0.5+9.2,0)  to  (-0.5+9.2,2) ; 
         \draw[densely dotted, black]  (-0.5+10.1,0)  to  (-0.5+10.1,2) ; 
               \node [below] at (0.4+10,0)  {\tiny $i_{r+3}$};
        \draw  (0.4+10,0)  to  (0.4+10,2) ;         \draw[densely dotted, black]  (0.4+9.2,0)  to  (0.4+9.2,2) ; 
        \draw[densely dotted, black]  (0.4+10.1,0)  to  (0.4+10.1,2) ; 
                           \draw[fill,white]   (10.8,2.1) rectangle (11.6,-0.1);  
      \draw[thick,dotted](10.8,2) --  (11.6,2);         \draw[thick,dotted] (10.8,0) --  (11.6,0);    
                  \node [below] at (12,0)  {\tiny $i_{n}$};
        \draw  (12,0)  to  (12,2) ;    
   \end{tikzpicture}\end{minipage}
\end{align*}
 {Thus we obtain many distinct presentations for the same Hecke algebra, $\mathcal H_n^\Bbbk(\underline{s})$, 
 one for each possible lift of $\underline{s}\in(\ZZ/e\ZZ)^\ell$ to the integers.  While these algebras are all isomorphic, we have already seen that each of these distinct lifts   has a different combinatorial flavour.  We shall see what these different lifts tell us about the structure of $\mathcal  H_n^\Bbbk(\underline{s})$ in \cref{mainresults,manydiffgrad,isaididexplain,proj}. }
 \end{thm}

 \begin{rmk}
 The reader might think notice that the diagrams in \cref{isomer} are only dependent on
  $\underline{s}\in \NN_{>1}\times(\ZZ/e\ZZ)^\ell$ and not on the integral lift  $\theta\in \weight$.  
  We remind the reader that this is not the case because the $x$-coordinates of the red strands are determined by     $\theta\in \weight$ (even though their residue-decorations, $s_{\ell-1}, \dots s_0$, are independent of   the integral lift).  
 \end{rmk}

Before proving this result, we provide a new set of generators of  
${\sf E}^\theta_\omega \algebra^\Bbbk_n (\theta)  {\sf E}^\theta_\omega$ 
which is highly compatible with the desired isomorphism.  

  \begin{prop}\label{andrew}
The algebra ${\sf E}^\theta_\omega \algebra^\Bbbk_n (\theta)  {\sf E}^\theta_\omega$ is
 generated by 
   \begin{equation}\label{gens1}
   \langle {\sf 1}_\omega^{\imath}, \sigma  (y_s(e(\imath))), 
 \sigma  (\psi_r(e(\imath))) \mid \text{  $1\leq r <n$, $1\leq s \leq n$, and $\imath \in (\ZZ/e\ZZ)^\ell$ }\rangle
\end{equation}   subject to  \ref{rel1} to \ref{rel15}.  
In particular,  
 the  $\Bbbk$-linear map $\sigma : \mathcal  H_n^\Bbbk(\underline{s}) \to  {\sf E}^\theta_\omega \algebra^\Bbbk_n(\theta)  {\sf E}^\theta_\omega $ is surjective.  
  
  \end{prop}

\begin{proof}
We begin by rewriting every diagram in ${\sf E}^\theta_\omega \algebra^\Bbbk_n (\theta)  {\sf E}^\theta_\omega$ 
  so that the red strands are ``replaced" with 
something more akin to the usual cyclotomic relation for the classical KLR-algebra.  

\medskip
\noindent{\bf Claim:}
Let $\Diag \in {\sf 1}^{\imath}_\omega\algebra^\Bbbk_n(\theta){\sf 1}_\omega^{\jmath} $ be an arbitrary diagram.  
 Let $\overline{\Diag }$ be the diagram obtained from $\Diag $ by the following procedure:  
 \begin{itemize}[leftmargin=*]
 \item [(1)]
 pull  each red strands  sufficiently 
 far to the right that   it longer intersects any solid or ghost strands;
  \item [(2)] for each crossing   involving a solid $i$-strand (from southwest to northeast) and a red-strand (from southeast to northwest) place a dot at the position of the crossing in $\Diag $.  
 \end{itemize}
We claim   that $\Diag =\overline{\Diag }$.  
 
\begin{figure}[ht!] \[\begin{tikzpicture}[baseline, thick,yscale=0.45,xscale=0.75]
  
          \draw[wei2 ] (0.6,-2)--(0.6,2)   node[below,at start]{$0 $ } ;

     \draw  (0,-2) to[out=90,in=-60] 
       node[below,at start]{$0 $ }  (1,0)   to[out=120,in=-90] 
        (-0.8,2);

   \draw  (0,2) to[out=-90,in=60] 
 (1,0)   to[out=-120,in=90] 
        node[below,at end]{$0 $ }         (-0.8,-2);

\end{tikzpicture}
\; =
 \begin{tikzpicture}[baseline, thick,yscale=0.45,xscale=0.75]
  
          \draw[wei2] (0.7,-2)--(0.7,2)   node[below,at start]{$0 $ } ;

     \draw  (0,-2) to[out=60,in=-30] 
       node[below,at start]{$0 $ } 
 (1,-0.4)        to[out=150,in=-30]  
       (0.25,0)  to[out=150,in=-90] 
        (-0.8,2);

   \draw  (0,2) to[out=-60,in=30] 
 (1,0.4)        to[out=-150,in= 30]  
       (0.25,0)  to[out=-150,in=90] 
           node[below,at end]{$0 $ }     (-0.8,-2);

\end{tikzpicture}
-
 \begin{tikzpicture}[baseline, thick,yscale=0.45,xscale=0.75]
  
          \draw[wei2] (0.7,-2)--(0.7,2)   node[below,at start]{$0 $ } ;

  \draw  (-0.8,-2) to[out=90,in=-90] 
       node[below,at start]{$0 $ } 
 (1.3-1.1,0)       to[out=90,in=-90] 
        (-0.8,2);

 \draw  (0,-2) to[out=90,in=-90] 
       node[below,at start]{$0 $ } 
 (1.3,0)       to[out=90,in=-90] 
        (0,2);
 
 \end{tikzpicture} 
\; =
 \begin{tikzpicture}[baseline, thick,yscale=0.45,xscale=0.75]
  
          \draw[wei2] (0.7,-2)--(0.7,2)   node[below,at start]{$0 $ } ;

     \draw  (0,-2) to[out=60,in=-60] 
       node[below,at start]{$0 $ }      (-0.8,2);

     \draw  (0,2) to[out=-60,in=60]  
 node[below,at end]{$0 $ }            (-0.8,-2) ; 
            
            \fill(0.34,1) ellipse (4pt and 7pt);
            \fill(0.34,-1) ellipse (4pt and 7pt);

\end{tikzpicture}
-
 \begin{tikzpicture}[baseline, thick,yscale=0.45,xscale=0.75]
  
          \draw[wei2] (0.7,-2)--(0.7,2)   node[below,at start]{$0 $ } ;

  \draw  (-0.8,-2) to[out=90,in=-90] 
       node[below,at start]{$0 $ } 
 (-0.8,0)       to[out=90,in=-90] 
        (-0.8,2);

 \draw  (0,-2) to[out=90,in=-90] 
       node[below,at start]{$0 $ } 
 (0,0)       to[out=90,in=-90] 
        (0,2);
        
        \fill(0,0) ellipse (4pt and 7pt);
 
 \end{tikzpicture}\;=
  \begin{tikzpicture}[baseline, thick,yscale=0.45,xscale=0.75]
  
          \draw[wei2] (0.7,-2)--(0.7,2)   node[below,at start]{$0 $ } ;

     \draw  (0,-2) to[out=70,in=-70] 
   node[below,at start]{$0 $ }          (-0.8,2);

     \draw  (0,2) to[out=-70,in=70]  
     node[below,at end]{$0 $ }       coordinate [pos=.85]   (hi)  (-0.8,-2) ; 
            
            \fill(-0.38,-1) ellipse (4pt and 7pt);
            \fill(0.2,-1) ellipse (4pt and 7pt);

\end{tikzpicture}
 \]
 \caption{Passing a crossing through a red strand using relation \ref{rel12}, then relation  \ref{rel11}, then \ref{rel6}).  All ghost   strands commute  with red strands, and so we do not picture these here.  }\label{forandrew}
 \end{figure}
 
 We now prove the claim.   
 We consider the effect of pulling the $m$th red strand to the right.  
 We can restrict our attention to the effect on the sub-diagram of $\Diag _{\underline{w}}$ (with underlying word, $[\Diag ]$, equal to  $\underline{w}$) consisting only of the  solid strands of residue $s_m\in \ZZ/e\ZZ$  which cross the $m$th red strand (we can ignore  their ghosts, by \ref{rel13}).  
 We can further  restrict our attention to only consider the  crossings  of these strands within the    region $  (\sigma_m-m/\ell, \infty)\times [0,1]$ as  it is only these strands within this region   to which  we need apply    non-trivial relations (while pulling the $m$th red strand to the right).   
    Thus we write $\Diag _{\underline{w}}= 
 \Diag _{\underline{u}} \Diag _{\underline{w'}} \Diag _{\underline{v}} $ where $\Diag _{\underline{w'}} $ consists of the  
 $s_m$-strands intersected with the region $  (\sigma_m-m/\ell, \infty)\times [y_1,y_2]$ 
 (where $y_1$ and $y_2$ are chosen appropriately) and 
 $\Diag _{\underline{u}}$ and $\Diag _{\underline{v}} $ consist of the $s_m$-strands 
 with the regions $  (-\infty, \sigma_m-m/\ell,  ]\times [y_2,1]$  and 
  $  (-\infty, \sigma_m-m/\ell,  ]\times  [0,y_1] $ 
respectively.   
We set $k$ to be the number of solid $i$-strands in $\Diag _{\underline{w}'}$. We proceed by induction on  $\ell(\underline{w'})$.  
 
 We have that $\ell(\underline{w'})\geq 2k$ because each solid strand intersects the red strand at least twice.  
 If $\ell(\underline{w'})=2k$ then the claim follows by $k$ applications of the leftmost relation in \ref{rel11}.  
We can now suppose $\underline{w'}=\underline{x}  s_{1,2} \underline{y}$
 (if $s_{1,2}$ does not appear in $\underline{w'}$ then we can  apply the leftmost relation of \ref{rel11}) for $\underline{x}$ and $\underline{y}$ subwords such that  $\ell(\underline{x})+\ell(\underline{y})<\ell(\underline{w})$.  
We pull the red-strand through the crossing corresponding to this  $s_{1,2}$  at the expense of an error term (with coefficient $-1$) in which we undo this crossing.  
 By our inductive assumption we can 
  move the red strand through 
 $\Diag _{\underline{x}\underline{y}}$ and obtain  
$\overline{\Diag }_{\underline{x}}   \Diag _{s_{1,2}}  \overline{\Diag }_{\underline{y}}
- \overline{\Diag }_{{\underline{x}}{\underline{y}}}$ but with the red strand all the way to the right.  
By construction, we have a dot   at the bottom of $\Diag _{\underline{x}}'$ on the leftmost strand; 
using \ref{rel6}   we pull this dot down to the left through $ \Diag _{s_{1,2}} $  to obtain the required 
diagram
 $\overline{\Diag }_{\underline{w'}}$ at the expense of an error term equal to $ \overline{\Diag }_{{\underline{x}}{\underline{y}}}$ (with coefficient $+1$).  The error terms cancel and the claim holds.  An example  is depicted in \cref{forandrew} (for $
 {\underline{w}}={\underline{w'}}=s_{1,2}$).

 \begin{figure}[ht!]
   \[
\scalefont{0.8}\begin{tikzpicture}[very thick,xscale=1.6,yscale=0.7,baseline]
\draw[densely dotted, ]   (-2.6,-1) -- +(-.8,2)node[below,at start]{$i$};
\draw[densely dotted, ]   (-3.4,-1) -- +(.8,2)node[below,at start]{$i$}; 
 \draw (-3,-1) .. controls (-2.5,0) ..  +(0,2)
node[below,at start]{$j$};
\end{tikzpicture} 
 \qquad \quad
\begin{tikzpicture}[very thick,xscale=1.6,yscale=0.7,baseline]
\draw[densely dotted, ]  (-1.5+.4,-1) -- +(+-.8,2)node[below,at start]{$i $};
\draw[densely dotted, ]   (-1.5+-.4,-1) -- +(+.8,2)node[below,at start]{$i $};
 \draw  (-1.5+0,-1) .. controls (-1.5-.5,0) ..  +(0,2)
node[below,at start]{$j $};
 
\end{tikzpicture}
\qquad \quad  \begin{tikzpicture}[very thick,xscale=-0.5,yscale=0.5,baseline,yscale=1.3]
\draw (-2.8,-1) .. controls (-1.2,0) ..  +(0,2)
node[below,at start]{$i$};
\draw[densely dotted, ] (-1.2,-1) .. controls (-2.8,0) ..  +(0,2)
node[below,at start]{$j$};
\end{tikzpicture}
\qquad  \begin{tikzpicture}[very thick,xscale=0.5,yscale=0.5,baseline,yscale=1.3]
\draw (-2.8,-1) .. controls (-1.2,0) ..  +(0,2)
node[below,at start]{$i$};
\draw[densely dotted, ] (-1.2,-1) .. controls (-2.8,0) ..  +(0,2)
node[below,at start]{$j$};
\end{tikzpicture}\]
\caption{Crossings within a  diagram  $\Diag  \in {\sf E}^\theta_\omega\algebra^\Bbbk_n (\theta)  {\sf E}^\theta_\omega$. }
\label{critical}
\end{figure}

 With the claim in place, we may now assume   $\Diag  \in {\sf E}^\theta_\omega\algebra^\Bbbk_n (\theta)  {\sf E}^\theta_\omega$
   has no crossings involving red strands.  Any crossing in such a diagram is of one of the forms depicted in \cref{critical}.

We can undo any double-crossing using relation \ref{rel4} and \ref{rel5} to obtain a sum of diagrams of the required form.  
We must now consider any triple-crossings as in \cref{critical} and show that this belongs to a wider product  of the form 
\begin{equation}\label{reqforms}    \sigma  (\psi_{r+1}(e(\imath)))   \sigma  (\psi_r(e(\imath)))   \sigma  (\psi_{r+1}(e(\imath))) 
\qquad \text{or} \qquad 
 \sigma  (\psi_r(e(\imath)))  \sigma  (\psi_{r+1}(e(\imath)))  \sigma  (\psi_r(e(\imath)))    \end{equation} for $1\leq r <n$.  
 There are precisely 4 different wider products (up to isotopy) 
 to which such a diagram can belong, these are depicted in \cref{critical2}.  The first and fourth of these diagrams are already of the respective forms  in \cref{reqforms}.

 \begin{figure}[ht!]
   \[\scalefont{0.8}\begin{tikzpicture}[scale=0.6]
\draw[densely dotted, rounded corners]   (3.2-1.4,-0.3)--  (3.2-1.4,0) to[out=90,in=-90]  (0-1.4,4) -- (0-1.4,4.3)   ;
\draw[densely dotted, rounded corners ]      (3.2-1.4,4.3) --   (3.2-1.4,4) to[out=-90,in=90 ] (0-1.4,0)--(0-1.4,-0.3) ;
\draw[very thick, rounded corners ]   (3.2,-0.3) -- (3.2,0)  to[out=90,in=-90]  (0,4)--  (0,4.3) ;
\draw[very thick, rounded corners ]   (3.2,4.3) -- (3.2,4)    to[out=-90,in=90 ]   (0,0)--  (0,-0.3)  ;  

\draw[very thick, rounded corners ]  (1.6 ,-0.3) node[below] {$j$}; \draw(0 ,0-0.3) node[below] {$i$}; \draw(3.2 ,0-0.3) node[below] {$i$}; 
\draw[very thick, rounded corners ]   (1.6 ,0-0.3) to   (1.6 ,0) to [out=60,in=-90]	( 3.4,2)		to [out=90,in=-60]  (1.6 ,4)--(1.6 ,4.3) ; 
\draw[densely dotted, rounded corners] (1.6 -1.4,-0.3)-- (1.6 -1.4,0) to [out=60,in=-90]	( 3.4 -1.4,2)		to [out=90,in=-60]  (1.6 -1.4 ,4)--(1.6 -1.4 ,4.3);
 
\end{tikzpicture}
\qquad\qquad
\begin{tikzpicture}[scale=0.6]
\draw[densely dotted, rounded corners]   (3.2-1.4,-0.3)--  (3.2-1.4,0) to[out=110,in=-90]  (0-1.4,4) -- (0-1.4,4.3)   ;
\draw[densely dotted, rounded corners ]      (3.2-1.4,4.3) --   (3.2-1.4,4) to[out=-110,in=90 ] (0-1.4,0)--(0-1.4,-0.3) ;
\draw[very thick, rounded corners ]   (3.2,-0.3) -- (3.2,0)  to[out=110,in=-90]  (0,4)--  (0,4.3) ;
\draw[very thick, rounded corners ]   (3.2,4.3) -- (3.2,4)    to[out=-110,in=90 ]   (0,0)--  (0,-0.3)  ;  

\draw[very thick, rounded corners ]  (1.6 ,-0.3) node[below] {$j$}; \draw(0 ,0-0.3) node[below] {$i$}; \draw(3.2 ,0-0.3) node[below] {$i$};

\draw[very thick  ]   (1.6 ,0-0.3) to   (1.6 ,0) to [out=90,in=-90]	( 2.3,2)		to [out=90,in=-90]    (1.6 ,4)-- (1.6 ,4.3) ; 
\draw[densely dotted ] (1.6 -1.4,-0.3)-- (1.6 -1.4,0) to [out=90,in=-90]	( 2.3  -1.4,2)		to [out=90,in=-90]  (1.6 -1.4 ,4)--    (1.6 -1.4 ,4.3);
 
\end{tikzpicture} 
\qquad\qquad
\begin{tikzpicture}[scale=0.6]
\draw[densely dotted, rounded corners]   (3.2-1.4,-0.3)--  (3.2-1.4,0) to[out=90,in=-90]  (0-1.4,4) -- (0-1.4,4.3)   ;
\draw[densely dotted, rounded corners ]      (3.2-1.4,4.3) --   (3.2-1.4,4) to[out=-90,in=90 ] (0-1.4,0)--(0-1.4,-0.3) ;
\draw[very thick, rounded corners ]   (3.2,-0.3) -- (3.2,0)  to[out=90,in=-90]  (0,4)--  (0,4.3) ;
\draw[very thick, rounded corners ]   (3.2,4.3) -- (3.2,4)    to[out=-90,in=90 ]   (0,0)--  (0,-0.3)  ;  

\draw[very thick, rounded corners ]  (1.6 ,-0.3) node[below] {$j$}; \draw(0 ,0-0.3) node[below] {$i$}; \draw(3.2 ,0-0.3) node[below] {$i$};

\draw[very thick  ]   (1.6 ,0-0.3) to   (1.6 ,0) to [out=90,in=-90]	( 0.55,2)		to [out=90,in=-90]    (1.6 ,4)-- (1.6 ,4.3) ; 
\draw[densely dotted ] (1.6 -1.4,-0.3)-- (1.6 -1.4,0) to [out=90,in=-90]	( 0.55  -1.4,2)		to [out=90,in=-90]  (1.6 -1.4 ,4)--    (1.6 -1.4 ,4.3);
 
\end{tikzpicture} 
\qquad\qquad
\begin{tikzpicture}[scale=0.6]
\draw[densely dotted, rounded corners]   (3.2-1.4,-0.3)--  (3.2-1.4,0) to[out=90,in=-60]  (0-1.4,4) -- (0-1.4,4.3)   ;
\draw[densely dotted, rounded corners ]      (3.2-1.4,4.3) --   (3.2-1.4,4) to[out=-90,in=60 ] (0-1.4,0)--(0-1.4,-0.3) ;
\draw[very thick, rounded corners ]   (3.2,-0.3) -- (3.2,0)  to[out=90,in=-60]  (0,4)--  (0,4.3) ;
\draw[very thick, rounded corners ]   (3.2,4.3) -- (3.2,4)    to[out=-90,in=60 ]   (0,0)--  (0,-0.3)  ;  

\draw[very thick, rounded corners ]  (1.6 ,-0.3) node[below] {$j$}; \draw(0 ,0-0.3) node[below] {$i$}; \draw(3.2 ,0-0.3) node[below] {$i$};

\draw[very thick  ]   (1.6 ,0-0.3) to
 [out=90,in=-90]	( .3 ,2)		to [out=90,in=-90]   (1.6 ,4.3) ; 
\draw[densely dotted ] (1.6 -1.4,-0.3) 
to [out=90,in=-90]	( .3 -1.4,2)		to [out=90,in=-90]  (1.6 -1.4 ,4.3);
 
\end{tikzpicture}   \]
\caption{The wider diagrams    $\Diag \in {\sf E}^\theta_\omega\algebra^\Bbbk_n (\theta)  {\sf E}^\theta_\omega$ 
 containing a triple-crossing of the form depicted in \cref{critical}. The first and final of which are  
already of the required form.   }
\label{critical2}
\end{figure}
For the remaining two diagrams in \cref{critical2}, we can pull the ghost $j$-strand (respectively, solid $j$-strand)  to the right 
(to the left) of the  solid (respectively ghost) $i$-crossing using \ref{rel6} or \ref{rel5}.  In either case, the resulting diagram is now of the required form. 
 \end{proof}

\begin{proof} [Proof of Theorem 6.14]
 We have already seen that ${\sf E}^\theta_\omega \algebra^\Bbbk_n (\theta)  {\sf E}^\theta_\omega$ is
 generated by  
  the elements $ {\sf 1}_\omega^{\imath},$   $\sigma  (y_s(e(\imath))), $ 
  $
 \sigma  (\psi_r(e(\imath))) $ 
 for $1\leq s \leq n$ and $1\leq r <n$ and $\imath\in (\ZZ/e\ZZ)^n$.  
 We define $\sigma^{-1}$ to be the obvious inverse map.  We will now verify that  $\sigma$ respects \ref{rel1.1} to \ref{rel1.12} and that 
 $\sigma^{-1}$ respects  relations \ref{rel1} to \ref{rel15}.

The images of relations  \ref{rel1.1}  to \ref{rel1.7}  under $\sigma$ 
  follow from the diagrammatic definition of the multiplication and \ref{rel1}.   
  Conversely, the image under $\sigma^{-1}$ of the implicit diagrammatic   relations 
   (i.e. that strands carry residues, products are zero for non-matching  residue sequences,  and \ref{rel1})
   immediately follow from  \ref{rel1.1}  to \ref{rel1.7}.  
The images of relations  \ref{rel1.8}  to \ref{rel1.9}  under $\sigma$   
  follow from \ref{rel2} and   \ref{rel3}.  
  Conversely the image  of relations  \ref{rel2}  and \ref{rel3}  under $\sigma^{-1}$   
  follow from \ref{rel1.8} and   \ref{rel1.9}.

We shall now show that the image of \ref{rel1.10} under $\sigma$ holds in ${\sf E}^\theta_\omega \algebra^\Bbbk_n (\theta)  {\sf E}^\theta_\omega$ and 
that the images of  \ref{rel4}, \ref{rel5} and  \ref{rel6} under $\sigma^{-1}$ hold in $\mathcal H_n^\Bbbk(\underline{s})$.  
Relation \ref{rel1.10} has four parts; the images of the   
   $i_r\neq i_{r+1}\pm 1$ cases follow from   relation  \ref{rel4}   and  \ref{rel5}.    
   If $ {i}_r -1 = i_{r+1}$, then 
 we first apply relation  \ref{rel7}  to the diagram   $\sigma   (\psi_r e(\imath) )\sigma   (  \psi_r  e(s_{r,r+1}\imath))$ in order to undo the double-crossing  of the   ghost 
$(i-1)$-strand with the solid $i$-strand; now if $e\neq 2$, then this implies that $i_r  \neq i_{r+1}-1$ and so  the double-crossing  of the 
ghost $i$-strand with the solid $(i-1)$-strand can be undone without cost by relation  \ref{rel5}.  We hence obtain that   $\sigma   (\psi_r e(\imath) )\sigma   (  \psi_r  e(s_{r,r+1}\imath)) = 
\sigma   (y_{r+1} e(\imath))- \sigma   (y_{r} e(\imath))$, as required.   
  We now assume that  $ {i}_r +1 = i_{r+1}$ with $e\neq 2$.     Here we have that the 
 double-crossing  of the   ghost 
$i $-strand with the solid $( i+1)$-strand can be undone without cost by relation \ref{rel5}; now   the double-crossing  of the 
ghost $(i-1)$-strand with the solid $i$-strand can be undone   by relation  \ref{rel6} to
 obtain    $\sigma   (\psi_r e(\imath) )\sigma   (  \psi_r  e(s_{r,r+1}\imath)) = 
\sigma   (y_{r} e(\imath))- \sigma   (y_{r+1} e(\imath))$, as required.  
 Finally the $e=2$ case can be obtained in the same fashion as above, except noting that $i_r=i_{r+1}+1=i_{r+1}-1$ and so we need  apply relation  \ref{rel7} twice and hence obtain four terms. 
 
  Conversely, 
 let $\Diag  \in {\sf E}^\theta_\omega \algebra^\Bbbk_n(\theta )  {\sf E}^\theta_\omega$ be any diagram written as a product of the generators in \cref{gens1}.   
Any double crossing  in $\Diag $ of the form  depicted in \ref{rel4}, \ref{rel5}, \ref{rel7} must occur within a region of $\Diag $ of the form 
\begin{equation}\label{pairs12}
 \sigma (\psi_r) 
\sigma (\psi_r) 
{\sf 1}_\omega^{(\dots, i_{r-1},i_r, \dots )  } .  
\end{equation}
 In particular,  
none of the double-crossings in  \ref{rel4}, \ref{rel5}, \ref{rel7} ever appear by themselves; they always appear with a complementary pair of double-crossings strands (depending on $i_{r-1}=i_r-1, i_r, i_r+1$ or otherwise).
  Thus we do not need to show that  the images under $\sigma^{-1}$ of \ref{rel4}, \ref{rel5}, \ref{rel7} themselves hold, but rather we need only check that all possible pairs of these relations (which can appear in \cref{pairs12}) hold.  
These pairs correspond precisely to the subcases of \ref{rel1.10} and  can be  argued identically to the above (but backwards).

Now we consider relation \ref{rel1.11} at the same time as    \ref{rel9} and \ref{rel8}.    
We first check that $\sigma$ respects \ref{rel1.11} for each of the four cases.  
In the first case  of \ref{rel1.11}, we can move the ghost $(i-1)$-strand in  $\sigma  ( \psi_{r+1}\psi_{r} \psi_{r+1}e(\dots, i ,i-1,i \dots ))$    through the solid $i$-crossing
 using 
   \ref{rel9}    at the expense of an error term (with coefficient $-1$) in which we undo the solid $i$-crossing; 
  we can then simplify the former diagram  using   \ref{rel8}  and the error diagram using \ref{rel5} in order  to obtain $\sigma (  (\psi_{r} \psi_{r+1}\psi_{r}-1) e(\dots, i ,i+1,i \dots ))$.   \Cref{critical2} depicts the three steps in this process, with only the   step from the first to the second diagram producing an error term (i.e. we can get from the second to the third to the fourth  diagram in \cref{critical2} using only \ref{rel8}).  The second case is similar.  
 The third case is an amalgamation of cases 1 and 2, 
but the error terms must be simplified with \ref{rel7} due to the additional residue adjacencies.  
  The fourth case follows directly from relation  \ref{rel8}.  
  Thus $\sigma$ preserves relation \ref{rel1.11}.  

   Conversely, by  \cref{andrew} 
 any triple crossing of the form in   \ref{rel9} or \ref{rel8} 
must occur within a region of the diagram of the form 
\begin{equation}\label{residuu}
\sigma (\psi_r) 
\sigma (\psi_{r+1}) \sigma (\psi_r) 
{\sf 1}_\omega^{(\dots, i_{r-1},i_r,i_{r+1} \dots )  }
\qquad \text{or} \qquad 
\sigma (\psi_{r+1})  \sigma (\psi_r) \sigma (\psi_{r+1})
{\sf 1}_\omega^{(\dots, i_{r-1},i_r,i_{r+1} \dots )  }
\end{equation} for some $i_{r-1},i_r, i_{r+1}\in \ZZ/e\ZZ$.  
In particular, 
none of the triple-crossings  in  \ref{rel9} or \ref{rel8}   ever appear by themselves; they always appear with 
 a complementary triple-crossing strands (breaking down into cases according to whether  $i_{r-1}=i_r-1, i_r, i_r+1$ or otherwise).
  Thus we do not need to show that  the images under $\sigma^{-1}$ of  \ref{rel9} or \ref{rel8} 
   themselves hold, but rather we need only check that the possible combinations of these relations (which can appear in \cref{residuu})  hold.  
These pairs correspond precisely to the subcases of \ref{rel1.11} and  can be  argued identically to the  above (but backwards).

Finally it remains to consider  relation \ref{rel1.12} for $\mathcal H^\Bbbk_n(\underline{s})$ at the same time as the  red strand relations in 
${\sf E}^\theta_\omega \algebra^\Bbbk_n (\theta)  {\sf E}^\theta_\omega$.  
We have that 
$$\sigma( y_1^{\sharp\{s_m \mid s_m= i_1 	\}} e(\imath))  =0$$
by $\ell$ applications of \ref{rel11}.  We now consider the image under $\sigma^{-1}$ of the relations involving red strands.  
We have shown in the proof of \cref{andrew} that  any diagram 
   $\Diag \in {\sf E}^\theta_\omega \algebra^\Bbbk_n (\theta)  {\sf E}^\theta_\omega$ 
 is equal to some (decorated) diagram $\overline{\Diag }$ with no crossings involving red strands.  
Thus 
 we need only to verify that if $\Diag $ is unsteady, then 
$\sigma^{-1}(\overline{\Diag })=0$.  If $\Diag $ is unsteady, then we can move the  strands  
back towards the left, with the rightmost solid strand (of residue $i\in \ZZ/e\ZZ$, say)
 picking up a total of ${\sharp\{s_m \mid s_m= i	\}}$ dots; the resulting diagram factors through an idempotent of the form  
 $  y_{(1,1,\ell)}^{\sharp\{s_m \mid s_m= i_1 	\}} {\sf 1}^ \imath_\omega  =\sigma( y_1^{\sharp\{s_m \mid s_m= i_1 	\}} e(\imath)) $ and thus 
  $\sigma^{-1} (\Diag ) =0$ by  \ref{rel1.12}, as required.  
  \end{proof}

\subsection{Cellularity and quasi-heredity of quiver Cherednik algebras}

We shall now show that the spanning set of \cref{cellularitybreedscontempt} is in fact a cellular basis of the quiver Cherednik algebra.  
We first require a new ordering on the boxes in an $\ell$-multipartition.

\begin{defn}\label{defsucc}  
Given $\SSTS\in \SStd_{\theta}(\la,\mu)$  and   two distinct boxes $(r,c,m), (r',c',m') \in \boxla $ we write 
$(r,c,m)\succ (r',c',m')$ if one of the following holds 
 \begin{itemize} 
 \item[$(i)$]  ${\bf I}^{\theta}_{(r,c,m)}> {\bf I}^{\theta}_{(r',c',m')}  \pm  \g$ and $\SSTS(r,c,m) < \SSTS(r',c',m')  \pm  \g$ or  
 \item[$(ii)$]  $  {\bf I}^{\theta}_{(r,c,m)} > {\bf I}^{\theta}_{(r',c',m')}  $ and $\SSTS(r,c,m) < \SSTS(r',c',m')  $ or 
 \item[$(iii)$] $(r,c,m)$ and $(r',c',m')$ appear in the same row (respectively column) of the same component of $\la$  and $c=c'+1$ (respectively $r=r'+1$).
 \end{itemize}
We write $(r,c,m)\succeq (r',c',m')$ if either  $(r,c,m)= (r',c',m')$ or  $(r,c,m)\succeq (r',c',m')$.  
\end{defn}
 \begin{defn} Given $\SSTS\in \SStd_{\theta}(\la,\mu)$ and $\stt\in \Std_{\theta}(\la)$, 
  we say that that   $\stt$ factors through $\SSTS$ if  
  $(r,c,m)\succ (r',c',m')$ implies $\stt(r,c,m) > \stt(r',c',m')$ for all 
  pairs of distinct   $(r,c,m), (r',c',m') \in \boxla $. 

\end{defn}
%

\begin{prop}\label{complement}
Given $\SSTS\in \SStd_{\theta}(\la,\mu)$   there exists 
 an   ${\overline{\sts}} \in \Std_{\theta}(\la)$ such that 
 ${\overline{\sts}} $ factors through $\SSTS$.  
For such a pair, $(\SSTS,{\overline{\sts}})$, there exists a  tableau $ \SSTS^c$ of shape $\mu$ and weight $\omega$  
 such that  $ \Cell_{\SSTS^c}   \Cell_\SSTS  =  \Cell_{\varphi({\overline{\sts}} )}$. 
\end{prop}

\begin{proof}
 Let  $\SSTS\in\SStd_{\theta}(\la,\mu)$. 
Let $(r_1,c_1,m_1), (r_2,c_2,m_2) \in \la$ be any pair of distinct boxes such that  $(r_1,c_1,m_1) \succ (r_1,c_2,m_2)$.  
By  \cref{semistandard:defn}, if $(r_1,c_1,m_1), (r_2,c_2,m_2)$ are as in $(iii)$ then 
$$|\SSTS(r_1,c_1,m_1)-\mathbf{I}^\theta_{(r_1,c_1,m_1)}|\geq |\SSTS(r_2,c_2,m_2)-\mathbf{I}^\theta_{(r_2,c_2,m_2)}|$$
 and if $(r_1,c_1,m_1), (r_1,c_2,m_2)$ are as in $(i)$ or $(ii)$ then
$$|\SSTS(r_1,c_1,m_1)-\mathbf{I}^\theta_{(r_1,c_1,m_1)}|>|\SSTS(r_2,c_2,m_2)-\mathbf{I}^\theta_{(r_2,c_2,m_2)}|.$$
Therefore if 
$ (r_1,c_1,m_1)\succ(r_2,c_2,m_2)\succ \dots \succ (r_k,c_k,m_k)$ 
then 
\begin{equation}\label{succ}
|\SSTS(r_1,c_1,m_1)-\mathbf{I}^\theta_{(r_1,c_1,m_1)}|\geq |\SSTS(r_k,c_k,m_k)-\mathbf{I}^\theta_{(r_k,c_k,m_k)}|
\end{equation}
with equality only if $m_1=m_k$ and  $r_1 \geq r_k$, $c_1\geq c_k$ (and there are no crossings between the  strands labelled by these boxes).  
 
 We consider the transitive closure of the relation $\succeq$ (by abuse of notation we also denote this by $\succeq$); this relation is transitive and reflexive by definition.  
 If $(r,c,m) \succeq (r',c',m')$ and $(r,c,m) \preceq (r',c',m')$ then 
 by \cref{succ} we have that $(r,c,m)=(r',c',m')$; hence the relation is antisymmetric.
 Therefore $\succeq$ defines a partial ordering on 
  the boxes of  $\la \in \mptn \ell n$.  
 
 Regard $\succ$ as a partial ordering on the boxes of  $\SSTS(\la)=\mu$ by identifying the nodes of $\mu$ with the corresponding nodes of $\la$.  
We can encode any  total refinement, $\succ_{\rm t}$, of  $\succ$  as a tableau,
 $\SSTS^c$, of shape $\mu$  and weight ${\omega}$. 
  This is simply given by  letting  $\SSTS^c(r,c,m) < \SSTS^c (r',c',m') $ if  and only if 
  $(r,c,m) \succ_{\rm t} (r',c',m')$ for $(r,c,m), (r',c',m')\in \mu$.  

It remains to show that $ \Cell_{\SSTS ^c}  \Cell_\SSTS  = \Cell_{\varphi({\overline{\sts}})}$  for some ${\overline{\sts}} \in \Std_{\theta}(\la)$.  
Suppose $(r,c,m)$ and $(r',c',m')$ are two boxes in $\la$ whose solid or ghost strands cross in the diagram $ \Cell_\SSTS$.  
In which case,   $(r,c,m) \succ (r',c',m')$ (or vice versa) and   we are as in one of  cases $(i)$,  $(ii)$, or  $(iii)$ of \cref{defsucc}.  
By definition, $\SSTS^c(r,c,m)<\SSTS^c (r',c',m') $ and so the crossing strands from $ \Cell_{\SSTS}$ do not
cross again in  $ \Cell_{\SSTS^c}$.  Therefore the diagram $ \Cell_{\SSTS ^c}  \Cell_\SSTS  $ contains no double-crossings and so is equal to 
$ \Cell_{\overline{\SSTS}}$ for ${\overline{\SSTS}}$ some tableau of shape $\la$ and weight $\omega$.  
Now, by condition $(iii)$ of \cref{defsucc}, 
 we have that ${\overline{\SSTS}}(r,c+1,m)< {\overline{\SSTS}}(r,c,m)$ and ${\overline{\SSTS}}(r+1,c,m)< {\overline{\SSTS}}(r,c,m)$  for all $1\leq r,c\leq n$ and $1\leq m \leq \ell$.  
 Since any pair of  boxes of the $\ell$-partition $\omega$ 
 are at   $\g $ unit  apart, ${\overline{\SSTS}}$ satisfies conditions $(i)$ and $(ii)$ of \cref{semistandard:defn}.  
Finally for  any  $1\leq m \leq \ell$,  we have that
  ${\overline{\SSTS}}(1,1,m) = (r,1,\ell)$ for some $1\leq r \leq n$ and so
${\overline{\SSTS}}$ satisfies condition $(iii)$ of  \cref{semistandard:defn}.  
  Therefore ${\overline{\SSTS}}$ is semistandard.   
Finally, we let ${\overline{\sts}}$ be the standard tableau determined by $\varphi({\overline{\sts}})={\overline{\SSTS}}$ and this completes the proof.  
 \end{proof}

Finally,  we generalise  {\cite[Theorem 4.11]{MR3732238}}  to an arbitrary     integral domain.  

\begin{thm} \label{cellularitybreedscontempt2}
Let $\Bbbk$ be an arbitrary   integral domain.   
The algebra $\algebra^\Bbbk_n(\theta )$ is free as an $\Bbbk$-module and  has a  graded cellular    basis 
\[
  \{{ \Cell}_{\SSTS  \SSTT} \mid \SSTS \in \SStd_{\theta}(\lambda,\mu), \SSTT\in \SStd_{\theta}(\lambda,\nu), 
 \lambda  \in \mptn {\ell}n,  \mu, \nu  \in \con {\ell}n\} 
\]
 with respect to the $\theta$-dominance order on $\mptn \ell n$ and the involution $\ast$ given by
 horizontal reflection. 
 We let $\Delta_{\sigma }^\Bbbk(\la)$ denote the corresponding cell-module for $\la \in \mptn \ell n$.
\end{thm}

\begin{proof}
 We shall prove this by contradiction.    By \cref{cellularitybreedscontempt0.5} and the fact that $\sum_{\alpha \in \con \ell n,  {\imath} \in (\ZZ/e\ZZ)^n} {\sf 1}^{\imath}_\alpha$ is the identity of $\algebra^\Bbbk_n(\theta )$, it is enough to show that if  there exist   $\alpha_{\SSTU\SSTV}\in \Bbbk$ such that 
\begin{align}\label{an equation}
\sum_{
\begin{subarray}c
\SSTU\in \SStd_{\theta}(\la,\mu) \\
\SSTV\in \SStd_{\theta}(\la,\nu) 
\end{subarray}
}
\alpha_{\SSTU\SSTV}{ \Cell}_{\SSTU\SSTV} = 0  \mod \algebra^{\rhd \la}_n(\theta)  
\end{align}
then this implies that  $\alpha_{\SSTU\SSTV}=0$ for all $\SSTU\in \SStd_{\theta}(\la,\mu), \SSTV\in \SStd_{\theta}(\la,\nu)$.  
We set  
$\sharp(\SSTS , 
 \SSTT)=\ell[ \Cell_\SSTS]+\ell[ \Cell_\SSTT] $. 
 We let $\SSTS, \SSTT$ be any pair such that $\sharp(\SSTS , 
 \SSTT)\geq \sharp(\SSTU , 
 \SSTV)$  for any pair  of tableaux $\SSTU , 
 \SSTV$ with $\alpha_{\SSTU\SSTV}\neq0$.  We let $\SSTS^c$ (respectively $\SSTT^c$) denote any tableau of shape $\mu$ (respectively $\nu$) and shape $\omega$ as in \cref{complement}.   
We shall show   that the coefficient $\alpha_{\SSTS\SSTT}$ is necessarily zero
   (and so the result immediately follows by repeating this argument).  We multiply 
   \cref{an equation} on the left by $ \Cell^\ast_{\SSTS^c}$ and on the  right  by $ \Cell_{\SSTT^c}$; 
      it is  enough to show that if 
\begin{align}\label{another eqution}
\alpha_{\SSTS\SSTT} \Cell_{\overline\SSTS\; \overline\SSTT} + 
\sum_{
\begin{subarray}c
\SSTU\in \SStd_{\theta}(\la,\mu) \\
\SSTV\in \SStd_{\theta}(\la,\nu) 
\end{subarray}
}
\alpha_{\SSTU\SSTV}  \Cell_{\SSTS^c}{ \Cell}_{\SSTU\SSTV}  \Cell^\ast_{\SSTT^c}= 0  \mod \algebra^{\rhd \la}_n(\theta) , 
\end{align} then $\alpha_{\SSTS\SSTT}=0$ (where $\overline\SSTS=\varphi(\overline\sts)$ and $\overline\SSTT=\varphi(\overline\stt)$ as in \cref{complement}).  
There are two cases to consider.  Firstly, if one of $ \Cell_{\SSTS^c}   \Cell_{\SSTU}$ or  
 $ \Cell_{\SSTV}^\ast  \Cell_{\SSTT^c}^\ast$ contains a double-crossing, then 
$$
\Cell _{\SSTS^c} {\Cell}_{\SSTU\SSTV} \Cell^\ast_{\SSTT^c}=
 \sum_{\begin{subarray}c
\SSTU', \SSTV'\in \SStd_{\theta}(\la,\omega) \\
\sharp(\SSTU' ,
 \SSTV')< 
\sharp(\SSTT , 
 \SSTV)
\end{subarray}} \beta_{\SSTU'\SSTV'} {\Cell}_{\SSTU'\SSTV'} \mod\algebra _n^{\rhd \la}(\sigma)
$$
for some $\beta_{\SSTU'\SSTV'}\in \Bbbk$, by \cref{mastumoto2}.     
 We now  consider the case in which 
$ \Cell_{\SSTS^c} \Cell_{\SSTU}$ and   $ \Cell_{\SSTV}^\ast  \Cell_{\SSTT^c}^\ast$ 
 contain no double-crossings. 
We have that $(\SSTS,\SSTT)\neq (\SSTU,\SSTV)$.
  Assume $\SSTU\neq \SSTS$, 
then the bijection traced out by $\SSTU$ is different to that traced out by $\SSTS$; therefore the bijection traced out by $  \Cell_{\SSTS^c}  \Cell_\SSTU   $  is not equal to that traced out by  $   \Cell_{\SSTS^c} \Cell_\SSTS = \Cell_{\varphi(\overline{\sts})}$.  In particular, if $  \Cell_{\SSTS^c}  \Cell_\SSTU  $  contains no double-crossings, then it is equal to $ \Cell_{{\overline{\SSTU}}}$ for ${{\overline{\SSTU}}}$ some   (not necessarily semistandard) tableau of shape $\la $ and weight $\omega$ which is {\em not} equal to ${\varphi(\overline{\sts})}$. 
 Arguing similarly  for the case  $\SSTV\neq \SSTT$, we
therefore deduce that  $$ \Cell _{\SSTS^c}{\Cell}_{\SSTU\SSTV}  \Cell_{\SSTT^c}^\ast =  \Cell_{{\overline{\SSTU}}\; {\overline{\SSTV}}}$$ for ${\overline{\SSTU}}, {\overline{\SSTV}}$ two (not necessarily semistandard) tableaux of shape $\la$ such that $({\overline{\SSTU}} ,{\overline{\SSTV}})\neq (\varphi(\sts),\varphi(\stt))$.  
  Now, if ${\overline{\SSTU}}$ and ${\overline{\SSTV}}$ are not semistandard, then 
 $$ \Cell_{{\overline{\SSTU}}\; {\overline{\SSTV}}} = \sum_{\begin{subarray}c
\SSTU', \SSTV'\in \SStd_{\theta}(\la,\omega) \\
 \sharp(\SSTU',
 \SSTV')< 
\sharp(\SSTS ,
 \SSTT)
\end{subarray}} \gamma_{\SSTU'\SSTV'} {\Cell}_{\SSTU'\SSTV'} \mod  \algebra_n^{\rhd \la}(\sigma)
$$
for some $\gamma_{\SSTU'\SSTV'}\in \Bbbk$, by  as in the proof of \cref{cellularitybreedscontempt}.  If ${\overline{\SSTU}}$ and ${\overline{\SSTV}}$ are   semistandard, then we  set 
$ \overline{\SSTU} =\SSTU'$ and $  \overline{\SSTV} =\SSTV'$     for convenience.   
Putting all of this  together,  we have that \cref{another eqution} is equivalent to 
 $$
  \alpha_{\SSTS\SSTT} {\Cell}_{\overline\SSTS\;\overline\SSTT}   +  \sum_{
 \begin{subarray}c
\SSTU', \SSTV'\in \SStd_{\theta}(\la,\omega) \\
(\SSTU' , \SSTV') \neq  (\SSTS,\SSTT)
 \end{subarray} 
} 
\alpha_{\SSTU\SSTV}(\beta_{\SSTU'\SSTV'} +\gamma_{\SSTU'\SSTV'}) {\Cell}_{\SSTU'\SSTV'}=0 \mod  \algebra_n^{\rhd \la}(\sigma).
$$
Now, the set $\{ \Cell_{\SSTQ\SSTR} \mid \SSTQ,\SSTR\in\SStd_{\theta}(\la,\omega)\}$ is a basis of 
${\sf E}^\theta_\omega \algebra^\Bbbk_n(\theta ){\sf E}^\theta_\omega$ by \cref{isomer} and so $\alpha_{\SSTS\SSTT}=0$, as required.  
Therefore we have verified condition $(2)$ of \cref{defn1}.  
 Conditions $(1)$ and $(4)$ of  \cref{defn1} follow immediately from the diagrammatic definitions.  
 Condition $(3)$ follows from   \cref{mastumoto2,move a dot down}.     \end{proof}

\begin{cor}[{\cite[Cor  2.26]{MR3732238}}]\label{qaausfuduifddfheredtiary}
Let $\Bbbk$ be a field.  The algebra $\algebra^\Bbbk_n(\theta )$ is quasi-hereditary and 
the $   L^\Bbbk_\sigma(\la)=\Delta_{\sigma }^\Bbbk(\la) / \rad(\Delta_{\sigma }^\Bbbk(\la))\text{ for }\la \in \mptn \ell n $
provide   a complete set of non-isomorphic  irreducible  modules.  
\end{cor}
 
\begin{proof}Let    $\SSTT$ denote the unique element of   $\SStd_\theta(\la,\la)$.  
The  element $\Cell_{\SSTT \SSTT} ={\sf 1}_\la \in \algebra^{\trianglerighteq \la} _n(\theta)  
 $ is an idempotent.  
Therefore the radical of the bilinear form is not the whole cell module. Therefore the algebra is quasi-hereditary with the prescribed set of  irreducible  modules.  
\end{proof}

  \renewcommand{\H}{{\mathcal H}}
  
\section{The many integral cellular bases of quiver Hecke algebras}  \label{mainresults}

 We now proceed to apply  the many Schur functors in order to obtain our many graded cellular bases of Hecke algebras.  Given $\underline{s}=(e;s_0,s_2,\dots,s_{\ell-1})\in \NN_{>1} \times (\ZZ/e\ZZ)^\ell$ we let 
$\theta=(e;\sigma_0,\sigma_1,\dots,\sigma_{\ell-1}) \in \weight$ denote any choice of  {integral lift}.  
We have seen that $\mathcal{H}_n ^\Bbbk (\underline{s}) \cong {\sf E}^\theta_\omega \algebra^\Bbbk_n(\theta ){\sf E}^\theta_\omega$ is   generated by 
 $$ 
   \langle \sigma (y_1),\dots,\sigma (y_n),  \sigma (\psi_1), \dots, \sigma (\psi_{n-1}),   \sigma (e(\imath)) \mid  \imath\in (\ZZ/e\ZZ)^n 
\rangle  $$ 
subject to relations \ref{rel1} to \ref{rel15}.  
This idea should be very familiar to those working with Cherednik algebras.  Given a fixed Hecke algebra $\mathcal H_n^\Bbbk(\underline{s})$ there are many associated quiver Cherednik algebras 
  $ \algebra^\Bbbk_n(\theta )$ (namely, one for each integral lift $\theta\in \weight$).  
 Each of these   distinct quiver Cherednik algebras casts its own ``charged  shadow" on the representation theory of our fixed Hecke algebra.   
    Given $\sts,\stt \in \Std_{\theta}(\la)$
 we set $$\cell^\theta_{\sts\stt}:={\sf E}^\theta_\omega \Cell_{\SSTS\SSTT}{\sf E}^\theta_\omega \in \H_n^\Bbbk(\underline s) $$   
 where $\varphi(\sts)=\SSTS \in \SStd_{\theta}(\la,\omega) $ and  $\varphi(\stt)=\SSTT \in \SStd_{\theta}(\la,\omega) $.

 \begin{thm}\label{corollary}
  For  a charge $\theta\in \weight$,  
  the $\Bbbk$-algebra  $\H_n^\Bbbk(\underline s) $ 
    admits a  graded cellular structure with respect to the poset     $(\mptn \ell n,\rhd_\theta)$ and  the     basis
   $$ 
   \{ \cell^\theta_{ \sts\stt}  \mid \lambda \in \mptn \ell n, \sts,\stt \in \Std_{\theta}(\lambda)\}
   $$  
   and the involution $\ast$.
 In particular,  $\deg(\cell^{\theta }_{\sts \stt})=\deg(\sts)+\deg(\stt)$ for  $\sts,\stt \in  \Std_{\theta}(\lambda)$. 
 \end{thm}
 
 \begin{proof}
 The  elements ${\sf E}^\theta_\omega \Cell_{\SSTS\SSTT} {\sf E}^\theta_\omega =\cell^\theta_{\SSTS\SSTT}  $ satisfy property $(2)$  for  ${\sf E}^\theta_\omega \algebra^\Bbbk_n(\theta ){\sf E}^\theta_\omega \cong \H_n^\Bbbk(\underline s) $ and property $(4)$   immediately.    
 We have that  $$({\sf E}^\theta_\omega    A {\sf E}^\theta_\omega)(   {\sf E}^\theta_\omega \Cell_{\SSTS\SSTT} {\sf E}^\theta_\omega  )
 = ({\sf E}^\theta_\omega    A)( \Cell_{\SSTS\SSTT})$$ for $\SSTS, \SSTT \in \SStd_{\theta}(\lambda,\omega)$ 
  and therefore $(3)$ for ${\sf E}^\theta_\omega \algebra^\Bbbk_n(\theta ){\sf E}^\theta_\omega$ follows from condition $(3)$ for $  \algebra^\Bbbk_n(\theta )$.  
 To see that condition $(1)$ holds, we proceed by induction on $n\in \NN$. 
  The $n=0$ case holds trivially.   
 Now, let $\sts \in \Std_{\theta}(\la)$ and  let $(r,c,m)\in \boxla $ be such that 
  $\sts(r,c,m)=n$.  
   By induction, we may assume that 
   $$\deg(\sts{\downarrow}_{\{1,\dots,n-1\}})=\deg (\cell^\theta_{\sts{\downarrow}_{\{1,\dots,n-1\}}}) $$
having trivially verified that the $n=1$ case holds.     Now, we have that 
   $$
\cell^\theta_{\sts}=    \overline{\cell}^\theta_{\sts{\downarrow}_{\{1,\dots,n-1\}}} \times
 {\sf 1}^{\la +(n,1,\ell)}_{\la +(r,c,m)}
   $$
   where the diagrams on the  right\color{black}hand-side  are constructed as follows 
    \begin{itemize}[leftmargin=*] \item  we obtain $\overline{\cell}^\theta_{\sts{\downarrow}_{\{1,\dots,n-1\}}}$
    from  
   $\cell ^\theta_{\sts{\downarrow}_{\{1,\dots,n-1\}}}  $ by adding a  vertical solid strand with $x$-coordinate 
    $ {\bf I}^{\theta}_{(n,1,\ell)}$;
    \item   we obtain
${\sf 1}^{\la +(n,1,\ell)}_{\la +(r,c,m)}$  from  ${\sf 1}_\la$
by adding a solid strand, $S_n$,  from        $(\mathbf{I}^\theta_{(r,c,m)},0)$ to  $({\bf I}^{\theta}_{(n,1,\ell)} ,1)$; 
\end{itemize}
and both diagrams are drawn  in such a way as to create no double-crossings.  
The degree of $\cell^\theta_{\sts}$ can be calculated inductively as follows, 
$$
\deg(\cell^\theta_{\sts})=    \deg(\overline{\cell}^\theta_{\sts{\downarrow}_{\{1,\dots,n-1\}}})
+ \deg({\sf 1}^{\la +(n,1,\ell)}_{\la +(r,c,m)})
\quad\text{
where} \quad\deg(\overline{\cell}^\theta_{\sts{\downarrow}_{\{1,\dots,n-1\}}})= \deg(\cell^\theta_{\sts{\downarrow}_{\{1,\dots,n-1\}}})$$ by construction.   The degree of $  {\sf 1}^{\la +(n,1,\ell)}_{\la +(r,c,m)}$ is calculated in terms of the 
number of crossings as in \cref{grsubsec}.    
We calculate this brick-by-brick and diagonal-at-a-time as follows.

 If  $S_n$  passes through a brick $\mathbf{B}_k$ for $k=1$, $k=2,3$,
  $k=4,5$ or $k=6$, then the degree contribution of this crossing is
$0$, $-1$, $+1$, or $-2$ respectively.  
  Let  $\mathbf{D}$ be a  diagonal in the diagram $ {\sf 1}_{\la }$  and suppose that 
 $S_n$  passes through $\mathbf{D}$.   
An addable   diagonal is built out of 
  a single  $\mathbf{B}_k$ brick  for $k\in\{4,5,6\}$ and
  some number (possibly zero) of $\mathbf{B}_1$ bricks.   
 A removable (respectively  invisible)  diagonal  has an extra  
  single   $\mathbf{B}_k$ brick    for $k=6$ (respectively $k\in\{2,3\}$).  
  Summing over the degrees, we conclude that the 
  crossing of  $S_n$  with an addable, removable, or invisible 
  $i$-diagonal has degree $+1$, $-1$, or $0$. 
    Finally, we observe that the $i$-diagonals in  $  {\sf 1}_{\la}$ which the   $S_n$ strand crosses are precisely those to the left of ${\bf I}^{\theta}_{(r,c,m)}$ and so the total degree contribution of this strand is 
$	|{\mathcal A}_\stt(n)|-|{\mathcal R}_\stt(n)|	$.  Therefore condition $(1)$ holds.  
   \end{proof}
 
   \renewcommand{\Theta}{{\Sigma^\ell_n}}
Using the standard facts about cellular algebras which we recalled in \cref{gradedcell}, we are now able to define the many different families of  cell/Specht modules and many different parameterisations/constructions of  irreducible  modules promised in the introduction.  

 \begin{defn}\label{construction}
Fix an $e$-charge 
 $\underline{s}=(e;s_0,s_1,\dots,s_{\ell-1})\in \NN_{>1}\times(\ZZ/e\ZZ)^\ell$.  
 For each integral lift $\sigma\in\weight$ of the $e$-charge  and $\lambda \in \mptn \ell n $, we let 
\begin{equation} 
\Specht^\Bbbk_{\theta} (\la)=\{ \cell^\theta_{\sts} \mid \sts \in \Std_{\theta}(\la)\}   
\end{equation}
denote the corresponding  $\H^\Bbbk_n(\underline{s})$ cell-module.   
For  $\Bbbk$   a  field and an  integral lift $\sigma\in\weight$ of the $e$-charge, we set   
$$\Theta   =\{ \lambda\in \mptn \ell n\mid {\rm rad}(\langle\ ,\ \rangle_\lambda)\neq \Specht^\Bbbk_{\theta}(\lambda)\}\subseteq \mptn \ell n $$
 and we let 
$$\{ \Simple^\Bbbk  _\theta (\la):=   \Specht^\Bbbk_\theta(\lambda)/ {\rm rad}^\Bbbk(\langle\ ,\ \rangle_\lambda)  \mid \lambda \in \Theta \}.$$

  \end{defn}

%

\begin{prop}\label{huuuuu}
For $\sigma_0 \gg \sigma_1 \gg \dots \gg \sigma_{\ell-1}$ an asymptotic charge, the basis 
  $$ 
   \{ \cell^\theta_{ \sts\stt}  \mid \lambda \in \mptn \ell n, \sts,\stt \in \Std_{\theta}(\lambda)\}
   $$  
differs from that of \cite[Main Theorem]{hm10} by $\rhd_\sigma$-unitriangular change of basis matrix.  
For each $\mu \in \mptn \ell n$ the cell module  $\Specht_\sigma(\mu)$ is isomorphic  as a graded $\mathcal{H}_n ^\Bbbk (\underline{s}) $-module  to the usual graded Specht module 
(with the same label)    defined in \cite[Theorem 4.10]{bkw11}. 
\end{prop}

\begin{proof}
For $\stt^\la\in \Std_\sigma(\la)$ the {\em unique} tableau satisfying $\stt^\la \trianglerighteq \sts$ for all $\sts \in  \Std_\sigma(\la)$. 
 The  chain of 2-sided cell ideals in   \cite{hm10} is  given by the  $\mathcal{H}^\Bbbk_n(\underline{s}) y_\la \mathcal{H}^\Bbbk_n(\underline{s} )$ for 
$$y_\la = \prod _{k=1}^n y_k ^{|{\mathcal A}_{\stt^\la}(k)|} 
e_{{\rm res}( \stt^\la)} \quad  \text{and we claim that}
\quad y_\la =\cell^\sigma_{\stt^\la\stt^\la}+\sum _{
\begin{subarray}c 
\mu\rhd_\sigma \la \\  \sts,\stt \in \Std_\sigma (\mu)
\end{subarray}}\alpha_{\stu\stv}\cell^\sigma_{\sts\stt}$$ for $\alpha_{\stu\stv}\in \ZZ$.  
Once the claim is established, we have that the chains of two-sided ideals are isomorphic and the result follows.  
We now prove the claim by induction on $\rhd _\sigma$ and $1\leq k\leq n$.  At the $k$th step,  we pull the $k$th strand to the right using the leftmost relation in \ref{rel11} until we encounter a solid $j$-strand (from some earlier step in the process) with $|j-i_k|\leq 1$.
If the $k$th strand is undecorated (either because $|{\mathcal A}_{\stt^\la}(k)|=0$ or because 
we have already applied \ref{rel11} a total of $|{\mathcal A}_{\stt^\la}(k)|$ times) then the $k$th strand is now part of
 either an ${\sf M}_2$-brick or has comes to rest in the required position (in which case we are done).  
 In the former case, we use the ${\bf M}_2$ analogue of \cref{M1relation} to move the strand rightwards at the expense of an error term (which is zero by  \cref{aprop}).   
Otherwise, we have that the $k$th strand carries a {\em single} dot (by definition of $\stt^\la$) and we use \ref{rel7} to move the strand rightwards at the expense of an error term (which is again zero by  \cref{aprop}).   
Repeating as necessary, the process terminates when the $k$th strand reaches the $x$-coordinate of the box $(\stt^\la)^{-1}(k)$.  
 \end{proof}

We now consider 
$\H_2^\Bbbk(\underline{s} )$ for  $\underline{s}= (e; s_0,s_1)=(2;0,1)$.  This algebra has two    irreducible  modules, $\Simple^\Bbbk(0,1)$ and $\Simple^\Bbbk(1,0)$, 
 which are generated by  $e(0,1)$ and $e(1,0)$ respectively, and which are annihilated by all the other generators of $
 \H_2^\Bbbk(\underline{s})$.   
 Here we   label  the  irreducibles   by the corresponding idempotents (not by $\ell$-partitions) because this labelling is independent of the charge; we will  reconcile this with the charged labelings in \cref{labelllllleres,labelllllleres2}.  
 There are two charges, $\sigma=(e;\theta_0,\sigma_1)=(2;4,1) $ and $\sigma=(e;\theta_0',\sigma_1')=(2;2,1)$, which give rise to distinct    cellular structures (every other charge gives   a basis  equivalent to one of these).   
We  show that  there   is no isomorphism relating the sets of cell modules obtained from these distinct charges.  
  
\begin{eg}\label{countereg1} 
Let  $\sigma=(e; \theta_0,\sigma_1)=(2;4,1)$, we remark that this is a well-separated charge.  
The $\sigma$-dominance order is   as follows,
$$
((2),\varnothing)\rhd_{\sigma} 
((1^2),\varnothing)\rhd_{\sigma} 
((1),(1))\rhd_{\sigma} 
( \varnothing,(2))\rhd_{\sigma} 
( \varnothing,(1^2)).   
$$
We let 
\begin{align*}
 \stw \in \Std_{\sigma}((2),\varnothing))  \quad \stv \in \Std_{\sigma}((1^2),\varnothing)  \quad
 \stt,\stu \in \Std_{\sigma}((1),(1))  \quad 
 \sts \in   \Std_{\sigma}( \varnothing, (2))  \quad
 \str \in  \Std_{\sigma}( \varnothing,(1^2)).
 \end{align*}   We choose $\stt$ so that the box $\stt^{-1}(1)$ has residue 0.  
     We leave constructing the diagrammatic version of this  basis as an exercise for the reader.  Instead, 
we describe the basis  as a linear combination of products of the KLR generators 
(using the process described in \cref{isomer}) as follows, 
 
 $$
  \addtolength{\tabcolsep}{3pt}    
\setlength\extrarowheight{5pt}  \begin{tabular}{|rlrl|}
 \hline $\cell^{{\sigma}}_{\str\str}$ &$=
  e(1,0)$&&
\\
\hline  $\cell^{{\sigma}}_{\sts\sts}$ &$=
 y_2 e(1,0)$  && \\ \hline
  $\cell^{{\sigma}}_{\stt\stt}$ &$=
  e(0,1) 						$ \quad \quad 
  &
   $\cell^{{\sigma}}_{\stt\stu}$ &$=
   \psi_1 e(1,0)$
  \\
 $\cell^{{\sigma}}_{\stu\stt}$ &$=
  \psi_1e(0,1)$  
&  $\cell^{{\sigma}}_{\stt\stt}$ &$=
 (y_2-y_1)(y_1-y_2)e(1,0)$
  \\  \hline
   $\cell^{{\sigma}}_{\stv\stv}$ &$=
  y_2 e(1,0)$  &&
\\  \hline $\cell^{{\sigma}}_{\stw\stw}$ &$=
  y_2^2   e(1,0)$ && \\ \hline
\end{tabular} 
 $$
   \end{eg}

 \label{many examples}

\begin{eg}\label{countereg2} 
Let $ \sigma=(e;\theta_0,\sigma_1)=(2;2,1)$.  
The $(2;2,1)$-dominance order is given as follows,
$$
 ((1^2),\varnothing), 
( \varnothing,(1^2)) \lhd_{\sigma} 
((1),(1))\lhd_{\sigma} 
((2),\varnothing),
( \varnothing,(2)).
$$
We let 
$$ \cell_{\sf rr}^\sigma=\begin{minipage}{6.4cm}
  \begin{tikzpicture}[baseline, thick,yscale=0.6,xscale=1.6]
\scalefont{0.8}
 
  \draw(0,0) rectangle (4,4);
 
  \draw[wei]  (1.5,4)  -- (1.5,0)   [below, at end] node{$1$};
  
  \draw[wei]  (2,4)  -- (2,0)   [below, at end] node{$0$}; 
 
 \draw[thick](0.5-0.06*3,4) to 
 [out=-90, in =90] (0.5-0.06*3,0)  [below, at end] node{$0$}; 
 
 \draw[densely dotted,thick](1.5-0.06*3,4) to 
 [out=-90, in =90] (1.5-0.06*3,0);

 \draw[thick](1.5-2*0.06,4) to 
 [out=-90, in =90] (1.5-2*0.06,0)  [below, at end] node{$1$}; 
 \draw[thick,densely dotted](2.5-2*0.06,4) to 
 [out=-90, in =90] (2.5-2*0.06,0);

\end{tikzpicture}\end{minipage}
\qquad
\cell_{\sf ss}^\sigma=
\begin{minipage}{6.4cm} \begin{tikzpicture}[baseline, thick,yscale=0.6,xscale=1.6]
\scalefont{0.8}
 
  \draw(0,0) rectangle (4,4);
 
  \draw[wei]  (1.5,4)  -- (1.5,0)   [below, at end] node{$1$};
  
  \draw[wei]  (2,4)  -- (2,0)   [below, at end] node{$0$}; 
 
 \draw[thick](0.5-0.06*3,4) to  [out=-90, in =90] (1-3*0.06,2) to 
 [out=-90, in =90] (0.5-0.06*3,0)  [below, at end] node{$1$}; 
 
 \draw[densely dotted,thick](1.5-0.06*3,4)  to  [out=-90, in =90] (2-3*0.06,2) to 
 [out=-90, in =90] (1.5-0.06*3,0);

 \draw[thick](1.5-2*0.06,4) to  [out=-90, in =90] (2-2*0.06,2) to 
 [out=-90, in =90] (1.5-2*0.06,0)  [below, at end] node{$0$}; 
 \draw[thick,densely dotted](2.5-2*0.06,4) to  [out=-90, in =90](3-2*0.06,2)  to 
 [out=-90, in =90] (2.5-2*0.06,0);

\end{tikzpicture}\end{minipage}
$$
$$ \cell_{\sf tt}^\sigma=\begin{minipage}{6.4cm}
  \begin{tikzpicture}[baseline, thick,yscale=0.6,xscale=1.6]
\scalefont{0.8}
 
  \draw(0,0) rectangle (4,4);
 
  \draw[wei]  (1.5,4)  -- (1.5,0)   [below, at end] node{$1$};
  
  \draw[wei]  (2,4)  -- (2,0)   [below, at end] node{$0$}; 
 
 \draw[thick](0.5-0.06*3,4)to  [out=-90, in =90] (2-2*0.06,2) to 
 [out=-90, in =90] (0.5-0.06*3,0)  [below, at end] node{$0$}; 
 
 \draw[densely dotted,thick](1.5-0.06*3,4)to  [out=-90, in =90] (3-2*0.06,2) to 
 [out=-90, in =90] (1.5-0.06*3,0);

 \draw[thick](1.5-2*0.06,4)  to 
 [out=-90, in =90] (1.5-2*0.06,0)  [below, at end] node{$1$}; 
 \draw[thick,densely dotted](2.5-2*0.06,4) to 
 [out=-90, in =90] (2.5-2*0.06,0);

\end{tikzpicture}\end{minipage}
\qquad
\cell_{\sf tu}^\sigma=
\begin{minipage}{6.4cm} \begin{tikzpicture}[baseline, thick,yscale=0.6,xscale=1.6]
\scalefont{0.8}
 
  \draw(0,0) rectangle (4,4);
 
  \draw[wei]  (1.5,4)  -- (1.5,0)   [below, at end] node{$1$};
  
  \draw[wei]  (2,4)  -- (2,0)   [below, at end] node{$0$}; 
 
  \draw[thick](0.5-0.06*3,4)to  [out=-90, in =90] (2-2*0.06,2) to 
 [out=-90, in =90]  		(1.5-2*0.06,0.4) --++(-90:0.4)	  [below, at end] node{$0$}; 
 
 \draw[densely dotted,thick](1.5-0.06*3,4)to  [out=-90, in =90] (3-2*0.06,2) to 
 [out=-90, in =90] (2.5-2*0.06,0.4)-++(-90:0.4);

 \draw[thick](1.5-2*0.06,4) 
to  [out=-90, in =90]   (1.5-2*0.06,2)to
 [out=-90, in =90] (0.5-0.06*3,0)  [below, at end] node{$1$}; 
 \draw[thick,densely dotted](2.5-2*0.06,4)to  [out=-90, in =90]   (2.5-2*0.06,2)to      
 [out=-90, in =90]  (1.5-0.06*3,0);

\end{tikzpicture}\end{minipage}
$$
$$
\cell_{\sf ut}^\sigma=
\begin{minipage}{6.4cm} \begin{tikzpicture}[baseline, thick,yscale=0.6,xscale=1.6]
\scalefont{0.8}
 
  \draw(0,0) rectangle (4,4);
 
  \draw[wei]  (1.5,4)  -- (1.5,0)   [below, at end] node{$1$};
  
  \draw[wei]  (2,4)  -- (2,0)   [below, at end] node{$0$};

  \draw[thick](0.5-0.06*3,0)to  [out=90, in =-90] (2-2*0.06,2) to 
 [out=90, in =-90]  		(1.5-2*0.06,4) ;
 
 \draw[densely dotted,thick](1.5-0.06*3,0)to  [out=90, in =-90] (3-2*0.06,2) to 
 [out=90, in =-90] (2.5-2*0.06,4);   
   \draw[thick](0.5-3*0.06,0)   [below, at start] node{$0$};

 \draw[thick](1.5-2*0.06,0) 
to  [out=90, in =-90]   (1.5-2*0.06,2)to
 [out=90, in =-90] (0.5-0.06*3,4-0.4)--++(90:0.4) ; 
 \draw[thick](1.5-2*0.06,0)   [below, at start] node{$1$}; 
 \draw[thick,densely dotted](2.5-2*0.06,0)to  [out=90, in =-90]   (2.5-2*0.06,2)to      
 [out=90, in =-90]  (1.5-0.06*3,4-0.4)--++(90:0.4) ;

\end{tikzpicture}\end{minipage}
\qquad
\cell_{\sf uu}^\sigma=
\begin{minipage}{6.4cm} \begin{tikzpicture}[baseline, thick,yscale=0.6,xscale=1.6]
\scalefont{0.8}
 
  \draw(0,0) rectangle (4,4);
 
  \draw[wei]  (1.5,4)  -- (1.5,0)   [below, at end] node{$1$};
  
  \draw[wei]  (2,4)  -- (2,0)   [below, at end] node{$0$};

  \draw[thick](1.5-2*0.06,0)--++(90:0.3) to  [out=90, in =-90] (2-2*0.06,2) to 
 [out=90, in =-90]  		(1.5-2*0.06,4-0.3)--++(90:0.3)  ;
 
 \draw[densely dotted,thick](2.5-2*0.06,0)--++(90:0.3) to  [out=90, in =-90] (3-2*0.06,2) to 
 [out=90, in =-90] (2.5-2*0.06,4-0.3)--++(90:0.3) ;   
   \draw[thick](0.5-3*0.06,0)   [below, at start] node{$1$};

 \draw[thick](0.5-0.06*3,0) 
to  [out=90, in =-90]   (1.5-2*0.06,2)to
 [out=90, in =-90] (0.5-0.06*3,4)  ; 
  \draw[thick](1.5-2*0.06,0)   [below, at start] node{$0$}; 
 \draw[thick,densely dotted](1.5-0.06*3,0)to  [out=90, in =-90]   (2.5-2*0.06,2)to      
  [out=90, in =-90]  (1.5-0.06*3,4) ;

\end{tikzpicture}\end{minipage}$$
$$
\cell_{\sf vv}^\sigma=
\begin{minipage}{6.4cm} \begin{tikzpicture}[baseline, thick,yscale=0.6,xscale=1.6]
\scalefont{0.8}
 
  \draw(0,0) rectangle (4,4);
 
  \draw[wei]  (1.5,4)  -- (1.5,0)   [below, at end] node{$1$};
  
  \draw[wei]  (2,4)  -- (2,0)   [below, at end] node{$0$};

  \draw[thick](1.5-2*0.06,0) to 
 [out=90, in =-90]  	(1.5-2*0.06,2)
  to 
 [out=90, in =-90]  	(1.5-2*0.06,4-0.3)--++(90:0.3)  ;
 
 \draw[densely dotted,thick](2.5-2*0.06,0)--  (2.5-2*0.06,4-0.3)--++(90:0.3) ;   
   \draw[thick](0.5-3*0.06,0)   [below, at start] node{$0$};

 \draw[thick](0.5-0.06*3,0) 
to  [out=90, in =-90]   (2.5-3*0.06,2)to
 [out=90, in =-90] (0.5-0.06*3,4)  ; 
  \draw[thick](1.5-2*0.06,0)   [below, at start] node{$1$}; 
 \draw[thick,densely dotted](1.5-0.06*3,0)to  [out=90, in =-90]   (3.5-3*0.06,2)to      
  [out=90, in =-90]  (1.5-0.06*3,4) ;

\end{tikzpicture}\end{minipage}\qquad
\cell_{\sf ww}^\sigma=
\begin{minipage}{6.4cm} \begin{tikzpicture}[baseline, thick,yscale=0.6,xscale=1.6]
\scalefont{0.8}
 
  \draw(0,0) rectangle (4,4);
 
  \draw[wei]  (1.5,4)  -- (1.5,0)   [below, at end] node{$1$};
  
  \draw[wei]  (2,4)  -- (2,0)   [below, at end] node{$0$};

  \draw[thick](1.5-2*0.06,0) to 
 [out=90, in =-90]  	(2-2*0.06,2)
  to 
 [out=90, in =-90]  	(1.5-2*0.06,4-0.3)--++(90:0.3)  ;
  
   \draw[densely dotted,thick](2.5-2*0.06,0) to 
 [out=90, in =-90]  
 (3-2*0.06,2)to 
 [out=90, in =-90]  
   (2.5-2*0.06,4) ;   
   \draw[thick](0.5-3*0.06,0)   [below, at start] node{$1$};

 \draw[thick](0.5-0.06*3,0) 
to  [out=90, in =-90]   (3-3*0.06,2)to
  [out=90, in =-90] (0.5-0.06*3,4)  ; 
  \draw[thick](1.5-2*0.06,0)   [below, at start] node{$0$}; 
 \draw[thick,densely dotted](1.5-0.06*3,0)to  [out=90, in =-90]   (4-3*0.06,2)to      
  [out=90, in =-90]  (1.5-0.06*3,4) ;

\end{tikzpicture}\end{minipage}
$$ where  we have chosen $\stt$ so that the box $\stt^{-1}(1)$ has residue 1.
 The elements $\cell^{\theta}_{\stw\stw}, \cell^{\theta}_{\stv\stv}$ are equal to the idempotents $e(1,0)$ and $e(0,1)$ respectively.   Therefore these basis elements generate the corresponding  irreducible  modules $L^\Bbbk(1,0)$ and $L^\Bbbk(0,1)$ 
 (modulo more dominant terms).  One can rewrite the above basis elements using relation \ref{rel1} to \ref{rel15} 
 (as in the proof of \cref{isomer}) to obtain a  basis of this algebra in terms of a  linear combination of  products of the  KLR-generators  as follows,
 $$
 \addtolength{\tabcolsep}{3pt}    
\setlength\extrarowheight{5pt}
  \begin{tabular}{|rlrl|}
 \hline$\cell^{\sigma}_{\str\str}$ &$=
  e(1,0)$
 & 
  $\cell^{\sigma}_{\sts\sts}$ &$=
 e(0,1)$ \\ \hline
  $\cell^{\sigma}_{\stt\stt}$ &$=
  y_2e(1,0)- y_1e(1,0)$ \quad 
  &
   $\cell^{\sigma}_{\stt\stu}$ &$=
  \psi_1e(0,1)$
  \\
 $\cell^{\sigma}_{\stu\stt}$ &$=
  \psi_1e(1,0)$  
&  $\cell^{\sigma}_{\stt\stt}$ &$=
  y_2e(0,1)- y_1e(0,1)$
  \\  \hline
   $\cell^{\sigma}_{\stv\stv}$ &$=
  y_2^2e(1,0)- y_1y_2e(1,0)$ \quad  
&  $\cell^{\sigma}_{\stw\stw}$ &$=
  y_2^2e(0,1)- y_1y_2e(0,1)$ 
\\ \hline  \end{tabular} 
 $$
 We remark that any  term  with a $y_1$ in the product is zero by relation \ref{rel1.12}. These terms have been included in order to facilitate comparison with the diagrams.  
  \end{eg}

  \begin{eg}The graded dimension of $\H_2^\Bbbk(2;0,1)$   can be calculated using either the $(2;4,1)$ or $(2;2,1)$  cellular structure  
$$
(1)^2 +(t)^2 + (1+t^2)^2 +(t)^2+(t^2)^2 
=
2+ 4t^2 + 2t^4
=(1)^2 +(1)^2 + (t+t)^2 +(t^2)^2+(t^2)^2, 
 $$
respectively and  is (of course!) independent of the choice of cellular structure.  
   \end{eg}

  In  \cite{bkw11} the authors prove a series of results on asymptotic cellular structures:
they construct graded tableaux-theoretic bases, analyse the restriction of asymptotic cell modules, 
  and provide quasi-Garnir relations which serve as  a warm-up to \cite{MR3004104}.  We have already generalised the graded tableaux and resulting bases in \cref{corollary} and we will prove the generalised  restriction rule in \cref{brancher}.  We now generalise their quasi-Garnir relations to all charges.

\begin{thm}\label{almostgarnir}
Given  $\stt$ a   tableau of shape $\la$, we  let $\cell_\stt$ be a   reduced diagram for $\stt$.   We have that 
$$ 
e(\imath)  \cell^{\theta}_{\stt}
 =\delta_{\imath,\res(\stt)}\cell^{\theta}_{\stt}
\quad
y_r \cell^{\theta}_{\stt}
=
 \sum_{
 \sts\rhd_\theta \stt
 }  \alpha_\sts \cell^{\theta}_\sts 
 \quad
 \psi_r \cell^{\theta}_{\stt}
=
\begin{cases}
 \cell^{\theta}_{s_{r,r+1}(\stt)}		&\text{if }  \stt \rhd_\theta s_{r,r+1}(\stt)\in \Std_\theta(\la) \\
 \sum_{\sts\rhd_\theta s_{r,r+1}(\stt)}\beta_\sts \cell^{\theta}_\sts
	&\text{otherwise. }  
\end{cases}
   $$
 \end{thm}

\begin{proof}
 This result follows directly from \cref{move a dot down,mastumoto2} once we have shown that the  Bruhat ordering on diagrams  coincides with the dominance order on $\Std_\theta(\la)$.
 Each solid (respectively ghost) strand in  
  $\cell_\stt^\sigma$ terminates at some northern point  ${\bf I}^\sigma_{(p,1,\ell)}$
  (respectively ${\bf I}^\sigma_{(p,1,\ell)}+1$) and
the corresponding southern point ${\bf I}^\sigma_{\stt^{-1}(p,1,\ell)}$
  (respectively ${\bf I}^\sigma_{\stt^{-1}(p,1,\ell)}+1$)   for some associated integer $1\leq p  \leq n$. 
A pair of solid or ghost strands in $\cell_\stt$  associated to integers $1\leq p < q \leq n$ 
crosses if and only if   ${\bf I}^\sigma_{(p,1,\ell)}<{\bf I}^\sigma_{(q,1,\ell)}$. 
   Thus undoing a crossing of  strands is equivalent to swapping the entries $p$ and $q$ in $\stt$ to obtain a diagram $\cell_{\sts}$ associated to  $\sts=s_{p,q}(\stt)$.  Finally, we observe that ${\bf I}^\sigma_{(p,1,\ell)}<{\bf I}^\sigma_{(q,1,\ell)}$ implies  $\sts \lhd _\sigma \stt$, as required.    
 \end{proof}

\tikzset{wei3/.style={darkgreen,double=white,double
distance=0.7pt}}

  \section{  Generic semisimplicity and 
  the decomposition map  over $\mathbb Q$}\label{Section2}
\label{Brelationspageofstuff}   

   We now recall Webster's definition of a generically semisimple algebra
   which specialises to be isomorphic to the (graded) quiver Cherednik and Hecke algebras of this paper (over $\mathbb Q$ or $\mathbb  C$).  
  This    generic  semisimplicity  allows us to  
 understand  the many  ``charged'' families of Specht modules as   specialisations 
 of a single family of 
 semisimple modules (via many {\em different}  integral forms on these modules).    
In more detail, we  now recall Webster's  definition of an algebra    $\Balgebra^{\Bbbk}_n(\sigma,(q;Q_0,Q_1,\dots, Q_{\ell-1}))$ associated to
  $n\in \mathbb{N}$ and parameters $ q$ and $Q_0,Q_1,\dots, Q_{\ell-1}$.  
When we specialise  $q=\xi$ to a primitive $e$th root of unity and $Q_m= q^{\sigma_m}$ for $0\leq m <\ell$ we will see that 
$
\algebra^{\mathbb Q}_n({\sigma})\cong 
\Balgebra^{\mathbb Q}_n(\sigma,(\xi;\xi^{\sigma _0},\xi^{\sigma _1},\dots, \xi^{\sigma _{\ell-1}} ))$.

\begin{figure}[ht]\captionsetup{width=0.9\textwidth}
\!\!\!\!\!\!\!\!\!\!\!
$$\scalefont{0.8}
\begin{tikzpicture}[scale=1]
\clip(-2.5,-4.5) rectangle (3.5,4);
 \draw[wei3](-2,-4)--(4,-4);
 \draw (0,-4.25) node {0};
  \path(0,-4.25)--++(45:0.5)--++(-45:0.5) coordinate (hello) node {1};
  \path(0,-4.25)--++(45:-0.5)--++(-45:-0.5) coordinate (hello2) node {$-1$};
  \path(hello2)--++(45:-0.5)--++(-45:-0.5) coordinate (hello2) node {$-2$};  
    \path(hello)--++(45:0.5)--++(-45:0.5) coordinate (hello) node {2};
        \path(hello)--++(45:0.5)--++(-45:0.5) coordinate (hello) node {3};
            \path(hello)--++(45:0.5)--++(-45:0.5) coordinate (hello) node {4};
 
\begin{scope}{   \path(0,0) coordinate (origin);
 \clip(origin)--(45:8)--(135:8)--(-135:8)--(origin);
  \clip(origin)--++(45:4)--++(135:1)--++(-135:3)--++(135:2)--++(-135:1)--(origin);

  \foreach \i\j in {1,...,19}
  {
    \path (origin)++(45:1*\i cm)  coordinate (a\i);
    \path (origin)++(135:1*\i cm)  coordinate (b\i);
    \path (a\i)++(135:1*\j cm) coordinate (ca\i,\j);
    \path (b\i)++(45:1*\j cm) coordinate (cb\i,\j);

  }
    \foreach \i in {1,...,19}
{  \draw[darkgreen,thick] (a\i)--++(135:8);
    \draw[darkgreen,thick] (b\i)--++(45:8);}
      \draw[darkgreen,thick] (origin)--++(135:8);
            \draw[darkgreen,thick] (origin)--++(45:8);

}
\end{scope}
\draw[wei](origin)--++(-90:4);

\path (origin)--++(45:0.5)--++(135:0.5) coordinate (origin);

  \foreach \i in {1,...,4}
  {
     \path (origin)++(45:1*\i cm)  coordinate (xx\i);
    \path (origin)++(135:1*\i cm)  coordinate (xy\i);

  }
\draw[wei3](xx1) node {$q  Q_0$};
\draw[wei3](xx2) node {$q^{2} Q_0$};
\draw[wei3](xx3) node {$q^{3} Q_0$};
\draw[wei3](xy1) node {$q^{-1} Q_0$};  
\draw[wei3](xy2) node {$q^{-2} Q_0$};  
\path(origin)--++(45:0.5)--++(135:0.5) coordinate (xc3);
\draw[wei3](origin) node {$  Q_0$};  
\path(xc3)--++(45:0.5)--++(135:0.5) coordinate (xcd);


\path(0,0)--++(45:0.25)--++(-45:0.25)--++(-90:3.2) coordinate (origin);
\draw[wei](origin)--++(-90:0.8);
\begin{scope}
{   
\draw[darkgreen,thick](origin)--++(45:3)--++(135:1)--++(-135:1)--++(135:1)--++(-135:1)--++(135:1)--++(-135:1)--(origin);
 \clip(origin)--++(45:8)--++(135:8)--++(-135:8)--(origin);
 \clip(origin)--++(45:3)--++(135:1)--++(-135:1)--++(135:1)--++(-135:1)--++(135:1)--++(-135:1)--(origin);
   \foreach \i\j in {1,...,19}
  {
    \path (origin)++(45:1*\i cm)  coordinate (a\i);
    \path (origin)++(135:1*\i cm)  coordinate (b\i);
    \path (a\i)++(135:1*\j cm) coordinate (ca\i,\j);
    \path (b\i)++(45:1*\j cm) coordinate (cb\i,\j);

  }
    \foreach \i in {1,...,19}
{  \draw[darkgreen,thick] (a\i)--++(135:8);
    \draw[darkgreen,thick] (b\i)--++(45:8);}
      \draw[darkgreen,thick] (origin)--++(135:8);
            \draw[darkgreen,thick] (origin)--++(45:8);
}

  \clip(origin)--++(45:8)--++(135:8)--++(-135:8)--(origin);

\path (origin)--++(45:0.5)--++(135:0.5) coordinate (origin);

  \foreach \i in {1,...,4}
  {
     \path (origin)++(45:1*\i cm)  coordinate (x\i);
    \path (origin)++(135:1*\i cm)  coordinate (y\i);

  }
\draw[wei3](x1) node {$ qQ_1$};
\draw[wei3](x2) node {$ q^{2}Q_1$};
\draw[wei3](y1) node {$ q^{-1}Q_1$};  
\draw[wei3](y2) node {$ q^{-2}Q_1$};  
\path(origin)--++(45:0.5)--++(135:0.5) coordinate (c3);
\draw[wei3](origin) node {$   Q_1$};  
\path(c3)--++(45:0.5)--++(135:0.5) coordinate (cd);
\draw[wei3](cd) node {$  Q_1$};

\end{scope}
\end{tikzpicture} 
\qquad
\begin{tikzpicture}[scale=1.]
\clip(-2.5,-4.5) rectangle (4.7,4);
 \draw[wei3](-2,-4)--(4,-4);
 \draw (0,-4.25) node {0};
  \path(0,-4.25)--++(45:0.5)--++(-45:0.5) coordinate (hello) node {1};
  \path(0,-4.25)--++(45:-0.5)--++(-45:-0.5) coordinate (hello2) node {$-1$};
  \path(hello2)--++(45:-0.5)--++(-45:-0.5) coordinate (hello2) node {$-2$};  
    \path(hello)--++(45:0.5)--++(-45:0.5) coordinate (hello) node {2};
        \path(hello)--++(45:0.5)--++(-45:0.5) coordinate (hello) node {3};
            \path(hello)--++(45:0.5)--++(-45:0.5) coordinate (hello) node {4};
            \path(hello)--++(45:0.5)--++(-45:0.5) coordinate (hello) node {5};
 
\begin{scope}{   \path(0,0) coordinate (origin);
 \clip(origin)--(45:8)--(135:8)--(-135:8)--(origin);
  \clip(origin)--++(45:4)--++(135:1)--++(-135:3)--++(135:2)--++(-135:1)--(origin);

  \foreach \i\j in {1,...,19}
  {
    \path (origin)++(45:1*\i cm)  coordinate (a\i);
    \path (origin)++(135:1*\i cm)  coordinate (b\i);
    \path (a\i)++(135:1*\j cm) coordinate (ca\i,\j);
    \path (b\i)++(45:1*\j cm) coordinate (cb\i,\j);

  }
    \foreach \i in {1,...,19}
{  \draw[darkgreen,thick] (a\i)--++(135:8);
    \draw[darkgreen,thick] (b\i)--++(45:8);}
      \draw[darkgreen,thick] (origin)--++(135:8);
            \draw[darkgreen,thick] (origin)--++(45:8);

}
\end{scope}
\draw[wei](origin)--++(-90:4);

\path (origin)--++(45:0.5)--++(135:0.5) coordinate (origin);

  \foreach \i in {1,...,4}
  {
     \path (origin)++(45:1*\i cm)  coordinate (xx\i);
    \path (origin)++(135:1*\i cm)  coordinate (xy\i);

  }
\draw[wei3](xx1) node {$q  Q_0$};
\draw[wei3](xx2) node {$q^{2} Q_0$};
\draw[wei3](xx3) node {$q^{3} Q_0$};
\draw[wei3](xy1) node {$q^{-1} Q_0$};  
\draw[wei3](xy2) node {$q^{-2} Q_0$};  
\path(origin)--++(45:0.5)--++(135:0.5) coordinate (xc3);
\draw[wei3](origin) node {$  Q_0$};  
\path(xc3)--++(45:0.5)--++(135:0.5) coordinate (xcd);


\path(0,0)--++(45:1.75)--++(-45:1.75)--++(-90:3.2) coordinate (origin);
\draw[wei](origin)--++(-90:0.8);
\begin{scope}
{   
\draw[darkgreen,thick](origin)--++(45:3)--++(135:1)--++(-135:1)--++(135:1)--++(-135:1)--++(135:1)--++(-135:1)--(origin);
 \clip(origin)--++(45:8)--++(135:8)--++(-135:8)--(origin);
 \clip(origin)--++(45:3)--++(135:1)--++(-135:1)--++(135:1)--++(-135:1)--++(135:1)--++(-135:1)--(origin);
   \foreach \i\j in {1,...,19}
  {
    \path (origin)++(45:1*\i cm)  coordinate (a\i);
    \path (origin)++(135:1*\i cm)  coordinate (b\i);
    \path (a\i)++(135:1*\j cm) coordinate (ca\i,\j);
    \path (b\i)++(45:1*\j cm) coordinate (cb\i,\j);

  }
    \foreach \i in {1,...,19}
{  \draw[darkgreen,thick] (a\i)--++(135:8);
    \draw[darkgreen,thick] (b\i)--++(45:8);}
      \draw[darkgreen,thick] (origin)--++(135:8);
            \draw[darkgreen,thick] (origin)--++(45:8);
}

  \clip(origin)--++(45:8)--++(135:8)--++(-135:8)--(origin);

\path (origin)--++(45:0.5)--++(135:0.5) coordinate (origin);

  \foreach \i in {1,...,4}
  {
     \path (origin)++(45:1*\i cm)  coordinate (x\i);
    \path (origin)++(135:1*\i cm)  coordinate (y\i);

  }
\draw[wei3](x1) node {$ qQ_1$};
\draw[wei3](x2) node {$ q^{2}Q_1$};
\draw[wei3](y1) node {$ q^{-1}Q_1$};  
\draw[wei3](y2) node {$ q^{-2}Q_1$};  
\path(origin)--++(45:0.5)--++(135:0.5) coordinate (c3);
\draw[wei3](origin) node {$   Q_1$};  
\path(c3)--++(45:0.5)--++(135:0.5) coordinate (cd);
\draw[wei3](cd) node {$  Q_1$};

\end{scope}
\end{tikzpicture} 
$$

\!\!\!\caption{We picture the $(q,Q)$-contents of the  2-partition $((4,1^2) \mid (3,2,1))$ 
 for $\sigma=(0,1)$ and $(0,4)$ respectively.  In each box we have placed the $(q,Q)$-content of the box.  
}\label{ALoading3333345}
\end{figure}

  \begin{figure}[ht!]
 \[
 \begin{tikzpicture}[baseline, thick,yscale=-0.9,xscale=-3.3]
\scalefont{0.8}

  \draw[wei]  (-1.1+0.5, -2)  to[out=90,in=-90] (-1.1+0.5, 2)  [below, at end] node{$\; Q_0$};
  
 \draw[darkgreen](-1,-2) rectangle (1.6,2);

\draw[gray,  densely dotted, very thick, darkgreen!80]  (-.5-0.4,-2)    to[out=60,in=-90](1-0.4,.2) to[out=90,in=-90] 
   (0-0.4,2);

   \draw[  densely dotted, very thick, darkgreen!80] (.5-0.4,-2) to[out=93,in=-93]   (.5-0.4,0) to[out=87,in=-90]  
  (1-0.4,2);

  \draw[  densely dotted, very thick, darkgreen!80]  (-0.4+1,-2) to[out=90,in=-90]  (-0.5-0.4,1) to[out=90,in=-90] (-0.4+.5,2);

  \draw[   densely dotted, very thick, darkgreen!80] (-0.4+0,-2)to     (-0.4,0) to[out=90,in=-90]  node[midway,circle,fill=darkgreen!55, densely dotted, very thick,  inner sep=2pt]{}
  (-0.4+-.5,2);

  \draw[wei3] (-.5,-2)    to[out=60,in=-90](1,.2) to[out=90,in=-90] 
     (0,2)  ; 
    \draw[wei3]    (0,2)  [below ] node{$j_2$}; 
 
  \draw[wei3] (.5,-2) to[out=93,in=-93]   (.5,0) to[out=87,in=-90]  
    (1,2)  [below, at end] node{$j_4$};
 
  \draw[wei3]  (1,-2) to[out=90,in=-90]    (-0.5,1) to[out=90,in=-90]
     (.5,2)  [below, at end] node{$j_3$};
  \draw[wei3] (0,-2) to[out=90,in=-90]     (-0.05,0)  to[out=90,in=-90]  
   node[midway,circle,fill=darkgreen,inner sep=2pt]{}  (-.5,2)
   [below, at end] node{$j_1$};
   
  \draw[wei3]  (1.5,-2) to[out=90,in=-90]    (1.15,0.2) to[out=90,in=-90]
    (1.5,2) [below, at end] node{$j_5$};

  \draw[ densely dotted, very thick, darkgreen!80]  (1.5-0.4,-2) to[out=90,in=-90]    (1.15-0.4,0.2) to[out=90,in=-90]
    (1.5-0.4,2);

\end{tikzpicture} \]

\!\!\!\!\!\caption{
A $\theta$-diagram, $B\in \Balgebra^\Bbbk_5(0,(q;Q_0))$,   with northern and southern loading
  ${\bf I}^{\theta}_\omega$ for $\omega= (1^5)$.   Here $j_k=q^{i_k}Q_0 $ for some $i_k \in \ZZ$  and $1\leq k \leq 5$.    }
\label{initialdiagram2}
\end{figure}

\subsection{Algebra definition}
We now define the algebra of interest.  Our definition  is slightly reverse-engineered in order to make it easier to understand the isomorphism (see \cref{whyican laebel}).  
We first require a ``generic" versions of   definitions of the residues and contents from \cref{sec1}.    We assume all the notation and definitions of \cref{sec1}.  
We define the $(q,Q)$-{\sf content of a box} as follows, 
$$
{\sf ct}_{q,Q}(r,c,m)= q^{c-r}Q_m
$$ 
and we set $\ct_{q,Q}(\la)=\sum_{(r,c,m)\in\la}{\sf ct}_{q,Q}(r,c,m)$.  
Upon specialisation of $q=\xi$   and  $Q_m = \xi^{\sigma_m}$ for $0\leq m <\ell$,   we have 
$$
 {\sf ct}_{q,Q}(r,c,m)|_{q=\xi,Q_m=\xi^m } 
:= \xi^{c-r}\xi^{\sigma_m} = \xi^{{\sf res}(r,c,m)}.
$$

\begin{defn}
Given   a    $\theta$-{\sf diagram} of type $(\mu,\la)$, we define a 
corresponding {\sf  degraded}   
  $\theta$-{\sf diagram} to be any diagram obtained by relabelling (and recolouring) as follows.  
  We recolour  each solid strand as a 
   green double-edged line 
and  replace the residue   of this solid strand with some $(q,Q)$-{\sf content} $q^i Q_m$ for $i\in \ZZ$ and  $0\leq m <\ell$; we  relabel  the residue of the $\sigma_m$ red strand with $Q_m$.  
  \end{defn}

  \!\!
 \begin{defn}  \label{defintino2}
 The     associative   $ \Bbbk$-algebra,  $\Balgebra^{ \Bbbk}_n( \sigma,(q;Q_0,Q_1,\dots, Q_{\ell-1})  )$, is  generated (as a $ { \Bbbk}$-module) by   all inequivalent    degraded $\theta$-diagrams  modulo the   local relations  \ref{Brel1} to \ref{Brel15} below.
    The product $b_1 b_2$ of two diagrams $b_1,b_2 \in \Balgebra^{ \Bbbk}_n( \sigma,(q;Q_0,Q_1,\dots, Q_{\ell-1})  )$ is then given by putting $b_1$ on top of $b_2$.
This product is defined to be $0$ unless the southern border of $b_1$ is given by the same loading as the northern border of $b_2$ with  $(q;Q)$-contents matching in the obvious manner,
 in which case we obtain a new diagram with loading   inherited from those of $b_1$ and $b_2$.  
\end{defn}
    
      \begin{enumerate}[label=(B\arabic*),leftmargin=*] 
\item\label{Brel1}  Any diagram may be deformed isotopically; that is,
 by a continuous deformation
 of the diagram which  
 avoids  tangencies, double points and dots on crossings. 
\item\label{Brel2} 
Any solid dot     can pass through  an arbitrary crossing involving a  ghost strand.  Namely: 
\[   \scalefont{0.8} \begin{tikzpicture}[yscale=0.45,xscale=-0.45,baseline]
  \draw[darkgreen, densely dotted, very thick,very thick](-4,0) +(-1,-1) -- +(1,1) node [at start,below] {$j$} ;
   \draw[wei3] (-4,0) +(1,-1) -- +(-1,1)  node [at start,below] {$i$} ; 
    \fill[darkgreen]  (-4.5,.5) circle  (5pt);
    \node at (-2,0){$=$}; 
    \draw[darkgreen, densely dotted, very thick,very thick](0,0) +(-1,-1) -- +(1,1)
node [at start,below] {$j$} ; 
   \draw[wei3] (0,0) +(1,-1) --
  +(-1,1)node [at start,below] {$i$} ; 
 ;
    \fill[darkgreen]   (.5,-.5) circle (5pt);
  \node at (4,0){ };
\end{tikzpicture} \quad
\begin{tikzpicture}[yscale=-0.45,xscale=-0.45,baseline]
  \draw[darkgreen, densely dotted, very thick,very thick](-4,0) +(-1,-1) -- +(1,1) node [at end,below] {$i$} ; 
; \draw[wei3] (-4,0) +(1,-1) -- +(-1,1) node [at end,below] {$j$} ; 
;\fill [darkgreen] (-4.5,.5) circle  (5pt);
    \node at (-2,0){$=$}; \draw[darkgreen, densely dotted, very thick,very thick](0,0) +(-1,-1) -- +(1,1)
node [at end,below] {$i$} ; 
 ; \draw[wei3] (0,0) +(1,-1) --
  +(-1,1) node [at end,below] {$j$} ; 
; \fill [darkgreen]   (.5,-.5) circle (5pt);
  \node at (4,0){ };
\end{tikzpicture}\]
for   $i,j$ any $(q,Q)$-contents and their  images through reflection in the vertical axis hold.  
 \item\label{Brel3}  We can pass a solid  dot through a
  crossing   at the expense of  an error term:
\[
\scalefont{0.8}\begin{tikzpicture}[yscale=0.45,xscale=0.45,baseline]
  \draw[wei3](-4,0) +(-1,-1) -- +(1,1) node [at start,below] {$i$} ; 
; \draw[wei3](-4,0) +(1,-1) -- +(-1,1)node [at start,below] {$j$} ; 
 ;\fill [darkgreen] (-3.5,.5) circle (5pt);
   \node at (-2,0){$=$}; 
   \draw[wei3](0,0) +(-1,-1) -- +(1,1)node [at start,below] {$i$} ; 
 ; \draw[wei3](0,0) +(1,-1) --
  +(-1,1)node [at start,below] {$j$} ; 
; \fill [darkgreen] (-.5,-.5) circle (5pt);
  \node at (2,0){$+$};
   \draw[wei3](4,0) +(-1,-1) -- +(-1,1)
node [at start,below] {$i$} ; 
 ;
   \draw[wei3](4,0) +(0,-1) --
  +(0,1) node [at start,below] {$j$} ; 
;
\end{tikzpicture}  \quad \quad \quad
\scalefont{0.8}\begin{tikzpicture}[yscale=0.45,xscale=0.45,baseline]
  \draw[wei3](-4,0) +(-1,-1) -- +(1,1) node [at start,below] {$i$} ; 
   \draw[wei3](-4,0) +(1,-1) -- +(-1,1) node [at start,below] {$j$} ; 
   \fill [darkgreen] (-3.5,-.5) circle (5pt);
       \node at (-2,0){$=$}; 
       \draw[wei3](0,0) +(-1,-1) -- +(1,1)
node [at start,below] {$i$} ;  
 \draw[wei3](0,0) +(1,-1) --
  +(-1,1)node [at start,below] {$j$} ;  
  \fill [darkgreen] (-.5,.5) circle (5pt);
  \node at (2,0){$+$}; 
  \draw[wei3](4,0) +(-1,-1) -- +(-1,1)node [at start,below] {$i$} ; 
 ; \draw[wei3](4,0) +(0,-1) --
  +(0,1)node [at start,below] {$j$} ; 
\end{tikzpicture}\]
for   $i,j$ any $(q,Q)$-contents.  Ghost dots can pass through any crossing of strands  freely. 
\end{enumerate} 
 \begin{enumerate}[resume, label=(B\arabic*),leftmargin=*]  
\item\label{Brel4} For $i,j$ any $(q,Q)$-contents, 
a double-crossings of solid strands is zero, that is   
\[
\scalefont{0.8}\begin{tikzpicture}[ xscale=-0.45,yscale=0.45,baseline ]
\draw[wei3](-2.8,-1.25) .. controls (-1.2,0) ..  +(0,2.5)
node [at start,below] {$j$} ; 
\draw[wei3](-1.2,-1.25) .. controls (-2.8,0) ..  +(0,2.5) node [at start,below] {$i$} ; 
\node at (-.5,0) {$=$};
\node at (0.4,0) {$0$};
\end{tikzpicture} 
\] 
\end{enumerate}
 \begin{enumerate}[resume, label=(B\arabic*),leftmargin=*]  
\item\label{Brel6}\label{Brel7} 
For       $i,j$ any $(q,Q)$-contents,  a  double-crossing of ghost and solid strands can be undone as follows:  
\[\scalefont{0.8}
 \begin{tikzpicture}[ xscale=0.45,yscale=0.45,baseline, ]

\draw[ densely dotted, very thick,darkgreen!80,  ] (-2.8-0.3,-1-0.25) .. controls (-1.2-0.3,0) ..  +(0,2+0.5)
node [at start,below] {$i$} ; 
 \draw[wei3]  (-1.2-0.3,-1-0.25) .. controls (-2.8-0.3,0) ..  +(0,2+0.5)
node [at start,below] {$j$} ; 
\node at (-.4,0) {$=$};

\draw[ densely dotted, very thick,darkgreen!80,  ] (0.9,-1-0.25) .. controls (1.4,0) ..  +(0,2+0.5)
node [at start,below] {$i$} ; ;
 \draw[wei3]  (2.5,-1-0.25) .. controls (2.1,0) ..  +(0,2+0.5)
   node[midway,fill=darkgreen,inner sep=2.5pt,circle]{} node [at start,below] {$j$} ;

\node at (4.4,0) {\scalefont{1.1}$- \quad q$}; 

\draw[ densely dotted, very thick,darkgreen!80,  ] (0.9+4.5,-1-0.25) .. controls (1.4+4.5,0) ..  +(0,2+0.5)
 node[midway,fill=darkgreen!45,inner sep=2.5pt,circle]{} node [at start,below] {$i$} ; ; 
 \draw[wei3]  (2.5+4.5,-1-0.25) .. controls (2.1+4.5,0) ..  +(0,2+0.5)
node [at start,below] {$j$} ; 
 
\end{tikzpicture}
 \qquad\qquad\qquad \begin{tikzpicture}[ xscale=0.5,yscale=0.5,baseline, ]
 \draw[wei3]  (-2.8-0.3,-1-0.25) .. controls (-1.2-0.3,0) ..  +(0,2.5)
 node [at start,below] {$i$} ; ;
\draw[ densely dotted, very thick,darkgreen!80,  ]  (-1.2-0.3,-1-0.25) .. controls (-2.8-0.3,0) ..  +(0,2.5)
node [at start,below] {$j$} ; 
\node at (-.4,0) {$=$}; 

  \draw[wei3]  (0.9,-1-0.25) .. controls (1.4,0) ..  +(0,2.5)
  node[midway,fill=darkgreen,inner sep=2.5pt,circle]{}node [at start,below] {$i$} ; 

\draw[ densely dotted, very thick,darkgreen!80,  ] (2.5,-1-0.25) .. controls (2.1,0) ..  +(0,2.5)
node [at start,below] {$j$} ;

\node at (4.4,0) {\scalefont{1.1}$- \quad q$};

 \draw[wei3]  (0.9+4.5,-1-0.25) .. controls (1.4+4.5,0) ..  +(0,2.5)
  node [at start,below] {$i$} ; 

\draw[ densely dotted, very thick,darkgreen!80,  ]  (2.5+4.5,-1-0.25) .. controls (2.1+4.5,0) ..  +(0,2.5)
  node[midway,fill=darkgreen!45,inner sep=2.5pt,circle]{} node [at start,below] {$j$} ;

\end{tikzpicture}
\] 
 \item\label{Brel9}\label{Brel10}    We can pull a solid strand through a  ghost-crossing (or a ghost strand through a solid-crossing) at the expense of an error term: for   $i,j,k$ any $(q,Q)$-contents we have
  \[
\scalefont{0.8}\begin{tikzpicture}[ xscale=1.3,yscale=0.525,baseline]
\draw[ densely dotted, very thick,darkgreen!80,  ]   (-2.6,-1) -- +(-.8,2)node [at start,below] {$k$} ;  
\draw[ densely dotted, very thick,darkgreen!80,  ]   (-3.4,-1) -- +(.8,2) 
node [at start,below] {$i$} ;   
 \draw[wei3] (-3,-1) .. controls (-2.5,0) ..  +(0,2)
node [at start,below] {$j$} ;  

\node at (-2.25,0) {$=$};

\draw[ densely dotted, very thick,darkgreen!80,  ]  (-1.5+.4,-1) -- +(+-.8,2)node [at start,below] {$k$} ;  
\draw[ densely dotted, very thick,darkgreen!80,  ]   (-1.5+-.4,-1) -- +(+.8,2)node [at start,below] {$i$} ;  
 \draw[wei3]  (-1.5+0,-1) .. controls (-1.5-.5,0) ..  +(0,2)
node [at start,below] {$j$} ;  

\node at (-.7,0) {$- \quad q$};

 \draw[ densely dotted, very thick,darkgreen!80,  ]  (-2.9+3.4,-1) -- +(0,2)node [at start,below] {$k$} ;  
\draw[ densely dotted, very thick,darkgreen!80,  ]   (-2.9+2.6,-1) -- +(0,2)node [at start,below] {$i$} ;   \draw[wei3] (-2.9+3,-1) -- +(0,2)  node [at start,below] {$j$} ;  
 \end{tikzpicture}
\qquad\qquad   \begin{tikzpicture}[ xscale=1.3,yscale=0.525,baseline]
 \draw[wei3]  (-2.6,-1) -- +(-.8,2)node [at start,below] {$k$} ;  
 \draw[wei3]  (-3.4,-1) -- +(.8,2)node [at start,below] {$i$} ;  
 \draw[ densely dotted, very thick,darkgreen!80,  ] (-3,-1) .. controls (-2.5,0) ..  +(0,2)
node [at start,below] {$j$} ;  

\node at (-2.25,0) {$=$};

 \draw[wei3] (-1.5+.4,-1) -- +(+-.8,2)node [at start,below] {$k$} ;  
 \draw[wei3]  (-1.5+-.4,-1) -- +(+.8,2) node [at start,below] {$i$} ;  
 \draw[ densely dotted, very thick,darkgreen!80,  ] (-1.5+0,-1) .. controls (-1.5-.5,0) ..  +(0,2)
node [at start,below] {$j$} ;  

\node at (-.7,0) {$+$};

 \draw[wei3] (-2.9+3.4,-1) -- +(0,2)node [at start,below] {$k$} ;  
 \draw[wei3]  (-2.9+2.6,-1) -- +(0,2) node [at start,below] {$i$} ;  
\draw[ densely dotted, very thick,darkgreen!80,very thick ] (-2.9+3,-1) -- +(0,2) node [at start,below] {$j$} ;  
\end{tikzpicture}
\] 
\end{enumerate}
    \begin{enumerate}[resume, label=(B\arabic*),leftmargin=*]  
\item \label{Brel8} All other triples of  solid and ghost strands satisfy the naive braid relation:
 \[\Yvcentermath1
\scalefont{0.8}\begin{tikzpicture}[ scale=0.5 ,baseline]
 \draw[wei3] (-2,-1) -- +(-2,2)  node [at start,below] {$k$} ;
 \draw[wei3] (-4,-1) -- +(2,2) node [at start,below] {$i$} ;
 \draw[wei3] (-3,-1) .. controls (-4,0) ..  +(0,2)
node [at start,below] {$j$} ;
\node at (-1,0) {$=$};
 \draw[wei3] (2,-1) -- +(-2,2)node [at start,below] {$k$} ;
 \draw[wei3] (0,-1) -- +(2,2)node [at start,below] {$i$} ;
 \draw[wei3] (1,-1) .. controls (2,0) ..  +(0,2)node [at start,below] {$j$} ;
\end{tikzpicture}\qquad
\quad
\begin{tikzpicture}[ scale=0.5 ,baseline]
\draw[ densely dotted, very thick,darkgreen!80,very thick] (-2,-1) -- +(-2,2)  node [at start,below] {$k$} ;
 \draw[wei3] (-4,-1) -- +(2,2) node [at start,below] {$i$} ;
 \draw[wei3] (-3,-1) .. controls (-4,0) ..  +(0,2)node [at start,below] {$j$} ;
\node at (-1,0) {$=$};
\draw[ densely dotted, very thick,darkgreen!80,very thick] (2,-1) -- +(-2,2)
node [at start,below] {$k$} ;
 \draw[wei3] (0,-1) -- +(2,2)node [at start,below] {$i$} ;
 \draw[wei3] (1,-1) .. controls (2,0) ..  +(0,2)node [at start,below] {$j$} ;
\end{tikzpicture}
\qquad
\quad
\begin{tikzpicture}[ scale=0.5 ,baseline]
 \draw[wei3] (-2,-1) -- +(-2,2)  node [at start,below] {$k$} ;
\draw[ densely dotted, very thick,darkgreen!80,very thick] (-4,-1) -- +(2,2) node [at start,below] {$i$} ;
\draw[ densely dotted, very thick,darkgreen!80,very thick] (-3,-1) .. controls (-4,0) ..  +(0,2)node [at start,below] {$j$} ;
\node at (-1,0) {$=$};
 \draw[wei3] (2,-1) -- +(-2,2)node [at start,below] {$k$} ;
\draw[ densely dotted, very thick,darkgreen!80,very thick] (0,-1) -- +(2,2)
node [at start,below] {$i$} ;
\draw[ densely dotted, very thick,darkgreen!80,very thick] (1,-1) .. controls (2,0) ..  +(0,2)node [at start,below] {$j$} ;
\end{tikzpicture}
\]
for any $i,j,k$ any $(q,Q)$-contents and their mirror images through reflection in the vertical axis hold.
  Performing the leftmost relation  implicitly involves manipulating  a braid  of three ghost strands at the same time (which we do not picture) and vice versa.  

 \item\label{Brel11}Double-crossings of solid and red strands can be undone at the expense of an error term 
\[
\scalefont{0.8}
\begin{tikzpicture}[ xscale=0.45,yscale=0.45  ]

\draw[  wei3,    ] (-2.8-0.3,-1-0.25) .. controls (-1.2-0.3,0) ..  +(0,2+0.5) node [at start,below] {$i$} ;
 \draw[wei]  (-1.2-0.3,-1-0.25) .. controls (-2.8-0.3,0) ..  +(0,2+0.5)
node[at start,below]{$Q_m$};  
\node at (-.4,0) {$=$};

\draw[   wei3  ] (0.9,-1-0.25) .. controls (1.4,0) ..  +(0,2+0.5)
node[midway,fill=darkgreen,inner sep=2.5pt,circle]{}
node [at start,below] {$i$} ;
 \draw[wei]  (2.5,-1-0.25) .. controls (2.1,0) ..  +(0,2+0.5)
node[at start,below]{$Q_m$};
    ;

\node at (4.4,0) {\scalefont{1.1}$- \!\!\!\quad Q_m  $}; 

 \draw[ wei3  ] (0.9+4.5,-1-0.25) .. controls (1.4+4.5,0) ..  +(0,2+0.5)
 node[midway,fill=darkgreen ,inner sep=2.5pt,circle] {} 
 node [at start,below] {$i$} ; 
 \draw[wei]  (2.5+4.5,-1-0.25) .. controls (2.1+4.5,0) ..  +(0,2+0.5)
node[at start,below]{$Q_m$};
  ;
 
\end{tikzpicture}  \]
for $0\leq m <\ell$ and $i $ any $(q,Q)$-content;  the  mirror image  through reflection through the vertical axis also holds.   Ghost strands and ghost dots may pass through red strands freely.   
 \Item\label{Brel12}Solid crossings and dots can pass through  red strands, with a
correction term,
\[
\scalefont{0.8}\begin{tikzpicture}[  scale=0.55,baseline]
\path(0,2.2)--(1,2.2);
\draw[wei] (-3,-1) .. controls (-2,0) ..  +(0,2)
node[at start,below]{$Q_m$};
 \draw[wei3] (-2,-1)  -- +(-2,2) node [at start,below] {$j$} ; 
 \draw[wei3] (-4,-1) -- +(2,2)  node [at start,below] {$i$} ; 
 
\node at (-1,0) {$=$};
\draw[wei] (1,-1) .. controls (0,0) .. +(0,2)
node[at start,below]{$Q_m$};

 \draw[wei3] (2,-1) -- +(-2,2) node [at start,below] {$j$} ; 
 \draw[wei3] (0,-1) -- +(2,2) node [at start,below] {$i$} ; 
 
\node at (2.8,0) {$+ $};
\draw[wei] (6.5-0.5-0.5,-1) -- +(0,2)
node[at start,below]{$Q_m$};

 \draw[wei3] (7.5-0.5-0.5,-1) -- +(0,2) node [at start,below] {$j$} ; 
 \draw[wei3] (5.5-0.5-0.5,-1) -- +(0,2) node [at start,below] {$i$} ;

\end{tikzpicture}
\]
for $0\leq m <\ell$ and   $i,j$ any $(q,Q)$-contents. 
 \item\label{Brel13}
 Any   braid involving a red strand and   not of the form in \ref{Brel12} can be undone without cost:  
\[
\scalefont{0.8}\begin{tikzpicture}[scale=0.5, baseline=2cm]
\draw[wei] (-2,2) -- +(-2,2)node[at start,below]{\scalefont{0.8} $Q_m$};
 \draw[wei3] (-3,2) .. controls (-4,3) ..  +(0,2)node[at start,below]{$j$};
 \draw[wei3] (-4,2) -- +(2,2) node[at start,below]{$i$};
\node at (-1,3) {$=$};

\draw[wei] (2,2) -- +(-2,2) node[at start,below]{\scalefont{0.8} $Q_m$};
 \draw[wei3] (1,2) .. controls (2,3) ..  +(0,2) node[at start,below]{  $j$};
 \draw[wei3] (0,2) -- +(2,2) node[at start,below]{  $i$};

\end{tikzpicture}
\quad\quad\quad  
\begin{tikzpicture}[scale=0.5, baseline=2cm]
\draw[wei] (-2,2) -- +(-2,2)node[at start,below]{\scalefont{0.8} $Q_m$};
\draw[ densely dotted, very thick,darkgreen!80,very thick] (-3,2) .. controls (-4,3) ..  +(0,2) node[at start,below]{  $j$};
 \draw[wei3] (-4,2) -- +(2,2)node[at start,below]{  $i$};

\node at (-1,3) {$=$};

\draw[wei] (2,2) -- +(-2,2)node[at start,below]{\scalefont{0.8} $Q_m$};
\draw[ densely dotted, very thick,darkgreen!80,very thick] (1,2) .. controls (2,3) ..  +(0,2)node[at start,below]{  $j$};
 \draw[wei3] (0,2) -- +(2,2)node[at start,below]{  $i$};

\end{tikzpicture}
\quad\quad\quad
\begin{tikzpicture}[scale=0.5, baseline=2cm]
\draw[wei] (-2,2) -- +(-2,2)node[at start,below]{\scalefont{0.8} $Q_m$};
\draw[ densely dotted, very thick,darkgreen!80,very thick] (-3,2) .. controls (-4,3) ..  +(0,2)node[at start,below]{  $j$};
\draw[ densely dotted, very thick,darkgreen!80,very thick] (-4,2) -- +(2,2)node[at start,below]{  $i$};

\node at (-1,3) {$=$};

\draw[wei] (2,2) -- +(-2,2)node[at start,below]{\scalefont{0.8} $Q_m$};
\draw[ densely dotted, very thick,darkgreen!80,very thick] (1,2) .. controls (2,3) ..  +(0,2)node[at start,below]{  $j$};
\draw[ densely dotted, very thick,darkgreen!80,very thick] (0,2) -- +(2,2)node[at start,below]{  $i$};
 
\end{tikzpicture}
\]

\!\!\!\[\scalefont{0.8}
\begin{tikzpicture}[scale=0.5, baseline=2cm]
 \draw[wei3]  (-2,2) -- +(-2,2) node[at start,below]{  $j$};;
\draw[wei]  (-3,2) .. controls (-4,3) ..  +(0,2)node[at start,below]{\scalefont{0.8} $Q_m$};
 \draw[wei3] (-4,2) -- +(2,2)node[at start,below]{  $i$};;

\node at (-1,3) {$=$};

 \draw[wei3]   (2,2) -- +(-2,2)node[at start,below]{  $j$};;
\draw[wei] (1,2) .. controls (2,3) ..  +(0,2)node[at start,below]{\scalefont{0.8} $Q_m$};
 \draw[wei3] (0,2) -- +(2,2)node[at start,below]{  $i$};;
 
\end{tikzpicture}
\quad\quad\quad
\begin{tikzpicture}[scale=0.5, baseline=2cm]
\draw[ densely dotted, very thick,darkgreen!80,very thick]  (-2,2) -- +(-2,2)node[at start,below]{  $j$};;
\draw[wei]  (-3,2) .. controls (-4,3) ..  +(0,2)node[at start,below]{\scalefont{0.8} $Q_m$};
\draw[ densely dotted, very thick,darkgreen!80,very thick] (-4,2) -- +(2,2)node[at start,below]{  $i$};;

\node at (-1,3) {$=$};

\draw[ densely dotted, very thick,darkgreen!80,very thick]   (2,2) -- +(-2,2)node[at start,below]{  $j$};;
\draw[wei] (1,2) .. controls (2,3) ..  +(0,2)node[at start,below]{\scalefont{0.8} $Q_m$};
\draw[ densely dotted, very thick,darkgreen!80,very thick] (0,2) -- +(2,2)node[at start,below]{  $i$};;
 
\end{tikzpicture}
\quad\quad\quad
\begin{tikzpicture}[scale=0.5, baseline=2cm]
 \draw[wei3]   (-2,2) -- +(-2,2)node[at start,below]{  $j$};;
\draw[wei]  (-3,2) .. controls (-4,3) ..  +(0,2)node[at start,below]{\scalefont{0.8} $Q_m$};
\draw[ densely dotted, very thick,darkgreen!80,very thick] (-4,2) -- +(2,2)node[at start,below]{  $i$};;

\node at (-1,3) {$=$};

 \draw[wei3]  (2,2) -- +(-2,2)node[at start,below]{  $j$};;
\draw[wei] (1,2) .. controls (2,3) ..  +(0,2)node[at start,below]{\scalefont{0.8} $Q_m$};
\draw[ densely dotted, very thick,darkgreen!80,very thick] (0,2) -- +(2,2)node[at start,below]{  $i$};;
 
\end{tikzpicture}
\]
for $0\leq m <\ell$ and   $i,j$ any $(q,Q)$-contents;    their   reflections through the vertical axis   hold.    \item
\label{Brel14} Finally, any solid or ghost dot can be pulled through a red strand without cost: 
\[
\scalefont{0.8}\begin{tikzpicture}[ baseline,scale=0.5]
\draw[wei](-3,0) +(1,-1) -- +(-1,1)node[at start,below]{\scalefont{0.8} $Q_m$};
\draw[wei3](-3,0) +(-1,-1) -- +(1,1)node[at start,below]{  $i$};;
\draw[wei](1,0) +(1,-1) -- +(-1,1)node[at start,below]{\scalefont{0.8} $Q_m$};

\fill [darkgreen]  (-3.5,-.5) circle (5pt);
\node at (-1,0) {$=$};
\draw[wei3](1,0) +(-1,-1) -- +(1,1)node[at start,below]{  $i$};;
 \fill [darkgreen] (1.5,.5) circle (5pt);

\draw[wei](1,0) +(3,-1) -- +(5,1)node[at start,below]{\scalefont{0.8} $Q_m$};
\draw[ densely dotted, very thick,darkgreen!80,very thick](1,0) +(5,-1) -- +(3,1)node[at start,below]{  $i$};;
\fill[darkgreen!45] (5.5,-.5) circle (5pt);
\node at (7,0) {$=$};
\draw[wei](5,0) +(3,-1) -- +(5,1)node[at start,below]{\scalefont{0.8} $Q_m$};
\draw[ densely dotted, very thick,darkgreen!80,very thick] (5,0) +(5,-1) -- +(3,1)node[at start,below]{  $i$};;
\fill [darkgreen] [darkgreen!45] (8.5,.5) circle (5pt);
\end{tikzpicture}
\]
for $0\leq m <\ell$ and   $i $ any $(q,Q)$-content;  their   reflections through the vertical axis also hold.
\end{enumerate}
   We note that the diagram 
  $\upepsilon_\la  $ whose solid points along the  northern and southern boundaries are given by 
  $\mathbf{I}^\theta_\la$ for $\la \in \mathscr{C}^\ell _n$ 
and with {\em no crossing strands} is an idempotent by construction.  We refer to any such  diagram as a {\sf degraded weight idempotent}.  
Finally, we have the following non-local idempotent relation.    
 \begin{enumerate}[resume, label=(B\arabic*),leftmargin=*]
\item\label{Brel15}
Any degraded weight idempotent,   $\upepsilon_\la $,   in which a solid strand is at least  $  n$ units to the  right  of the  rightmost red-strand is referred to as unsteady and set to be equal to zero.
\end{enumerate}

\begin{rmk}
Given   (degraded) weight idempotents 
 $\upepsilon_\la^\imath$ 
and 
 ${\sf 1}_\la^\imath$, we enumerate the solid
 (green or black, respectively)  strands $1,\dots, n$ 
  from right-to-left.  
  We  let $X_p \upepsilon_\la^\imath$  
  and $y_p{\sf 1}_\la^\imath$ denote the 
  diagrams obtained by adding a single dot on the $p$th solid strand (and a corresponding ghost dot on its ghost strand) and we let  
$e^{y_p}{\sf 1}_\la^\imath = \sum_{i\geq 0} \tfrac{1}{(i!)}y_p^ i {\sf 1}_\la^\imath $.  
We set $\upepsilon  ^\theta_\omega
= \sum_{\imath  } \upepsilon ^{\imath}_{\omega} $ where the sum is over all $(q,Q)$-content  sequences.   
\end{rmk}

  \begin{thm}[{\cite[Theorem 4.6]{MR3732238} and 
  \cite[Theorem 6.9]{Webster(b)} 
  }]
   We have an isomorphism of $\mathbb Q$-algebras  $
\zeta : \Balgebra^{\mathbb Q}_n(\sigma(\xi; \xi^{\sigma _0},\xi^{\sigma _1},\dots, \xi^{\sigma _{\ell-1}}) )\xrightarrow{} \algebra^{\mathbb Q}_n({\sigma})$.  This isomorphism is given by specifying what happens on every local region  of a diagram,  as follows.  We have that 
    \vspace{-0.4cm}$$
      \upepsilon_\la^\imath\mapsto {\sf 1}_\la^\imath 
 \quad  \quad 
 X_p  \upepsilon_\la^\imath
 \mapsto e^{y_p}{\sf 1}_\la^\imath 
\qquad\quad 
 \begin{minipage}{1.2cm} \begin{tikzpicture}[yscale=0.425,xscale=-0.425,baseline]
  \draw[densely dotted, darkgreen, very thick](-4,0) +(-1,-1) -- +(1,1) node [at start,below] { \scalefont{0.8} $\ \xi^{i_p}$} ; 
  \draw[densely dotted, very thick, darkgreen](-4,0) +(1,-1) -- +(-1,1) node [at start,below] { \scalefont{0.8} $\ \xi^{ i_{p+1}}$} 
  node [at end,above] { $ \color{white}\xi^{i_r}$} ;   
 \end{tikzpicture}\end{minipage}
 \;  \;  \;  \;  \;  \;  \mapsto   \;  \; 
  \begin{minipage}{1.2cm} \begin{tikzpicture}[yscale=0.425,xscale=-0.425,baseline]
  \draw[densely dotted,very thick](-4,0) +(-1,-1) -- +(1,1) node [at start,below] {$i_p$} ; \draw[densely dotted, very thick] (-4,0) +(1,-1) -- +(-1,1) node [at start,below] { $ i_{p+1}$} 
  node [at end,above] { $ \color{white}i_r$} ;   
 \end{tikzpicture}\end{minipage}
 \qquad\qquad 
  \begin{minipage}{1.2cm} \begin{tikzpicture}[yscale=0.425,xscale=-0.425,baseline]
  \draw[darkgreen, very thick ,densely dotted](-4,0) +(-1,-1) -- +(1,1) node [at start,below] { \scalefont{0.8} $\ \xi^{i_p}$} ;
   \draw[wei](-4,0) +(1,-1) -- +(-1,1) node [at start,below] {\scalefont{0.8}$ {\color{white}\xi^{i_p} \; }Q_m{\color{white}i_p}$} 
  node [at end,above] {\scalefont{0.8}$ \color{white}Q_m$} ;   
 \end{tikzpicture}\end{minipage}
 \;  \;  \;   \;  \;  \;   \;  \;  \;  \mapsto  \!\!\!
  \begin{minipage}{1.2cm} \begin{tikzpicture}[yscale=0.425,xscale=-0.425,baseline]
  \draw[very thick,densely dotted] (-4,0) +(-1,-1) -- +(1,1) node [at start,below] {$i_p$} ; \draw[wei](-4,0) +(1,-1) -- +(-1,1) node [at start,below] { $ { \color{white}i_p \; } s_m{\color{white}i_p}$} 
  node [at end,above] { $ \color{white}s_m$} ;   
 \end{tikzpicture}\end{minipage}$$
together with the    flips of the latter two diagrams  through the horizontal axis.  We have that 
  \vspace{-0.4cm}$$
 \begin{minipage}{1.2cm} \begin{tikzpicture}[yscale=0.425,xscale=-0.425,baseline]
  \draw[wei3](-4,0) +(-1,-1) -- +(1,1) node [at start,below] { \scalefont{0.8} $\  \xi^{i_p}$} ; \draw[wei](-4,0) +(1,-1) -- +(-1,1) node [at start,below] {\scalefont{0.8}$ {\color{white}i_p \; }Q_m{\color{white}\scalefont{0.8} \xi^{i_p}}$} 
  node [at end,above] {\scalefont{0.8}$ \color{white}Q_m$} ;   
 \end{tikzpicture}\end{minipage}
 \;  \;  \;  \;  \;  \;  \;  \;  \;  \mapsto  \!\!
  \begin{minipage}{1.2cm} \begin{tikzpicture}[yscale=0.425,xscale=-0.425,baseline]
  \draw[very thick] (-4,0) +(-1,-1) -- +(1,1) node [at start,below] {$i_p$} ; \draw[wei](-4,0) +(1,-1) -- +(-1,1) node [at start,below] { $ { \color{white}i_p \; } s_m{\color{white}i_p}$} 
  node [at end,above] { $ \color{white}s_m$} ;   
 \end{tikzpicture}\end{minipage}
  \qquad \qquad \qquad 
   \begin{minipage}{1.2cm} \begin{tikzpicture}[yscale=0.425,xscale=0.4,baseline]
  \draw[wei3](-4,0) +(-1,-1) -- +(1,1) node [at start,below] { \scalefont{0.8} $ \xi^{i_p}$} ; \draw[wei](-4,0) +(1,-1) -- +(-1,1) node [at start,below] {\scalefont{0.8}$ {\color{white}i_p \; }Q_m{\color{white}\scalefont{0.8} \xi^{i_p}}$} 
  node [at end,above] {\scalefont{0.8}$ \color{white}Q_m$} ;   
 \end{tikzpicture}\end{minipage}
 \;  \;  \; \mapsto  \; \; 
\frac{  (i_p  e^{y_p} - Q_m)}
{  y_p^{\delta_{i_p,Q_m}}}
\!\!\!\!  \begin{minipage}{1.2cm} \begin{tikzpicture}[yscale=0.425,xscale=0.4,baseline]
  \draw[very thick] (-4,0) +(-1,-1) -- +(1,1) node [at start,below] {$i_p$} ; \draw[wei](-4,0) +(1,-1) -- +(-1,1) node [at start,below] { $ { \color{white}i_p \; } s_m{\color{white}i_p}$} 
  node [at end,above] { $ \color{white}s_m$} ;   
 \end{tikzpicture}\end{minipage}
 $$
 \vspace{-0.2cm}
 $$ \quad \;\; \begin{minipage}{1.2cm} \begin{tikzpicture}[yscale=0.425,xscale=-0.425,baseline]
  \draw[wei3](-4,0) +(-1,-1) -- +(1,1) node [at start,below] {  \scalefont{0.8} $\;  \xi^{i_p}$} ; \draw[densely dotted, very thick, darkgreen](-4,0) +(1,-1) -- +(-1,1) node [at start,below] { \scalefont{0.8} $\;  \xi^{i_r}$} 
  node [at end,above] { $  \color{white}\scalefont{0.8} \xi^{i_p}$} ;   
 \end{tikzpicture}\end{minipage}
\; \;  \;  \;  \;  \mapsto 
  \begin{minipage}{1.2cm} \begin{tikzpicture}[yscale=0.425,xscale=-0.425,baseline]
  \draw[very thick](-4,0) +(-1,-1) -- +(1,1) node [at start,below] {$i_p$} ; \draw[densely dotted, very thick] (-4,0) +(1,-1) -- +(-1,1) node [at start,below] { $ i_r$} 
  node [at end,above] { $ \color{white}i_r$} ;   
 \end{tikzpicture}\end{minipage}
 \qquad \qquad\quad
  \begin{minipage}{1.2cm} \begin{tikzpicture}[yscale=0.425,xscale=0.425,baseline]
  \draw[wei3](-4,0) +(-1,-1) -- +(1,1) node [at start,below] {  \scalefont{0.8} $\;  \xi^{i_p}$} ;
   \draw[densely dotted, very thick, darkgreen](-4,0) +(1,-1) -- +(-1,1) node [at start,below] { \scalefont{0.8} $\;  \xi^{i_r}$} 
  node [at end,above] { $  \color{white}\scalefont{0.8} \xi^{i_p}$} ;   
 \end{tikzpicture}\end{minipage}
\; \;  \;  \;  \;  \mapsto 
\frac{ i_p e^{y_p} -q i_r e^{y_r} }   { 	(y_r - y_p)	^{\delta_ { qi_r, i_p		}		}	}   
  \begin{minipage}{1.2cm} \begin{tikzpicture}[yscale=0.425,xscale=0.425,baseline]
  \draw[black, very thick](-4,0) +(-1,-1) -- +(1,1) node [at start,below] {  \scalefont{0.8} $ {i_p}$} ;
   \draw[densely dotted, very thick ](-4,0) +(1,-1) -- +(-1,1) node [at start,below] { \scalefont{0.8} $ {i_r}$} 
  node [at end,above] { $  \color{white}\scalefont{0.8} \xi^{i_p}$} ;   
 \end{tikzpicture}\end{minipage}$$
 and finally, we have that 
    \vspace{-0.2cm}$$ 
     \begin{minipage}{1.2cm} \begin{tikzpicture}[yscale=0.425,xscale=-0.425,baseline]
  \draw[wei3](-4,0) +(-1,-1) -- +(1,1) node [at start,below] {\scalefont{0.8}$\xi^{i_p}$} ; 
  \draw[wei3](-4,0) +(1,-1) -- +(-1,1) node [at start,below] {\scalefont{0.8}$ \xi^{i_{p+1}}$} 
  node [at end,above] { $ \color{white}i_r$} ;   
 \end{tikzpicture}\end{minipage}
 \;  \;  \;   \mapsto  
  \begin{cases}
   \begin{minipage}{7cm} \begin{tikzpicture}[yscale=0.425,xscale=0.425,baseline]
  
  \clip(-12,-2.2) rectangle (-2+9,1.6);

  \draw[very thick](-4,0) +(-1,-1) -- +(1,1) node [at start,below] {\scalefont{0.9}$i_{p+1}$} ; \draw[ very thick] (-4,0) +(1,-1) -- +(-1,1) node [at start,below] {\scalefont{0.9} $ i_p$} 
  node [at end,above] { $ \color{white}i_r$} ;   
\draw[thick] (-5.6-0.1,-2.1) to [out=110,in=-110] (-5.6-0.1,1.4); 
 \draw(-2,0) node {$-$}; 
\draw(-9.1,-0.2) node {\scalefont{1.3}$ \frac{   
1	}
{ i_{p+1} 		e^{y_{p+1}}	-	 i_{p} 		e^{y_{p}}		}
$};

\draw(5,0) node {$i_p\neq i_{p+1} $};  
 \draw[ very thick](-5+4+0.2,-1) to [out=50,in=-50] (-5+4+0.2,1) 
;
  \draw[  very thick] (-3+4,-1) to [out=130,in=-130]
   (-3+4,1);
   \draw(-3+4,-1) node [below] { \scalefont{0.9}$ i_p$} ;
      \draw(-5+4+0.2,-1) node [below] { \scalefont{0.9}$\; i_{p+1}$} ;
      
\draw(-3+4,1) node [above] {\color{white} $ i_r$} ;

\draw[thick] (-3+.6+4,-2.1) to [out=70,in=-70] (-3+.6+4,1.4);

  \end{tikzpicture}\end{minipage}
  \\
   \begin{minipage}{1.2cm} \begin{tikzpicture}[yscale=0.425,xscale=0.425,baseline]
  \draw[very thick](-4,0) +(-1,-1) -- +(1,1) node [at start,below] {\scalefont{0.9}$i_{p+1}$} ; \draw[ very thick] (-4,0) +(1,-1) -- +(-1,1) node [at start,below] { \scalefont{0.9} $ i_{p }$} 
  node [at end,above] { $ \color{white}i_r$} ;

\draw(-9.1,-0.2) node {\scalefont{1.3}$ \frac{   
y_{p+1}-y_{p}	}
{ i_{p+1} 		e^{y_{p+1}}	-	 i_{p} 		e^{y_{p}}		}
$};

  \clip(-12,-2.2) rectangle (-2+9,1.6);  

\draw(5,0) node {$i_p= i_{p+1} $}; 

  \end{tikzpicture}\end{minipage}
  \end{cases}
 $$

  \end{thm}

\begin{rmk}  \label{whyican laebel}
Webster takes a slightly different approach the definition of this algebra.  He defines the algebra as above, but does  not attach $(q;Q)$-contents to the strands.  
Webster then observes the following.
Let   $M$ be a finite dimensional $\Balgebra^{\Bbbk}_n( \sigma,(q;Q_0,Q_1,\dots, Q_{\ell-1})  )$-module,  the eigenvalues of each $X_p \upepsilon_\la$ on $M$ are of the form 
  $q^iQ_m$ for some $i\in \ZZ$ and $0\leq m <\ell$.  
  So $M$ decomposes as a direct sum 
  of its  weight spaces 
  $$M_{\underline{i}}=\{ v\in M \mid (X_p-q^{i_p}Q^{m_p})^ N v =0 \text{ for all } p=1,\dots, n \text { and }N\gg0			\}$$
  Considering the weight-space decomposition of the regular module, one deduces that there is a system $\{
   \upepsilon_\la^\imath \mid \imath_p = 
  q^{i_p}Q_{m_p} \text{   for some $i_p\in \ZZ$ and $0\leq m_p <\ell$}, 1\leq p \leq n\}$.  With this notation in place, Webster  then decomposes the identity of his algebra as a sum of these idempotents --- thus obtaining our $(q;Q)$-content decorated diagrams.  
   In what follows, we will consider the  
  summation  over all possible the   $(q,Q)$-content  decorations on any given solid green strand; we denote the resulting 
   diagram  without decorations on solid strands.  
\end{rmk}

\begin{rmk}
The following proposition and theorem are due to Webster and we provide citations here.  However, we remark that,
 with only a few minor modifications, one can repeat all the arguments of  \cref{cellllllllllll}  almost  verbatim (one simply has to ``forget the residues" of solid and ghost strands on account of  relations \ref{Brel1} to \ref{Brel15} being residue-free).  Indeed, all the results and proofs  of \cref{cellllllllllll} were very heavily based on ideas from \cite{MR3732238} and \cite{bkw11}.   

 There is only one significant change to the analogues of the results from \cref{cellllllllllll}.  Namely, we must replace the statement $y_{(r,c,m)}{\sf 1} _\mu \in \algebra^{\rhd \mu}_n(\theta)$ in  \cref{aprop}
with 
$X_{(r,c,m)}\upepsilon _\mu \in \ct_{q,Q}(r,c,m)\upepsilon _\mu+ \Balgebra^{\rhd \mu}_n(\theta,((q;Q_0,Q_1,\dots, Q_{\ell-1})))$.  To see this, one should compare the extra   scalar  on the righthand-side of relation \ref{Brel11} versus its  analogue   
in relation \ref{rel11} and the scalar on the righthand-side of 
relation \ref{Brel7} versus its analogue, the pair of relations  \ref{rel5}, \ref{rel6}.  We revisit this idea in \cref{scakarrrr}, below.  
\end{rmk}

 \begin{prop}[{\cite[Proposition 5.7]{Webster(b)}}]\label{ppfpfpfpddddd}
 Let $\Bbbk$ be a    integral domain.  
  We have  an isomorphism of   $\Bbbk$-algebras  
 $$   \sigma_{q,Q} :   H_n^\Bbbk(q;Q_0,Q_1,\dots,Q_{\ell-1}) \to  \upepsilon ^\theta_\omega \Balgebra^\Bbbk_n( \sigma, ( q;Q_0,Q_1,\dots,Q_{\ell-1}) )  \upepsilon  ^\theta_\omega. $$
  \begin{align*}
 \sigma_{q,Q}  ( X_k ) &= 
\begin{minipage}{145mm} \scalefont{0.7}    \begin{tikzpicture}[xscale=-0.95,yscale=0.95] 
  \draw[thick,darkgreen] (-2,0) rectangle (12.6,2);
   \foreach \x in {-1.5,-0.5,2.5}  
     {\draw[wei] (\x,0)--(\x,2); }
     \node [wei,below] at (-1.5,-0.07)  {\tiny ${s}_0$};
          \node [wei,below] at (2.5,-0.07)  {\tiny ${s}_{\ell\!-\!1}$};
        \draw[fill,white]   (0.2,1.9) rectangle (1.75,2.1);  
     \draw[fill,white] (0.2,0.1) rectangle (1.75,-0.1);  
     \draw[thick,densely dotted,darkgreen]  (0.1,0.) --  (1.75,-0);       \draw[thick,densely dotted,darkgreen]  (0.1,2) --  (1.75,2);  
          \node [wei,below] at (-0.5,-0.07)  {\tiny ${s}_1$};
            \draw[wei3]  (2.9,0)  to  (2.9,2) ;          \draw[thick,densely dotted,darkgreen]   (2.3,0)  to  (2.3,2) ; 
                    \draw[wei3]  (3.8,0)  to  (3.8,2) ;          \draw[thick,densely dotted,darkgreen]   (3,0)  to  (3,2) ; 
                      \draw[thick,densely dotted,darkgreen]   (3.9,0)  to  (3.9,2) ; 
                        \draw[thick,densely dotted,darkgreen]   (3.9,0)  to  (3.9,2) ; 
                       \draw[fill,white]   (4.2,2.1) rectangle (-0.5+5.8,-0.1);  
     \draw[thick,densely dotted,darkgreen]  (4.18,0.) --  (-0.5+5.8,-0);         \draw[thick,densely dotted,darkgreen]  (-0.5+4.2,2) --  (-0.5+5.8,2);    
         \draw[wei3]  (-0.5-0.9+7,0)  to  (-0.5-0.9+7,2) ; 
                  \draw[thick,densely dotted,darkgreen]   (-0.4-0.9+7,0)  to  (-0.4-0.9+7,2) ; 
            \draw[wei3]  (-0.5+7.3,0)  to  (-0.5+7.3,2) ;                  
        \draw[wei3]  (-0.5+8.2,0)  to  (-0.5+8.2,2) ;             \draw[wei3]  (-0.5+9.1,2)  to  (-0.5+9.1,0) ;     
        \draw[thick,densely dotted,darkgreen]     (-0.5-0.8+8.2,0)  to  (-0.5-0.8+8.2,2) ;             \draw[thick,densely dotted,darkgreen]     (-0.5-0.8+9.13,2)  to  (-0.5-0.8+9.13,0) ;      
        \draw[wei3]  (-0.5+10,0)  to  (-0.5+10,2) ;         \draw[thick,densely dotted,darkgreen]   (-0.5+9.2,0)  to  (-0.5+9.2,2) ; 
         \draw[thick,densely dotted,darkgreen]   (-0.5+10.1,0)  to  (-0.5+10.1,2) ; 
         \draw[wei3]  (0.4+10,0)  to  (0.4+10,2) ;         \draw[thick,densely dotted,darkgreen]   (0.4+9.2,0)  to  (0.4+9.2,2) ; 
        \draw[thick,densely dotted,darkgreen]   (0.4+10.1,0)  to  (0.4+10.1,2) ; 
                           \draw[fill,white]   (10.8,2.1) rectangle (11.6,-0.1);  
      \draw[thick,densely dotted,darkgreen] (10.8,2) --  (11.6,2);         \draw[thick,densely dotted,darkgreen]  (10.8,0) --  (11.6,0);    
         \draw[wei3]  (12,0)  to  (12,2) ;        \draw[fill,darkgreen]   (-0.5+8.2,1)           circle (2pt); 
   \end{tikzpicture}\end{minipage}\\
  \sigma_{q,Q}  (T_k+1) & = 
\begin{minipage}{145mm} \scalefont{0.7}    \begin{tikzpicture}[xscale=-0.95,yscale=0.95] 
  \draw[thick,darkgreen] (-2,0) rectangle (12.6,2);
   \foreach \x in {-1.5,-0.5,2.5}  
     {\draw[wei] (\x,0)--(\x,2); }
     \node [wei,below] at (-1.5,-0.07)  {\tiny ${s}_0$};
          \node [wei,below] at (2.5,-0.07)  {\tiny ${s}_{\ell\!-\!1}$};
        \draw[fill,white]   (0.2,1.9) rectangle (1.75,2.1);  
     \draw[fill,white] (0.2,0.1) rectangle (1.75,-0.1);  
     \draw[thick,densely dotted,darkgreen]  (0.1,0.) --  (1.75,-0);       \draw[thick,densely dotted,darkgreen]  (0.1,2) --  (1.75,2);  
          \node [wei,below] at (-0.5,-0.07)  {\tiny ${s}_1$};
                 \draw[wei3]  (2.9,0)  to  (2.9,2) ;          \draw[thick,densely dotted,darkgreen]   (2.3,0)  to  (2.3,2) ; 
                    \draw[wei3]  (3.8,0)  to  (3.8,2) ;          \draw[thick,densely dotted,darkgreen]   (3,0)  to  (3,2) ; 
                      \draw[thick,densely dotted,darkgreen]   (3.9,0)  to  (3.9,2) ; 
                           \draw[thick,densely dotted,darkgreen]   (3.9,0)  to  (3.9,2) ; 
                       \draw[fill,white]   (4.2,2.1) rectangle (-0.5+5.8,-0.1);  
     \draw[thick,densely dotted,darkgreen]  (4.18,0.) --  (-0.5+5.8,-0);         \draw[thick,densely dotted,darkgreen]  (-0.5+4.2,2) --  (-0.5+5.8,2);    
          \draw[wei3]  (-0.5-0.9+7,0)  to  (-0.5-0.9+7,2) ; 
                  \draw[thick,densely dotted,darkgreen]   (-0.4-0.9+7,0)  to  (-0.4-0.9+7,2) ; 
            \draw[wei3]  (-0.5+7.3,0)  to  (-0.5+7.3,2) ;       
         \draw[wei3]  (-0.5+8.2,0)  to [out=90,in=-90] (-0.5+9.1,2) ;             \draw[wei3]  (-0.5+8.2,2)  to [out=-90,in=90] (-0.5+9.1,0) ;     
        \draw[thick,densely dotted,darkgreen]   (-0.5+8.2-0.8,0)  to [out=90,in=-90] (-0.5+9.1-0.8,2) ;             \draw[thick,densely dotted,darkgreen]   (-0.5+8.2-0.8,2)  to [out=-90,in=90] (-0.5+9.1-0.8,0) ;     
         \draw[wei3]  (-0.5+10,0)  to  (-0.5+10,2) ;         \draw[thick,densely dotted,darkgreen]   (-0.5+9.2,0)  to  (-0.5+9.2,2) ; 
         \draw[thick,densely dotted,darkgreen]   (-0.5+10.1,0)  to  (-0.5+10.1,2) ; 
        \draw[wei3]  (0.4+10,0)  to  (0.4+10,2) ;         \draw[thick,densely dotted,darkgreen]   (0.4+9.2,0)  to  (0.4+9.2,2) ; 
        \draw[thick,densely dotted,darkgreen]   (0.4+10.1,0)  to  (0.4+10.1,2) ; 
                           \draw[fill,white]   (10.8,2.1) rectangle (11.6,-0.1);  
      \draw[thick,densely dotted,darkgreen] (10.8,2) --  (11.6,2);    
           \draw[thick,densely dotted,darkgreen]  (10.8,0) --  (11.6,0);    
         \draw[wei3]  (12,0)  to  (12,2) ;    
   \end{tikzpicture}\end{minipage}
\end{align*}
where the dot is on the $k$th  solid strand (from the right) and the crossing strands are the $k$th and $(k+1)$st solid  strands from the right.

\end{prop}

We let $B_{\SSTS\SSTT}$ denote the degraded $\sigma$-diagram obtained from $A_{\SSTS\SSTT}$ by relabelling each $s_m$-strand with the content $Q_m$ and forgetting the residues of all other strands.  
In other words, we sum over all possible $(q,Q)$-contents on each green strand in $B_{\SSTS\SSTT}$.

 \begin{thm}{\cite[Theorem 2.24]{MR3732238}}\label{glkjhdfsglkjshdfglskdfhj}
Let $\Bbbk$ be an integral domain. 
 The $\Bbbk$-algebra  $\Balgebra^{\Bbbk}_n(
\sigma,(q;Q_0,Q_1,\dots, Q_{\ell-1}))$ 
   is free as an $\Bbbk$-module and  has a cellular    basis 
\[
  \{ B_{\SSTS  \SSTT} \mid \SSTS \in \SStd_{\theta}(\lambda,\mu), \SSTT\in \SStd_{\theta}(\lambda,\nu), 
 \lambda  \in \mptn {\ell}n,  \mu, \nu  \in \con {\ell}n\} 
\]
 with respect to the $\theta$-dominance order on $\mptn \ell n$ and the involution $\ast$ given by
 horizontal reflection. We let $\Delta^\Bbbk_{\sigma,q,Q}(\la)$ denote the corresponding cell-module for $\la \in \mptn \ell n$.
     \end{thm}

\begin{cor}\label{semisimpleeeeee2}
Specialise  $q=\xi$ and $Q_m=\xi^m$ for $0\leq m <\ell$ and $\Bbbk=\mathbb Q$. 
 The $\H_n^{\mathbb Q}(\underline{s} )$-modules   
$\Specht ^{\mathbb Q}_{\sigma }(\la)$  and  $
\zeta({\upepsilon}_\omega\Delta^{\mathbb Q}_{\sigma,q,Q}(\la))   $ are isomorphic.  
\end{cor}
\begin{proof}
We  use the notation from the proof of  \cref{cellularitybreedscontempt2} and we  let $\SSTS,\SSTT \in \SStd_{\theta}(\la,\omega)$.  
We let $\widehat{ A}_{\SSTS'\SSTT'}\in \H_n^{\mathbb Q}(\underline{s} )$ denote some diagram obtained from the graded cellular basis element ${ A}_{\SSTS'\SSTT'}\in \H_n^{\mathbb Q}(\underline{s} )$ by adding some number (possibly zero) of dots along the strands.  
By the definition of the map $\zeta $,
we have  that  $$\zeta (B_{\SSTS\SSTT}) = A_{\SSTS\SSTT}+ \sum_{\begin{subarray}c
\SSTU', \SSTT'\in \SStd_{\theta}(\la,\omega) \\
 \sharp(\SSTS',
 \SSTT')< 
\sharp(\SSTS ,
 \SSTT)
\end{subarray}}\widehat{ A}_{\SSTS'\SSTT'}
.$$    
Therefore  by  \cref{move a dot down,aprop}, we have that 
 $$\zeta (B_{\SSTS\SSTT}) \in A_{\SSTS\SSTT}+ \sum_{\begin{subarray}c
\SSTU', \SSTT'\in \SStd_{\theta}(\la,\omega) \\
 \sharp(\SSTS',
 \SSTT')< 
\sharp(\SSTS ,
 \SSTT)
\end{subarray}}  k_{\SSTS'\SSTT'}{ A}_{\SSTS'\SSTT'}+  {\sf E}_\omega \algebra^{\rhd \la}_n(\theta){\sf E}_\omega $$    for some $k_{\SSTS'\SSTT'} \in \Bbbk$.   
 Thus the bases $\{ \zeta (B_{\SSTS\SSTT}) \mid \SSTS,\SSTT \in \SStd_{\theta}(\lambda,\omega) ,\la\in \mptn \ell n  \}  $ and $\{A_{\SSTS\SSTT}  \mid \SSTS,\SSTT \in \SStd_{\theta}(\lambda,\omega),\la\in \mptn \ell n  \} $   differ by uni-triangular change of basis matrix and the result follows.  
 \end{proof}


\begin{prop}\label{scakarrrr}
For each $\la \in \mathscr{P}_\ell(n)$, the element 
 $
X_\la =  \textstyle \sum_{(r,c,m)\in \la}X_{(r,c,m)}  \upepsilon_\lambda
 $
is central within $\upepsilon_\la 
\Balgebra^\Bbbk_n( \sigma(q;Q_0,Q_1,\dots,Q_{\ell-1}) )
\upepsilon_\la $ and acts on  
$\upepsilon_\la \Delta^\Bbbk_{\sigma,q,Q}(\mu)$ as the scalar 
$\ct_{q,Q}(\mu)$.  
\end{prop}

\begin{proof}
Centrality follows immediately from the relations \ref{Brel1} to \ref{Brel15} and therefore $X_\la$ acts as a scalar on any module, it remains to calculate this scalar.  
 We have that
  $$X_{(r,c,m)} \upepsilon_\la=
  \begin{cases}
   q X_{(r-1,c,m)}\upepsilon_\la + \Balgebra^{\rhd \la}		& \text{ for }r>1 \\
 q^{-1} X_{(r,c-1,m)}\upepsilon_\la + \Balgebra^{\rhd \la}		&\text{ for }c>1\\
 Q_m X_{(r,c,m)}\upepsilon_\la + \Balgebra^{\rhd \la}		&\text{ for }r=c=1 
   \end{cases}
   $$
by the first   case  of  relation \ref{Brel7}, second  case  of relation  \ref{Brel7}, and relation \ref{Brel11} respectively.  
It follows that  $X_\la  \upepsilon_\lambda = \ct_{q,Q}(\la)   \upepsilon_\lambda+ \Balgebra^{\rhd \la}	$. 
 For  a cellular basis element  of   $\upepsilon_\la \Delta^\Bbbk_{\sigma,q,Q}(\mu)$, the result follows 
by induction on the Bruhat order and the analogue  of   \cref{move a dot down}.  
\end{proof}

 \noindent Bringing together \cref{ppfpfpfpddddd,glkjhdfsglkjshdfglskdfhj,scakarrrr} we immediately deduce the following: 

\begin{thm}\label{semisimpleeeeee}
  For $\la \in \mptn \ell n$, the  $H_n^{{\mathbb Q}}(q;Q_0,Q_1,\dots,Q_{\ell-1})   $-module   ${\upepsilon}_\omega\Delta^{\mathbb Q}_{\sigma,q,Q}(\la)   $
  is  isomorphic to the  irreducible  module   
   upon which the central element $X_1+\dots + X_n$ acts as the scalar $\ct_{q,Q}(\la)$.  

\end{thm}

\section{The many different graded decomposition matrices}\label{manydiffgrad}

In this section let $\Bbbk$ be an arbitrary field.  We now prove the first statement of Theorem~B: namely that the decomposition matrices of Hecke algebras are uni-triangular with respect to any of Lusztig's ${\bf a}_\sigma$-orderings.   
 By \cref{semisimpleeeeee2,semisimpleeeeee} the    
modules $ \Specht     _\theta^\C(\la)$ are obtained from the usual semisimple Specht modules    after  specialisation of $q=\xi$ and $Q_m=\xi^m$ for $0\leq m <\ell$.  
Thus by \cref{uniqueneess},  upon forgetting the grading, the  cellular  decomposition matrices coincide with the usual definition of a decomposition matrix coming from a modular system.  

   \begin{thm}\label{QUIVERTLdecompositionmatrix}
 Given a fixed $\theta \in \weight$,  the graded decomposition matrix of $\H_n^\Bbbk(\underline s) $ with respect to the $\theta$-cellular structure appears as a submatrix of the decomposition matrix of  $\algebra^\Bbbk_n(\theta )$   as follows, 
$$
\sum_{k\in\ZZ}[\Specht_{\sigma}^\Bbbk(\la):  \Simple_\sigma^\Bbbk (\mu) \langle k \rangle ]= \sum_{k\in\ZZ}[\Delta_{\sigma}^\Bbbk(\la): L^\Bbbk_\sigma(\mu) \langle k \rangle ]
$$
for $\la\in \mptn \ell n$, $\mu \in \Theta $.  Here $\Sigma^\ell_n \subseteq   \mptn \ell n$ is the  subset  for which
$ \Simple^\Bbbk_\sigma (\mu):={\sf E}^\sigma_\omega  L^\Bbbk_\sigma(\mu) \neq 0$; the set $\{  \Simple^\Bbbk_\sigma (\mu)\mid \mu \in \Sigma^\ell_n\}$ provides a complete set of non-isomorphic  irreducible  $\H_n^\Bbbk(\underline s) $-modules.  
This matrix is uni-triangular with respect to  the ordering $\rhd_\theta$  on $\mptn\ell n$.  
 
\end{thm}
\begin{proof}
The unitriangularity result is immediate   from  \cref{corollary} and standard results on cellular algebras recalled  explicitly in \cref{gradedcell}.  The equality is immediate from    \cite[(6.6b)Lemma]{MR2349209}.    
\end{proof}

%
%
 In \cite{MR1443748,MR3732238,RSVV,losev} it is shown that the decomposition matrix of $\algebra^\Bbbk_n(\theta)$ 
  is given by the Kashiwara--Lusztig canonical basis for an irreducible highest weight $U(\widehat{\mathfrak{sl}}_e)$-module
and the entries are given by certain Kazhdan--Lusztig polynomials; these can be computed using an 
algorithm due to Uglov \cite{Uglov}.   

One of the main advantages of our new  $\ZZ$-lattices is that they allow us to  define generalisations of James' adjustment matrices.    The theory of adjustment matrices gives us a way of factorising representation theoretic questions into two steps:
 firstly {\em specialise the   parameter} $\theta\in\weight$  and study the non-semisimple algebra $\H^{\mathbb{Q}}_n(\sigma)$;
 then {\em reduce modulo $p$}  by studying $\H^\Bbbk_n(\theta)=\H^\ZZ_n(\theta) \otimes _\ZZ \Bbbk$.
This allows us to  factorise the problem of understanding decomposition matrices  as follows,
\begin{align}\label{adjust}
[\Specht     ^\Bbbk_\theta(\la): \Simple _\theta^\Bbbk(\mu)]   &=\sum_\nu
[\Specht     _\theta^\C(\la): \Simple ^\C_\theta(\nu)] \times
[\Simple _\theta^\C(\nu)\otimes_\ZZ\Bbbk: \Simple ^\Bbbk_\theta(\mu)] .
 \end{align}
 On the right-hand side of the equality we have two matrices: the first is the {$\sigma$-decomposition matrix for $\H^{\mathbb Q}_n(\underline{s})$} 
and the second is what we refer to as the  {\sf generalised James' $\theta$-adjustment matrix} ${\bf Ad}^\Bbbk_{\theta}(t) $.    
We emphasise  that the definition of ${\bf Ad}^\Bbbk_{\theta}(t) $ only makes sense   because, by Theorem A,  we have  $\ZZ$-forms for the cell and  irreducible  modules which allow us to reduce modulo $p$ in \cref{adjust}.

\begin{eg}\label{countereg} 
   The action of the generators on the basis of $\Specht^\Bbbk _{(2;2,1)} ((1),(1))$ in  \cref{countereg2} is given as follows, 
$$
\psi_1, y_1, y_2, e(0,0), e(1,1) \mapsto \left(\begin{array}{cc}0 & 0 \\0 & 0\end{array}\right) \quad 
e(0,1)\mapsto \left(\begin{array}{cc}1 & 0 \\0 & 0\end{array}\right)\quad 
e(1,0)\mapsto \left(\begin{array}{cc}0 & 0 \\0 & 1\end{array}\right)\quad 
$$
and therefore this module is a direct sum of the  irreducible  modules $L^\Bbbk(0,1)$ and $L^\Bbbk(1,0)$.  Clearly the module   $\Specht^\Bbbk _{(2;2,1)} ((1),(1))$ is not cyclic.  
 The action of the generators on the basis of  $\Specht _{(2;4,1)} ((1),(1))$ in  \cref{countereg1} is given as follows, 
$$
y_1, y_2, e(0,0), e(1,1) \mapsto \left(\begin{array}{cc}0 & 0 \\0 & 0\end{array}\right) \quad 
e(0,1)\mapsto \left(\begin{array}{cc}1 & 0 \\0 & 0\end{array}\right)\quad 
e(1,0)\mapsto \left(\begin{array}{cc}0 & 0 \\0 & 1\end{array}\right)\quad 
\psi_1 \mapsto \left(\begin{array}{cc}0 & 0 \\1 & 0\end{array}\right)
$$
and therefore this module is a non-split extension of the  irreducible  modules $L^\Bbbk(0,1)$ and $L^\Bbbk(1,0)$.   
All  the cell-modules for the charge $ (2;2,1)$ are all indecomposable, whereas this is not the case for the charge $ (2;4,1)$. 
Hence,  there is no isomorphism relating the sets of cell modules  from these two distinct charges.  Notice that the two modules have the same composition factors, but not the same structure.  
 \end{eg}

 \begin{eg}\label{countereg3}
The graded decomposition matrices with respect to  these cellular bases are,  $$
{\bf  D}_{ (2;2,1)}(t)
 =
 \begin{array}{r|cc}
& \Simple^\Bbbk (1,0) &  \Simple^\Bbbk(0,1) \\ \hline 
 (\varnothing,(1^2)) & 1 & 0 \\
 ((1^2),\varnothing)& 0 & 1 \\
((1),(1))& t & t \\
(\varnothing,(2))& t^2 & 0 \\
 ((2),\varnothing) & 0 & t^2 
 \end{array} 
\qquad
 \Simple _{ (2;4,1)}(t)
 =
 \begin{array}{r|cc}
&  \Simple^\Bbbk(1,0) &  \Simple^\Bbbk(0,1) \\ \hline 
 (\varnothing,(1^2)) & 1 & 0 \\
(\varnothing,(2)) & t & 0 \\
((1),(1))& t^2  & 1 \\
 ((1^2),\varnothing) &  0 & t \\
 ((2),\varnothing) & 0 & t^2
 \end{array}.
 $$
We cannot obtain $ {\bf D} ^\Bbbk_{ (2;2,1)}(t)$ by   permuting  the rows of  $ {\bf D}  ^\Bbbk_{ (2;4,1)}(t)$.  
However,  letting $t=1$ we find that 
  $ {\bf D}  ^\Bbbk_{ (2;2,1)}(1)$ can be obtained from  by   permuting  the rows of  ${\bf D}  ^\Bbbk_{ (2;4,1)}(1)$.  
   \end{eg}

 \section{Uglov  combinatorics  and the many different \\ constructions of  irreducible  modules    } 
\label{isaididexplain}

\newcommand{\lhdc}{\lhd_s}
\newcommand{\rhdc}{\rhd_s}
\newcommand{\leqc}{\leq_s}
\newcommand{\geqc}{\geq_s}

In this section let $\Bbbk$ be an arbitrary field. 
In this section we  complete the proof of  Theorem B of the introduction.  Namely,  we provide many explicit constructions of the  irreducible  modules of (quiver) Hecke algebras   (in terms of cellular bilinear forms) over arbitrary fields.

\begin{defn}\label{uglovv1}
 Fix $\theta\in\weight $. Given $\lambda \in \mptn \ell n$ and $i \in \ZZ/e\ZZ$, we define the $i$-sequence of $\lambda$ to be the sequence of addable and removable nodes (recorded by $A$ and $R$ respectively) in increasing order  with respect to $\lhd_\theta$.   
We define the reduced $i$-sequence to be the sequence of the form 
$ R,R,\dots,R, A,A,\dots,A$    obtained from the above by repeatedly removing all pairs of the form $(A,R)$.    
We say that the  removable $i$-node of $\lambda$ is $\theta$-{\sf good} if it corresponds to the rightmost   $R$ in the reduced  $i$-sequence.   
\end{defn}

\begin{defn}\label{uglovv}
Given a fixed  $ \theta \in\weight $, the set of Uglov $\ell$-partitions $\Theta\subseteq \mptn \ell n$ is defined recursively as follows.  
We have that $\varnothing\in \Theta$.  
For   $\la\in \mptn \ell n$, we have that $\la \in \Theta$ if and only if there   exists  $i\in \ZZ/e\ZZ$  and  a good  $i$-node $\upalpha \in {\rm Rem}_i(\la)$ such that $\lambda -\upalpha  \in \Theta$.  
\end{defn}

\begin{eg}
For asymptotic charges the  Uglov $\ell$-partitions defined above are better known as the   {\em Kleshchev} $\ell$-partitions. 
\end{eg}

   We now recall \cite[Main Theorem]{MR2350227}, modifying the statement slightly by  inputting the definition of a canonical basic set (\cref{uglovv2}) and by having  explicitly defined $\Theta$ in \cref{uglovv} using the ``crystal combinatorics" made explicit in \cref{uglovv1}.

\begin{thm}[{\cite[Main Theorem]{MR2350227}}]\label{uglox}
For each $\sigma \in \weight$, the algebra $H_n^{\mathbb Q} (\underline{s})$ has canonical basic set 
 $\Theta\subseteq \mptn \ell n$ with respect to the ordering $>_\sigma$.  
\end{thm}

   The following theorem extends \cref{uglox} to arbitrary fields and also gives an explicit construction of the
   irreducible modules labelled by  $\Theta\subseteq \mptn \ell n$ 
   (this is new even in the case of $\Bbbk=\mathbb{Q}$).  
 This 
 extends  Ariki--Mathas's results for {\em asymptotic} charges \cite{MR1864465,MR1750939} to arbitrary charges.  

\begin{thm}\label{simplllle}
Let $\Bbbk$ be an arbitrary field and let $\theta \in \weight$.  
 The  irreducible   $\H_n^\Bbbk (\underline{s})$-modules are constructed as follows
  $$  \Simple ^\Bbbk  _\theta (\la):=   \Specht^\Bbbk_\theta(\lambda)/ {\rm rad}^\Bbbk(\langle\ ,\ \rangle_\lambda)  $$  
 are indexed by the set 
 $  \Theta=\{\text{Uglov  $\ell$-partitions with respect to the pair 
  $\theta\in\weight $}\}.$ 
  \end{thm}

\begin{proof}
We first fix our field to be $\QQ$.  
In \cref{semisimpleeeeee,semisimpleeeeee2} we proved that   our cell-modules 
$\Specht_\sigma^{\mathbb Q}(\la)$ 
are obtained via specialisation from  the 
   irreducible  modules 
 of the semisimple Hecke algebra; moreover we showed that this preserved the labelling of these modules.  
 We saw in \cref{QUIVERTLdecompositionmatrix} that  the cellular structure of \cref{corollary} 
gives  rise to a unitriangular decomposition matrix (and hence a canonical basic set) with respect to  to the ordering $\rhd_\theta$ (which is a coarsening  of $>_\sigma$)  on $\mptn \ell n$.  
By \cref{uniqueneess,uglox} it follows that $  \Theta=\{\text{Uglov  $\ell$-partitions with respect to the pair 
  $\theta\in\weight $}\} $ for $\Bbbk=\mathbb{Q}$.   
 It remains to prove that the result extends to arbitrary fields, by reduction modulo $p$ (with respect to the $\ZZ$-lattices of \cref{corollary}).  

Let $\Bbbk$ be an arbitrary field.  From the above, we know that 
${\rm rad}^{\mathbb Q}(\langle\ ,\ \rangle_\lambda)= \Specht^{\mathbb Q}_{\theta} (\lambda)$ for any $\lambda\not \in \Theta$.  
We also know that the number of  irreducibles   of the Hecke algebra is independent of the characteristic of the field \cite{MR1750939} 
and that all  irreducibles   (regardless of the field) are obtained as quotients of these radicals by cellularity.   
Therefore, 
${\rm rad}^\Bbbk(\langle\ ,\ \rangle_\lambda)= \Specht^\Bbbk_{\theta} (\lambda)$ for any $\lambda\not \in \Theta$   
 by base change (as our bases are constructible  over $\ZZ$).  
 \end{proof}

 \begin{eg}\label{labelllllleres}
The irreducible modules for $\theta=(2;2,1)$ are labelled by $  \{ ((1^2),\varnothing), 
( \varnothing,(1^2)) \}$. 
\end{eg}
 \begin{eg}\label{labelllllleres2}
The  irreducible  modules for $\theta=(2;4,1)$ are labelled  by $\{( \varnothing,(2)), (\varnothing,(1^2)) \}$. 
\end{eg} 
To summarise:  the parameterisations of   irreducible  modules given by 
\cref{simplllle}  are precisely those of Ariki's categorification theorem. 
  Thus  \cref{corollary} provides the  integral cellular bases  ``predicted" by Ariki's categorification theorem.  
 \cref{simplllle} explicitly constructs these  irreducible  modules in terms of radicals of cellular bilinear forms for the first time.

\section{The many different   filtrations of projective modules}\label{proj}
Our many cellular bases allow us to obtain many different filtrations on any fixed projective $\H_n^\Bbbk(\underline s) $-module.
 The search for these   different filtrations was initiated by Geck--Rouquier \cite{MR1889346}. 

\begin{thm}
 Fix  $\Bbbk$ a field, ${\underline{s}}=(e;s_0,s_1,\dots,s_{\ell-1})\in \mathbb{N}_{>1}\times  (\ZZ/e\ZZ)^\ell$,   and 
let $ {\bf P}_{\underline{s}} $ be a fixed projective indecomposable $\H^\Bbbk_n({\underline{s}})$-module. 
For each and every integral lift, $\sigma\in \weight$, the  projective module $ {\bf P}_{\underline{s}} $ admits a filtration  
$$
0= M^{\sigma}_1  \subset
M^{\sigma}_2
\subset \dots
\subset M^{\sigma}_z= {\bf P}_{ \underline{s}}
$$
such that for each $1\leq r\leq z$, we have
$
M^{\sigma}_r/ M^{\sigma}_{r-1}
$ is isomorphic to some $
 \Specht _{\sigma}^\Bbbk(\mu^{(r)})$  for  $\mu^{(r)}\in \mptn \ell n$ 
  up to grading shift.    
We have  that $ \mu^{(r)}\lhd_{\sigma}\mu^{(r-1)}$ for $1\leq r \leq z$.  
 In particular, every projective module admits many different cell-filtrations (up to grading shift), one for each cellular structure in Theorem A, or equivalently, one for each quasi-hereditary cover  $\algebra^\Bbbk_n(\theta)$ of $\H_n^\Bbbk(\underline{s})$.  
 \end{thm}

\begin{proof}
Fix an arbitrary integral lift $\theta \in \weight$. 
For $\lambda \in \mptn \ell n$, let ${\bf P}_\theta(\lambda)$ denote    
the corresponding projective $\algebra^\Bbbk_n(\theta)$-module.
Then ${\bf P}_\theta(\lambda)$ admits a cell-filtration 
 (with respect to the cellular structure of \cref{cellularitybreedscontempt}) by standard facts concerning quasi-hereditary algebras (and \cref{qaausfuduifddfheredtiary}).
Therefore $  {\sf E}^\theta_\omega{\bf P}_\theta(\lambda)$ is 
 an indecomposable $\H_n^\Bbbk(\underline s) $-module with a 
filtration by $\Specht^\Bbbk_\theta(\mu)$ such that $\mu \trianglerighteq_\theta \lambda$.   
 Given any integral lift  of our $e$-charge, a full set of   projective $H_n^\Bbbk(\underline{s})$-modules are 
given by $\{{\sf E}^\theta_\omega{\bf P}_\theta(\lambda) \mid \lambda \in \Theta		\}$ and so the result follows.  
\end{proof}

\begin{figure}[ht!]
$$\scalefont{0.9}
\begin{tikzpicture} 
\draw(0,0)--(1,1);
\draw(0,0)--(-1,1);
\draw(0,2)--(1,1);
\draw(0,2)--(-1,1);
\fill[white] (0,0) circle (9pt); 
\fill[white](1,1) circle (9pt); 
\fill[white](-1,1) circle (9pt); 
\fill[white](0,2) circle (9pt); 

\draw (0,0) node {$L^\Bbbk(1,0)$};
\draw(1,1) node {$L^\Bbbk(1,0)$};
\draw(-1,1) node {$L^\Bbbk(0,1)$};
\draw(0,2) node {$L^\Bbbk(1,0)$};

\end{tikzpicture}
\qquad
\begin{tikzpicture} 
\draw(0,0)--(1,1);
\draw(0,0)--(-1,1);
\draw(0,2)--(1,1);
\draw(0,2)--(-1,1);
\fill[white] (0,0) circle (9pt); 
\fill[white](1,1) circle (9pt); 
\fill[white](-1,1) circle (9pt); 
\fill[white](0,2) circle (9pt); 

\draw[rounded corners] (-.35, 0.3)--(.65, 1.3)--(2, 1.3)--(1.35, 0.7)--(.35, -0.3)-- (-1, -0.3)-- (-.35, 0.3)  ;

\draw(-2.75,0) node  {$\Specht_{0,3} ( (1),(1))$};
\draw[rounded corners] (-1.65, -0.3+1) rectangle(-0.35,0.3+1)  ; \draw(-2.75,1) node  {$\Specht_{0,3}(\varnothing,(2))$};
\draw[rounded corners] (-1.65, -0.3+1+1) rectangle(1.65,0.3+1+1)  ; \draw(-2.75,1+1) node  {$\Specht_{0,3}
(  \varnothing,(1^2)) $};

\draw (0,0) node {$L^\Bbbk(1,0)$};
\draw(1,1) node {$L^\Bbbk(1,0)$};
\draw(-1,1) node {$L^\Bbbk(0,1)$};
\draw(0,2) node {$L^\Bbbk(1,0)$};

\end{tikzpicture}
\qquad
\begin{tikzpicture} 
\draw(0,0)--(1,1);
\draw(0,0)--(-1,1);
\draw(0,2)--(1,1);
\draw(0,2)--(-1,1);
\fill[white] (0,0) circle (9pt); 
\fill[white](1,1) circle (9pt); 
\fill[white](-1,1) circle (9pt); 
\fill[white](0,2) circle (9pt); 

\draw[rounded corners] (-1.65, -0.3) rectangle(1.65,0.3)  ; \draw(-2.75,0) node  {$\Specht_{0,1}((2),\varnothing)$};
\draw[rounded corners] (-1.65, -0.3+1) rectangle(1.65,0.3+1)  ; \draw(-2.75,1) node  {$\Specht_{0,1}((1),(1))$};
\draw[rounded corners] (-1.65, -0.3+1+1) rectangle(1.65,0.3+1+1)  ; \draw(-2.75,1+1) node  {$\Specht_{0,1}( \varnothing,(1^2))$};

\draw (0,0) node {$L^\Bbbk(1,0)$};
\draw(1,1) node {$L^\Bbbk(1,0)$};
\draw(-1,1) node {$L^\Bbbk(0,1)$};
\draw(0,2) node {$L^\Bbbk(1,0)$};

\end{tikzpicture}
$$ 
\caption{The projective cover of the  irreducible  module $L^\Bbbk(1,0)$ and its 2 distinct cell-filtrations. 
As both $n$ and $\ell$  increase, we obtain many more distinct filtrations on each projective module.}
\label{distinctfiltration}
\end{figure}

 \begin{eg}\label{projfilterer}
The algebra $\H_2^\Bbbk(2;0,1)$ has two indecomposable projective modules.  We picture the full submodule  structure of the projective $P^\Bbbk(1,0)$ in    \cref{distinctfiltration}.  We also picture the two distinct cell-filtrations of this module for  $\theta=(2;4,1)$ and $\theta=(2;2,1)$.

 \end{eg}

\section{The restriction of a cell module for the quiver Hecke algebra}
\label{brancher}

For every charge $\theta\in\weight$, we prove that  the (graded) restriction of cell-module  (down the tower of Hecke algebras) has a cell-filtration. We thus complete our program of generalising all the results from \cite{bkw11} to arbitrary charges.  
 This result is to be expected, given the  2-categorical origins of our $\ZZ$-bases \cite{MR3732238} (where  $\theta$-diagrams arise in  categorifying quantum knot variants).  
 This result provides the key ingredient to the   construction of resolutions of unitary modules for Cherednik algebras and 
  algebraic varieties in 
 \cite{BNS}.    

\begin{thm}\label{hfsaklhsalhskafhjksdlhjsadahlfdshjksadflhafskhsfajk}Let $\Bbbk$ be a   integral domain.
Let  $\lambda \in \mptn \ell n$ and let $\upalpha_1\rhd_\theta  \upalpha_2 \rhd_\theta \dots \rhd_\theta \upalpha_z$ denote the removable boxes of $\lambda$, totally ordered according to the $\theta$-dominance ordering.  Then
the restriction of a cell-module has  an $\H_{n-1}^\Bbbk(\underline{s})$-module filtration 
\begin{align}\label{restriction}
0= \Specht^{z+1,\lambda}_\theta \subset 
\Specht^{z,\lambda}_\theta
\subset \dots 
\subset \Specht^{1,\lambda}_\theta=\Res_{\H _{n-1}^\Bbbk(\underline{s})}( \Specht^\Bbbk_{\theta} (\la))
\end{align}such that, for each $1\leq r\leq z$, we have that  
\begin{align}\label{restriction2}
\Specht_{\theta}^\Bbbk (\la-\upalpha_r)\langle 	 \deg (\upalpha_r)	\rangle  
\cong
\Specht^{r,\lambda}_\theta / \Specht^{r+1,\lambda}_\theta .
\end{align} 
\end{thm}

\begin{proof}

For $1\leq r \leq z$, we  define 
$$\Specht^{r,\lambda}_\theta
=
R\{ \Cell_{\stu\SSTT^\lambda} \mid \stu\in \Std_{\theta}(\lambda) \text{ and } \Shape( \stu{\downarrow}_{\{1,\dots, n-1\}})
\trianglerighteq\lambda-\upalpha_r \} .
$$
On the level of graded $\Bbbk$-modules, the chain of inclusions in \cref{restriction} is clear.  
For $\stt \in \Std_{\theta}(\lambda-\upalpha_r)$, we define 
$\varphi_r(\stt) \in \Std_{\theta}(\lambda)$ to be the tableau obtained
by adding the box $\upalpha_r$ with entry $n$   to the tableau
$\stt  \in \Std_{\theta}(\lambda-\upalpha_r)$. 
Abusing notation, we define 
 $$
\varphi_r(\Cell_{ \stt\SSTT^{(\lambda-\upalpha_r)} })=  
\cell^\theta_{\varphi_r(\stt) \SSTT^\lambda }.
$$
 We assume that $\upalpha_r$ is a box of residue $i\in  \ZZ/e\ZZ$.  
It is clear that $\varphi_r$ provides the required graded $\Bbbk$-module isomorphism of  \cref{restriction2}.  
It remains to verify that the chain of inclusions and the resulting isomorphisms hold on the level of $\H_{n-1}^\Bbbk(\underline{s})$-modules.  
We shall prove this by downward induction on the ordering on the removable nodes of $\lambda$.
Let $\rho ,\tau \in \con \ell {n }$   
and suppose that $\rho \setminus (\rho \cap \tau )= \square_\rho :=(r_\rho ,c_\rho ,m_\rho ) $
and 
$\tau \setminus (\rho \cap \tau )= \square_\rho :=(r_\tau ,c_\tau ,m_\tau ) $. 
 Let $\gamma\in \con \ell {n}$ and 
Given 
$\SSTS\in \SStd _{\theta}(\rho , \tau )$,  
 $\stt\in \Std _{\theta}(\gamma)$  
we define 
$$
\overline{\SSTS}_{\rho }^{\tau } (r,c,m)
=
\begin{cases}
{\bf I}^\sigma_{ (r,c,m)} &\text{if  }(r,c,m) \in \rho \cap \tau  \\
{\bf I}^\sigma_{(r_\tau ,c_\tau ,m_\tau ) }	&\text{if  }(r,c,m) = \square_\rho  
\end{cases}
\quad
\overline{\stt}(r,c,m)
=
\begin{cases}
\stt(r,c,m) &\text{if  }(r,c,m) \in \nu \\
n	&\text{if  }(r,c,m) = \square_\rho 
\end{cases}
$$
    We have that 
   \begin{align}\label{423570937045905374928}
\Cell_{\varphi_r(\stt)}=    {\Cell} _{ \overline\stt  }
 \times  \Cell_{\SSTT^{\lambda-\upalpha_r+(n,1,\ell)}_\lambda} 
\end{align}   for  $\stt \in \Std_{\theta}(\lambda-\upalpha_r)$. 
For $a \in \mathscr{H}^\Bbbk_{n-1} (\theta)$ and $\stt \in \Std_{\theta}(\lambda-\upalpha_r)$, it follows from \cref{corollary} that 
\begin{align} \label{some coefffs}
a 
\Cell_{ \stt }  
=
\sum_
{
\begin{subarray}c
\nu \trianglerighteq\lambda-\upalpha_r \\
\sts \in \Std_{\theta}(\nu)
\\
\SSTS \in \SStd_{\theta}(\nu, \lambda-\upalpha_r)
\end{subarray}
}
k_{\sts \SSTS}\Cell_{\sts \SSTS}
\end{align}
for some $k_{\sts \SSTS} \in \Bbbk$.  By \cref{423570937045905374928}, we have that 
\begin{align}\label{sfadhjklafsdhljkafdshjlkdafshjkldsaflkhjfdas}
\varphi_r(a 
\Cell_{ \stt  } )
=
\sum_
{
\begin{subarray}c
\nu \trianglerighteq\lambda-\upalpha_r \\
\sts \in \Std_{\theta}(\nu)
\\
\SSTS \in \SStd_{\theta}(\nu, \lambda-\upalpha_r)
\end{subarray}
}
k_{\sts \SSTS} {\Cell}_{\overline\sts \SSTS}\Cell_{\SSTT^{\nu+(n,1,\ell)}_{\la}} 
\end{align}
 and so it will suffice to show that if $k_{\sts\SSTS}\neq0$ and $\nu \neq \lambda-\upalpha_r$, then
\begin{align}\label{moduloshit}
   \Cell_{\overline\sts}   \Cell^\ast _{\SSTS}   \Cell _{\SSTT_{\nu+(n,1,\ell)}^{\la}}  \in 
\Specht^{r+1,\lambda}_\theta 
\qquad \text{\rm mod }  \algebra ^{\rhd \lambda}_n(\theta) .
\end{align}
 By \cref{aprop}, a necessary condition for
$ {   \Cell}_ { \SSTS}^\ast    \Cell _{\SSTT^{\nu+(n,1,\ell)}_{\la}} \not \in \algebra ^{\rhd \lambda}_n(\theta) $ is that 
 $ 
\lambda 
\trianglerighteq 
\nu + (n,1,\ell) 
.$    
 On the other hand, we have that $\nu \vartriangleright  \lambda-\upalpha_r$ by \cref{some coefffs}. 
Therefore, we need only consider terms in the sum  \ref{sfadhjklafsdhljkafdshjlkdafshjkldsaflkhjfdas} labelled by  
$\nu \in \mptn \ell n$ such that 
both 
\begin{align}\label{oiuyqwyuqwoiutrqoiuyrteiuyoertqoiuyertw} 
\nu \vartriangleright  \lambda-\upalpha_r \text{ and }\lambda 
\trianglerighteq
\nu + (n,1,\ell). 
\end{align}  
  \cref{oiuyqwyuqwoiutrqoiuyrteiuyoertqoiuyertw}  implies that  
   $\lambda$ and $\nu$   only  
    differ by moving   some number (possibly zero) of  boxes  of residue  $\res(\upalpha_r)=i \in \ZZ/e\ZZ$.  
 Again by \cref{oiuyqwyuqwoiutrqoiuyrteiuyoertqoiuyertw}, this implies that 
  $\nu$ is obtained from $\lambda-\upalpha_r$   by  removing  a  non-zero (since $\nu=\la-\upalpha_r $) set of $i$-boxes 
 \begin{equation}\label{afhljkhjkladfls}
\mathcal{R}=\{\upalpha_{i_1}\rhd_\theta \upalpha_{i_2} \rhd_\theta  \dots \rhd_\theta  \upalpha_{i_S}
 \mid  {\upalpha_{r}}\vartriangleright_\theta 
  {\upalpha_{i_s}} \text{ for }1\leq s \leq S  \} \subset   {\rm Rem}_{i}(\lambda-\upalpha_r)
  \end{equation}
   and  adding a set of  $i$-boxes
 \begin{equation}\label{afhljkhjkladfls}
\mathcal{A}=\{\upalpha_{j_1} \rhd_\theta  \upalpha_{j_2} \rhd_\theta  \dots \rhd_\theta  \upalpha_{j_S}\mid 
  {\upalpha_{r}}\trianglerighteq_\theta 
 {\upalpha_{j_s}}\rhd_\theta  {\upalpha_{i_s}} \text{ for }1\leq s \leq S\} \subseteq   {\rm Add}_{i}(\lambda-\upalpha_r) 
 \end{equation}
such that $\mathcal{R}\neq \mathcal{A}$. We let $N$ denote the set of all $\nu \in \mptn \ell {n-1} \setminus\{\lambda-\upalpha_r\}$ which can be obtained from $\lambda-\upalpha_r$ in this fashion.  
Putting all this together   
it will suffices to show that if 
$k_{\sts\SSTS}\neq 0$ and $\SSTS \in \SStd_\theta(\nu,\la-\upalpha_r)$ then  
 in \ref{sfadhjklafsdhljkafdshjlkdafshjkldsaflkhjfdas}  and  $\nu\in N$,
  then \begin{equation}\label{diagram} 
 {   \Cell}_{  \SSTS}^\ast \Cell_{\SSTT^{\nu+(n,1,\ell)}_{\la}}      = 
 \Cell_{\SSTT_{\nu+(n,1,\ell)}^{\la}}  {\Cell}_{  \SSTS}   \in \algebra^{\vartriangleright \lambda } \end{equation}where we have applied the involution $\ast$ to simplify notation.  
By \cref{mastumoto}, we have that 
$$
 \Cell_{\SSTT_{\nu+(n,1,\ell)}^{\la}}  {\Cell}_{  \SSTS}   =
 {\Cell}_{  \SSTS}   \Cell_{\SSTT_{\nu+(n,1,\ell)}^{\nu+\upalpha_r}}  
 +\sum_{\Diag''  } \Diag'' f_{\Diag''}(y) \quad \text{ for some }
 \Diag' \rhd_\sigma {\Cell}_{  \SSTS}   \Cell_{\SSTT_{\nu+(n,1,\ell)}^{\nu+\upalpha_r}} \text { and }f_{\Diag''}(y) \in \mathcal{Y}_{\nu+(n,1,\ell)} .
$$
All terms on the righthand-side factor through the idempotent labelled by $\nu+\upalpha_r$ which is strictly more dominant than $\lambda$ by \cref{oiuyqwyuqwoiutrqoiuyrteiuyoertqoiuyertw} and the result follows.     
  \end{proof}

    \section{The generalised blob algebras and beyond}\label{sec2}

   In the case of the symmetric groups  
  modular representation theorists have long focussed on the subcategory  of   representations  labelled by partitions with at
  most $h$ columns for some $h\in \ZZ_{>0}$ over a field, $\Bbbk$,  of characteristic (possibly much) greater than $h$.  
  This subcategory is highest weight and  far more amenable to study via the tools of Kazhdan--Lusztig theory  \cite{MR1272539,geordie} (in terms of  the alcove geometry of type $A_{h-1} \subseteq \widehat{A}_{h -1}$).  
   However, there is no obvious analogous subcategory/quotient algebra of $\H_n^\Bbbk(\underline s) $ in higher levels; hence    almost nothing is   known or even conjectured about   such Hecke algebras in positive characteristic.  
   The purpose of this section is to introduce a candidate for such a quotient algebra and prove Martin--Woodcock's conjecture.  
   The results of this section have been used by Libedinsky--Plaza as the basis of a modular analogue of Martin--Woodock's conjecture \cite{blobby}.  Cox, Hazi and the author have subsequently proven this conjecture in \cite{cell4us2} (using the ideas  from this section).

\subsection{ The {cylindric} charge}  \label{quiverTL defintion}
Given $h\in \NN$  we  assume that  our charge  satisfies 
 \label{quiverTL defintion}
  $h <  |\sigma_i - \sigma_{j}| <e-h$ for $0\leq i<j <\ell$.  
    Let $\mptn \ell n (h)\subseteq \mptn \ell n$ denote the saturated subset consisting of all $\ell$-partitions with at most $h$ columns in any given component, that is,
 $$\mptn \ell n (h)= \{\la=(\la^{(0)}, \dots , \la^{(\ell-1)}) \mid \lambda^{(m)}_1\leq h \text{ for all }0\leq m<\ell\}.$$
 For such a $\theta\in \weight$ our candidate quotient algebra is as follows: 
$$ \mathcal{Q}_{\ell,h, n}(\theta)= \H_n^\Bbbk(\underline s) /
\langle A^\sigma_{\sts\stt}
\mid \sts,\stt \in \Std_\sigma(\la),  \la \not \in \mptn \ell n (h)
 \rangle.  
$$ We will show in future work \cite{CoxBowman,cell4us2} that the category  $\mathcal{Q}_{\ell,h, n}(\theta)$-${\rm mod}$  is 
incredibly rich and yet far more tractable than the category  $\H_n^\Bbbk(\underline{s})$-${\rm mod}$: Under the restriction that $e>(h+1)\ell$,  we shall cast representation theoretic questions in terms of an alcove geometry of type 
$$
A_{h-1}\times A_{h-1}\dots \times A_{h-1}\subseteq \widehat{A}_{\ell h-1}.
$$
 In this section, we  prove that  $\mathcal{Q}_{\ell,h, n}(\theta)$-$\rm mod$ is a highest-weight category   over arbitrary field; thus generalises results on symmetric groups from \cite[Theorem 4.4]{MR1611011} and results on the blob  algebras of statistical mechanics \cite{MS94,MW03}.  We shall then prove Martin--Woodcock's conjecture.

 \subsection{The representation theory of the algebras  $\mathcal{Q}_{\ell,h, n}(\theta)$}
 With our definitions in place, we are now ready to prove that these algebras are quasi-hereditary and provide 
 presentations of these algebras solely in terms of the classical KLR generators.  
  
\begin{thm}\label{QUIVERTL}
 For  $\theta\in \weight$ such that $h<|\sigma_i-\sigma_j |<e-h$ for $0 \leq i <j  < \ell$, the  algebra $\mathcal{Q}_{\ell,h, n}(\theta)$ has a presentation  solely in terms of  the classical KLR generators as follows, 
\begin{equation}\label{theoremequation}
\mathcal{Q}_{\ell,h, n}(\theta)
=
\H_n^\Bbbk(\underline s) 		/
\langle e(\imath) 
\mid 
\imath\in (\ZZ/e\ZZ)^{h+1} \text{ and }
 \imath_{k+1} = \imath_{k}+1 \text{ for } 1\leq k \leq h 
\rangle.  
\end{equation}
Over a field, the algebra 
$\mathcal{Q}_{\ell,h, n}(\theta)$  
 is   quasi-hereditary  with  irreducible  modules indexed by $\mptn \ell n (h)$.      

 \end{thm}
\begin{proof}
   Consider an idempotent $e(\imath) $ of the   form stated in \cref{theoremequation}.  We pull the right most strand in $e(\imath) $ rightwards using the non-interacting relations.  
 If  $\imath_1\neq \theta_m$ for some $1\leq m \leq \ell$ and 
   we pull this strand rightwards until it is $>  n$ units right  of the red strand $\theta_0$ and the diagram is zero by relation \ref{rel15}.
 Otherwise     $\imath_1=s_m$ for some $1\leq m \leq \ell$, and  this process terminates when    the solid
   $\theta_m$-strand comes to rest upon  reaching the vertical line with $x$-coordinate $(1,1,m)$.  
     In other words, once it is 
  ever-so-slightly to the  right  of  the red $s_m$-strand with $x$-coordinate $\theta_m -m/\ell  $.   
 By our assumptions on $\theta \in \weight$,    we can pull  the solid
   $(s_m+1)$-strand rightwards until it 
  reaches the vertical line with $x$-coordinate $(1,2,m)$. 
    We then repeat this process until we have moved   the right most $(h+1)$ solid strands as far right  as possible.  
   We let  $\la=(\varnothing,\dots, \varnothing, (h+1), \varnothing,  \dots , \varnothing, (0^{h+1},1^{n-1-h}))$ where the $m$th and $\ell$th are the only non-empty components of $\la$.    
   The  diagram produced by the process above is equal to ${\cell}^\theta_{\stt\stt}$ where  $\varphi(\stt)=\SSTT$ is the tableau of shape $\la$ and weight $\omega$ which takes
   $\SSTT(r,1,m)={\bf I}^{\theta}_{(r,1,\ell)}$ for $1\leq r \leq h+1$ and $\SSTT(r,1,\ell)={\bf I}^{\theta}_{(r,1,\ell)}$ for $h+1<r \leq n$; therefore    $e(\imath)\in \{\cell_{\sts\stt}^\theta\mid  \sts,\stt\in\Std_\theta(\la),   \la \not\in \mptn \ell n(h)\}$.  
 
 Now for the reverse containment.   By definition,  there is no tableau $\SSTT \in \Std_\sigma(\la)$ for $\la \in  \mptn \ell n(h)$ with  residue sequence   $(s_m,s_m+1,\dots,s_m+h) $ and so $e(s_m,s_m+1,\dots,s_m+h) $ annihilates all these cell-modules.  
Hence the ideal by which we quotient in \cref{theoremequation} has   dimension less than or equal to $\sum_{\la \in \mptn \ell n \setminus \mptn \ell n(h)}|\Std_\sigma(\la)|^2$ and the reverse containment holds.  
   Finally, the $\la$th cell layer contains an idempotent $e_{\SSTT_\la}$ for   $ \mptn \ell n(h)$ and so the algebra is quasi-hereditary, as required.   
    \end{proof}
  
\begin{cor}
The algebra $\mathcal{Q}_{\ell,1, n}(\theta)$ is isomorphic to the generalised blob algebra of \cite{MW03}.  
\end{cor}
\begin{proof}
 We  have seen that  $\mathcal{Q}_{\ell,1, n}(\theta)$  is the quotient of $\H_n^\Bbbk(\underline s) $   by the two-sided ideal generated by 
  $\sum_{0\leq m < \ell}e(s_m,s_m+1)$.  Each  idempotent  $e(s_m,s_m+1)$ in this sum spans 
  a 1-dimensional  irreducible  $\H_{2}^\Bbbk(\underline{s}) $-module labelled by the $\ell$-partition $(\varnothing, \dots ,\varnothing,(2), \varnothing, \dots,\varnothing)$.  
  The result follows.  
\end{proof}

The grading on  tableaux  and the corresponding  graded cellular bases for the blob algebra  were first constructed  in the papers of 
Ryom-Hansen and Ryom-Hansen--Plaza \cite{RH12,PRH14}.

\begin{cor}
The algebra $\mathcal{Q}_{1,h, n}(\theta)$ is isomorphic to the generalised Temperley--Lieb algebra  of \cite{MR1680384,MR1611011}
and is Morita equivalent to the Ringel dual of the $q$-Schur algebra of ${\rm GL}_h$ (acting on   $n$-fold $q$-tensor space).  
\end{cor}
\begin{proof}
 We  have seen that  $\mathcal{Q}_{1,h, n}(\theta)$  is the quotient of $\H_n^\Bbbk(\underline s) $   by the two-sided ideal generated by 
  $ e(s_m,s_m+1, \dots, s_m+h )$ which labels the trivial representation of  $\H_{h+1}^\Bbbk(e;0)$.   
  The result follows by \cite[Theorem 4]{MR1680384}.   
\end{proof}


\begin{thm}[Martin Woodcock's conjecture \cite{MW03}]Let $\theta\in \weight$ be such that $1< |\sigma_i - \sigma_{j}|< e-1$ for $0\leq i<j \leq \ell$.  
 The graded decomposition matrix of $\mathcal{Q}_{\ell,1, n}(\theta)$  appears as a square submatrix of  that  of 
$\mathcal{H}^\Bbbk_n (\underline{s})$ with respect to the $\sigma$-cell structure.  
We have that
 $$
\sum_{k\in\ZZ}[\Specht_{\sigma}^{\mathbb Q}(\la):  \Simple_\sigma^{\mathbb Q} (\mu) \langle k \rangle ]=   n_{\lambda\mu}(t)
 $$ 
 for $\lambda, \mu \in \mptn \ell n (1)$  
where $n_{\lambda\mu}(t)$ is equal to a  non-parabolic  affine  Kazhdan--Lusztig polynomial of type $\widehat{A}_{\ell-1}$.  
 The   action of the affine Weyl group on   $\mptn \ell n (1)$   is given as in \cite[Section 3]{bcs15}.  
\end{thm}

\begin{proof}
The square shape of the decomposition matrix follows from quasi-heredity   \cref{QUIVERTL}.
The algebra   $\mathcal{Q}_{\ell,1, n}(\theta)$ is the quotient of $\H_n^\Bbbk(\underline s) $ 
 by the cell-ideal labelled by multipartitions with more than 1 column  in some component.  Thus the decomposition  matrix  of  $\mathcal{Q}_{\ell,1, n}(\theta)$ appears as the submatrix of that of  $\H_n^\Bbbk(\underline s) $ labelled by pairs $\lambda, \mu \in \mptn \ell n (1)$. 
 By \cref{QUIVERTLdecompositionmatrix},  the decomposition matrix 
 of  $\H_n^\Bbbk(\underline s) $ appears as a submatrix of that of  $\algebra^\Bbbk_n(\theta )$.  
 The  entries of this submatrix of the decomposition matrix of $\algebra^\Bbbk_n(\theta )$   
 were shown to be equal to $n_{\lambda\mu}(t)$  in \cite[Theorem 3.16]{bcs15}.  
\end{proof}

\appendix
\section{The many versions of this paper}
This paper has gone through many arXiv iterations, through which we have developed the combinatorics and diagrammatics.  The main results throughout  versions 1 to 4 were Theorems  A and C; we added   Theorem B  from  version 5 onwards.  
 Throughout versions 1 to 5, the diagrammatics and combinatorics remained very similar: we followed Webster's conventions from \cite{MR3732238} where ghost strands are drawn on the left and we encoded the ``charged" information   via a  weighting $\vartheta\in \RR^\ell$ and a separate $e$-charge in $(\ZZ/e\ZZ)^\ell$.   
  
 The most significant changes in the presentation of this work came in version 6 of the arXiv paper.  The  central  idea  was to dispense  with Webster's weighting  $\vartheta\in \RR^\ell$ in favour of the integral lifts (of $e$-charges in $(\ZZ/e\ZZ)^\ell$ to charges in $ \ZZ ^\ell$) used here, 
this allows us to transfer between the combinatorics of Webster, Lusztig, and Uglov seamlessly.   
In order to match-up this combinatorics with the diagrammatics, we had to   reflect   the earlier diagrams through the vertical axis (drawing ghost strands to the right) and hence obtained the diagrams which we work with in this paper (which is now available as version 7 on the arXiv).  This final big change in the diagrammatics and combinatorics 
 was inspired by the desire for an elementary proof of Theorem B and prompted by conversations with Nicolas Jacon and Maria Chlouveraki.  
In version 6 we introduced the idea of discretisation and we strengthened many of our intermediary results, hence 
simplifying  the presentation of \cref{span} --- we were inspired by (and mimicked) the presentation of similar material from  \cite{bkw11} (for well-separated charges).

In version 6 of the paper,  we chose to restrict to the reduced diagrams as our generating set for the algebra.  
 This was in order to highlight  the fact that the algebra is finitely generated, but this came  at the cost of having to prove associativity (which then follows from \cref{step1}).  In this paper, we have instead chosen to include all diagrams in the generating set and then deduce that the algebra is finitely generated as a corollary of \cref{step1}.  
 
  In this final version,  we also add a new Section 8 in which we 
 introduce the algebras $\Balgebra^\Bbbk_n(\sigma,(q:Q_0,\dots,Q_{\ell-1})$ 
 in order to explicitly match-up the cell-modules with   irreducible modules in the semisimple case.

\begin{Acknowledgements*}
This paper owes a special thanks to     Maria Chlouveraki, Joe Chuang,  Jun Hu,  Nicolas Jacon, Sinead Lyle,   Andrew Mathas,   Liron Speyer,   and  Ben Webster
   for   formative and interesting  conversations,  and for teaching me a great deal about cyclotomic Hecke and Cherednik algebras.   
 The anonymous referee went above and beyond with their detailed 
 mathematical feedback and suggestions for important improvements to      earlier versions of this paper;  I am thankful for their patience and hard work.  
 This research was funded by EPSRC fellowship grant EP/V00090X/1.
  \end{Acknowledgements*}


\end{document}